\documentclass[10pt]{article}
\usepackage{amsmath,color,bm, float}
\usepackage{amsfonts, xcolor}
\usepackage{amsthm,amssymb,mathabx,relsize}
\definecolor{dbl}{rgb}{0.1,0.0,0.97}
\usepackage[authoryear]{natbib}
\usepackage{graphics,graphicx}
\numberwithin{equation}{section} 
\usepackage{caption}
\usepackage{relsize}
\definecolor{gb}{rgb}{0.80, 0.80, 0.80}
\usepackage{sidecap}  

\usepackage{etoc} 

\usepackage{enumitem}

\usepackage{multirow}
\usepackage{authblk} 
\usepackage{tabularx,ragged2e,multirow}
\newcolumntype{C}{>{\Centering}X} 
\usepackage{mathtools}

\usepackage{pgfplots}
\pgfplotsset{compat=newest,compat/show suggested version=false}
\usepackage{adjustbox}
\usepackage{subfig}
\usepackage{subcaption}  
\usetikzlibrary{positioning}

\usepackage{enumitem} 
\newcommand{\appref}[1]{\hyperref[#1]{Appendix~\ref*{#1}}}

\usepackage{graphicx}
\usepackage{capt-of}
\usepackage{booktabs}
\usepackage{tikz}

\usepackage{float}

\newcommand{\twostar}[3]{%
\begin{tikzpicture}[scale=0.35, baseline=-0.5ex]
  \node[circle, fill, inner sep=1.2pt,
        label=above:{\scriptsize $#1$}] (a) at (0,0.9) {};
  \node[circle, fill, inner sep=1.2pt,
        label=below:{\scriptsize $#2$}] (b) at (-0.8,-0.4) {};
  \node[circle, fill, inner sep=1.2pt,
        label=below:{\scriptsize $#3$}] (c) at (0.8,-0.4) {};
  \draw (a)--(b);
  \draw (a)--(c);
\end{tikzpicture}
}

\newcommand{\trianglegraph}[3]{%
\begin{tikzpicture}[scale=0.35, baseline=-0.5ex]
  \node[circle, fill, inner sep=1.2pt,
        label=above:{\scriptsize $#1$}] (a) at (0,0.9) {};
  \node[circle, fill, inner sep=1.2pt,
        label=below:{\scriptsize $#2$}] (b) at (-0.8,-0.4) {};
  \node[circle, fill, inner sep=1.2pt,
        label=below:{\scriptsize $#3$}] (c) at (0.8,-0.4) {};
  \draw (a)--(b)--(c)--(a);
\end{tikzpicture}
}

\usepackage{silence}
\usepackage{ctable}

\usepackage[pdftex,hypertexnames=false,linktocpage=true,colorlinks]{hyperref}
\hypersetup{colorlinks=true,linkcolor=blue,anchorcolor=blue,citecolor=blue,filecolor=blue,urlcolor=blue,bookmarksnumbered=true}\RequirePackage{hypernat}

\allowdisplaybreaks[3]

\theoremstyle{plain}
\newtheorem{theorem}{Theorem}[section]
\newtheorem*{theorem*}{Theorem}
\newtheorem{lemma}[theorem]{Lemma}
\newtheorem{cor}[theorem]{Corollary}
\newtheorem{prop}[theorem]{Proposition}

\theoremstyle{definition}

\theoremstyle{remark}
\newtheorem{rem}{{\bf Remark}}

\let\oldthebibliography\thebibliography
\renewcommand{\thebibliography}[1]{%
  \oldthebibliography{#1}%
  \setlength{\itemsep}{0pt}%
}

\defcitealias{bhattacharya2022}{[BDM22]}

\newcommand{\pr}{\mathbf{P}}
\newcommand{\E}{\mathbf{E}}
\newcommand{\Var}{\mathbf{Var}}
\newcommand{\cov}{\mathbf{cov}}

\newcommand{\primet}{{\mathsmaller T}}

\newcommand{\kl}{\mathbb{K}}
\newcommand{\rl}{\mathbb{R}}

\newcommand{\Cb}{\mathbf{C}}

\newcommand{\Ub}{\mathbf{U}}

\newcommand{\Xb}{\mathbf{X}}
\newcommand{\Wb}{\mathbf{W}}

\newcommand{\Yb}{\mathbf{Y}}

\newcommand{\nsr}{{[N]}_R}

\newcommand{\PiB}{\boldsymbol{\Pi}}

\newcommand{\Dc}{\mathcal{D}}
\newcommand{\Gc}{\mathcal{G}}

\newcommand{\Mc}{\mathcal{M}}

\newcommand{\jbb}{\mathbf{j}}
\newcommand{\kbb}{\mathbf{k}}

\newcommand{\sbb}{\mathbf{s}}

\newcommand{\rb}{\mathbf{r}}

\newcommand{\ub}{\mathbf{u}}

\newcommand{\vb}{\mathbf{v}}

\newcommand{\xb}{\mathbf{x}}

\newcommand{\yb}{\mathbf{y}}

\newcommand{\rai}{\rightarrow\infty}

\newcommand{\parw}{\stackrel{P}{\rightarrow}}
\newcommand{\darw}{\stackrel{d}{\rightarrow}}
\newcommand{\eqd}{\overset{d}{=}}

\newcommand{\eps}{\epsilon}
\newcommand{\al}{\alpha}
\newcommand{\la}{\lambda}

\newcommand{\thb}{\boldsymbol{\theta}}

\newcommand{\lamb}{\boldsymbol{\lambda}}
\newcommand{\phib}{\boldsymbol{\phi}}

\newcommand{\norm}[2]{{\|{#1}\|}_{#2}}

\DeclareMathOperator*{\argmaxA}{argmax}

\begin{document}

\title{{\bf Statistical inference for subgraph counts and clustering coefficient using network sampling in a sparse Stochastic Block Model framework}}
\author{A.~Mandal\thanks{Supported by PhD scholarship from Indian Statistical Institute;
{\small{email:}} \texttt{anirban22r@isid.ac.in}}\ \ and\ \ A.~Chatterjee\thanks{
 {\small{email:}} \texttt{cha@isid.ac.in}} 
\\Statistical Sciences Unit\\Indian Statistical Institute, Delhi
}
\date{}
\maketitle

\begin{abstract}
This article develops limit laws for network sampling based estimates of subgraph counts and clustering coefficient of a large population network, and uses them for predictive inference. A model based approach is used, where the population network is assumed to be generated from a {\it sparse} Stochastic Block Model (SBM). In order to quantify the effects of node sampling under resource constraints, a {\it sparse} Bernoulli node sampling scheme is introduced, where the node selection probability is allowed to decay to zero as the population size increases. Both induced and ego-centric network formation approaches are explored. Quantitative bounds on the speed of normal approximation for estimated subgraph counts are obtained in a joint model and design based asymptotic framework, and these bounds show that the accuracy of statistical inference depends intricately on the level of sparsity in the model, the sparsity level of the sampling scheme, and on features like the {\it edge density} and {\it minimum vertex cover size} of the target subgraph. We find that the ego-centric approach can handle higher levels of sparsity in the model and the sampling scheme, compared to the induced approach. We also show that if the model sparsity level remains below an initial threshold, then the quality of statistical inference remains unaffected by the model sparsity level. Beyond this threshold, the quality degrades rapidly. The sufficient conditions for obtaining a Gaussian limit law also turn out to be {\it necessary}. For {\it strictly balanced} target subgraphs, we obtain sharp transitions from Gaussian to various types of Poisson based limit laws, as the model and sampling sparsity levels increase. A complete description of all possible limit laws for estimated subgraph counts is found in the induced case, and a near-complete description is obtained in the ego-centric case. As an application of our results on estimated subgraph counts, we obtain Gaussian and Poisson based limit laws for the estimated clustering coefficient, in both induced and ego-centric cases. A simulation study provides strong support for the theoretical results at various levels of model sparsity and sampling sparsity. Finally, the proposed methodology is applied to a real data set. 

\end{abstract}

\noindent%
\vspace{5mm}

\section{Introduction}
\label{sec:intro}

In recent years, network sampling has emerged as a practical and feasible method to observe data from large networks (\cite{newman2018}), especially in cases where the {\it entire} target network (population) is difficult to access  (cf. \cite{kolaczyk09}, \cite{zhang2015}) or the entire data set is too large to use for storage and analysis (\cite{leskovec06}). Applied researchers are interested in accurate estimation of various population network level summary statistics using the sampled network data. However, a large amount of empirical research indicates that estimates obtained through network sampling can be severely biased, and the bias depends on the structure of the population network, the sampling mechanism, and the target summary statistic (cf. \cite{lee2006}, \cite{bliss-14}, \cite{rib-12}, \cite{ebbes2008}, \cite{gonzal-bias-14}, \cite{maiya-bias-11}, \cite{cost-03}, \cite{agc-lewis-16}, \cite{illenberger-12}, \cite{ruggeri19}, \cite{ko-jecon24}, and the references therein for a selective overview). In spite of extensive empirical research on network sampling, there has been a lack of rigorous theoretical results on the distributional properties of network sampling based estimates and their use for statistical inference about their population level counterparts. Initial investigations on network sampling have focused on a design based approach (cf. \cite{frank1978sampling}, \cite{franksjs-78}, \cite{frank2011survey}) and model based approaches have been explored more recently (cf. \cite{thompson-frank-2000}, \cite{shi-2019}, \cite{handcock-gile10} and \cite{gile2015}). Recently, \cite{bhattacharya2022} made a seminal contribution in this area and proved a CLT for estimated subgraph counts under a Bernoulli node sampling scheme (using induced network formation) in a design based framework, assuming that the population network is generated {\it non-stochastically}. Similar inferential questions arise in a model based setting, where the population network is generated stochastically. The question is relevant due to the extensive use of stochastic models in network modeling (see \cite{newman2018}, \cite{kolaczyk09}, \cite{van2017random}, \cite{fienberg-12}, \cite{izenman2023}). To the best of our knowledge, this question has not been explored in the existing literature. Neither exist any rigorous limit law results for alternative network formation methods, like the ego-centric approach. In this paper, we investigate the properties of network sampling in a model based framework and focus on the prediction of two different types of population summary statistics: subgraph counts and the clustering coefficient. We explore the effects of both induced and ego-centric network formation. We allow for sparsity in the population network and also introduce a {\it sparse} Bernoulli node sampling scheme to study the effects of sampling under {\it resource constraints}. We obtain different Gaussian and Poisson based limit laws for estimated subgraph counts and clustering coefficient, allowing for {\it sparsity} in both the model and the Bernoulli node sampling mechanism.

Consider an undirected population network on $N$ nodes. We assume that the population network is generated from a $K$-class Stochastic Block Model (SBM) (cf. \cite{holland1983}), where $K(\geq 2)$ denotes the number of classes in the SBM. 
We focus on the SBM framework primarily due to its simplicity and flexibility and because it is widely used in applications (cf. \cite{lee2019review}, \cite{abbe2018community}, \cite{airoldi2013stochastic}).  
It is assumed that every node in the population network has an unique and fixed (non-stochastic) class label and the edge probability between any two nodes depends only on their respective class labels. The population network is allowed to be dense or sparse and the edge probabilities in the SBM are allowed to remain fixed or decay to zero, as the size of the population network $N$ increases. Let $\pi_{N,r,s}$ denote the edge probability between two nodes in the population network, one of them having class label $r$, and another having class label $s$, where $r,s\in \{1,\ldots,K\}$. We allow $\pi_{N,r,s}$ to depend on $N$, and we assume that there exists a {\it model sparsity level} $\beta\geq 0$, such that 
\begin{align}
N^\beta\pi_{N,r,s} \rightarrow c_{r,s},\quad\text{as $N\rai$, for some $c_{r,s}\in [0,\infty)$, for all $r,s\in \{1,\ldots,K\}$,}
\label{model-sp}
\end{align}
where at least one of the $c_{r,s}$ values is assumed to be positive. If $\beta = 0$, then are some $\pi_{N,r,s}$ values that do not decay to zero and the population network will be dense. If $\beta>0$, then $\pi_{N,r,s}\rightarrow 0$, for all $r,s\in\{1,\ldots,K\}$, as $N\rai$, and the population network becomes sparse. A higher value of $\beta$ implies higher levels of sparsity. A Bernoulli node sampling scheme is used for selecting nodes from this population network. The node sampling probability $p_N$ is allowed to depend on $N$ and we set 
\begin{align}
p_N = \frac{c}{N^{\alpha}},\quad\text{for some $\alpha\in [0,1]$, $c>0$, and for all $N\geq 1$.}
\label{samp-sp}
\end{align}
The quantity $\alpha$ is called the {\it sampling sparsity level}. If $\al = 0$, we call the sampling {\it dense}; if $\alpha\in(0,1)$, the sampling is called {\it intermediately sparse}; and if $\al = 1$, the sampling is called {\it extremely sparse}. If $\al = 0$, then $p_N$ remains fixed (and bounded away from $0$ and $1$) and the expected number of sampled nodes is of the same order as $N$. If $\alpha\in (0,1]$, then $p_N$ decays to zero as $N\rai$. In the intermediately sparse case, the expected number of sampled nodes increases with $N$, but at a much slower rate than $N$. In the extremely sparse case, the expected number of sampled nodes remains constant even as $N\rai$. To the best of our knowledge, this sort of {\it sparse node sampling} framework has not been explored in the existing literature. This framework is introduced to quantify the effects of sampling very few nodes, compared to population size $N$ (cf. \cite{lee2006}, \cite{xu2017}, \cite{luo-15}, \cite{kaan-19}), and how this affects statistical inference when the underlying population network is sparse or dense. This type of situation can arise even in cases where $N$ is not very large but there are cost, storage, or access constraints on data collection or in situations where $N$ is extremely large and the sample size cannot be scaled up at the same rate. 

Following the initial sampling of nodes, we explore two different approaches for creating the sampled network: (a) induced approach, where the edge status between two population nodes is observed only if both nodes are sampled, and (b) ego-centric approach, where the edge status is observed if at least one of the two nodes is sampled. Induced and ego-centric network formation schemes have been widely studied in the existing literature (see \cite{kolaczyk09}, \cite{bhattacharya2022}, \cite{handcock-gile10}, \cite{ko-jecon24}, \cite{agc-lewis-16}). The node sampling scheme and the SBM are assumed to be {\it independent}. In order to describe our main contributions, we provide some initial definitions.

Throughout this article, we focus on a fixed, simple, undirected, and connected target subgraph $H$, having $R(\geq 2)$ vertices and $T$ edges. For any $t\in\{2,\ldots,R\}$ we define $m(t:H)$ as the maximum number of edges in {\it any} $t$ vertex induced subgraph of $H$ (see \eqref{mtH-def}), and write $m(H) = \displaystyle\max_{2\leq t\leq R}\displaystyle m(t:H)/t$, as the {\it maximum edge density} of $H$ and its subgraphs (see \eqref{def-mH}). We define $m(1:H)=0$, for any choice of $H$. We say $H$ is {\it strictly balanced} if $m(t:H)/t$ is uniquely maximized at $t=R$ (see p. 64 of \cite{janson-rucinski-randomgraphs}). A collection of nodes $C$ of $H$ is called a {\it vertex cover of $H$}, if an endpoint of {\it any} edge in $H$ is contained in $C$. Let $\tau(H)$ denote the {\it size of the minimum vertex cover of $H$} (see \eqref{def-tauH}). Let $\kl_{1,R-1}$, $\kl_R$ and $\mathbb{L}_R$, respectively, denote a star graph, a complete graph and a line graph on $R$ vertices. 
For any two random variables $X$ and $Y$, $d_1(X,Y)$ denotes a distance measure between $X$ and $Y$ (cf. \eqref{d1-def}) based on the class of all bounded maps with bounded first derivative, such that convergence in the $d_1$ distance implies convergence in distribution. Let $S_N(H)$, $\widehat{S}_{N,1}(H)$ and $\widehat{S}_{N,2}(H)$ (see \eqref{s-all-main}) denote the subgraph count (number of copies) of $H$ in the population network, the estimate obtained from the network sample constructed by induced subgraph formation, and the estimate obtained from the network sample constructed by ego-centric subgraph formation, respectively. The centered and scaled sample based subgraph counts $T_{N,1}(H)$ (induced case) and $T_{N,2}(H)$ (ego-centric case) are defined as
\begin{equation}
\left. \begin{aligned}
T_{N,l}(H) & = \frac{\widehat{S}_{N,l}(H) - f_l(p_N:H)\cdot S_N(H)}{\sigma_{N,l}(H)},\quad\text{where,}\\
\sigma^2_{N,l}(H) & = \Var\left(\widehat{S}_{N,l}(H) - f_{l}(p_N:H)\cdot S_N(H)\right),
\end{aligned}\qquad\right\} \quad\text{for $l = 1,2$.}
\label{TN-def-12}
\end{equation}
The variance term $\sigma^2_{N,l}(H)$ in \eqref{TN-def-12} is found jointly with respect to the model (SBM) and the node sampling design, taking into account the stochasticity from both the population network model and the node sampling scheme. The terms $f_l(p_N:H)$, $l=1,2$, are defined in \eqref{fl-def}, and are used for correct centering and to account for the effects of sampling.

\subsection{Overview of Gaussian approximation results for estimated subgraph counts}
\label{sec:intro:gauss}

One of the primary contributions of this article is to provide quantitative bounds on the speed of normal approximation for $T_{N,l}(H)$, for $l=1,2$, for suitable choices of $H$, and describe sufficient conditions on the model and sampling sparsity levels $\beta$ and $\al$, which will ensure asymptotic normality. Let $Z\sim N(0,1)$. Suppose $\beta$ and $\al$ denote the underlying model and sampling sparsity levels. Define 
\begin{align}
\Delta_{N,l}(H;\alpha,\beta)\equiv d_1\left(T_{N,l}(H),Z\right), \quad\text{for any $\al,\beta\geq 0$, and $l=1,2$ (cf. \eqref{TN-def-12}),}
\label{Delta-def}
\end{align}
where both $\al$ and $\beta$ are included in the definition of $\Delta_{N,l}(H;\alpha,\beta)$ to account for the effects of the underlying sampling and model sparsity levels. In Theorem \ref{thm-1-main} we show that for any target subgraph $H$ that satisfies certain assumptions, and for any $\al,\beta\geq 0$, the $d_1$ distance $\Delta_{N,l}(H;\al,\beta)$ (cf. \eqref{Delta-def}) satisfies the following upper bound,
\begin{equation}
\left.
\begin{aligned}
&\Delta_{N,l}(H;\al,\beta) = O\left(N^{~\displaystyle \left\{g_{l,\al,\beta,H}\left(t_l(\al,\beta;H)\right)/2\right\}\displaystyle}\right),\quad\text{for $l=1,2$, as $N\rai$, where}\\
& t_l(\al,\beta;H) \in \operatorname{argmax}\left\{g_{l,\al,\beta,H}(t): t = 1,\ldots,R\right\},\quad\text{for $l=1,2$, with}\\
&\left.
\begin{aligned}
& g_{1,\al,\beta,H}(t) = \displaystyle{}-t + t\cdot \al + m(t:H)\cdot \beta\displaystyle,\quad\text{and}\\
& g_{2,\al,\beta,H}(t) = \displaystyle{}-t + \min\{\tau(H),t\}\cdot \al + m(t:H)\cdot \beta,\displaystyle
 \end{aligned}\quad\right\}\ \text{for $t=1,\ldots,R$.}
\end{aligned}\right\}
\label{delta-bdd-temp}
\end{equation}
The bound in \eqref{delta-bdd-temp} shows that the $d_1$ distance between $T_{N,l}(H)$ and $Z$ is controlled by the model and the sampling sparsity levels $\beta$ and $\al$, and this also depends on two different features of the target subgraph: (i) the collection of all edge densities $\{m(t:H)/t:t\in\{1,\ldots,R\}\}$, for both induced and ego-centric cases; (ii) the minimum vertex cover size $\tau(H)$, for the ego-centric case. To obtain an asymptotic normality result, one has to show that the bound on $\Delta_{N,l}(H;\al,\beta)$ converges to zero. This can be achieved if the sparsity levels $\beta$ and $\al$ are restricted within a threshold, and this threshold will depend on the target subgraph $H$. In the induced case ($l=1$), for {\it any} simple, connected, and undirected target subgraph $H$, we define the following region of model and sampling sparsity levels, 
\begin{align}
  C_1(H) &= \left\{(\al,\beta):\displaystyle \al + \frac{\beta}{\frac{t}{m(t:H)}}<1\displaystyle,~\text{for $t=2,\ldots,R$},~\al\in[0,1),\beta\geq 0\right\}.
\label{c-set-ind}
\end{align}
Using the definitions provided in \eqref{delta-bdd-temp}, we can write $C_1(H)=\{(\al,\beta):g_{1,\al,\beta,H}(t_1(\al,\beta;H))<0,\al\in[0,1),\beta\geq 0\}$. Clearly, the above description shows that if $(\al,\beta)\in C_1(H)$, then $\Delta_{N,1}(H;\al,\beta)\rightarrow 0$, as $N\rai$, thus ensuring the asymptotic normality of $T_{N,1}(H)$. 

In the ego-centric case, the allowable choices of $H$ depend on the sampling sparsity level $\al$. If the sampling is sparse, {\it i.e.}, $\al\in (0,1)$, in addition to the features described above in the induced case, we require $H$ to satisfy a specific assumption (see Assumption \ref{assump-5} in Section \ref{sec:main}). Common target subgraphs such as the complete graph or the star graph, as well as numerous other graphs, satisfy this assumption. If the sampling is dense ($\al = 0$), then $H$ does not need to satisfy this assumption. In fact, in the ego-centric case ($l=2$), the bound on $\Delta_{N,2}(H:\al,\beta)$ in \eqref{delta-bdd-temp} can be obtained at an $\al>0$, if $H$ satisfies this assumption. In the ego-centric case, for a given choice of $H$, we define the following region of model and sampling sparsity levels $\beta$ and $\al$,
\begin{align}
\label{c-set-ego}
C_2(H) &=\left\{(\al,\beta):~\displaystyle \frac{\al}{\frac{t}{\min\{t,\tau(H)\}}} + \frac{\beta}{\frac{t}{m(t:H)}} <1\displaystyle,~
\text{for $t = 2,\ldots,R$, $\al\in[0,1)$, $\beta\geq 0$}\right\}.
\end{align}
Rewriting the constraints within $C_2(H)$ in terms of $g_{2,\al,\beta,H}(\cdot)$ and $t_2(\al,\beta;H)$ (see \eqref{delta-bdd-temp}), we can show that if $(\al,\beta)\in C_2(H)$, then $T_{N,2}(H)\darw N(0,1)$. 

This asymptotic normality result is obtained under a joint model and design based framework. To the best of our knowledge, the asymptotic distribution of ego-centric sampling based estimates have not been studied in the existing literature under any type of framework (design or model based), and the same has not been done for the induced case in a model based setting. If the sampling is extremely sparse ($\al = 1$), a finite amount of data is observed in the induced case and it can be shown that $T_{N,1}(H)$ will not converge in distribution. In the ego-centric case, at $\al=1$, our results are limited to the case of $H=\kl_{1,R-1}$ (star graph), where we show that a Gaussian limit law does not arise.

The bound in \eqref{delta-bdd-temp} provides deep insight about the quality of normal approximation and how it is affected by the levels of model and sampling sparsity, the mechanism of network formation after node sampling, and the edge structure of the target subgraph. The speed of normal approximation for $T_{N,l}(H)$ depends on the magnitude of the term $g_{l,\al,\beta,H}(t_l(\al,\beta;H))$ in the exponent of $N$ (cf. \eqref{delta-bdd-temp}). From the definition of $g_{l,\al,\beta,H}(\cdot)$, it follows that the influence of the sparsity levels $\beta$ and $\al$ is transmitted through the edge structure of $H$. For example, in the ego-centric case, an increase in the sampling sparsity level $\al$ contributes at most $\tau(H)\cdot \al$ to the growth of $g_{2,\al,\beta,H}(t)$, even as $t$ increases beyond $\tau(H)$. The same is not true in the induced case. This explains how in the case of a star graph $H=\kl_{1,R-1}$, with $\tau(H)=1$, the ego-centric approach can handle relatively high levels of model sparsity $\beta$ and provide a Gaussian limit law for $T_{N,2}(\kl_{1,R-1})$, even if the sampling sparsity level $\al$ is high. Similarly, the impact of increasing the model sparsity level $\beta$ is regulated by the values of the edge densities $\{m(t:H)/t:t=1,\ldots,R\}$.

Each $C_l(H)$, $l=1,2$ (cf. \eqref{c-set-ind} and \eqref{c-set-ego}), can be further split into {\it sub-regions}, based on the value of the maxima $t_l(\al,\beta;H)$ (see \eqref{delta-bdd-temp}). Define
\begin{align}
 C_{l,j}(H) = \left\{(\al,\beta)\in C_l(H): t_l(\al,\beta;H) = j\right\},\quad \text{$j=1,\ldots,R$, $l=1,2$.}
 \label{Clj-def}
\end{align}
For some choices of $j$, $C_{l,j}(H)$ can be empty, but $\{C_{l,j}:j=1,\ldots,R\}$ forms a partition of $C_l(H)$. The amount of influence exerted by the model and sampling sparsity levels $\beta$ and $\al$ and the speed of normal approximation in \eqref{delta-bdd-temp} changes sharply when $(\al,\beta)$ move over different sub-regions. Similarly, the growth rate of the variance term in \eqref{TN-def-12} also changes over these sub-regions. To explain this clearly and to illustrate the utility of the bound in \eqref{delta-bdd-temp}, we consider the specific example of $H = \kl_3$ (triangle). 
\begin{figure}[htbp]
    \centering

    \begin{minipage}[b]{0.48\textwidth}
        \centering
        \includegraphics[scale = 0.27]{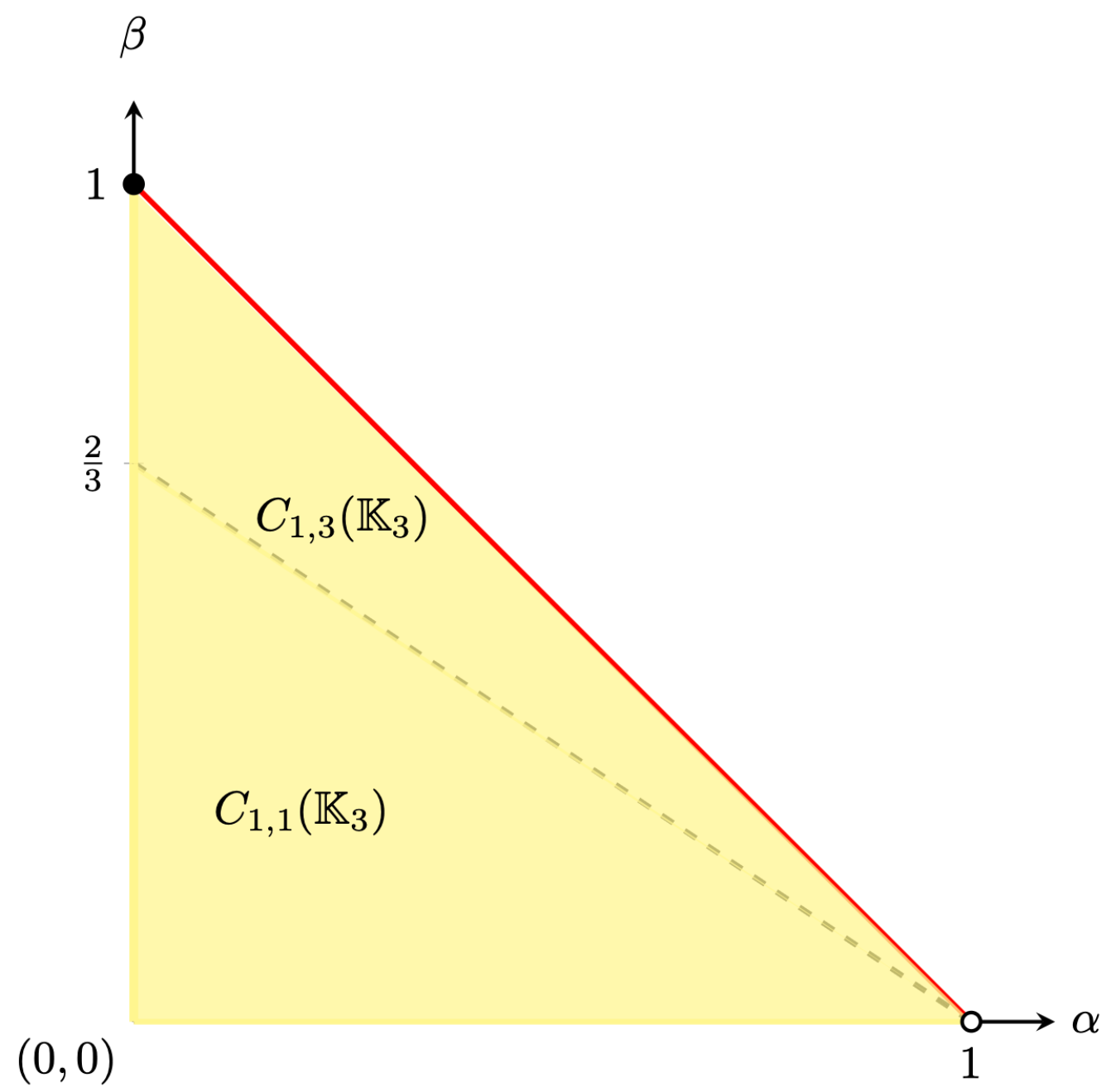}
        \\[-0.3ex]
        \footnotesize (a) Induced case: $C_{1}(\kl_3)$ and its sub-regions.
    \end{minipage}
    \hspace{1.5ex}
    \begin{minipage}[b]{0.48\textwidth}
        \centering
        \includegraphics[scale = 0.27]{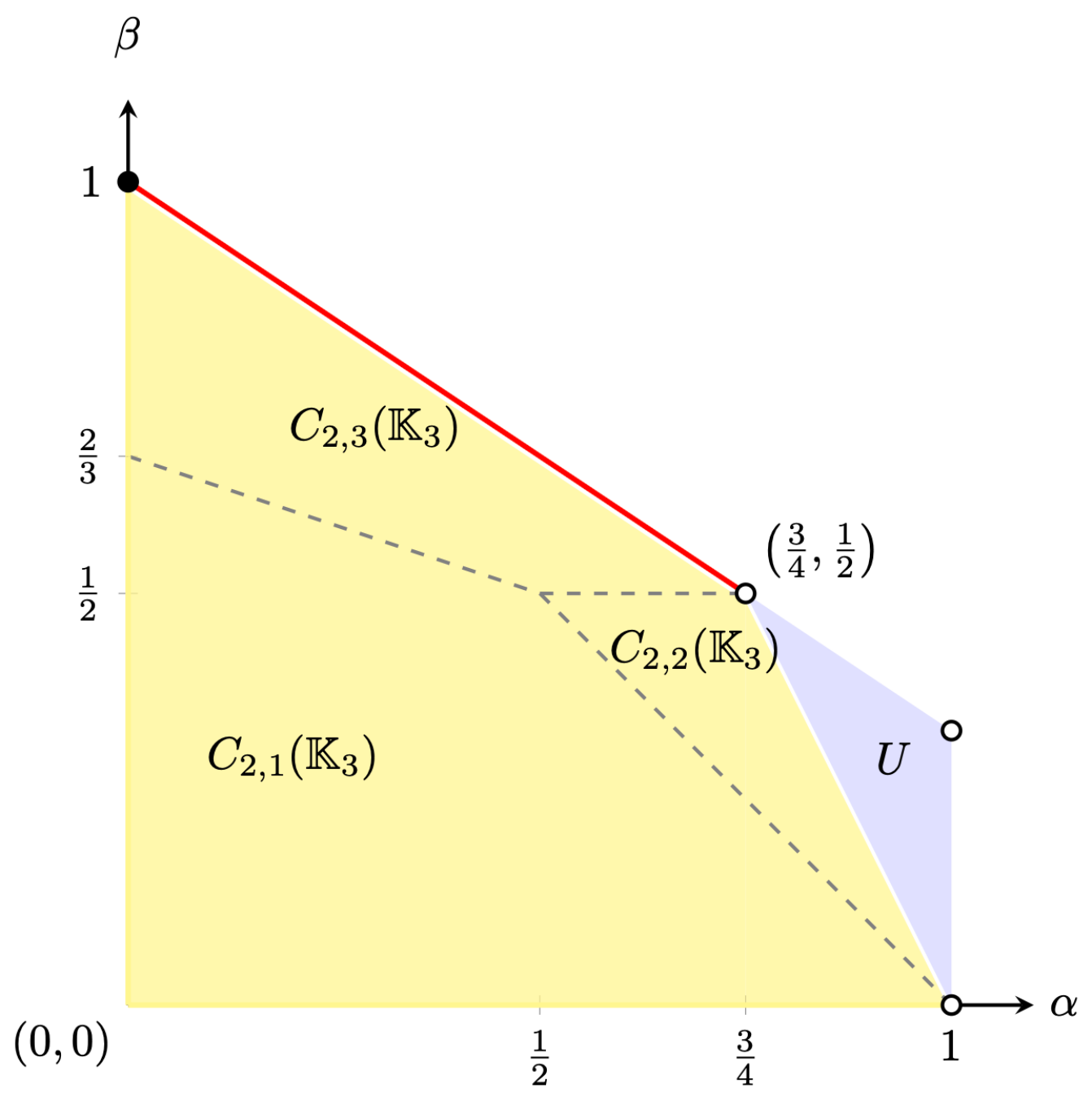}
        \\[-0.3ex]
        \footnotesize (b) Ego-centric case: $C_{2}(\kl_3)$ and its sub-regions.
    \end{minipage}
    \caption{{\small{{\bf Sparsity levels and limit laws for estimated triangle counts:} $T_{N,1}(\kl_3)$ and $T_{N,2}(\kl_3)$ have a Gaussian limit on $C_1(\kl_3)$ and $C_2(\kl_3)$ (yellow shaded regions), with different speeds of normal approximation within each sub-region (cf. \eqref{rate-tri-ind} and \eqref{rates-tri-ego}).  Similar types of Poisson based limit laws arise at $(\al,\beta)=(0,R/T)=(0,1)$ (marked by filled black dot on both figures) for both $T_{N,l}(\kl_3)$, $l=1,2$, see Theorem \ref{thm-2-main-temp}(a). A different type of Poisson based limit law arises on the upper boundaries of $C_{1,3}(\kl_3)$ and $C_{2,3}(\kl_3)$ (marked by solid red line), see Theorem \ref{thm-2-main-temp}(b). The limiting behavior of $T_{N,2}(\kl_3)$ remains unknown on the triangular region $U$, on and beyond the boundary of the sub-region $C_{2,2}(\kl_3)$. The triangle graph satisfies Assumption \ref{assump-5}.}}}
    \label{fig:main:tri}
\end{figure}

The yellow shaded regions in Figures \ref{fig:main:tri}(a) and \ref{fig:main:tri}(b) show the sets $C_{1}(\kl_3)$ and $C_{2}(\kl_3)$ (see \eqref{c-set-ind} and \eqref{c-set-ego}) and their non-empty sub-regions: $C_{1,1}(\kl_3)$ and $C_{1,3}(\kl_3)$, in the induced case; and $C_{2,1}(\kl_3)$, $C_{2,2}(\kl_3)$ and $C_{2,3}(\kl_3)$, in the ego-centric case. Using the bound in \eqref{delta-bdd-temp}, we can obtain the following results for $H=\kl_3$,
\begin{align}
&\Delta_{N,1}(\kl_3;\al,\beta) = \begin{cases}
O\left(N^{{}-\left(1-\al\right)/2}\right) & \text{if $(\al,\beta)\in C_{1,1}(\kl_3) = \left\{(\al,\beta):0\leq \beta < \frac{2(1-\al)}{3}\right\}$,}\\
O\left(N^{{}-3\left(1-\al-\beta\right)/2}\right) & \text{if $(\al,\beta)\in C_{1,3}(\kl_3) =\left\{(\al,\beta):\frac{2(1-\al)}{3}\leq \beta < (1-\al)\right\}$,}
\end{cases}
\label{rate-tri-ind}
\end{align}
and,
\begin{align}
\Delta_{N,2}(\kl_3;\al,\beta)&= \begin{cases}
O\left(N^{ {}-(1-\al)/2}\right) & \text{if $(\al,\beta)\in C_{2,1}(\kl_3)=\left\{(\al,\beta):\al+\beta<1,\frac{\al}{2}+\frac{3\beta}{2}<1\right\}$,}\\
O\left(N^{ {}-(1-\al)+\beta/2}\right) & \text{if $(\al,\beta)\in C_{2,2}(\kl_3)=\left\{(\al,\beta):\al+\beta\geq 1, \beta<\frac{1}{2}, \al+\frac{\beta}{2}<1\right\}$,}\\
O\left(N^{{}-3\left(1-2\al/3-\beta\right)/2}\right) & \text{if $(\al,\beta)\in C_{2,3}(\kl_3)=\left\{(\al,\beta):\frac{\al}{2}+\frac{3\beta}{2}>1,\beta>\frac{1}{2},\frac{2\al}{3}+\beta<1\right\}$.}
\end{cases}
\label{rates-tri-ego}
\end{align}
The bounds in \eqref{rate-tri-ind} and \eqref{rates-tri-ego} provide a precise understanding of the influence of the model and sampling sparsity levels $\beta$ and $\al$, as $(\al,\beta)$ vary over different sub-regions within $C_l(\kl_3)$, $l=1,2$. Similar results can be obtained for other choices of $H$. This information can be used by practitioners to choose the sampling sparsity level during data collection. One can show that at any common $(\al,\beta)\in C_1(H)\cap C_2(H)$, the ego-centric approach is always preferable, in terms of speed of convergence to normality (see Remark \ref{ind-ego-comp}). Remarkably, our results on Poisson approximation for $T_{N,l}(H)$, $l=1,2$, show that the thresholds on the sparsity levels $(\al,\beta)$, provided in $C_1(H)$ (cf. \eqref{c-set-ind}) and $C_2(H)$ (cf. \eqref{c-set-ego}) also turn out to be {\it necessary} for obtaining a Gaussian limit law. 

Obtaining the bound on the $d_1$ distance in \eqref{delta-bdd-temp} for an arbitrary choice of $H$ is a non-trivial task, especially in the ego-centric case. 
The first difficulty is to analyze the dependency within the components of $T_{N,l}(H)$ (cf. \eqref{TN-def-12}), which arises from the joint effect of the sampling design and the  network model. The strength of sampling based dependency is related to the number of common nodes while model based dependency is related to the number of common edges. There is no simple approach to relate these two sources of dependency. Secondly, we require the {\it exact} rate of growth of the variance $\sigma^2_{N,l}(H)$ (cf. \eqref{TN-def-12}) and to obtain this we have to identify the covariance term that has the highest contribution. This task is exceptionally difficult in the ego-centric case due to the presence of the minimum vertex cover size term. To address this issue, we have developed some novel upper and lower bounds on the minimum vertex cover size of the union of two or three arbitrary isomorphic copies of $H$. Lastly, we had to devise an approach to control the third order moments of the pivotal quantity in \eqref{TN-def-12} and compare them with the growth rate of the variance term $\sigma^2_{N,l}(H)$.

\subsection{Overview of Poisson based approximation results for estimated subgraph counts and results about the estimated clustering coefficient}
\label{sec:intro:pois}

As the sparsity levels $\beta$ and $\al$ increase and reach the boundaries of the sets $C_1(H)$ and $C_2(H)$ (see \eqref{c-set-ind} and \eqref{c-set-ego}), we observe sharp distributional transitions. We study the limit laws of $T_{N,l}(H)$, $l=1,2$, at these boundary values of $(\al,\beta)$, but focus only on those $H$, which are strictly balanced. As discussed earlier, in the ego-centric case, if $\al>0$, we assume that $H$ satisfies Assumption \ref{assump-5}. Most common target graphs are strictly balanced. Theorem \ref{thm-2-main-temp} provides three different types of Poisson distribution based limit laws for $T_{N,l}(H)$, $l=1,2$, depending on the values of $(\al,\beta)$ and the choice of $H$. Our results are exhaustive in the induced case. In the ego-centric setting, we provide a partially complete description, and some open questions remain. Our findings are listed below.

\begin{enumerate}
\item[(a)] When the sampling is dense ($\al=0$), Theorem \ref{thm-2-main-temp}(a) shows that both $T_{N,1}(H)$ and $T_{N,2}(H)$ have similar Poisson based limit laws at the extreme model sparsity level $\beta = R/T$. These limit laws are expressible as weighted sum of two independent Poisson random variables, and they differ only in their underlying parameters. In Figures \ref{fig:main:tri}(a) and \ref{fig:main:tri}(b), this corresponds to the boundary point $(\al,\beta)=(0,1)$ (marked by a solid black dot). 

\item[(b)] If the sampling is intermediately sparse ($\al\in (0,1)$), a second type of Poisson based limit law is obtained for $T_{N,1}(H)$, in the boundary region of $C_1(H)$ (see Theorem \ref{thm-2-main-temp}(b)), for all choices of $\al\in (0,1)$. In the ego-centric case, a similar type of Poisson based limit law (with different underlying parameters) is obtained for $T_{N,2}(H)$ in the boundary region of $C_2(H)$, when $\tau(H)=1$, {\it i.e.}, when $H = \kl_{1,R-1}$ (star graph). 

In case of $\tau(H)>1$, we obtain similar results for $T_{N,2}(H)$ at the boundary of $C_2(H)$, {\it but} our results are restricted and are only available for the range $\al\in (0,\al^\star_H)$, where $\al^\star_H\in (0,1)$ (cf. \eqref{al-star-def}) and depends on $H$. See Theorem \ref{thm-2-main-temp}(b) for more details. Thus, for $\tau(H)>1$, the limit law of $T_{N,2}(H)$ at certain boundary values of $C_2(H)$ remains unknown when $\al\in[\al^\star_H,1)$. 

\item[(c)] In the case of extremely sparse sampling ($\al=1$), too few sampled data points are available to obtain a proper limit law for $T_{N,1}(H)$ in the induced case. In the ego-centric case, we were able to obtain a Poisson based limit law for $T_{N,2}(H)$ only in the case of $\tau(H)=1$, {\it i.e.}, when $H=\kl_{1,R-1}$ is a star graph (see Theorem \ref{thm-2-main-temp}(c)). If $\tau(H)>1$, then the limit law of $T_{N,2}(H)$ remains unknown in the case of $\al=1$.

\end{enumerate}

To find these three types of limit laws, different techniques were used. This includes the application of multivariate size biased Poisson approximation results (cf. \cite{nualart2025}) in the context of network sampling to handle the dense sampling case and the use of total variation bounds for Poisson approximation (see Section \ref{sec:pois} for further details) to handle the intermediately sparse case. The emergence of these Poisson-based limit laws shows that the upper thresholds on the sparsity levels $(\al,\beta)\in C_l(H)$, $l=1,2$, which ensure a Gaussian limit law for $T_{N,l}(H)$, $l=1,2$, are also {\it necessary}. 
An overview of the different criteria and restrictions that were used to derive the  Gaussian and Poisson based limit laws is provided in Table \ref{tab:limit-laws}. The conditions are classified according to the underlying sampling sparsity level. 

\ctable[
  pos=!htbp,
  caption={An overview of conditions under which Gaussian and Poisson based limit laws for $T_{N,l}(H)$, $l=1,2$, are obtained at different sampling sparsity regimes.},
  label={tab:limit-laws},
  doinside={\scriptsize\setlength{\tabcolsep}{2pt}\renewcommand{\arraystretch}{0.95}}
]{%
  >{\centering\arraybackslash}p{0.15\textwidth}
  @{}
  >{\RaggedRight\arraybackslash}p{0.10\textwidth}
  @{}
  >{\RaggedRight\arraybackslash}p{0.15\textwidth}
  >{\RaggedRight\arraybackslash}p{0.35\textwidth}
  >{\RaggedRight\arraybackslash}p{0.15\textwidth}
}{
  \tnote[a]{We require $H$ to be strictly balanced.}
}{
\toprule
& &
\multicolumn{3}{c}{\textbf{Sampling sparsity level ($\alpha$)}} \\
\cmidrule(lr){3-5}

\parbox[c]{0.15\textwidth}{\centering\textbf{Network formation}\\\textbf{approach}} &
\textbf{Limit law} &
\multicolumn{1}{c}{\textbf{Dense} $(\alpha = 0)$} &
\multicolumn{1}{c}{\textbf{Intermediate} $(0 < \alpha < 1)$} &
\multicolumn{1}{c}{\textbf{Extreme} $(\alpha = 1)$} \\
\midrule

Induced  &
\textbf{Gaussian}  &
\multicolumn{2}{>{\RaggedRight\arraybackslash}p{0.50\textwidth}}%
{Need $(\al,\beta)\in C_1(H)$, for {\it any} choice of $H$. See Theorem \ref{thm-1-main}(a).} &
No limit law exists due to limited data. \\[4ex]
&
\textbf{Poisson based\tmark[a]}  & Only at $\beta = R/T$. See Theorem \ref{thm-2-main-temp}(a).  &
On $\beta = R(1-\al)/T$, for {\bf all} $\alpha\in (0,1)$. See Theorem \ref{thm-2-main-temp}(b). &
No limit law exists due to limited data.\\

\midrule

Ego-centric &
\textbf{Gaussian} &
\multicolumn{2}{>{\RaggedRight\arraybackslash}p{0.50\textwidth}}%
{Need $(\al,\beta)\in C_2(H)$. When $\al\in (0,1)$, $H$ needs to satisfy Assumption \ref{assump-5}. If $\al=0$, {\it any} choice of $H$ is allowed. See Theorem \ref{thm-1-main}(b).} &
A Gaussian limit law cannot arise.  \\[5ex]
&\textbf{Poisson based\tmark[a]} 
&
Only at $\beta = R/T$. See Theorem \ref{thm-2-main-temp}(a). 
&
On $\beta = R\!\left(1-\frac{\tau(H)\alpha}{R}\right)/T$, at {\bf all} $\alpha\in (0,1)$,
if $\tau(H)=1$. {\bf Only} for $\alpha\in (0,\alpha^\star_H)$ (see \eqref{al-star-def}) if $\tau(H)>1$. Boundary cases are unknown when
$\tau(H)>1$ and $\alpha\in [\alpha^\star_H,1)$. Require $H$ to satisfy Assumption \ref{assump-5}. See Theorem \ref{thm-2-main-temp}(b). 
&
Results are known only for $\tau(H)=1$ for $\beta\in [0,1)$. Limit laws are unknown when $\tau(H)>1$. See Theorem \ref{thm-2-main-temp}(c). \\

\bottomrule
}

As an application of the limit law results for subgraph counts we derive the asymptotic distribution of the sample based clustering coefficient statistic, in both induced and ego-centric cases (see Theorems \ref{thm-cc-ind} and \ref{thm-cc-ego}). The clustering coefficient is an important measure of network cohesion (see \cite{kolaczyk09}). Various researchers have focused on the estimation of the clustering coefficient through network sampling (cf. \cite{bliss-14}, \cite{lee2006}, \cite{illenberger-12}). 
To the best our knowledge, the asymptotic distribution of the appropriately centered sample based clustering coefficient, which is a function of both the sample based and population based wedge and triangle counts, has not been studied in the existing literature under any sort of framework or under any type of sampling scheme. Our results show that for both induced and ego-centric cases, the estimated clustering coefficient has a Gaussian limit law if the sparsity levels remain below the thresholds defined by $C_{l}(\kl_3)$, $l=1,2$ (cf. \eqref{c-set-ind}, \eqref{c-set-ego}). In the boundary regions of $C_l(\kl_3)$, we obtain various Poisson-based limit laws. The Gaussian limit law is driven by the joint effect of the wedge ($\kl_{1,2}$) and triangle counts, if $(\al,\beta)\in C_{1,1}(\kl_3)$ (induced case) and if $(\al,\beta)\in C_{2,1}(\kl_3)$ (ego-centric case). If $(\al,\beta)$ move beyond these sub-regions, the Gaussian limit law is influenced only by the triangle count. The asymptotic variance expressions are complicated, and the details are provided in Section \ref{sec:cc}.

Using these limit laws, we obtain asymptotic prediction intervals for the population subgraph counts and the clustering coefficient. A simulation study explores the coverage accuracy of these prediction intervals at various levels of model sparsity and sampling sparsity, including extreme and intermediate levels. Our simulations strongly support our theoretical results. One of the primary challenges is to compute the variance term $\sigma^2_{N,l}(H)$ (cf. \eqref{TN-def-12}), for any arbitrary choice of $H$. In Remark \ref{rem:var-comp} we provide detailed expressions of the variance $\sigma^2_{N,l}(H)$, for $l=1,2$, for {\it any} choice of $H$. The code for computing this complicated variance expression is also included in the
\href{https://github.com/AnirbanM13/subgraph-count-estimate}{github page}, along with all other codes and data used in this article. We also successfully use our proposed methodology to the {\it 2004 US political blog network} data from \cite{adamic2005}, which is known to have a block structure (cf. \cite{le2019estimating}, \cite{Amini_2013}).


The remainder of the article is organized as follows. The sparse SBM and the sparse node sampling scheme are introduced in Section \ref{sec:inf-frame}, where we provide reasons for using the joint model and design based  inferential approach. In Section \ref{sec:main}, we provide the main theoretical results of this article. This includes CLT's for $T_{N,l}(H)$, $l=1,2$, along with the decay rates for the $d_1$ distance, a description of the Poisson-based limit laws at the boundary of $C_l(H)$, $l=1,2$, and limit laws for the estimated clustering coefficient in both induced and ego-centric cases. A detailed simulation study and real data analysis is provided in Section \ref{sec:sim}. The Appendix contains proofs of main results, supplementary lemmas, description of essential results from other sources, and formulas for $\sigma^2_{N,l}(H)$ (cf. \eqref{TN-def-12}) for some common target subgraphs.

\section{The model for the population network and the node sampling scheme}
\label{sec:inf-frame}

Consider a population network on $N$ nodes, labeled  $1,\ldots,N$, and a symmetric binary relation among pairs of nodes. The terms {\it vertex} and {\it node} will be used interchangeably. We will represent this population network as an undirected, simple graph $G_N$, with a vertex set $V(G_N) = \{1,\ldots,N\}=[N]$ and an edge set $E(G_N)$. An edge exists (does not exist) between two nodes if the binary relation exists (does not exist). Let ${[V(G_N)]}^2$ denote the set of all two element subsets of $V(G_N)$. Then $E(G_N)\subseteq {[V(G_N)]}^2$, is a collection of two element subsets of $V(G_N)$ (cf. \cite{diestel2018graph}). An edge between nodes $i$ and $j$ will be represented by the two element set $\{i,j\}$. The network $G_N$ can be described in terms of a $N\times N$ symmetric adjacency matrix $\Yb_N$, where
\begin{equation}
\begin{aligned}
\Yb_N  =((Y_{N,i,j}:1\leq i,j\leq N)),\quad\text{and}\quad
Y_{N,i,j}  = Y_{N,j,i}= \begin{cases}
    1 & \text{if $\{i,j\}\in E(G_N)$,}\\
    0 & \text{if $\{i,j\}\not\in E(G_N)$.}
\end{cases} 
\end{aligned}
\label{yNdef}
\end{equation}
We assume that the network does not contain multiple edges and self-edges, {\it i.e.}, $Y_{N,i,i} = 0$, for each $i\in [N]$. 

In this article we use a model based approach and assume that the population network is stochastically generated from a {\it sparse} Stochastic Block Model (SBM). The SBM was proposed by \cite{holland1983} and it assumes that nodes in the population network are split into $K$ disjoint classes and edge probabilities depend on the class label of each node. We will assume $K\geq 2$ and $K$ remains fixed even as $N$ increases. For each $i\in [N]$, we assume that the $i$-th node in the population has an unique and non-stochastic class label $\al_i\in [K]=\{1,\ldots,K\}$, and $\al_i$ does not depend on $N$. Stochastic block models with fixed class labels have been widely used (see \cite{abbe2018community}, \cite{rohe2011spectral}, \cite{lei2015consistency}, \cite{celisse2012consistency}, \cite{anderson1992}), especially in the community detection literature. We define a sequence of $K\times K$ symmetric matrices of class-wise edge probabilities
\begin{equation}
 \PiB_N  = ((\pi_{N,i,j}, 1\leq i,j\leq K)),\quad \text{for all $N\geq 1$,}
    \label{pi-mat-def}
\end{equation}
where $\pi_{N,i,j}$ denotes the probability of an edge between a node with class label $i$ and another node with class label $j$, in the population network $G_N$. We assume that the probabilities $\pi_{N,i,j}$ can depend on the size ($N$) of population network. 
We also assume that the edges $\{Y_{N,i,j}:1\leq i<j\leq N\}$ are independently generated and 
\begin{align}
Y_{N,i,j}~{\sim} \text{ Bernoulli }(\pi_{N,\al_i,\al_j})\quad \text{for all $1\leq i<j\leq N$ and for every $N\geq 1$.}
\label{sbm:def}
\end{align}
If $K = 1$, the SBM reduces to the Erd\H{o}s-R\'{e}nyi model (cf. \cite{erdos-renyi60}). For ease of reference, we summarize the model description given above in the form of the following assumption.
\begin{enumerate}[label={(A.\arabic*)}]
\setcounter{enumi}{0}
\item\label{assump-1} We assume that the population network $G_N$ is stochastically generated by a $K$-class Stochastic Block Model with edge probability matrix $\PiB_N$ (cf. \eqref{pi-mat-def}), where $K(\geq 2)$ is fixed and does not depend on $N$. Each node in $G_N$ has a fixed and non-stochastic class label $\al_i\in[K]$, which does not depend on $N$. We also assume that the edges in $G_N$ are independently generated as per \eqref{sbm:def}.
\end{enumerate}
For simplicity of notation, we will drop the subscript $N$ and write $Y_{N,i,j} = Y_{i,j}$, for all $1\leq i,j\leq N$ and all $N\geq 1$. Write
\begin{align}
N_k = \sum_{i=1}^N \mathbf{1}(\al_i = k),\quad\text{and}\quad \la_{N,k} = \frac{N_k}{N},\quad\text{for all $k\in[K]$ and $N\geq 1$.}
\label{Nkdef}
\end{align}
These denote the count and the proportion of nodes in $G_N$, which have class label $k$. We assume the following:
\begin{enumerate}[label={(A.\arabic*)}]
\setcounter{enumi}{1}
\item\label{assump-2} There exists a collection of constants $\la_1,\ldots,\la_K\in (0,1)$, which satisfy the constraint $\sum_{k=1}^K \la_k = 1$, and the class size proportions $\lambda_{N,k}\rightarrow \lambda_k$, as $N\rai$, for each $k\in [K]$.
\end{enumerate}
Assumption \ref{assump-2} describes the growth rate of $N_k$ and assumes that the size of each class in the population grows linearly with $N$, as $N\rai$. It ensures that none of the $K$ classes in the population are negligible. If $\lambda_{k_0} = 0$, for some $k_0\in[K]$, then we can drop that class and work with the remaining $(K-1)$ classes. 

We focus on the SBM as it is one of the simplest network models which preserves edge independence, and yet it is flexible to approximate more complex network models (cf. \cite{airoldi2013stochastic},  \cite{olhede2014network}, \cite{matias2014modeling}). Edge independence has been the building block of other types of network models (cf. \cite{sischka2023graphon}, \cite{athreya2018statistical}). Additionally, SBM's are widely used in applications (cf. \cite{lee2019review}, \cite{nicola-22}, \cite{abbe2018community}). \cite{lee2019review} provide a detailed review about stochastic block models, its variations and applications. We now introduce {\it sparsity} into the SBM through the following assumption. 
\begin{enumerate}[label={(A.\arabic*)}]
\setcounter{enumi}{2}
\item\label{assump-3} There exists a constant $\beta \geq 0$, and a $K\times K$ symmetric matrix $\mathbf{C} = ((c_{i,j}:1\leq i,j\leq K))$, such that the edge probability matrix $\PiB_N$ defined in \eqref{pi-mat-def} satisfies
\begin{equation}
\left.\begin{aligned}
&\lim_{N\rai} N^\beta\cdot \boldsymbol{\Pi}_N = \Cb,\quad\text{where,}\\
& c_{i,j}\in [0,\infty),\quad\text{for all $1\leq i,j\leq K$,}\quad\text{and}\quad \max\{c_{i,i}:1\leq i\leq K\}\in (0,\infty).
\end{aligned}\right\}
\label{PiN-lim}
\end{equation}
To avoid trivialities we assume that there exists at least two unique entries in $\Cb$. If $\beta=0$, we assume $\PiB_N = \PiB$ for all $N\geq 1$, and $\Cb = \PiB$, where $\PiB$ is a fixed edge probability matrix. 
\end{enumerate}
Assumption \ref{assump-3} implies that if $\beta = 0$, then some edge probabilities in $\PiB$ remain positive and as a consequence we obtain a dense SBM through \eqref{sbm:def}. If $\beta>0$, then all edge probabilities in $\PiB_N$ decay to zero as $N\rai$, resulting in a sparse population network. The quantity $\beta$ is called the model sparsity level and it is introduced to quantify the amount of sparsity in the network model. Several researchers have constructed sparse network models in this manner by introducing such an extraneous sparsity controlling parameter (cf. \cite{pjb-chen-2011}, \cite{lei2015consistency}, \cite{sb-pjb-15}, \cite{agterberg2025}, \cite{stewart-aos-2025}, \cite{coulson2016}, \cite{otsu-24}). Similar approaches have been extensively used to model sparse Erd\H{o}s-R\'{e}nyi random graphs (see \cite{janson-rucinski-randomgraphs}). When $\beta>0$ and $c_{i,j}=0$, then $\pi_{N,i,j}\ll N^{-\beta}$. If $c_{i,j}>0$, then $\pi_{N,i,j}$ has the same rate as $N^{-\beta}$. As per Assumption \ref{assump-3}, some entries of $\Cb$ remain positive and all entries of $\Cb$ are finite. Thus, there cannot exist two different choices of $\beta$ which satisfy \eqref{PiN-lim}. We require $\max_{i\in [K]} c_{i,i}>0$, to ensure that all asymptotic variance terms obtained in our theoretical results remain positive. This is a very mild requirement.

Now we describe the sparse Bernoulli node sampling scheme and the induced and ego-centric network formation mechanisms. We make the following assumption about the Bernoulli node sampling scheme.
\begin{enumerate}[label={(A.\arabic*)}]
\setcounter{enumi}{3}
\item\label{assump-4} Each of the $N$ nodes in the population network $G_N$ is sampled independently with probability 
\begin{align}
    p_N = \frac{c}{N^\alpha},\quad\text{for some $\al\in[0,1]$ and some $c>0$, for all $N\geq 1$.}
\label{pN-def}
\end{align}
If $\al = 0$, we assume $p_N = c$, and $c\in (0,1)$.
\end{enumerate}
If $\al=0$, the node sampling scheme is called {\it dense}, otherwise we call it {\it sparse}. We consider three levels of sampling sparsity: $\al = 0$ (dense), $\al\in (0,1)$ (intermediate) and $\al = 1$ (extreme). 
Obviously, a dense sampling scheme with $\al = 0$ is always preferable if adequate resources are available. However, if there are resource constraints, it can limit the number of nodes that can be sampled (cf. \cite{leskovec06}, \cite{lee2006}, \cite{luo-15}). For example, if $N$ is extremely large, it may be sensible to use a probability of node selection that depends on the size of the population $N$. The sparse node sampling framework is introduced to understand the effects of network sampling under resource constraints. We quantify the effects of sparse and dense node sampling and how their effects are related to the structure of the target graph, the sparsity level of the population model, and the network formation mechanism. In many situations, accurate inferential conclusions can be obtained even when the model or the sampling, or both, are sparse. To the best of our knowledge, this type of sparse node sampling framework has not been studied in the existing literature on network sampling, especially in a model based setup. The sparse sampling approach provides a more finer understanding of the node sampling methodology and its effects. We define the node inclusion indicators
\begin{equation}
\begin{aligned}
W_{N,i} &= \begin{cases}
    1 & \text{if the $i$-th node is selected,}\\
    0 & \text{otherwise,}
\end{cases}\quad \text{for all $i\in[N]$, and write}\\
\Wb_N & = {(W_{N,1},\ldots,W_{N,N})}^\primet,\quad \text{for each $N\geq 1$.}
\end{aligned}
\label{WNi}
\end{equation}
Note, $\{W_{N,i}:i\in [N]\}$ are independent and identically distributed (i.i.d.) Bernoulli $(p_N)$ random variables (see Assumption \ref{assump-4}) for all $N\geq 1$. Define the maps,
\begin{align}
h_1,h_2:\{0,1\}\times \{0,1\}\mapsto \{0,1\},\quad\text{where,}\quad h_1(x,y) = x\cdot y,\quad \text{and}\quad h_2(x,y) = \max\{x,y\}.
\label{h1h2def}
\end{align}
Once the nodes are sampled, we consider two distinct network formation methods based on the sampled nodes.
\begin{enumerate}
\item[(a)] Induced subgraph formation: the $(i,j)^{th}$ entry of $\Yb_N$ is observed in the induced subgraph formation approach if both nodes $i$ and $j$ are sampled, {\it i.e.}, $h_1(W_{N,i},W_{N,j})= W_{N,i}W_{N,j}=1$.
\item[(b)] Ego-centric subgraph formation: the $(i,j)^{th}$ entry of $\Yb_N$ is observed in the ego-centric subgraph formation approach if either of the nodes $i$ and $j$ is sampled, {\it i.e.}, $h_2(W_{N,i},W_{N,j})= \max\{W_{N,i},W_{N,j}\}=1$. 
\end{enumerate}



Node sampling schemes are very common (cf. \cite{kolaczyk09}) and both induced and ego-centric sampling schemes were considered in \cite{handcock-gile10}, \cite{ko-jecon24} and \cite{agc-lewis-16}, while \cite{bhattacharya2022} only considered induced subgraph formation. Note, $\Wb_N$ (cf. \eqref{WNi}) is independent of the population adjancency matrix $\Yb_N$. 
As pointed out by \cite{handcock-gile10}, network sampling based inference can be carried out using a design based (cf. \cite{frank2005}, \cite{franksjs-78}) or a model based approach (cf. \cite{thompson-frank-2000}, \cite{vincent-21}, \cite{snijders1992estimation}, \cite{gile2015}, \cite{handcock-gile10}). In our setting, the node sampling scheme is independent of the population network model. We need to decide whether the effect of the sampling design should be included in our inferential framework, along with population model, when our primary objective is to predict population network summary statistics. The SBM formulation described above (see Assumption \ref{assump-1}) uses {\it non-stochastic} and fixed class labels $\{\al_i:i\in[N]\}$. This implies, the distributions of $Y_{i,j}$, $1\leq i<j\leq N$, are not identical and the nodes in $G_N$ are {\it not} homogeneous in terms of their edge forming characteristics. In fact, when the class labels are fixed, the SBM is no longer {\it vertex exchangeable} (see \cite{crane18}, \cite{sischka2023graphon}) and the stochastic block model formulation used in Assumption \ref{assump-1} cannot be cast in a graphon based framework. This implies, ignoring the design effect by conditioning on $\Wb_N$ (see \eqref{WNi}) and focusing on a {\it pre-selected} set of nodes would create bias. This type of bias can be removed if we include the effects of the sampling design (along with the population model) in our inferential framework. Joint design and model based approaches are frequently used in survey sampling (see Chapter 39 of \cite{pfeffermann-2000}, \cite{rubin-2005}, \cite{boistard2017}). Recall the definition of $h_1$ and $h_2$ from \eqref{h1h2def} and write
\begin{align}
{I}_{l,N} = \{(i,j)\in [N]^2: h_l(W_{N,i},W_{N,j}) = 1\}\quad \text{and}\quad \mathcal{Y}_{l,N} = \{Y_{U,V}:(U,V)\in I_{l,N}\},\quad \text{for $l=1,2$.}
\label{Ir-def}
\end{align}
Here, $I_{l,N}$ and $\mathcal{Y}_{l,N}$ denote the collection of indices $(i,j)$ and the corresponding $Y_{i,j}$ values, which have been observed through the induced ($l=1$) and the ego-centric ($l=2$) approaches. Note, for any $(U,V)\in I_{l,N}$, $\pr(Y_{U,V}=1)\neq \pr(Y_{i,j}=1)$, for any $(i,j)\in {[N]}^2$. Also, if $(U,V_1),(U,V_2)\in I_{l,N}$, then $Y_{U,V_1}$ and $Y_{U,V_2}$ are dependent, even though the $Y_{i,j}$'s in the population remain independent (see \eqref{sbm:def}). Thus, once the effect of sampling design is included, the original distribution of the $Y_{i,j}$'s is distorted and the observed data no longer remain independent. This also implies that asymptotic results developed in a design based framework (with a non-stochastic population adjacency matrix) can not be carried over trivially to a joint model and design based framework.

\section{Theoretical results}
\label{sec:main}

The main theoretical results of this paper are presented in this section. Firstly, we introduce some required notation and definitions. For any positive integer $m$, we write $[m]=\{1,\ldots,m\}$ to denote the set of the first $m$ positive integers. For any two sequences of positive real numbers $\{a_N:N\geq 1\}$ and $\{b_N:N\geq 1\}$, we write $a_N\asymp b_N$, if $a_N/b_N\rightarrow 1$, as $N\rai$. Similarly, we write $a_N\sim b_N$, if there exist constants $0<m_1 < m_2<\infty$, and an integer $N_0\geq 1$, such that $m_1<a_N/b_N <m_2$, for all $N\geq N_0$. We will always assume that the target subgraph $H$ is a fixed, simple, undirected and connected graph, with vertex set $V(H) = [R]$ and edge set $E(H) \subseteq {[V(H)]}^2$. Thus, we will not consider target graphs $H$ with isolated nodes. We assume $R(\geq 2)$ is a fixed positive integer and  $N\geq R$. The population network $G_N$ is represented as a graph with adjacency matrix $\Yb_N = ((Y_{i,j}:1\leq i,j\leq N))$ (cf. \eqref{yNdef}). Recall that $V(G_N) = [N]$. Define the following collections of $R$ dimensional vectors,
\begin{equation}
\begin{aligned}
&{[N]}^R = \left\{\mathbf{s}={\left(s_1,\ldots,s_{R}\right)}^\primet:s_i\in [N],\ \text{for all $i\in[R]$}\right\},\quad 
{[N]}_{R} = \left\{\mathbf{s}\in {[N]}^R:\text{$s_1,\ldots,s_{R}$ are distinct}\right\},\ \text{and}\\
&{[N]}_{<,R} = \left\{\mathbf{s}\in {[N]}_R:s_1<s_2<\cdots <s_R \right\}.
\label{n-sub-r}
\end{aligned}
\end{equation}
For any $\ub = {(u_1,\ldots,u_m)}^\primet\in \mathbb{R}^m$, we write $A(\ub)$ to denote the set of unique components of $\ub$. For example, if $\ub = {(1,4,3,1,3)}^\primet$, then $A(\ub)=\{1,3,4\}$. Two graphs $H_1$ and $H_2$ with vertex sets $V(H_1)$ and $V(H_2)$, and edge
sets $E(H_1)$ and $E(H_2)$ are said to be isomorphic to each other, if there exists a bijection $\phi: V(H_1)\mapsto V(H_2)$,
such that $\{u,v\}\in E(H_1)$ if and only if $\{\phi(u),\phi(v)\}\in E(H_2)$. We use the notation $H_1\simeq H_2$ to denote an isomorphism between $H_1$ and $H_2$. If $V(H_1)=V(H_2)$ and $H_1\simeq H_2$, then $\phi$ is called an automorphism. If $H_1$ and $H_2$ are automorphic, then $E(H_1)=E(H_2)$. We now define the edge densities of $H$ and its subgraphs. Fix a $t\in \{2,\ldots,R\}$ and write let $B_1,\ldots,B_{\binom{R}{t}}$ denote the collection of all possible $t$-element subsets of $[R]$. Consider the induced subgraph of $H_r$ of $H$, formed by the nodes in $B_r$, {\it i.e.}, $V(H_r) = B_r$ and $E(H_r) = \left\{\{i,j\}\in E(H):~i,j\in B_r\right\}$. Define, 
\begin{align}
m(t:H) = \max\left\{|E(H_r)|: r = 1,\ldots, \binom{R}{t}\right\},\quad \text{for all $t=2,\ldots,R$.} 
\label{mtH-def}
\end{align}
Thus, $m(t:H)$ is the maximum number of edges contained in {\it any} $t$-vertex induced subgraph of $H$, and $m(t:H)/t$ denotes the highest {\it edge density} arising from any such $t$-vertex induced subgraph. It can be computed for any given choice of $t$ and $H$ and analytical expressions can be found for specific types of $H$. Note, the maximum in \eqref{mtH-def} can be attained at different choices of $H_r$. The {\it maximum edge density of $H$} (see pp. 56 of \cite{janson-rucinski-randomgraphs}) is defined as   
\begin{align}
m(H) = \max_{2\leq t\leq R} \frac{m(t:H)}{t}.
\label{def-mH}
\end{align}
We will define $m(1:H) = 0$, for any choice of $H$. A graph is said to be \textit{strictly balanced} if $t\mapsto m(t:H)/t$, is {\it uniquely maximized} at $t=R$ (see pp. 64 of \cite{janson-rucinski-randomgraphs}). Most common target graphs are strictly balanced, {\it viz.}, complete graph, 
star graph, line graph. The {\it vertex cover} (see pp. 36 of \cite{diestel2018graph}) of a graph $H$ is a set $D\subseteq V(H)$ such that 
$$
\{u,v\}\in E(H) \quad\Rightarrow \quad u \in D ~\text{or} ~v \in D.
$$
A set $D\subseteq V(H)$ is called a {\it minimum vertex cover} of $H$ if, $D$ is a vertex cover of $H$ and for any other vertex cover $\tilde{D}$ of $H$, $|D|\leq |\tilde{D}|$. A minimum vertex cover set need not be unique. The {\it size of the minimum vertex cover} of $H$ is defined as 
\begin{equation}
  \label{def-tauH}
  \tau(H) = |D|,\quad\text{where $D$ is a minimum vertex cover of $H$.}
\end{equation}
One can show that $\tau(H) = 1$ if and only if $H$ is a star graph (see Lemma \ref{lem-4}).

Let $S_N(H)$, $\widehat{S}_{N,1}(H)$ and $\widehat{S}_{N,2}(H)$ denote the subgraph counts for $H$ in the population network $G_N$, within the sample based network formed by the induced approach, and within the sample based network formed by the ego-centric approach, respectively. They can be expressed as 
\begin{equation}
\begin{aligned}
S_N(H) &= \sum_{\mathbf{s}\in {[N]}_{R}} \prod_{\{i,j\}\in E(H)} Y_{s_i,s_j},\quad\text{and}\quad
\widehat{S}_{N,l}(H)  = \sum_{\mathbf{s}\in {[N]}_{R}} \prod_{\{i,j\}\in E(H)} h_l\left(W_{N,s_i},W_{N,s_j}\right)\cdot Y_{s_i,s_j},\quad \text{for $l=1,2$,}
\end{aligned}
\label{s-all-main}
\end{equation}
where, $h_l$, $l=1,2$, are defined in \eqref{h1h2def}. The standardized sample based subgraph counts $T_{N,l}(H)$ and $T_{N,2}(H)$ have been defined in \eqref{TN-def-12}. Let ${W}_1,\ldots,{W}_R$, denote a collection of i.i.d. Bernoulli $(p)$ random variables, where $p\in (0,1)$. For a given target graph $H$, with edge set $E(H)$, we define 
\begin{align}
\label{fl-def}
    f_l(p:H) = \E\left(\prod_{\{i,j\}\in E(H)}h_l\left({W}_i,{W}_j\right)\right),\quad\text{for $l=1,2$.}
\end{align}
where $h_1$ and $h_2$ are defined in \eqref{h1h2def}. Since $H$ is connected (with $R = |V(H)|$), we can write
\begin{align}
f_1(p:H) = p^R\quad\text{and}\quad f_2(p:H) = \sum_{D\in \mathcal{D}} p^{|D|}\cdot {(1-p)}^{(R-|D|)},
\label{f2-def}
\end{align}
where $\mathcal{D}$ is the collection of all vertex covers of $H$. When $p_N\rightarrow 0$, then  $f_2(p_N:H)\sim p^{\tau(H)}_N$. 
The pivotal quantities $T_{N,l}(H)$, $l=1,2$, are defined in \eqref{TN-def-12}. The $d_1$ distance between two real valued random variables $X$ and $Y$ is defined as (see pp. 158 of \cite{janson-rucinski-randomgraphs} and \cite{bkr-89})
\begin{align}
 d_1(X,Y) = \sup\left\{|\E h(X) - \E h(Y)|: \sup_{x\in\rl} \left(|h(x)| + |h^{(1)}(x)|\right)\leq 1\right\}.
 \label{d1-def}
\end{align}
The $d_1$ distance uses the class of all bounded real valued maps $h:\rl\mapsto \rl$, with a bounded derivative as the class of test functions. Such a map $h$ is also uniformly continuous. From Theorem 25.8 of \cite{bill-95} it now follows that convergence in the $d_1$ distance implies convergence in distribution. Next, we describe a crucial assumption which allows us to connect sets of vertices within $V(H)$, which provide a minimum vertex cover of $H$ and simultaneously form an induced subgraph of $H$ with the highest possible number of edges. 

\begin{enumerate}[label={(A.\arabic*)}]
\setcounter{enumi}{4}
    \item \label{assump-5} Let $\mathcal{D}_{\min}(H)$ denote the collection of all minimum vertex covers of $H$. For any $t\in \{2,\ldots,R\}$, define the following sub-collection of $t$ vertex induced subgraphs of $H$ which have the highest possible number of edges, 
\begin{align}
\label{Gcal-t}
    \Gc_{t,H} = \left\{H_r: r = 1,\ldots, \binom{R}{t},\quad |E(H_r)| =m(t:H)\right\},
\end{align}
where $m(t:H)$ is defined in \eqref{mtH-def}. For every $t\in \{2,\ldots,R\}$ we assume the following:
\begin{enumerate}
    \item[(i)] If $t\geq \tau(H)$, there exists a $H_{r_0}\in \Gc_{t,H}$ and a $D_0\in\mathcal{D}_{\min}(H)$, such that $D_0\subseteq V(H_{r_0})$.
    \item[(ii)] If $t<\tau(H)$, there exists a $H_{r_0}\in \Gc_{t,H}$ and a $D_0\in\mathcal{D}_{\min}(H)$, such that $V(H_{r_0})\subseteq D_0$.
\end{enumerate}
The choice of $D_0$ is allowed to depend on $t$.
\end{enumerate}

In case $t\geq \tau(H)$, Assumption \ref{assump-5} implies that there exists at least one collection of $t$ vertices, say $B_{r_0} \subset [R]$, such that the corresponding induced subgraph $H_{r_0}$ has $m(t:H)$ edges and $B_{r_0}$ also contains a minimum vertex cover of $H$. For $t<\tau(H)$, there exists a minimum vertex cover of $H$ and a vertex collection $B_{r_0}$, such that $B_{r_0}$ is contained within the minimum vertex cover and the induced subgraph $H_{r_0}$ has $m(t:H)$ edges. One can easily check that two of the most common target subgraphs, the complete graph $\kl_R$ and the star graph $\kl_{1,R-1}$ satisfy Assumption \ref{assump-5}. For other choices of $H$ a direct verification is needed.

\begin{theorem}[A bound on the $d_1$ distance]\label{thm-1-main}
Let $Z\sim N(0,1)$. Recall the definitions of $g_{l,\al,\beta,H}(\cdot)$ and $t_l(\al,\beta;H)$, for $l=1,2$, from \eqref{delta-bdd-temp}. Suppose that Assumptions \ref{assump-1}, \ref{assump-2}, \ref{assump-3} and \ref{assump-4} hold. Fix any $\al,\beta\geq 0$ and consider any target subgraph $H$. For both $l=1$ (induced case) and $l=2$ (ego-centric case), there exists a constant $b_l\in (0,\infty)$ and a positive integer $N_{0,l}$, which depend on the model parameters, sampling and model sparsity levels $\al$ and $\beta$, and the edge structure of the target subgraph $H$, such that
\begin{align}
d_1\left(T_{N,l}(H),Z\right) = \Delta_{N,l}(H;\al,\beta) \leq b_l\cdot \displaystyle N^{~\displaystyle \left\{g_{l,\al,\beta,H}\left(t_l(\al,\beta;H)\right)/2\right\}\displaystyle}\displaystyle,\quad\text{for $N\geq N_{0,l}$ and $l=1,2$.}
\label{bdd-2-ind-ego}
\end{align}
In addition to the above, if $\al>0$, then \eqref{bdd-2-ind-ego} holds in the ego-centric case ($l=2$), if the target subgraph $H$ satisfies Assumption \ref{assump-5}.

\end{theorem}

The bound in \eqref{bdd-2-ind-ego} shows that if the exponent of $N$ becomes negative, then the $d_1$ distance $\Delta_{N,l}(H;\al,\beta)$, $l=1,2$, will converge to zero. The definition of $g_{l,\al,\beta,H}(\cdot)$ in \eqref{delta-bdd-temp} shows that their values increase if the sparsity levels $\beta$ and $\al$ become higher. So, in order to ensure $g_{l,\al,\beta,H}(t_l(\al,\beta;H))<0$, for $l=1,2$, we need to place upper thresholds on the sparsity levels. Precisely the same is ensured if $(\al,\beta)\in C_l(H)$, for $l=1$ and $l=2$ (see \eqref{c-set-ind} and \eqref{c-set-ego}). For example, in the ego-centric case, from the definition of $g_{2,\al,\beta,H}(\cdot)$ in \eqref{delta-bdd-temp} we can write
\begin{align}
g_{2,\al,\beta,H}(t) &= t\cdot\left[-1 + \frac{\al}{\frac{t}{\min\{\tau(H),t\}}} + \frac{\beta}{\frac{t}{m(t:H)}}\right] < 0,\quad\text{for $t=2,\ldots,R$,}
\label{g2-bdd-ego}
\end{align}
and also at $t=1$, if $(\al,\beta)\in C_2(H)$. This is equivalent to $g_{2,\al,\beta,H}(t_2(\al,\beta;H))<0$. A similar argument works in the induced case. As a result, we obtain the following CLT for $T_{N,1}(H)$ and $T_{N,2}(H)$.

\begin{cor}[CLT for sample based subgraph count estimates]
Assume that the conditions of Theorem \ref{thm-1-main} hold. If $(\al,\beta)\in C_l(H)$, $l=1,2$ (cf. \eqref{c-set-ind}, \eqref{c-set-ego}), then the following statements are true for $l=1,2$.
\begin{enumerate}
    \item[(a)] $g_{l,\al,\beta,H}(t_l(\al,\beta;H))<0$, where $g_{l,\al,\beta,H}(\cdot)$, $l=1,2$, are defined in \eqref{delta-bdd-temp}.
    \item[(b)] $d_1(T_{N,l}(H),Z)= \Delta_{N,l}(H;\al,\beta) = O\left(N^{~\displaystyle \left\{g_{l,\al,\beta,H}\left(t_l(\al,\beta;H)\right)/2\right\}\displaystyle}\right)$, as $N\rai$.
    \item[(c)] $T_{N,l}(H)\darw N(0,1)$, as $N\rai$, where $T_{N,l}(H)$ is defined in \eqref{TN-def-12}.
\end{enumerate}
\label{cor-clt-all}
\end{cor}

Theorem \ref{thm-1-main} connects the decay rate of $d_1(T_{N,l}(H),Z)$ with the sparsity level in the model, the sparsity level of the sampling scheme and the edge structure of the target subgraph $H$. It is true for {\it any} choice of sparsity levels $\beta$ and $\al$. However, as shown in Corollary \ref{cor-clt-all}, the bound is {\it effective} only when $(\al,\beta)\in C_l(H)$. One can show that the components of the sum involved in the pivotal quantity $T_{N,l}(H)$, $l=1,2$ (cf. \eqref{TN-def-12}) possess a dependency structure which satisfies the criteria provided in \cite{bkr-89}. As stated above, the proof of Theorem \ref{thm-1-main} is based on the careful application of Theorem 6.32 of \cite{janson-rucinski-randomgraphs} (also see Section 2 of \cite{bkr-89}). We discuss some of the main challenges in the derivation of Theorem \ref{thm-1-main}. In order to use the above mentioned result from \cite{janson-rucinski-randomgraphs} a key step is to obtain the {\it precise} growth rate of the variance term $\sigma^2_{N,l}(H)$ (cf. \eqref{TN-def-12}). For any $\sbb\in {[N]}_R$, write $W_{N,l}(\sbb:H) = \prod_{\{i,j\}\in E(H)} h_l\left(W_{N,s_i},W_{N,s_j}\right)$, for $l=1,2$, where $h_l(\cdot,\cdot)$, $l=1,2$, are defined in \eqref{h1h2def}. 
For both $l=1,2$, we can write 
\begin{align}
&\sigma^2_{N,l}(H)  
= \sum_{t=1}^R \sum_{(\sbb_1,\sbb_2)\in\mathcal{M}_t} \cov\left({W}_{N,l}(\sbb_1:H),{W}_{N,l}(\sbb_2:H)\right) \cdot \E\left[\prod_{\{i,j\}\in E(H(\sbb_1)\cup H(\sbb_2))}Y_{i,j}\right],
\label{sig-temp-0}
\end{align}
where, $\mathcal{M}_t=\{(\sbb_1,\sbb_2):\sbb_1,\sbb_2\in \nsr,~|A(\sbb_1)\cap A(\sbb_2)|= t\}$ is the collection of $(\sbb_1,\sbb_2)$ pairs that have $t$ common components, $H(\sbb_1)$ and $H(\sbb_2)$ denote graphs which are isomorphic to $H$ and have vertex sets $V(H(\sbb_1))=A(\sbb_1)$ and $V(H(\sbb_2))=A(\sbb_2)$ and edge sets $E(H(\sbb_1))=\{\{s_{1,i},s_{1,j}\}:\{i,j\}\in E(H)\}$ and $E(H(\sbb_2)) = \{\{s_{2,i},s_{2,j}\}:\{i,j\}\in E(H)\}$, respectively. The union graph $H(\sbb_1)\cup H(\sbb_2)$ has vertex set $A(\sbb_1)\cup A(\sbb_2)$ and  edge set $E(H(\sbb_1))\cup E(H(\sbb_2))$. The intersection graph $H(\sbb_1)\cap H(\sbb_2)$ is defined in a similar way.

Using Assumptions \ref{assump-3} and \ref{assump-4}, it can be shown that in the induced case, the design based covariance term $\cov({W}_{N,1}(\sbb_1:H),W_{N,1}(\sbb_2:H))$ in \eqref{sig-temp-0} is an increasing function of $|V(H(\sbb_1)\cap H(\sbb_2))|$. Similarly, in the ego-centric case, this covariance decreases if $\tau(H(\sbb_1)\cup H(\sbb_2))$ increases. On the other hand, the order of magnitude of the model based expectation term involving $Y_{i,j}$'s in \eqref{sig-temp-0} is an increasing function of $|E(H(\sbb_1)\cap H(\sbb_2))|$. Note, even if $|V(H(\sbb_1)\cap H(\sbb_2))|=|A(\sbb_1)\cap A(\sbb_2)|$ increases, $|E(H(\sbb_1)\cap H(\sbb_2))|$ may {\it not} increase and can even decrease. Due to the method in which $H(\sbb_1)$ and $H(\sbb_2)$ are created, $|E(H(\sbb_1)\cap H(\sbb_2))|$ will depend on the positioning of the common components within the vectors $\sbb_1$ and $\sbb_2$. Similarly, no relation can be made between the values of $\tau(H(\sbb_1)\cup H(\sbb_2))$ and $|E(H(\sbb_1)\cup H(\sbb_2))|=2T-|E(H(\sbb_1)\cup H(\sbb_2))|$. Due to this reason, it is not possible to directly relate the dependencies between the sampling and model based components. These issues do not arise while studying the limit law of the population subgraph count $S_N(H)$ (cf. \cite{now-wierman-88}) or when using a purely design based approach (cf. \cite{bhattacharya2022}). To address this issue, one has to go through {\it all} possible structures of the union graph $H(\sbb_1)\cup H(\sbb_2)$. This can be achieved by further splitting $\mathcal{M}_t$ in terms of the positions of the common components of $\sbb_1$ and $\sbb_2$ and how these common components are mapped among each other. To describe this splitting, first consider the following example.

Let $N=5$, $R=4$, and consider the vectors $\sbb_1 = {(2,5,3,4)}^\primet$ and $\sbb_2 = {(1,4,5,2)}^\primet$. Thus, $(\sbb_1,\sbb_2)\in \mathcal{M}_3$ and the common elements are $\{2,4,5\}$. Write $\jbb = {(1,2,4)}^\primet$ and $\kbb = {(2,3,4)}^\primet$, to denote the positions at which these common elements are placed within $\sbb_1$ and $\sbb_2$, respectively. Consider the permutation $\xi:\{1,2,3\}\mapsto\{1,2,3\}$, defined by $\xi(1) = 3$, $\xi(2) = 2$ and $\xi(3) = 1$. Then, we can write, $
s_{1,j_1} = s_{2,k_3} = s_{2,k_{\xi(1)}}$, $s_{1,j_2} = s_{2,k_2} = s_{2,k_{\xi(2)}}$ and $s_{1,j_3}= s_{2,k_1} = s_{2,k_{\xi(3)}}$. For any $t\in[R]$, let $\mathcal{S}_t$ denote the collection of all permutation maps $\xi:[t]\mapsto [t]$. For any $\jbb,\kbb\in{[R]}_{<,t}$ (cf. \eqref{n-sub-r}) and $\xi\in\mathcal{S}_t$, define 
\begin{align}
\mathcal{M}_t(\jbb,\kbb,\xi) = \left\{(\sbb_1,\sbb_2)\in\mathcal{M}_t: s_{1,j_1} = s_{2,k_{\xi(1)}},\ldots,s_{1,j_t} = s_{2,k_{\xi(t)}}\right\}.
\label{Mt-jk-xi-def}
\end{align}
It is easy to check that $\{\mathcal{M}_t(\jbb,\kbb,\xi):\jbb,\kbb\in{[R]}_{<,t},\xi\in\mathcal{S}_t\}$ will form a disjoint partition of $\mathcal{M}_t$, and for any $(\sbb_1,\sbb_2)\in\mathcal{M}_t(\jbb,\kbb,\xi)$, the edge structure of the union graph $H(\sbb_1)\cup H(\sbb_2)$ will depend only on the underlying $(\jbb,\kbb,\xi)$, and not on the specific choice of $(\sbb_1,\sbb_2)$. This implies, for any $(\jbb,\kbb,\xi)$, we can study the edge structure of the union (or intersection) of $H(\sbb_1)$ and $H(\sbb_2)$ for any $(\sbb_1,\sbb_2)\in\mathcal{M}_t(\jbb,\kbb,\xi)$ by considering the following representative pair of vectors $(\tilde{\sbb}_1,\tilde{\sbb}_2)\in\mathcal{M}_t(\jbb,\kbb,\xi)$, where $A(\tilde{\sbb}_1)=[R]$ and $A(\tilde{\sbb}_1)\cup A(\tilde{\sbb}_2) = \{1,\ldots,2R-t\}$. Consider the union graph $H(\tilde{\sbb}_1)\cup H(\tilde{\sbb}_2)$. Then, for any $(\sbb_1,\sbb_2)\in\mathcal{M}_t(\jbb,\kbb,\xi)$, we can write
\begin{align*}
\E\left(W_{N,1}(\sbb_1:H)\cdot W_{N,1}(\sbb_2:H)\right)& = p^{|V(H(\sbb_1)\cup H(\sbb_2))|}_N = p^{(2R-t)}_N,\quad\text{and}\\
\E\left(W_{N,2}(\sbb_1:H)\cdot W_{N,2}(\sbb_2:H)\right)& = \sum_{D\in \mathcal{D}(H(\tilde{\sbb}_1)\cup H(\tilde{\sbb}_2))} p^{|D|}_N {(1-p_N)}^{(2R-t)-|D|},
\end{align*}
where $\mathcal{D}(H(\tilde{\sbb}_1)\cup H(\tilde{\sbb}_2))$ is the collection of all vertex covers of the union graph $H(\tilde{\sbb}_1)\cup H(\tilde{\sbb}_2)$. Further, one can show that under the stated assumptions, for all choices of $(\jbb,\kbb,\xi)$ and for large enough $N$,
\begin{align*}
\sum_{(\sbb_1,\sbb_2)\in\mathcal{M}_t(\jbb,\kbb,\xi)}\E\left[\prod_{\{i,j\}\in E(H(\sbb_1)\cup H(\sbb_2))} Y_{i,j}\right] & = N^{2R-t}\cdot N^{-\beta\cdot(2T-|E(H(\tilde{\sbb}_1)\cap H(\tilde{\sbb}_2))|)}\\
&\qquad \qquad \times\Theta_t\left(\jbb,\kbb,\xi: \Cb,\la_{1},\ldots,\la_{K},H\right)\cdot (1+o(1)),   
\end{align*}
where $\Theta_t\left(\jbb,\kbb,\xi: \Cb,\la_{1},\ldots,\la_{K},H\right)$ (cf. \eqref{Theta-t-def}) does not depend on $N$, and depends only on $(\jbb,\kbb,\xi)$, $\Cb$ and $\{\la_k:k\in[K]\}$ (cf. Assumptions \ref{assump-2} and \ref{assump-3}) and $H$. Continuing from \eqref{sig-temp-0}, in the induced case ($l=1$) we can now write
\begin{align*}
\sigma^2_{N,1}(H) & = \sum_{t=1}^R \left(p^{2R-t}_N-p^{2R}_N\right) N^{2R-t}\sum_{(\jbb,\kbb,\xi)}  N^{-\beta\cdot(2T-|E(H(\tilde{\sbb}_1)\cap H(\tilde{\sbb}_2))|)}\cdot \Theta_t\left(\jbb,\kbb,\xi: \Cb,\la_{1},\ldots,\la_{K}\right)(1+o(1)).
\end{align*}
Now we can easily see that the rate of growth of $\sigma^2_{N,1}(H)$ is determined by the maximum possible value of $|E(H(\tilde{\sbb}_1)\cap H(\tilde{\sbb}_2))|$. In Lemma \ref{lem-1} we identify the precise set of $(\jbb,\kbb)$ values where $|E(H(\tilde{\sbb}_1)\cap H(\tilde{\sbb}_2))|$ reaches its maximum possible value $m(t:H)$. 

In the ego-centric case, for $\al = 0$ (with $p_N =c\in (0,1)$), the growth rate of $\sigma^2_{N,2}(H)$ is influenced only by the highest value of $|E(H(\tilde{\sbb}_1)\cap H(\tilde{\sbb}_2))|$ and a similar argument can be applied. If $\al \in (0,1)$, the minimum vertex cover size of $H(\tilde{\sbb}_1)\cup H(\tilde{\sbb}_2)$ begins to play a role. In this case, one can easily verify that for all $(\sbb_1,\sbb_2)\in\mathcal{M}_t(\jbb,\kbb,\xi)$, and large enough $N$,
$$
\E\left(W_{N,2}(\sbb_1:H)\cdot W_{N,2}(\sbb_2:H)\right) = N^{-\alpha\cdot \tau(H(\tilde{\sbb}_1)\cup H(\tilde{\sbb}_2))}\cdot O(1), 
$$
where $\tau(H(\tilde{\sbb}_1)\cup H(\tilde{\sbb}_2))$ denotes the minimum vertex cover size of $H(\tilde{\sbb}_1)\cup H(\tilde{\sbb}_2)$, and depends only on the choice of $(\jbb,\kbb,\xi)$. For large enough $N$ and for $\al\in (0,1)$ we can show that
\begin{align*}
\sigma^2_{N,2}(H) & = \sum_{t=1}^R N^{-2T\beta + 2R-t}\sum_{(\jbb,\kbb,\xi)} N^{\displaystyle {}-\alpha\cdot \tau(H(\tilde{\sbb}_1)\cup H(\tilde{\sbb}_2)) + \beta\cdot |E(H(\tilde{\sbb}_1)\cap H(\tilde{\sbb}_2))|\displaystyle}\cdot O(1)\cdot (1+o(1)). 
\end{align*}
Thus, the growth rate of $\sigma^2_{N,2}(H)$ is determined by the choice of $(\jbb,\kbb,\xi)$, which maximizes the value of the linear combination 
\begin{align}
\psi_t(\jbb,\kbb,\xi:H) = {}-\alpha\cdot \tau(H(\tilde{\sbb}_1)\cup H(\tilde{\sbb}_2)) + \beta\cdot |E(H(\tilde{\sbb}_1)\cap H(\tilde{\sbb}_2))|.
\label{psi-def-new}
\end{align}
Lemma \ref{VC-lem-1} and Lemma \ref{VC-lem-2} provide a lower bound for $\tau(H(\tilde{\sbb}_1)\cup H(\tilde{\sbb}_2))$ and also establish that this lower bound is attained at some triplet $(\jbb_0,\kbb_0,\xi_0)$. However, there is {\it no guarantee} that $|E(H(\tilde{\sbb}_1)\cap H(\tilde{\sbb}_2))|$ will be maximized at $(\jbb_0,\kbb_0,\xi_0)$. As a result, there is no way to analytically find the choice of $(\jbb,\kbb,\xi)$ that maximizes $\psi_t(\jbb,\kbb,\xi:H)$. For specific choices of $H$, a brute-force search can be performed for every $t\in[R]$ and {\it all} possible triplets $(\jbb,\kbb,\xi)\in {[R]}_{<,t}\times {[R]}_{<,t}\times \mathcal{S}_t$. The pair of values $\tau(H(\tilde{\sbb}_1)\cup H(\tilde{\sbb}_2))$ and $E|H(\tilde{\sbb}_1)\cap H(\tilde{\sbb}_2)|$ has to be computed for each such triplet to find the maximizer of $\psi_t(\jbb,\kbb,\xi:H)$. Even this will become computationally infeasible for $R=10$ (which requires a search over nearly 234 million triplets). Obviously, this is not feasible for any arbitrary choice of $H$. This is due to the fact that for any arbitrary graph $H$, the graph features $\tau(H)$ and $|E(H)|$ are not {\it coherent} with each other. This means, if an additional edge is included in $H$, it does not {\it necessarily} imply that $\tau(H)$ will not decrease. Thus, it becomes impossible to find an analytical expression for the growth rate of $\sigma^2_{N,2}(H)$ for any choice of $(\al,\beta)$ and for any arbitrary choice of $H$. Consequently, Theorem 6.32 of \cite{janson-rucinski-randomgraphs} cannot be used to establish a Gaussian limit law result for $T_{N,2}(H)$.

It is precisely at this stage where Assumption \ref{assump-5} helps us avoid this tedious enumeration and ensures the existence of a triplet $(\jbb_0,\kbb_0,\xi_0)$, which will {\it simultaneously} minimize $\tau(H(\tilde{\sbb}_1)\cup H(\tilde{\sbb}_2))$ and maximize $E|H(\tilde{\sbb}_1)\cap H(\tilde{\sbb}_2)|$. This requires a special construction of the union graph $H(\tilde{\sbb}_1)\cup H(\tilde{\sbb}_2)$ using an appropriate choice of $(\jbb,\kbb,\xi)$ (see Lemma \ref{VC-lem-2} and \ref{VC-lem-1}). In addition to the above, in order to use Theorem 6.32 of \cite{janson-rucinski-randomgraphs} we also establish an upper bound on the third order moments of $\left(\widehat{S}_{N,2}(H)-f_2(p_N:H)\cdot S_N(H)\right)$ and then relate this to the previously found growth rate of $\sigma^2_{N,2}(H)$. For this task, we had to establish a lower bound on $\tau(\cup_{i=1}^3 H(\sbb_i))$ (see Lemma \ref{VC-lem-1}) and then achieve control on $|E(\cup_{i=1}^3 H(\sbb_i))|$ by using a carefully constructed splitting of $V(\cup_{i=1}^3 H(\sbb_i))$ (see Lemma \ref{maxEdge-lem} and the proof of Theorem \ref{thm-1-main}). This had to be carried out for any $\sbb_i\in \nsr$, for $i=1,2,3$, where, these vectors can have an arbitrary amount of overlap among themselves.

Now, we present some important aspects of the Gaussian approximation result presented in Theorem \ref{thm-1-main} and Corollary \ref{cor-clt-all}.

\begin{rem}[A special region of sparsity levels]
\label{rem:speed}
It is expected that if either the model or sampling sparsity levels increase, the quality of normal approximation for $T_{N,l}(H)$ should degrade. However, our results show that for any choice of $H$, there exists a collection of model and sampling sparsity $\beta$ and $\al$ (that depend on $H$ and the network formation mechanism), where the speed of normal approximation and the quality of statistical inference is {\it not} affected by the model sparsity level $\beta$, and is only affected by a change in the sampling sparsity level $\al$. In the case of the triangle, this corresponds to the regions $C_{1,1}(\kl_3)$ and $C_{2,1}(\kl_3)$ in Figure \ref{fig:main:tri} (see \eqref{rate-tri-ind} and \eqref{rates-tri-ego}) in the induced and ego-centric cases respectively. Recall the definition of the sets $C_{l,j}(H)$ (cf. \eqref{Clj-def}). 
It is easy to check, if $(\al,\beta)\in C_{l,1}(H)$, then $\Delta_{N,l}(H;\al,\beta) = O\left(N^{-(1-\al)/2}\right)$ for $l=1,2$. For example, in the induced case ($l=1$), for any $H$, $C_{1,1}(H) = \{(\al,\beta):\al + m(t:H) \beta/(t-1) < 1,\text{ for all $t=2,\ldots,R$}\}$. This implies, if the model sparsity level $\beta$ is not too high, {\it i.e.}, $\beta< \min\left\{{(t-1)}/{m(t:H)}:2\leq t\leq R\right\}$, and we use dense sampling ($\al = 0$), then the speed of normal approximation will be $O(N^{-1/2})$, and we will have the {\it same} inferential accuracy as in the case of a dense population network ($\beta = 0$). The same conclusion works at any other $(\al,\beta)\in C_{1,1}(H)$ with $\al>0$, except that the rate will be $O(N^{-(1-\al)/2})$. 
As the model sparsity level $\beta$ increases beyond the thresholds provided in $C_{1,1}(H)$, the inferential precision starts to degrade very rapidly. For example, in the case of the triangle (cf. \eqref{rate-tri-ind}) with $\al=0$, the decay rate stays at $O(N^{-1/2})$ if $\beta<2/3$, but in between $\beta\in[2/3,1)$, the exponent of $N$ in \eqref{rate-tri-ind} increases from $-1/2$ to $0$, as $\beta\uparrow 1$. The decay can be more rapid for other choices of $H$ and a similar pattern emerges at any $\al>0$.  

In the ego-centric case, if $\tau(H)=1$ ($H=\kl_{1,R-1}$, for some $R\geq 2$), then $C_{2,1}(\kl_{1,R-1})= {[0,1)}^2$ and
$C_2(\kl_{1,R-1}) = \{(\al,\beta):0\leq \al<1,~{\al} + (R-1)\beta < R\}$. Thus, for {\it any} $R$ vertex star graph, at any given sampling sparsity level $\al$, the inferential accuracy remains unaffected by the model sparsity level, as long as $\beta\in [0,1)$. The decay rate of $\Delta_{N,2}(\kl_{1,R-1})$ drops sharply when $\beta>1$, and the drop is faster for larger $R$. For any other $H$ satisfying Assumption \ref{assump-5}, with $\tau(H)\geq 2$, 
\begin{align*}
C_{2,1}(H) &= \left\{(\al,\beta):~\displaystyle \frac{\al}{\frac{t-1}{\min\{t,\tau(H)\}-1}} + \frac{\beta}{\frac{(t-1)}{m(t:H)}} <1\displaystyle,~\text{for $t = 2,\ldots,R$,} \right\}.
\end{align*}
\end{rem}

\begin{rem}[Comparing induced and ego-centric approaches]
\label{ind-ego-comp}
The ego-centric network formation approach collects far more observations in comparison to the induced approach. However, at $\al = 0$, $g_{1,0,\beta,H}(t) = g_{2,0,\beta,H}(t)$, for all $t\in [R]$ (see \eqref{delta-bdd-temp}). This implies, $\Delta_{N,1}(H;0,\beta)$ and $\Delta_{N,2}(H;0,\beta)$ have the same decay rate at the same model sparsity level. As $p_N$ remains bounded away from zero, enough number of nodes are sampled (as $N\rai$) in both the induced and ego-centric approaches, so that the ego-centric approach does not offer any advantages in terms of the speed of normal approximation. If $(\al,\beta)\in C_1(H)\cap C_2(H)$ and $\al>0$, we note that $g_{1,\al,\beta,H}(t) \geq g_{2,\al,\beta,H}(t)$, for all $t\in[R]$ (cf. \eqref{delta-bdd-temp}). This immediately implies that the maximum value achieved by $g_{1,\al,\beta,H}(\cdot)$ will always be greater than or equal to the maximum value achieved by $g_{2,\al,\beta,H}(\cdot)$. On $C_1(H)\cap C_2(H)$, the maximum values of both maps are negative (see Corollary \ref{cor-clt-all}(a)). This implies, if $(\al,\beta)\in C_1(H)\cap C_2(H)$, and if the sampling is sparse ($\al>0$), the ego-centric approach always provides a faster speed of normal approximation (cf. Corollary \ref{cor-clt-all}(b)) and better inferential accuracy at any model sparsity level $\beta$. At $\al=1$, the induced approach collects a finite size of data, even as $N\rai$, implying that no meaningful inferential conclusions can be obtained. The ego-centric approach observes $O(N)$ entries from the population adjacency matrix, allowing the possibility of developing inferential results.

In addition, the ego-centric approach also highlights finer differences among two different target graphs through the use of the minimum vertex cover size $\tau(H)$. For example, consider $H_1 = \kl_{1,3}$ (star graph on $R=4$ vertices) and $H_2 = \mathbb{L}_4$ (line on $R=4$ vertices). Both graphs satisfy Assumption \ref{assump-5} and have the same edge densities, but $\tau(\kl_{1,3}) = 1 < \tau(\mathbb{L}_4) = 2$. The asymptotic properties of $T_{N,1}(H)$ would be exactly similar for $H_1$ and $H_2$. However, in the ego-centric case, due to the smaller value of $\tau(H_1)$, we will have $C_2(\mathbb{L}_4)\subset C_2(\kl_{1,3})$, and the speed of normal approximation will be better for $T_{N,2}(H_1)$. Intuitively, if $\tau(H)$ is small, then it is easier to capture a copy of a graph $H$ in $G_N$ using the ego-centric approach. 

\end{rem}

\begin{rem}[Continuity of the exponent of $N$ in \eqref{bdd-2-ind-ego}]
\label{rem:g-cont}
As $(\al,\beta)$ vary over different sub-regions in $C_1(H)$ and $C_2(H)$ (cf. \eqref{rate-tri-ind} and \eqref{rates-tri-ego}), the exponent of $N$ in \eqref{bdd-2-ind-ego}, $g_{l,\al,\beta,H}(t_l(\al,\beta;H))$ remains a continuous function of $(\al,\beta)$, even though $(\al,\beta)\mapsto t_l(\al,\beta;H)$ can have jumps when $(\al,\beta)$ vary over sub-regions of $C_l(H)$. This is due to the fact that choices of $t_l(\al,\beta;H)$ are themselves determined by comparing $R$ different linear maps of $(\al,\beta)$ and as a result, the transitions in the values of $t_l(\al,\beta;H)$ arise when $g_{l,\al,\beta,H}(t)\geq g_{l,\al,\beta,H}(t^\prime)$, for some $t\neq t^\prime$. This ensures the continuity of $(\al,\beta)\mapsto g_{l,\al,\beta,H}(t_l(\al,\beta;H))$, for $l=1,2$.
\end{rem}

\begin{rem}[Comparison with the population subgraph count $S_N(H)$]
\label{rem:population-comp}
Using similar proof techniques, we can show that if $\beta\in [0, 1/m(H))$, then for any choice of $H$, 
\begin{equation*}
\begin{aligned}
d_1\left(\frac{S_N(H)-\E(S_N(H))}{\sqrt{\Var(S_N(H))}}, Z\right) &= O\left(N^{\displaystyle g_{1,0,\beta,H}(t_P(\beta;H))/2\displaystyle}\right),\quad\text{where, $Z\sim N(0,1)$, and}\\
t_P(\beta;H) &= \operatorname{argmax}\left\{g_{1,0,\beta,H}(t):t = 2,\ldots,R\right\}.
\end{aligned}
\end{equation*}
The same upper threshold on the model sparsity level was obtained by \cite{janson-rucinski-randomgraphs} for sparse Erd\H{o}s-R\'{e}nyi  random graphs. Unlike $t_l(\al,\beta;H)$ (cf. \eqref{delta-bdd-temp}), the maximization is carried out over $t\in \{2,\ldots,R\}$, because in the model based setup, edge dependency can not arise if $|A(\sbb_1)\cap A(\sbb_2)|\leq 1$. This also impacts the speed of convergence to normality. For example, at $\beta=0$, $S_N(H)$ converges in law to $N(0,1)$ at the speed $O(N^{-1})$, for any $H$. In comparison, for the sampling based case, we have $\Delta_{N,l}(H;0,0) = O\left(N^{-1/2}\right)$. The speed of convergence to normality slows down by an order of magnitude in the case of the sample based subgraph counts, only because of the extra dependency arising from the effect of the sampling design. 
\end{rem}

\begin{rem}
\cite{bhattacharya2022} made a seminal contribution in the area of network sampling based inference. They used the so called {\it fourth moment convergence criterion}, to obtain a CLT for the induced sample based subgraph count $\widehat{S}_{N,1}(H)$ in a purely design based framework, where the population network $G_N$ is {\it non-stochastic}. They did not consider the ego-centric approach. Our results are not directly comparable to those obtained in \cite{bhattacharya2022}, as we use a different inferential framework. However, some basic differences can be pointed out. Unlike the joint model and the design based variance expression in \eqref{sig-temp-0}, \cite{bhattacharya2022} used the design based variance for standardization. The proof techniques used in this article are not related to the approach used by \cite{bhattacharya2022}. In fact, we do not require the use of fourth order moments of the pivotal quantity $T_{N,l}(H)$ (cf. \eqref{TN-def-12}), neither we need to verify its convergence to the fourth moment of a standard normal random variable. We utilize a Stein's method based bound (cf. Theorem 6.32 of \cite{janson-rucinski-randomgraphs}), which uses second and third order moments of the pivotal quantity $T_{N,l}(H)$. Interestingly, their proposed fourth moment convergence criterion is usable only when the node selection probability is less than or equal to $1/20$.  
\label{rem:bb-aos}
\end{rem}

For the ease of practitioners, we now state the  bounds on $\Delta_{N,l}(H;\al,\beta)$, $l=1,2$, for some commonly used target subgraphs $H$. The following graphs are considered: complete graph $\kl_R$, star graph $\kl_{1,R-1}$ and line graph $\mathbb{L}_{R}$ on $R$ vertices. All the above mentioned graphs are strictly balanced. Assumption \ref{assump-5} is satisfied by any complete graph and star graph, but this assumption is not satisfied by a line graph if $R>4$. The edge ($H = \kl_2$) is treated as a star graph on two vertices.

\begin{cor}[Decay rates for $\Delta_{N,l}(H;\al,\beta)$ for specific graphs] 
\label{cor:spl-graphs}
Suppose the conditions of Theorem \ref{thm-1-main} hold and recall the definition of $C_{l,j}(H)$ (cf. \eqref{Clj-def}).
\begin{enumerate}
    \item[(a)] Consider the star graph $H=\kl_{1,R-1}$ on $R\geq 2$ vertices, with $T = (R-1)$, $m(t:\kl_{1,R-1}) = (t-1)$, for $t\in\{2,\ldots,R\}$, and $\tau(\kl_{1,R-1})=1$. In this case, $C_1(\kl_{1,R-1}) = \{(\al,\beta):\al + (R-1)\beta/R <1,~\al\in [0,1)\}$ and $C_2(\kl_{1,R-1}) = \{(\al,\beta):\al/R+ (R-1)\beta/R <1,~\al\in [0,1)\}$.
For any $(\al,\beta)\in C_1(\kl_{1,R-1})$, 
\begin{small}
    \begin{equation*}
        \begin{aligned}
            &\Delta_{N,1}(\kl_{1,R-1};\al,\beta)\notag\\
            &= \begin{cases}
                O\left(N^{-\frac{1}{2}(1-\al)}\right)&\quad\text{if $(\al,\beta)\in C_{1,1}(\kl_{1,R-1}) = \{(\al,\beta):0\leq \beta< 1-\al\}$,}\\
                O\left(N^{-\frac{(R-1)}{2}\left[\frac{R}{R-1}(1-\al)-\beta\right]}\right)&\quad \text{if $(\al,\beta)\in C_{1,R}(\kl_{1,R-1})=\{(\al,\beta): 1-\al\leq \beta< \frac{R(1-\al)}{R-1},\al\in [0,1)\}$.}
            \end{cases}
        \end{aligned}
    \end{equation*}
\end{small} 
For any $(\al,\beta)\in C_2(\kl_{1,R-1})$, 
\begin{small}    
    \begin{equation*}
        \begin{aligned}
            &\Delta_{N,2}(\kl_{1,R-1};\al,\beta)= \begin{cases}
                O\left(N^{-(1-\al)/2}\right)&\ \text{if $(\al,\beta)\in C_{2,1}(\kl_{1,R-1}) ={[0,1)}^2$,}\\
                O\left(N^{-\frac{(R-1)}{2}\left[\frac{(R-\al)}{R-1}-\beta\right]}\right)&\  \text{if $(\al,\beta)\in C_{2,R}(\kl_{1,R-1}) = \left\{(\al,\beta):1\leq\beta < \frac{(R-\al)}{R-1},\al\in[0,1)\right\}$.}
            \end{cases}
        \end{aligned}
    \end{equation*}
\end{small}    
\noindent
The expressions for $\sigma^2_{N,1}(\kl_{1,R-1})$ and $\sigma^2_{N,2}(\kl_{1,R-1})$ are provided in \eqref{var-star-1} and \eqref{var-star-2} respectively (see Appendix \ref{sec:app-E}). 

\item[(b)] Consider the complete graph $H=\kl_R$ on $R\geq 3$ vertices, with $T = R(R-1)/2$, $m(t:H)=t(t-1)/2$, for $t\in\{2,\ldots,R\}$ and $\tau(\kl_R) = (R-1)$. In this case, $C_1(\kl_R)=\{(\al,\beta):\al+\frac{(R-1)\beta}{2} <1,\al\in[0,1)\}$ and 
\begin{small}
\begin{align*}
C_2(\kl_R) & = \left\{(\al,\beta):\al+\frac{(R-2)\beta}{2}<1,\frac{(R-1)\al}{R} + \frac{(R-1)\beta}{2} <1,~\al\in [0,1)\right\}.
\end{align*}
\end{small}
For any $(\al,\beta)\in C_1(\kl_{R})$,
\begin{small}
\begin{equation*}
\begin{aligned}
&\Delta_{N,1}(\kl_R;\al,\beta)= 
\begin{cases}
O\left(N^{-\frac{1}{2}(1-\al)}\right)&\  \text{if $(\al,\beta)\in C_{1,1}(\kl_{R})=\left\{(\al,\beta):0\leq \beta < \frac{2(1-\al)}{R}\right\}$,}\\
O\left(N^{-\frac{R(R-1)}{4}\left(\frac{2}{R-1}(1-\al) -\beta\right)}\right)&\  \text{if $(\al,\beta)\in C_{1,R}(\kl_R) = \left\{(\al,\beta):\frac{2(1-\al)}{R}\leq \beta < \frac{2(1-\al)}{R-1}\right\}$.}
\end{cases}
\end{aligned}
\end{equation*}
\end{small}
Define the following sub-regions within $C_2(\kl_R)$,
\begin{footnotesize}
\begin{align*}
C_{2,1}(\kl_R) & = \left\{(\al,\beta):\al +\frac{\beta}{2/(R-1)}<1,~\frac{\al}{(R-1)/(R-2)}+\frac{\beta}{2/R}<1,\al\in[0,1)\right\},\\
C_{2,R-1}(\kl_R) & = \left\{(\al,\beta):\al+\frac{\beta}{2/(R-1)}>1,~\beta <\frac{1}{R-1},~\al+\frac{\beta}{2/(R-2)}<1,\al\in[0,1)\right\},\quad\text{and}\\
C_{2,R}(\kl_{R}) &= \left\{(\al,\beta):\frac{\al}{(R-1)/(R-2)}+\frac{\beta}{2/R}>1,~ \beta >\frac{1}{R-1},~\frac{(R-1)\al}{R} +\frac{(R-1)\beta}{2} <1,\al\in[0,1)\right\}.
\end{align*}
\end{footnotesize}
For any $(\al,\beta)\in C_2(\kl_{R})$, 
\begin{small}
\begin{align*}
\Delta_{N,2}(\kl_R;\al,\beta) &=\begin{cases}
O\left(N^{-\frac{1}{2}(1-\al)}\right) & \text{if $(\al,\beta)\in C_{2,1}(\kl_R)$,}\\
O\left(N^{\frac{1}{2}\left(-(R-1)(1-\al) + \frac{\beta}{2}(R-1)(R-2)\right)}\right) & \text{if $(\al,\beta)\in C_{2,R-1}(\kl_R)$,}\\
O\left(N^{\frac{1}{2}\left(-R+(R-1)\al + \frac{\beta}{2}R(R-1)\right)}\right)& \text{if $(\al,\beta)\in C_{2,R}(\kl_{R})$.}
\end{cases}
\end{align*}
\end{small}
The expression for $\sigma^2_{N,1}(\kl_{R})$ and $\sigma^2_{N,2}(\kl_{R})$ are provided in \eqref{var-comp-1} and \eqref{var-comp-2} respectively (see Appendix \ref{sec:app-E}). 

\item[(c)] Consider the line graph $H=\mathbb{L}_{R}$ on $R$ vertices, with $R\geq 4$. Here, $T = (R-1)$, $m(t:\mathbb{L}_{R})=(t-1)$, for $t\in \{2,\ldots,R\}$.
In this case, $C_1(\mathbb{L}_{R}) = C_1(\kl_{1,R-1})$ (see part (a) above). 
The expression for $\sigma^2_{N,1}(\mathbb{L}_{R})$ is provided in \eqref{var-path-1} (see Appendix \ref{sec:app-E}). 
\end{enumerate}
\end{cor}


\subsection{Poisson-based limit laws for estimated subgraph counts}
\label{sec:pois}

In this section, we investigate the limit laws of $T_{N,1}(H)$ and $T_{N,2}(H)$, when the model and sampling sparsity levels reach the boundaries of $C_1(H)$ and $C_2(H)$ respectively. We only focus on target graphs $H$ which are strictly balanced, {\it i.e.}, $m(t:H)$ (cf. \eqref{mtH-def}) is uniquely maximized at $t=R$, and $m(H) = T/R$ (cf. \eqref{def-mH}). In addition, for the ego-centric case, for $\al>0$, we consider $H$ satisfying Assumption \ref{assump-5} and for $\al=0$, all choices of $H$ are allowed. If $H$ is not strictly balanced, it is not clear if any proper limit law for $T_{N,l}(H)$ can be obtained at the boundary points of $C_l(H)$ for $l=1,2$. Most common target graphs of interest are strictly balanced. \cite{janson-rucinski-randomgraphs} obtained Poisson based limit laws for $S_N(H)$ in the case of sparse Erd\H{o}s-R\'{e}nyi  random graphs at the extreme model sparsity level $\beta = R/T$. They obtained these results only in the case of strictly balanced $H$. Let $C^b_l(H)$ denote the boundary of $C_l(H)$, $l=1,2$ (cf. \eqref{c-set-ind}, \eqref{c-set-ego}). As $H$ is strictly balanced, splitting $C^b_1(H)$ in terms of the level of sampling sparsity (dense, intermediate and extreme), we can write
\begin{align*}
C^b_1(H) 
=  \{(0,R/T)\} \cup \left\{\left(\al,{R(1-\al)}/{T}\right): \al\in (0,1)\right\} \cup \{(1,0)\}.
\end{align*}
Figures \ref{fig:main:tri}(a) and \ref{fig:main:star}(a) show $C^b_1(\kl_3)$ and $C^b_1(\kl_{1,3})$ respectively. In the ego-centric case, at $\tau(H)=1$, we can similarly split, 
\begin{align*}
C^b_2(\kl_{1,R-1}) = \left\{\left(0,{R}/
{(R-1)}\right)\right\}\cup \left\{\left(\al,{(R-\al)}/{(R-1)}\right): \al\in (0,1)\right\}\cup \{(1,\beta): 0\leq \beta \leq 1\}.
\end{align*}
Figure \ref{fig:main:star}(b) displays $C^b_2(\kl_{1,3})$. When $\tau(H)>1$, $C^b_2(H)$ is defined by the points of intersection of the lines
\begin{align}
\min\{t,\tau(H)\}\cdot {\al} + m(t:H)\cdot{\beta} = t,\quad t=2,\ldots,R.
\label{line-set}
\end{align}
The boundary $C^b_2(H)$ starts at $(\al,\beta)=(0,R/T)$, follows the equation of the line at $t=R$ within the collection of lines \eqref{line-set}, till it intersects another line from \eqref{line-set} at an intermediate point $(\al^\star_H,\beta^\star_H)$. For $j=1,\ldots,(R-1)$, let $(\al_{R,j}, \beta_{R,j})$ denote the point of intersection of the line at $t=R$ with the line at $t=j$ from the collection in \eqref{line-set}. Then 
\begin{align}
\al^\star_H = \min\{\al_{R,j}: j = 1,\ldots,(R-1)\} = 
\min_{j\in\{1,\ldots,R-1\} }\frac{j - \frac{R}{T}\cdot m(j:H)}{\min\{\tau(H),j\} - \frac{\tau(H)}{T}\cdot m(j:H)}
.
\label{al-star-def}
\end{align}
It can be checked that $\al^\star_H\in (0,1)$, for any choice of $H$ (with $\tau(H)>1$) and it can be easily computed for any choice of $H$. The boundary $C^b_2(H)$ continues as per the above procedure and the last line segment always ends at $(\al,\beta)=(1,0)$. For example, in the case of $H = \kl_3$, Figure \ref{fig:main:tri}(b) shows that the boundary of $C_2(\kl_3)$ consists of only two line segments with $\al^\star_{\kl_3} = 3/4$. We now summarize our findings on Poisson-based limit laws based on the underlying sampling sparsity level $\al$.
\begin{enumerate}
\item[(i)] {\bf Dense case ($\al=0$):} Theorem \ref{thm-2-main-temp}(a) establishes a Poisson based limit law at the boundary point $(\al,\beta)=(0,R/T)$, for {\it any} $H$, in both induced and ego-centric cases. These two limit laws are similar and differ in their underlying parameters. They are  expressible as a weighted sum of two independent Poisson random variables. We use a multivariate Poisson approximation result, developed by \cite{nualart2025} to obtain this result. It involves analyzing the joint limiting behavior of the pair of counts $(\widehat{S}_{N,l}(H),S_N(H))$, where both components are {\it dependent} and both converge marginally to different Poisson distributions. 
In Figures \ref{fig:main:tri} and \ref{fig:main:star}, this specific limit law arises at the solid black dotted points $(0,1)$ and $(0,4/3)$ respectively.

\item[(ii)] {\bf Intermediate case ($\al\in (0,1)$):} In this scenario, a second type of Poisson-based limit law is obtained for $T_{N,1}(H)$, for any $H$ that satisfies Assumption \ref{assump-5}, on the boundary $C^b_1(H)$ at {\bf all} choices of $\al\in (0,1)$ (see Theorem \ref{thm-2-main-temp}(b)). This limit law is expressible in terms of a scaled and centered Poisson random variable. In Figures \ref{fig:main:tri}(a) and \ref{fig:main:star}(a), this corresponds to points on the solid red line, at the boundary of the sub-regions $C_{1,3}(\kl_3)$ and $C_{1,4}(\kl_{1,3})$ respectively. 

In the ego-centric case, we are able to obtain the same limit law (with different parameters) for $T_{N,2}(H)$ (cf. Theorem \ref{thm-2-main-temp}(b)) on the boundary $C^b_2(H)$, when $\al\in (0,\widetilde{\al}_H)$, where
\begin{align}
\widetilde{\al}_H = \begin{cases}
    1 & \text{if $\tau(H)=1$,}\\
    \al^\star_H & \text{if $\tau(H)>1$,}
\end{cases}    
\label{al-tild-def}
\end{align}
where $\al^\star_H$ is defined in \eqref{al-star-def}. This implies, in the case of a star graph $H = \kl_{1,R-1}$, the limit law for $T_{N,2}(\kl_{1,R-1})$ on the boundary $C^b_2(\kl_{1,R-1})$ is obtained for the {\it entire} range of sampling sparsity levels $\al\in (0,1)$. In case of other graphs $H$ with $\tau(H)>1$, our results are available for $\al\in (0,\alpha^\star_H)$, which provides a {\it restricted} range of sampling sparsity levels. See Theorem \ref{thm-2-main-temp}(b) for more details. The existence and nature of limiting distributions of $T_{N,2}(H)$ on the boundary $C^b_2(H)$, for $\tau(H)>1$ and $\al\geq \al^\star_H$, is not yet known and needs to be investigated. Theorem \ref{thm-2-main-temp}(b) is obtained by using a well know total-variation distance bound for Poisson approximation (see Theorem 6.23 of \cite{janson-rucinski-randomgraphs}). The solid red line on top of the sub-region $C_{2,2}(\kl_{1,3})$ in Figure \ref{fig:main:star}(b), displays the boundary region $C^b_2(\kl_{1,3})$ for $\al\in (0,1)$. The solid red line on top of the sub-region $C_{2,3}(\kl_3)$ in Figure \ref{fig:main:tri}(b) displays the boundary region $C^b_2(\kl_{3})$ for $\al\in (0,\al^\star_{\kl_3})$, where $(\al^\star_{\kl_3},\beta^\star_{\kl_3}) = (3/4,1/2)$.

\item[(iii)] {\bf Extreme case ($\al=1$):} At this sampling sparsity level, in the induced case, only a finite amount of data is observed, even when the population size $N\rai$. This is true at all levels of model sparsity. This implies, there is no possibility of meaningful statistical inference using extremely sparse sampling and induced network formation, for any choice of $H$. 

The ego-centric approach provides an opportunity to obtain inferential conclusions even when $\al=1$. However, our existing results are limited in scope. We have been able to establish a third type of Poisson based limit law for $T_{N,2}(\kl_{1,R-1})$ on $C^b_2(\kl_{1,R-1})$, when the model sparsity level $\beta\in [0,1)$. This limit law is expressible as linear combination of $K$ independent Poisson random variables (see Theorem \ref{thm-2-main-temp}(c)) and is obtained by a characteristic function argument. We are unable to derive the limit law of $T_{N,2}(\kl_{1,R-1})$ at the corner point $(\al,\beta)=(1,1)$ (see Figure \ref{fig:main:star}(b)).

For $\tau(H)>1$, the existence and nature of limiting distributions of $T_{N,2}(H)$ on the boundary $C^b_2(H)$, at $\al=1$, is not yet known and needs to be investigated. Refer to Figure \ref{fig:main:tri}(b), which displays a region $U$, where the nature of limit laws of $T_{N,2}(H)$ remains unknown.
\end{enumerate}

Now we state the main result on Poisson based limit laws for $T_{N,l}(H)$, $l=1,2$, at boundary points of $C_1(H)$ and $C_2(H)$.

\begin{theorem}[Poisson based limit laws for $T_{N,l}(H)$]
\label{thm-2-main-temp}
Suppose that Assumptions \ref{assump-1}, \ref{assump-2}, \ref{assump-3} and \ref{assump-4} hold. Assume $H$ is a strictly balanced target subgraph with $m(H) = T/R$ (cf. \eqref{mtH-def}).   
\begin{enumerate}
\item[(a)] {\bf Dense sampling:} At $\al=0$, $p_N =c\in (0,1)$ for all $N\geq 1$ (cf. \eqref{pN-def}), and the only boundary point in $C_l(H)$, $l=1,2$, is $(\al,\beta) = (0,R/T)$, for both $l=1,2$. Recall the definitions of $f_l(\cdot:H)$, $l=1,2$,  $\Cb$ and $\{\la_k:k=1,\ldots,K\}$ from \eqref{f2-def}, \eqref{PiN-lim} and Assumption \ref{assump-2} respectively. Define the constant
\begin{align}
\kappa(\Cb,\la_1,\ldots,\la_K:H) & = \sum_{u_1,\ldots,u_R\in [K]}\la_{u_1}\cdot \la_{u_2}\cdots \la_{u_R}\cdot \prod_{\{i,j\}\in E(H)} c_{u_i,u_j}.
\label{kap-def}
\end{align}
For $l=1$ and $l=2$, we define two independent Poisson random variables $U_{1,l,H}$ and $U_{2,l,H}$ with means $f_l(c:H)\cdot \kappa(\Cb,\la_1,\ldots,\la_K:H)$ and $\left\{1-f_l(c:H)\right\}\cdot \kappa(\Cb,\la_1,\ldots,\la_K:H)$, respectively. Then, 
\begin{align}
T_{N,l}(H) &\darw \frac{\left(1-f_l(c:H)\right)\cdot U_{1,l,H} - f_l(c:H)\cdot U_{2,l,H}}{[f_l(c:H)\cdot\{1-f_l(c:H)\}\cdot \kappa(\Cb,\la_1,\ldots,\la_K:H)]^{1/2}},\quad\text{for $l=1,2$.}
\label{lim-pois-case-1}
\end{align}

\item[(b)] {\bf Intermediately sparse sampling:} Consider the following regions of boundary values within $C_1(H)$ and $C_2(H)$, corresponding to intermediate sampling sparsity levels:
\begin{equation}
\begin{aligned}
F_{1}(H) &= \left\{(\al,\beta): 0<\al<1,~\beta = \frac{R\cdot(1-\al)}{T}\right\},\quad\text{and}\\
F_2(H) &= \left\{(\al,\beta): 0 <\al < \widetilde{\al}_H,~\beta = \left(1-\frac{\tau(H)\cdot\al}{R}\right)\cdot\frac{R}{T}\right\},
\end{aligned}
\label{F1F2-set}
\end{equation}
where, $\widetilde{\al}_H$ is defined in \eqref{al-tild-def}. In the ego-centric case (for $l=2$), assume that $H$ satisfies Assumption \ref{assump-5}. Recall that $p_N = c/N^{\al}$ (cf. \eqref{pN-def}). Define the maps, $\phi_1(c:H) = c^R$ and $\phi_2(c:H)=c^{\tau(H)}$, for any $c>0$. Define two Poisson random variables $V_{1,H}$ and $V_{2,H}$, with means $\phi_1(c:H)\cdot \kappa(\Cb,\la_1,\ldots,\la_K:H)$ and $\phi_2(c:H)\cdot \kappa(\Cb,\la_1,\ldots,\la_K:H)$, respectively, where $\kappa(\Cb,\la_1,\ldots,\la_K:H)$ is defined in \eqref{kap-def}. Then,
\begin{equation}
T_{N,l}(H) \darw \frac{V_{l,H}-\phi_l(c:H)\cdot \kappa(\Cb,\la_1,\ldots,\la_K:H)}{\left[\phi_l(c:H)\cdot \kappa(\Cb,\la_1,\ldots,\la_K:H)\right]^{1/2}},\quad \text{if $(\al,\beta)\in F_l(H)$, for both $l=1,2$.}
\label{lim-pois-case-2}
\end{equation}

\item[(c)] {\bf Extremely sparse sampling with $\tau(H)=1$:} Consider $H = \kl_{1,R-1}$, with $\tau(H)=1$ and assume that $\al=1$ and $\beta\in [0,1)$. Let $Z_1,\ldots,Z_K$ denote a collection of independent Poisson random variables with means $\la_1,\ldots,\la_K$, respectively (see Assumption \ref{assump-2}). Then
\begin{equation}
\begin{aligned}
&T_{N,2}(H)\darw \frac{1}{[\nu(\Cb,\la_1,\ldots,\la_K:H)]^{1/2}}\sum_{u_1=1}^K(Z_{u_1} -c\cdot\la_{u_1})\cdot \left(\sum_{u_2,\ldots,u_R\in[K]} c_{u_1,u_2}\cdots c_{u_1,u_R}\times \la_{u_2}\cdots \la_{u_R}\right),\\
&\text{where,}\quad\nu\left(\Cb,\la_1,\ldots,\la_K:H\right)=
c\cdot\sum_{u_1=1}^K\la_{u_1}\left(\sum_{u_2,\ldots,u_R\in [K]}c_{u_1,u_2}\cdots c_{u_1,u_R}\times \la_{u_2}\cdots \la_{u_R}\right)^2.
\end{aligned}
\label{nu-def}
\end{equation}

\end{enumerate}
\end{theorem}

\begin{rem}[Limiting behavior of $T_{N,l}(H)$ beyond the boundaries of $C_l(H)$]
\label{rem:out-bound}
An important question is, what happens to the limit law of $T_{N,l}(H)$, $l=1,2$, when the sparsity levels $(\al,\beta)$ go beyond the boundaries of $C_l(H)$, $l=1,2$? We are able to provide partial answers to this question, and we focus only on strictly balanced $H$. Firstly, in the induced case, it can be shown that, $\widehat{S}_{N,1}(H) - f_1(p_N:H)\cdot S_N(H) \parw 0$, if $\al+m(H)\cdot \beta>1$ and $\al\in [0,1]$. For example, in Figures \ref{fig:main:tri}(a) and \ref{fig:main:star}(a), this is the region above the solid red line. Similarly, in the ego-centric case, for $\tau(H)=1$ (star graph), we can show $\widehat{S}_{N,2}(\kl_{1,R-1}) - f_2(p_N:\kl_{1,R-1})\cdot S_N(\kl_{1,R-1})\parw 0$, if ${\al}/{R}+(R-1)\cdot{\beta}/R>1$ and $\al\in[0,1]$. In Figure \ref{fig:main:star}(b), this is the region above the solid red line, including the region above the point $(\al,\beta)=(1,1)$. The situation becomes complicated if $\tau(H)>1$. In this case if we consider $\al\in [0,\al^\star_H)$ (cf. \eqref{al-star-def}), then the boundary of $C_2(H)$, on this range of $\al$ values is the line, $\tau(H)\cdot \al + T\cdot\beta = R$. It can be shown 
$$
\widehat{S}_{N,2}(H)-f_2(p_N:H)\cdot S_N(H) \parw 0,\quad \text{if $\tau(H)\cdot\al + T\cdot\beta>R$ and $\al\in [0,\al^\star_H)$.}
$$
If $\al\in [\al^\star_H,1]$ and $(\al,\beta)$ exceed (or are on) the boundary of $C_2(H)$, then the limit law of $T_{N,2}(H)$ is currently unknown (provided it exists). For example, this scenario arises on the region $U$, shown in Figure \ref{fig:main:tri}(b) for $H = \kl_3$. If $(\al,\beta)$ keeps increasing further, so that the point $(\al,\beta)$ is `higher' than all the $(R-1)$ lines provided in \eqref{line-set}, then $\widehat{S}_{N,2}(H)-f_2(p_N:H)\cdot S_N(H)\parw 0$, irrespective of the value of $\al$. In Figure \ref{fig:main:tri}(b), this happens when $(\al,\beta)$ is above the boundary of $C_{2,3}(\kl_3)$ and also above the region $U$.

\end{rem}

\begin{rem}[Necessary conditions for a Gaussian limit law]
\label{rem:necessity}
In the induced case and in the ego-centric case for the star graph, the thresholds on the sparsity levels, provided in $C_1(H)$ and $C_2(\kl_{1,R-1})$ are not only sufficient, but they are also {\it necessary} to achieve a Gaussian limit law. This can be clearly seen from the shape of the boundary regions, where Poisson based limit laws are found. In the ego-centric case, if $\tau(H)>1$, the thresholds in $C_2(H)$ are also necessary (if we leave out the case of $\al\geq \al^\star_H$, see \eqref{al-star-def}). Even then, in the case of the triangle ($H=\kl_3$), we conjecture that $T_{N,2}(\kl_3)$  {\it cannot} have a Gaussian limit law on the region $U$ (see Figure \ref{fig:main:tri}(b)). \cite{janson-rucinski-randomgraphs} exhibit similar sharp transitions from Gaussian to Poisson based limit laws for the population subgraph count $S_N(H)$ in the case of Erd\H{o}s-R\'{e}nyi random graphs, as the model sparsity level $\beta$ increases. However, it is remarkable to find that a similar sharp transition is observed even in the case of estimated subgraph counts, for both induced and ego-centric approaches, as either one or both of the model and sampling sparsity levels increases.
\end{rem}

\begin{rem}[Extensions to other models]
\label{rem:other-model}
Based on the techniques used for proving our main results, we believe that the above described Gaussian and Poisson based limit law results for the estimated subgraph counts can be extended to other population network models, including popular variants of the SBM, such as the degree corrected SBM (\cite{karrer2011stochastic}), or for random dot product graph models (\cite{athreya2018statistical}), at the cost of more complicated technical assumptions. We plan to investigate these issues in a future article. 
\end{rem}
\begin{figure}[tbp]
    \centering
    \begin{minipage}[b]{0.46\textwidth}
        \centering
        \includegraphics[scale = 0.3]{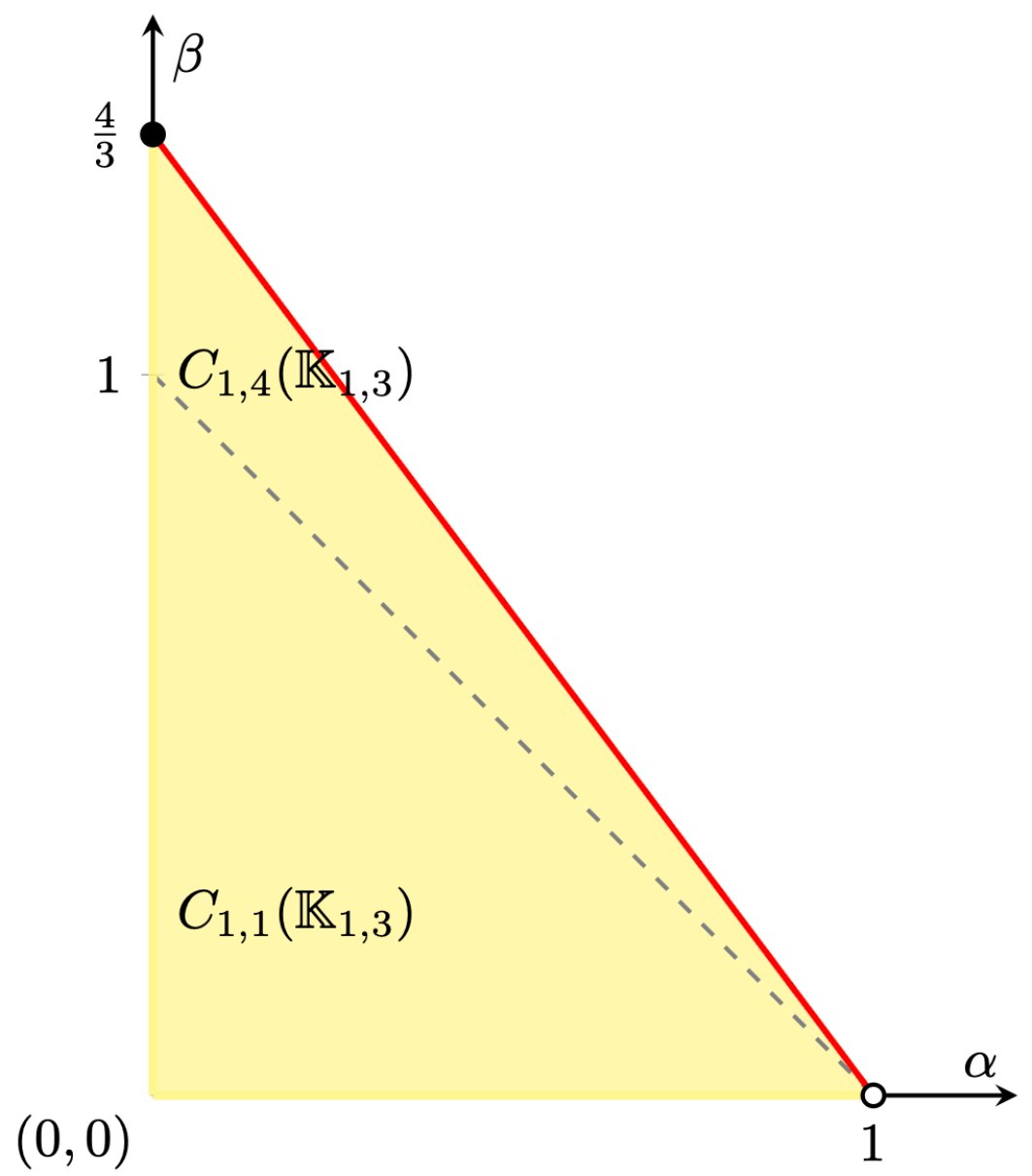}
        \\[-0.3ex]
        \footnotesize (a) Induced case: $C_{1}(\kl_{1,3})$ and its sub-regions.
    \end{minipage}
    \hspace{1.5ex}
    \begin{minipage}[b]{0.46\textwidth}
        \centering
        \includegraphics[scale = 0.3]{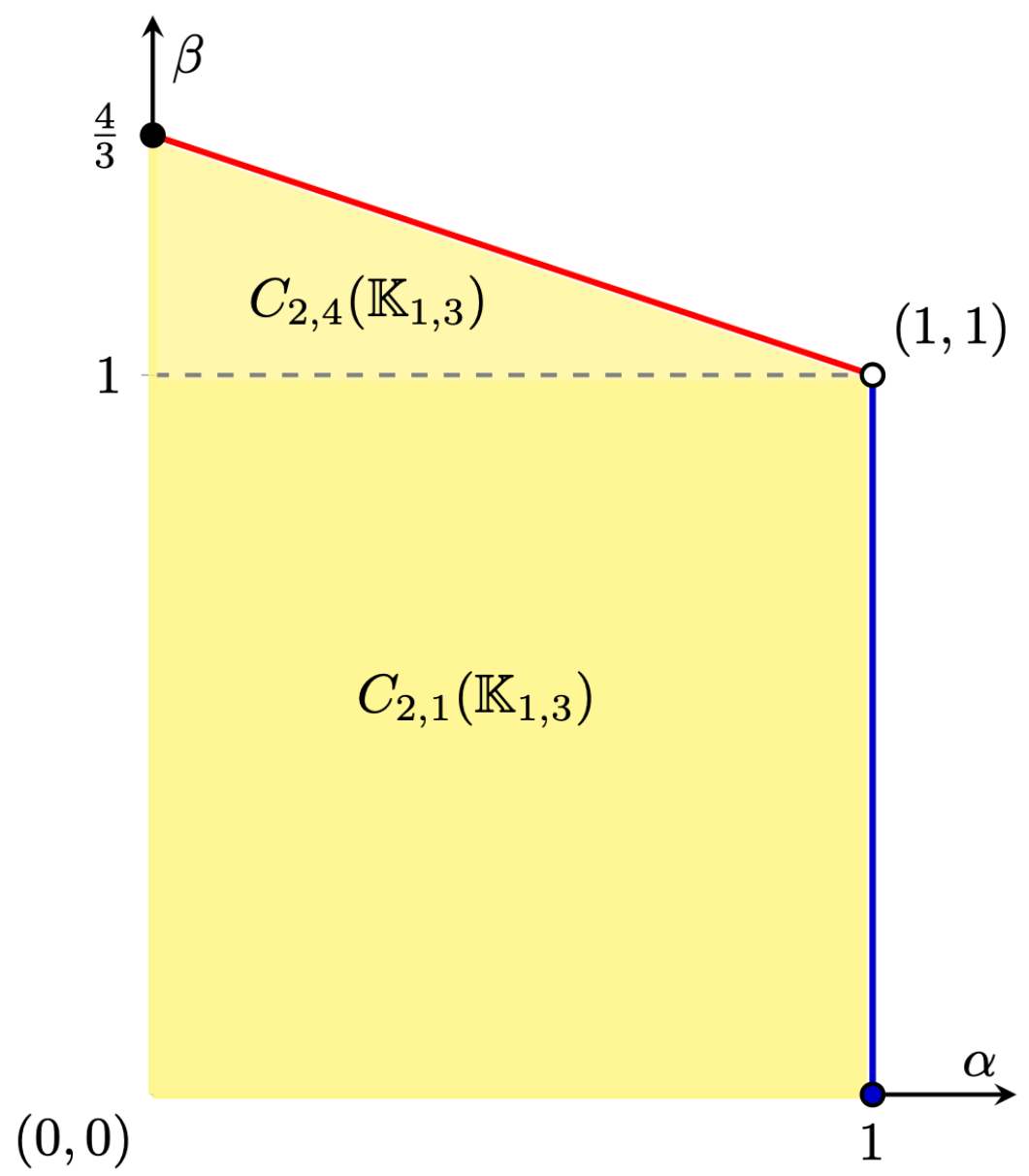}
        \\[-0.3ex]
        \footnotesize (b) Ego-centric case: $C_2(\kl_{1,3})$ and its sub-regions.
    \end{minipage}

    \caption{{\small{{\bf Sparsity levels and limit laws for estimated star graph counts:} $T_{N,1}(\kl_{1,3})$ and $T_{N,2}(\kl_{1,3})$ have a Gaussian limit on $C_1(\kl_{1,3})$ and $C_2(\kl_{1,3})$ (yellow shaded regions), with different decay rates for $\Delta_{N,l}(\kl_{1,3};\al,\beta)$, $l=1,2$, within each sub-region. Similar Poisson based limit laws arise at $(\al,\beta)=(0,R/T)=(0,4/3)$ (marked by filled black dot on both figures) for both $T_{N,l}(\kl_{1,3})$, $l=1,2$, see Theorem \ref{thm-2-main-temp}(a). A different Poisson based limit law arises on the upper boundaries of $C_{1,4}(\kl_{1,3})$ and $C_{2,2}(\kl_{1,3})$ (marked by solid red line), see Theorem \ref{thm-2-main-temp}(b). Another Poisson based limit law arises for $T_{N,2}(\kl_{1,3})$ on the line $\al=1$, for $\beta\in[0,1)$, see solid blue line in figure (b) and Theorem \ref{thm-2-main-temp}(c). Limit law of $T_{N,2}(\kl_{1,3})$ is unknown when $(\al,\beta)=(1,1)$.}}}
    \label{fig:main:star}
\end{figure}

\subsection{Limit laws for the estimated clustering coefficient}
\label{sec:cc}

In this section, we focus on the prediction of the clustering coefficient statistic (see Chapter 4 of \cite{kolaczyk09}) of the population network $G_N$. The clustering coefficient is a widely used measure of cohesion within the network. Various researchers have focused on the estimation of the clustering coefficient through network sampling (cf. \cite{bliss-14}, \cite{lee2006}, \cite{illenberger-12}, \cite{hardiman2013}, \cite{iwasaki2018}). To the best our knowledge, the asymptotic distribution of the sample based clustering coefficient estimate has not been studied in the existing literature, under any sort of framework or under any type of sampling scheme. According to \cite{bianconi-2018} (cf. p. 17), the clustering coefficient of $G_N$ can be written as
\begin{align}
\Gamma_N \equiv \frac{3\times \text{(number of triangles in ${G}_N$)}}{\text{(number of distinct paths of length two in ${G}_N$)}} = \begin{cases}
\displaystyle\frac{S_N(\kl_3)}{S_N(\kl_{1,2})} \displaystyle& \text{if $S_N(\kl_{1,2})> 0$,}\\
         0 & \text{otherwise,}.
         \end{cases}
         \label{CC-def}
\end{align}
where $S_N(\kl_{1,2})$ and $S_N(\kl_3)$ denote the number of wedges and triangles in the population network $G_N$. The induced and ego-centric sample based estimates of the clustering coefficient are defined as
\begin{align}
    \widehat{\Gamma}_{N,l} = 
    \begin{cases}        \displaystyle\frac{\widehat{S}_{N,l}(\kl_3)/{f_l(p_N:\kl_3)}}{\widehat{S}_{N,l}(\kl_{1,2})/{f_l(p_N:\kl_{1,2})}} \displaystyle &\quad \text{if $\widehat{S}_{N,l}(\kl_{1,2})>0$,}\\
         0&\quad \text{otherwise,}
    \end{cases},\quad\text{for $l=1,2$.}
\label{CC-est}    
\end{align}
where, $\widehat{S}_{N,l}(\kl_{1,2})/f_l(p_N:\kl_{1,2})$ and $\widehat{S}_{N,l}(\kl_3)/f_l(p_N:\kl_3)$, $l=1,2$, (cf. \eqref{s-all-main}) are Horvitz-Thompson estimators of $S_N(\kl_{1,2})$ and $S_N(\kl_3)$, $f_l(p_N:\kl_{1,2})$ and $f_l(p_N:\kl_3)$ have been defined in \eqref{fl-def}. Define the map $\psi(x,y) = x/y$, for all $x\in\rl$ and $y\neq 0$. The design unbiasedness of the Horvitz Thompson estimator implies $\E(\widehat{S}_{N,l}(H)) = f_l(p_N:H)\cdot\E(S_N(H))$, for any choice of $H$. Thus, if $\widehat{S}_{N,l}(\kl_{1,2})>0$, then we can use Taylor expansion, to write 
\begin{align}
&N^\beta\cdot\left(\widehat{\Gamma}_{N,l}-{\Gamma}_N \right)= \frac{a_{N,l}\cdot\sigma_{N,l}(\kl_3)}{N^{(3-3\beta)}\cdot f_l(p_N:\kl_3)}\cdot T_{N,l}(\kl_3) + \frac{b_{N,l}\cdot\sigma_{N,l}(\kl_{1,2})}{N^{(3-2\beta)}\cdot f_l(p_N:\kl_{1,2})}\cdot T_{N,l}(\kl_{1,2}) + R_N,
\label{lin-comb}
\end{align}
where, $a_{N,l}$ and $b_{N,l}$ depend on the model parameters, and involve the partial derivatives of $\psi(x,y)$, $R_N$ is a remainder term, $T_{N,l}(H)$ and $\sigma_{N,l}(H)$, $l=1,2$, are defined in \eqref{TN-def-12} and \eqref{sig-temp-0}. The limit law of $\left(\widehat{\Gamma}_{N,l}-\widehat{\Gamma}_N\right)$, $l=1,2$, is decided by the limiting distribution of the linear combination of $T_{N,1}(\kl_{1,2})$ and $T_{N,l}(\kl_{3})$ in \eqref{lin-comb}. As it turns out, for certain values of $(\al,\beta)$, the triangle based term in this linear combination becomes the stochastically dominant term, and the wedge based term has negligible contribution. In the induced case, this happens if $(\al,\beta)\in R_{1,3}\cup R_{1,4}\cup R_{1,5}$ and in the ego-centric case, the same situation arises if $(\al,\beta)\in R_{2,3}\cup R_{2,4}\cup R_{2,5}\cup R_{2,6}$ (see blue shaded area in Figures \ref{fig:main:cc}(a) and \ref{fig:main:cc}(b)). For all other choices of $(\al,\beta)$, both the triangle and wedge based terms in \eqref{lin-comb} contribute equally to the limiting behavior of this linear combination. In the induced case, this happens if $(\al,\beta)\in R_{1,1}\cup R_{1,2} = C_{1,1}(\kl_3)$ (see Figure \ref{fig:main:cc}(a) and Figure \ref{fig:main:tri}(a)). Similarly in the ego-centric case, the same situation is observed if $(\al,\beta)\in R_{2,1}\cup R_{2,2} = C_{2,1}(\kl_3)$ (see Figure \ref{fig:main:cc}(b) and Figure \ref{fig:main:star}(b)). If the $(\al,\beta)$ pair increases further, beyond the regions shown in Figures \ref{fig:main:cc}(a) and \ref{fig:main:cc}(b) (in the induced and ego-centric cases, respectively), then the triangle based term in the linear combination (in \eqref{lin-comb}) continues to be the dominant term, but it does not lead to any non-degenerate limit law. Note that the wedge ($\kl_{1,2}$) and triangle ($\kl_3$) graphs, which are involved in the definition of the clustering coefficient, satisfy Assumption \ref{assump-5}. In order to state our main result, we define the following regions of sparsity levels,

\begin{figure}[htbp]
    \centering
    \begin{minipage}[b]{0.48\textwidth}
        \centering
        \includegraphics[scale = 0.28]{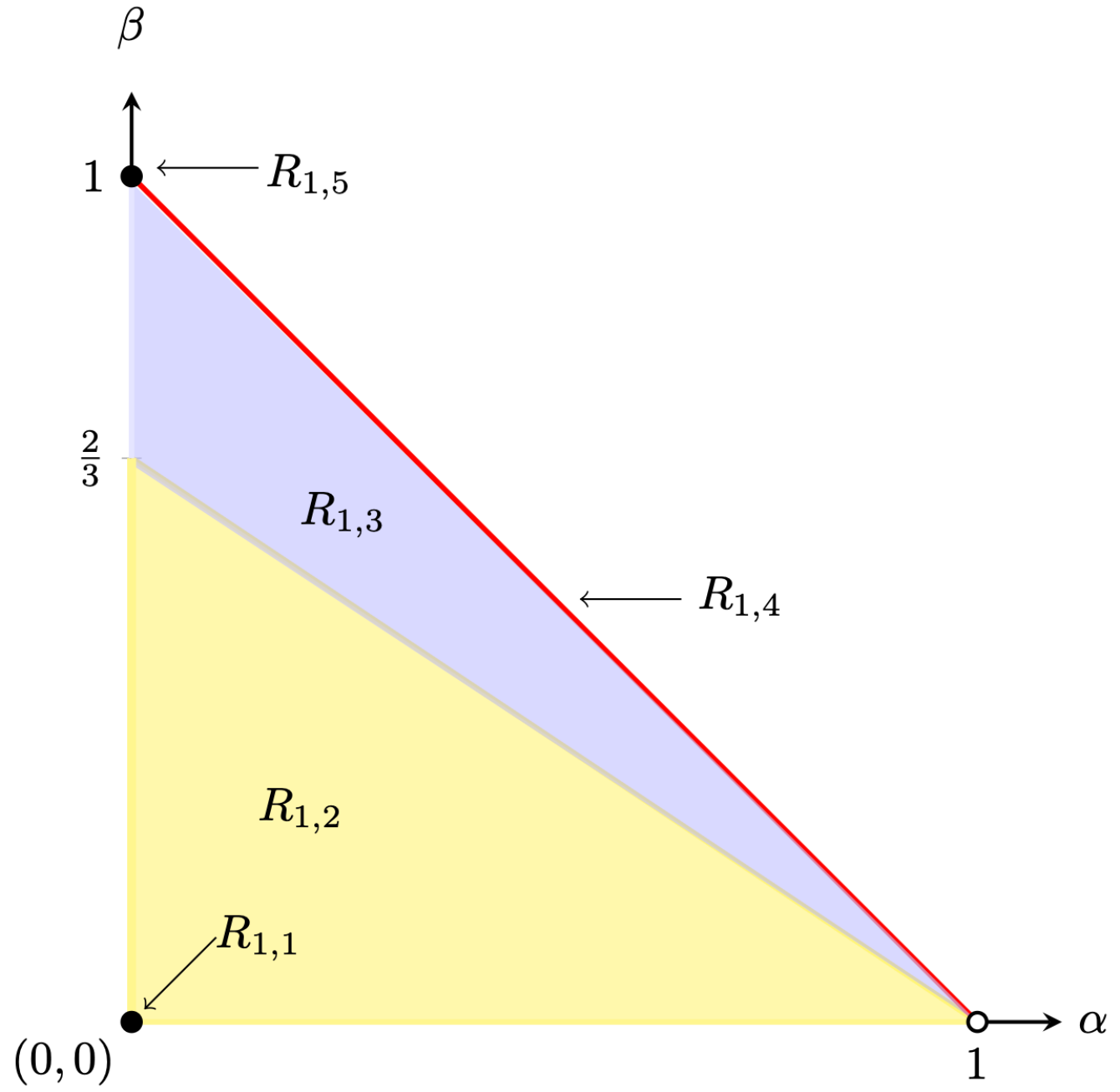}
        \\[-0.3ex]
        \footnotesize (a)  Sub-regions $R_{1,j}$, $j=1,\ldots,5$, in the induced case.
    \end{minipage}
    \hspace{1.5ex}
    \begin{minipage}[b]{0.48\textwidth}
        \centering
        \includegraphics[scale = 0.28]{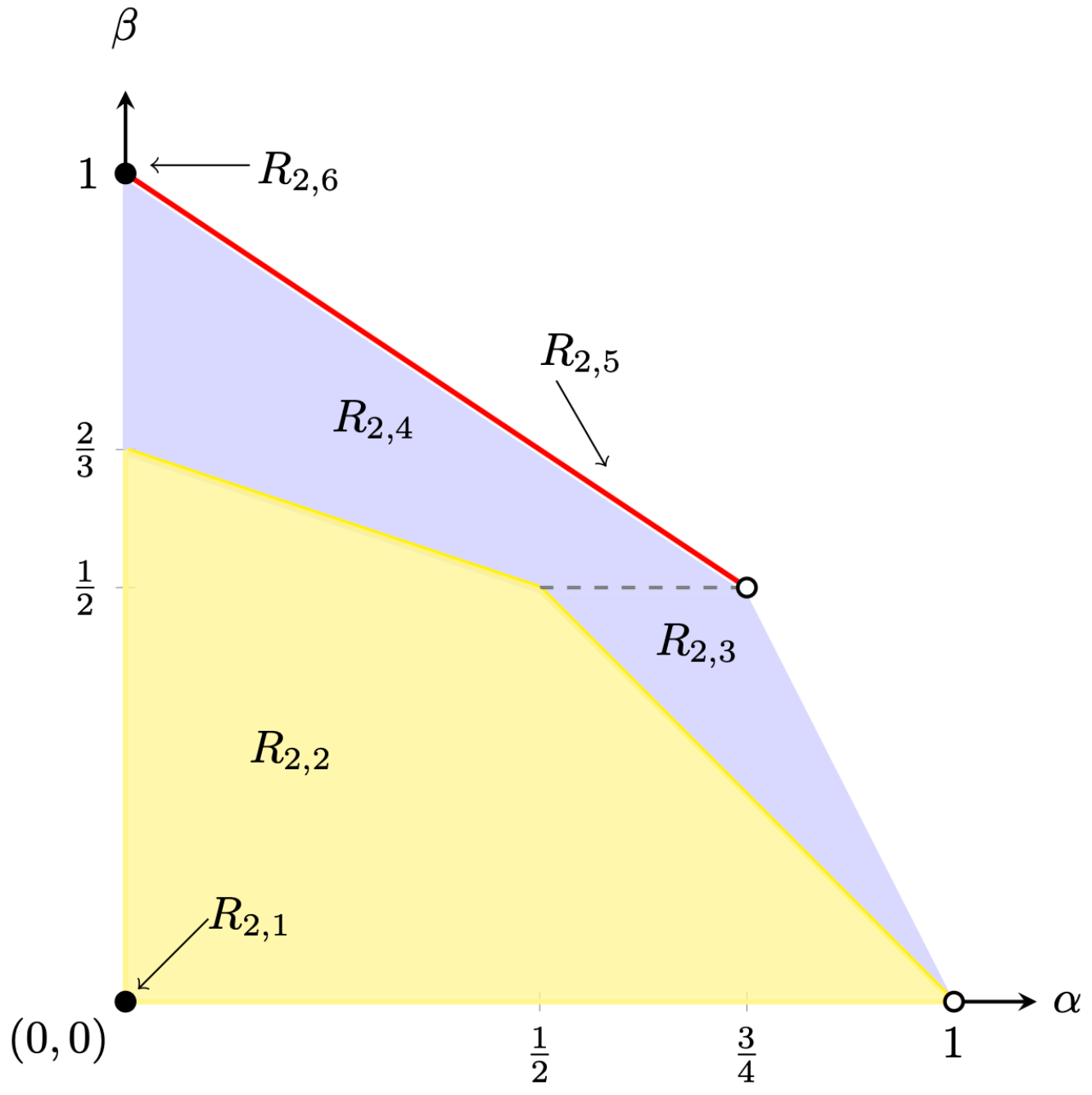}
        \\[-0.3ex]
        \footnotesize (b) Sub-regions $R_{2,j}$, $j=1,\ldots,6$, in the ego-centric case.
    \end{minipage}
    \caption{Figure showing various sub-regions arising during the estimation of the clustering coefficient (see \eqref{rsets-def} for details). The point $(\al,\beta)=(0,0)$ and all yellow shaded areas give rise to Gaussian limit laws which are influenced by both wedge and triangle counts. Blue shaded areas give rise to Gaussian limit laws which are influenced only by the triangle counts. Poisson based limit laws arise on the red solid lines and on the black filled dot at $(\al,\beta)=(0,1)$. See Theorems \ref{thm-cc-ind} and \ref{thm-cc-ego} for details.}
    \label{fig:main:cc}
\end{figure}

\begin{small}

\begin{equation}
\label{rsets-def}
\left. \begin{aligned}
        &R_{1,1} = R_{2,1} = \left\{(\al,\beta):\al=\beta=0\right\},\\
        &R_{1,2} = \left\{(\al,\beta): \al \in (0,1),~\al+\frac{\beta}{2/3}\leq 1\right\},~R_{2,2} = \left\{(\al,\beta):\al\in (0,1),\frac{\al}{2}+\frac{\beta}{2/3}\leq 1 ,~\al+\beta\leq 1\right\},\\
        &R_{1,3} =\left\{(\al,\beta):\al\in(0,1),~\al+\frac{\beta}{2/3}>1,\al+\beta <1 \right\},~R_{2,3} =\left\{(\al,\beta):\beta<\frac{1}{2},\al+\beta>1, \al+\frac{\beta}{2}<1\right\},\\
        &R_{1,4} = \left\{(\al,\beta):\al\in (0,1),~\al+\beta =1\right\},~R_{2,4} = \left\{(\al,\beta): \beta\geq \frac{1}{2},\frac{\al}{2}+\frac{\beta}{2/3}>1, \frac{\al}{3/2}+\beta<1\right\},\\
        &R_{1,5} = \left\{(\al,\beta) : \al = 0,\beta = 1\right\}, ~ ~R_{2,5} = \left\{(\al,\beta): \al\in \left(0,3/4\right), \frac{\al}{3/2} +\beta =1\right\}\quad\text{and}\\
        &R_{2,6} = \left\{(\al,\beta) : \al = 0,\beta = 1\right\}.
    \end{aligned}\right\}
\end{equation}
\end{small}
Recall the definition of $\Cb=((c_{u,v}:1\leq u,v\leq K))$ and $\la_1,\ldots,\la_K$ from \eqref{PiN-lim} and Assumption \ref{assump-2}. Define the constants,
\begin{align}
    &\theta_1 = \sum_{u,v,w=1}^K c_{u,v}c_{u,w}\cdot \la_u\la_v\la_w,\quad\text{and}\quad\theta_2=  \sum_{u,v,w=1}^Kc_{u,v}c_{u,w}c_{v,w}\cdot \la_u\la_v\la_w.
    \label{th1-th2-def}
\end{align}
Recall the definition of $p_N$ from \eqref{pN-def}. Note, when $\al = 0$, $p_N = c$, where $c\in (0,1)$. When $\al>0$, then $p_N = c/N^\al$, where $c>0$. We define the following asymptotic variance and asymptotic mean expressions, which are associated with limiting Gaussian and limiting Poisson based distributions, arising on the regions $R_{1,j}$, $j=1,\ldots,5$ (cf. \eqref{rsets-def}), in the induced case:
\begin{equation}
\left.\begin{aligned}
&\left[\rho_{1,1}(\kl_{1,2},\kl_{3})\right]^2 = \sum_{u,v,w,x,y=1}^K\pi_{u,v}\pi_{u,w}\cdot\left\{\frac{\theta_2^2}{\theta_1^4}\cdot \frac{c^5(1-c)}{c^6}\cdot \left\{\pi_{u,x}\pi_{u,y}  + 4\pi_{u,x}\pi_{x,y} + 4\pi_{v,x}\pi_{x,y}\right\}+\right.\\
            &\quad \left. {}+\frac{1}{\theta_1^2}\cdot \frac{9c^5(1-c)}{c^6}\cdot \pi_{v,w}\pi_{u,x}\pi_{x,y}\pi_{y,u}-\frac{2\theta_{2}}{\theta_1^3}\cdot \frac{c^5(1-c)}{c^6}\cdot\left\{3\pi_{u,x}\pi_{x,y}\pi_{y,u} + 6\pi_{v,x}\pi_{x,y}\pi_{y,v}\right\}\right\}\cdot\la_{u}\la_{v}\la_{w}\la_{x}\la_{y},
\end{aligned}\right.
\label{var-rho-1-1}
\end{equation}
\begin{align}
&\left[\rho_{1,2}(\kl_{1,2},\kl_{3})\right]^2 = \frac{1}{c \theta_1^2}\sum_{u,v,w,x,y=1}^Kc_{u,v}c_{u,w}\left\{\frac{\theta_2^2}{\theta_1^2}\left\{c_{u,x}c_{u,y}  + 4c_{u,x}c_{x,y} + 4c_{v,x}c_{x,y}\right\}+\right.\notag\\
            &\left. {}+ 9\cdot c_{v,w}c_{u,x}c_{x,y}c_{y,u}-\frac{2\theta_{2}}{\theta_1}\cdot\left\{3c_{u,x}c_{x,y}c_{y,u} + 6c_{v,x}c_{x,y}c_{y,v}\right\}\right\}\cdot\la_{u}\la_{v}\la_{w}\la_{x}\la_{y},
\label{var-rho-1-2}            
\end{align}
\begin{equation}
\left[\rho_{1,3}(\kl_3)\right]^2 = \frac{\theta_2}{c^3 \theta_{1}^2}
\quad\text{and}\quad
\kappa_{1,4}(\kl_3)  = c^3\theta_2.
\label{var-rho-1-3-kap-1-4}
\end{equation}
The subscript $j$ (in $\rho_{1,j}$ or $\kappa_{1,j}$) in the above expressions indicate that these asymptotic variance or mean expressions are associated with the limit law of the estimated clustering coefficient on the region $R_{1,j}$, for $j=1,\ldots,5$ (cf. \eqref{rsets-def}). Similarly, the use of both the triangle ($\kl_3$) and wedge ($\kl_{1,2}$) symbols indicates that the Gaussian limit law is influenced by both the triangle and wedge based terms in \eqref{lin-comb}, while if only the triangle symbol is used, it indicates that the limit law is influenced only by the triangle based term in \eqref{lin-comb}. Similar notational convention will be used in the ego-centric case. Now we describe the main result about the limit law for the estimated clustering coefficient in the induced case. 

\begin{theorem}[Limit laws for the estimated clustering coefficient in the induced case]
\label{thm-cc-ind}
Suppose that Assumptions \ref{assump-1}, \ref{assump-2}, \ref{assump-3} and \ref{assump-4} hold. Recall the definitions of $\theta_{1,N}$ and $\theta_{2,N}$ in \eqref{th1-th2-def} and the asymptotic variance expressions $[{\rho_{1,1}(\kl_{1,2},\kl_3)}]^2$, $[{\rho_{1,2}(\kl_{1,2},\kl_3)}]^2$, $[{\rho_{1,3}(\kl_3)}]^2$, from \eqref{var-rho-1-1}, \eqref{var-rho-1-2} and \eqref{var-rho-1-3-kap-1-4}. Define the independent Poisson random variables, $U_{1,4}$, $U_{1,5,1}$ and $U_{1,5,2}$, with means $\kappa_{1,4}(\kl_3)$ (cf. \eqref{var-rho-1-3-kap-1-4}), $c^3\cdot\kappa_{1,4}(\kl_3)$ and $(1-c^3)\cdot \kappa_{1,4}(\kl_3)$, respectively, where $c$ is defined in \eqref{pN-def}. Recall the definitions of the regions $R_{1,j}$, $j=1,\ldots,5$, in \eqref{rsets-def}. Then,
\begin{align}
 r_{N,1}(\al,\beta)\cdot\left(\widehat{\Gamma}_{N,1}-\Gamma_N\right)\darw V_{1,\al,\beta},   
 \label{ll-cc-ind}
\end{align}
where,
\begin{equation}
 \left.
 \begin{aligned}
 r_{N,1}(\al,\beta) =\sqrt{N}\quad &\text{and}\quad V_{1,\al,\beta} \stackrel{d}{=} N\left(0,{\left[\rho_{1,1}(\kl_{1,2},\kl_{3})\right]}^2\right),&\text{if $(\al,\beta)\in R_{1,1}$,}\\
 r_{N,1}(\al,\beta) = N^{\displaystyle\beta + (1-\al)/2\displaystyle}\quad &\text{and}\quad V_{1,\al,\beta} \stackrel{d}{=} N\left(0,\left[\rho_{1,2}(\kl_{1,2},\kl_{3})\right]^2\right),&\text{if $(\al,\beta)\in R_{1,2}$,}\\
 r_{N,1}(\al,\beta) = N^{\beta + 3(1-\al-\beta)/2}\quad &\text{and}\quad V_{1,\al,\beta} \stackrel{d}{=} N\left(0,{\left[\rho_{1,3}(\kl_3)\right]}^2\right),&\text{if $(\al,\beta)\in R_{1,3}$,}\\
 r_{N,1}(\al,\beta) = N^\beta\quad &\text{and}\quad V_{1,\al,\beta}\stackrel{d}{=}\frac{U_{1,4}-\kappa_{1,4}(\kl_3)}{\theta_1},&\text{if $(\al,\beta)\in R_{1,4}$,}\\
 r_{N,1}(\al,\beta) = N\quad&\text{and}\quad V_{1,\al,\beta}\stackrel{d}{=} \frac{(1-c^3)\cdot U_{1,5,1} - c^3\cdot U_{1,5,2}}{\theta_1},&\text{if $(\al,\beta)\in R_{1,5}$.}
 \end{aligned}
 \right\}
\end{equation}
\end{theorem}

Similarly to the induced case, we define some asymptotic variance and mean expressions which will be used to define the limit law of the estimated clustering coefficient in the ego-centric case:
\begin{align}
&{\left[\rho_{2,1}(\kl_{1,2},\kl_{3})\right]}^2\notag\\
&= \sum_{\substack{u,v,\\w,x,y=1}}^K\pi_{u,v}\pi_{u,w}\left\{\frac{\theta_2^2}{\theta_1^4}\cdot \frac{\left\{c(1 - c)^3(1+c)^2\cdot\pi_{u,x}\pi_{u,y} + 
4c^2{(1 - c)}^3(1 + c)\cdot \pi_{u,x}\pi_{x,y} +
4c^3{(1 - c)}^3\cdot\pi_{v,x} \pi_{x,y}\right\}}{{\left(c +(1-c)\cdot c^2\right)}^2}\right.\notag\\
&\left. {}+\frac{1}{\theta_1^2}\cdot\frac{36c^3{(1-c)}^3\cdot\pi_{v,w}\pi_{u,x}\pi_{x,y}\pi_{y,u}}{{(3c^2 -2c^3)}^2}\right.\notag\\
&\left.{}-\frac{2\theta_{2}}{\theta_1^3}\cdot\frac{\left\{6c^2{(1 - c)}^3(1+ c)\cdot\pi_{u,x}\pi_{x,y}\pi_{y,u} +
12c^3{(1 - c)}^3\cdot\pi_{v,x} \pi_{x,y}\pi_{y,v}\right\}}{(c+c{(1-c)}^2)\cdot(3c^2-2c^3)}\right\}\cdot\la_{u}\la_{v}\la_{w}\la_{x}\la_{y},
\label{var-rho-2-1}
\end{align}

\begin{align}
&{\left[\rho_{2,1}(\kl_{1,2},\kl_{3})\right]}^2= \frac{1}{c\theta_1^2}\sum_{\substack{u,v,\\w,x,y=1}}^Kc_{u,v}c_{u,w}c_{u,x}c_{u,y}\left\{\frac{\theta_2^2}{\theta_1^2}+4\cdot c_{v,w}c_{x,y}-\frac{4\theta_{2}}{\theta_1}c_{x,y}\right\}\cdot\la_{u}\la_{v}\la_{w}\la_{x}\la_{y},
\label{var-rho-2-2}
\end{align}
\begin{equation}
{\left[\rho_{2,3}(\kl_3)\right]}^2 = \frac{1}{9 c^2\theta_1^2}\cdot\sum_{u,v,w,x=1}^Kc_{u,v}c_{v,w}c_{w,u}c_{u,x}c_{w,x}\cdot\la_u\la_v\la_w\la_x,\quad\text{and}\quad {\left[\rho_{2,4}(\kl_3)\right]}^2 =\frac{\theta_2}{3c^2 \theta_1^2}.
\label{var-rho-2-3-2-4}
\end{equation}

\begin{theorem}[Limit laws for the estimated clustering coefficient in the ego-centric case]
\label{thm-cc-ego}
Suppose that Assumptions \ref{assump-1}, \ref{assump-2}, \ref{assump-3} and \ref{assump-4} hold. Recall the definitions of $\theta_{1,N}$ and $\theta_{2,N}$ in \eqref{th1-th2-def} and the asymptotic variance expressions ${[\rho_{2,1}(\kl_{1,2},\kl_3)]}^2$, ${[\rho_{2,2}(\kl_{1,2},\kl_3)]}^2$, ${[\rho_{2,3}(\kl_3)]}^2$, ${[\rho_{2,4}(\kl_3)]}^2$ defined in \eqref{var-rho-2-1}, \eqref{var-rho-2-2}, \eqref{var-rho-2-3-2-4}. Define the independent Poisson random variables, $U_{2,5}$, $U_{2,6,1}$ and $U_{2,6,2}$, with means $3c^2\cdot\theta_2$, $(3c^2-c^3)\cdot\theta_2$ and $\{1-(3c^2-c^3)\}\cdot\theta_2$, respectively, where $c$ is defined in \eqref{pN-def}. Define the constant $\delta_c = (3c^2-c^3)$. Recall the definitions of the regions $R_{2,j}$, $j=1,\ldots,6$, in \eqref{rsets-def}. Then
\begin{align}
 r_{N,2}(\al,\beta)\cdot\left(\widehat{\Gamma}_{N,2}-\Gamma_N\right)\darw V_{2,\al,\beta},   
 \label{ll-cc-ego}
\end{align}
where,
\begin{equation}
 \left.
 \begin{aligned}
 r_{N,2}(\al,\beta) =\sqrt{N}\quad&\text{and}\quad V_{2,\al,\beta} \stackrel{d}{=} N\left(0,{\left[\rho_{2,1}(\kl_{1,2},\kl_{3})\right]}^2\right),&\text{if $(\al,\beta)\in R_{2,1}$,}\\
 r_{N,2}(\al,\beta) = N^{\displaystyle\beta + (1-\al)/2\displaystyle}\quad &\text{and}\quad V_{2,\al,\beta} \stackrel{d}{=} N\left(0,\left[\rho_{2,2}(\kl_{1,2},\kl_{3})\right]^2\right),&\text{if $(\al,\beta)\in R_{2,2}$,}\\
 r_{N,2}(\al,\beta) = N^{\displaystyle\beta + (2-2\al-\beta)/2\displaystyle}\quad &\text{and}\quad V_{2,\al,\beta} \stackrel{d}{=} N\left(0,{\left[\rho_{2,3}(\kl_3)\right]}^2\right),&\text{if $(\al,\beta)\in R_{2,3}$,}\\
 r_{N,2}(\al,\beta) = N^{\displaystyle\beta+(3-2\al-3\beta)/3\displaystyle}\quad &\text{and}\quad V_{2,\al,\beta}\stackrel{d}{=}N\left(0,{\left[\rho_{2,4}(\kl_3)\right]}^2\right),&\text{if $(\al,\beta)\in R_{2,4}$,}\\ 
 r_{N,2}(\al,\beta) = N^\beta\quad&\text{and}\quad V_{2,\al,\beta}\stackrel{d}{=} \frac{U_{2,5}-3c^2\cdot\theta_2}{\theta_1},&\text{if $(\al,\beta)\in R_{2,5}$,}\\
 r_{N,2}(\al,\beta) = N\quad &\text{and}\quad V_{2,\al,\beta}\stackrel{d}{=} \frac{\{1-\delta_c\}\cdot U_{2,6,1}-\delta_c\cdot U_{2,6,2}}{\theta_1},&\text{if $(\al,\beta)\in R_{2,6}$.}
 \end{aligned}
 \right\}
\end{equation}

\end{theorem}

In the ego-centric case, the limiting distribution of the estimated triangle count remains unknown on the boundary of the region $R_{2,3}$ in Figure \ref{fig:main:cc}(b) (also see the region $U$ in Figure \ref{fig:main:tri}(b)). For this reason, the limit law of the estimated clustering coefficient also remains unknown on this region. The variance expressions obtained in Theorem \ref{thm-cc-ind} and Theorem \ref{thm-cc-ego} are complicated, especially when $\al$ or $\beta$ is zero. This happens due to the various complicated covariance terms, which arise while finding the limiting distribution of the linear combination of the wedge and triangle based terms in \eqref{lin-comb}.

\section{Simulation results and real data analysis}
\label{sec:sim}

If $T_{N,l}(H)\darw F$, where $F$ is a c.d.f., then we can construct the following asymptotic two-sided level $(1-\eta)$ {\it prediction interval} (PI) for the unknown population subgraph density $N^{-R}\cdot S_N(H)$ (where $R = |V(H)|$),
\begin{align}
\label{pred-int-1}
J_{\eta,l,F}(H)\equiv \left[\frac{\left(\widehat{S}_{N,l}(H)-\xi_{1-\eta/2}\cdot\sigma_{N,l}(H)\right)}{N^R\cdot f_{l}(p_N:H)},\frac{\left(\widehat{S}_{N,l}(H)-\xi_{\eta/2}\cdot\sigma_{N,l}(H)\right)}{N^R\cdot f_l(p_N:H)}\right],
\end{align}
for $l=1,2$, where $\xi_\eta$ denotes the $\eta$-quantile of the distribution $F$, $f_l(p_N:H)$ are defined in  \eqref{fl-def} and $\sigma_{N,l}(H)$ are defined in \eqref{sig-temp-0}. A similar approach can be used to construct an asymptotic two-sided PI for the unknown clustering coefficient $\Gamma_N$. 

In our simulation study, we will study the empirical coverage of $J_{\eta,l,F}(H)$ for various choices of $H$ (edge, wedge and triangle) and do the same with the PI for the clustering coefficient, in both induced and ego-centric cases. We will focus on various model and sampling sparsity levels. The expressions of $\sigma^2_{N,l}(H)$ for the above choices of $H$ are provided in Appendix \ref{sec:app-E}, and the expressions of asymptotic variance for the estimated clustering coefficient is provided in Section \ref{sec:cc}. Appendix \ref{sec:app-E} also contains the variance expressions for $\sigma^2_{N,l}(\kl_{1,R-1})$ and $\sigma^2_{N,l}(\kl_R)$, for $l=1,2$.

The main challenge is to estimate the unknown variance term $\sigma^2_{N,l}(H)$, since it depends on unknown model parameters like the limiting matrix $\Cb$ (cf. \eqref{PiN-lim}), the limiting class size proportions $\la_1,\ldots,\la_K$ (cf. Assumption \ref{assump-2}), the unknown edge probabilities $\pi_{N,u,v}$ (cf. \eqref{pi-mat-def}) and also on the model sparsity level $\beta$. If $F$ is a Poisson based limit law, then $\xi_\eta$ will also depend on these model parameters. Throughout the simulation study, we will assume that the sparsity level of the model ($\beta$) is {\it known}. Similarly, the sampling sparsity level $\al$ will be {\it known}. Since our results show sharp transitions from Gaussian to Poisson-based limit laws as the levels of $(\al,\beta)$ change, even an accurate estimate of $\beta$ (if available), can lead us to the use of an incorrect limit law for obtaining the PI's. In fact, $\beta$ is an extraneous quantity that is used to quantify the sparsity in $G_N$, and it is not a model (SBM) parameter. We construct the proposed PI's (cf. \eqref{pred-int-1}) in two different situations: when all above described model parameters are known, and when all of them are unknown and estimated from the sampled data.

We use the \texttt{Bethe-hessian clustering} algorithm for community detection proposed by \cite{saade2014spectral}, which uses spectral clustering to detect communities within a network.. We obtain an estimate $\widehat{K}$ of $K$ by using the largest estimated class label. Using the estimated class label information $\{\widehat{\al}_i: 1\leq i\leq N, W_{N,i}=1\}$ from sampled nodes, we construct a naive estimate of the class-wise edge probabilities,
\begin{align*}
\widehat{\pi}_{k,m}(l) & = \frac{\sum_{(i,j)\in I_{l,N}} Y_{i,j}\cdot \mathbf{1}\left(\widehat{\al}_i=k,\widehat{\al}_j=m\right)}{\sum_{(i,j)\in I_{l,N}} \mathbf{1}\left(\widehat{\al}_i=k,\widehat{\al}_j=m\right)},\quad\text{for $l=1,2$, for all $1\leq k,m\leq \widehat{K}$,}
\end{align*}
where $I_{l,N}$ is defined in \eqref{Ir-def}. The above mentioned community detection algorithm can not be used on the data set obtained from the ego-centric network formation approach. Thus, in case of $l=2$ (ego-centric), we make use of $\widehat{K}$ and class label estimates obtained from  $\mathcal{Y}_{1,N}$ (cf. \eqref{Ir-def}), but $\widehat{\pi}_{k,m}(2)$ are obtained by using the ego-centric sample based dataset $\mathcal{Y}_{2,N}$.

When model parameters are estimated using sampled network data $\mathcal{Y}_{l,N}$, $l=1,2$, this dataset consists of {\it dependent} entries, as the sampling design effect is included. Due to this dependence, the usual theoretical properties (consistency, etc.) of the estimates of $K$ or class labels (obtained through community detection algorithms) or estimates of edge probabilities may not hold. As a result, the `estimated' version $\widehat{J}_{\eta,l,F}(H)$ of the PI in \eqref{pred-int-1}, obtained by plugging in these estimated parameters, {\it does not} provide any theoretical guarantee about its nominal coverage level, even asymptotically. 


Five different simulation settings are explored, each focusing on a specific choice of model and sampling sparsity level. In all cases, empirical coverages of the proposed PI's were obtained by using 2000 Monte-Carlo replications. Each Monte-Carlo replication consists of initially generating $\Yb_N$ from the true model (SBM), independently selecting a sample of nodes from $[N]$ by Bernoulli sampling, and then observing the sampled data (through the induced/ego-centric approach). This process is repeated during every Monte-Carlo replication. This allows us to replicate the joint model and design based inferential framework. Throughout the simulation study we used a two class SBM ($K=2$), with equal class size proportions $\la_1 = \la_2 = 1/2$, while the edge probabilities, population size $N$ and the node sampling probability are changed.

\begin{enumerate}
\item[(S.1)] {\bf Dense network and dense sampling:} In this case, we set $\al=\beta = 0$. We use $\pi_{N,1,1}=0.20$, $\pi_{N,1,2}=0.05$, $\pi_{N,2,2}=0.10$ and $N = 20000$. We use $p_N\in\{0.01,0.05\}$ (cf. \eqref{pN-def}), which implies that the expected number of sampled nodes will be 200 and 1000, respectively. The empirical coverages are provided in Table \ref{tab:sim-1}.


\item[(S.2)] {\bf Sparse network and dense sampling:} In this case we set $\beta = 1/2$, $\al = 0$, $\pi_{N,1,1} = 0.00506$, $\pi_{N,1,2}=0.00206$, $\pi_{N,2,2}=0.00901$ and $N = 100000$. We use $p_N \in\{0.01,0.05\}$, which implies that the expected number of sampled nodes will be 1000 and 5000, respectively. The results are provided in Table \ref{tab:sim-2}.

\item[(S.3)] {\bf Sparse network and sparse sampling:} We use $\beta = 1/2$ and $\al = 1/4$, $\pi_{N,1,1} = 0.00112$, $\pi_{N,1,2} = 0.000112$, $\pi_{N,2,2}=0.000447$ and $N = 200000$. We use $p_N = 0.0002 =(0.00423)\cdot N^{-1/4}$. The expected number of sampled nodes is approximately 845. Table \ref{tab:sim-3} displays the results and only focuses on the case where the parameters are known. 

\item[(S.4)] {\bf Extreme model sparsity and dense sampling:} We focus on $H = \kl_2$ (edge, with $R=2$, $T=1$) and consider the highest possible model sparsity level $\beta = 2~(= R/T)$. We use $\al = 0$. Edge probabilities are  $\pi_{N,1,1} = 3\times 10^{-7}$, $\pi_{N,1,2}= 10^{-7}$ and $\pi_{N,2,2} = 2\times 10^{-7}$. We use $N = 10^6$ (1 million) and $p_N\in \{0.01,0.05\}$. The expected number of sampled nodes will be 10000 and 50000, respectively. Table \ref{tab:sim-4} displays the results and only focuses on the case where the true parameters are known. The quantiles of the limit law (cf. Theorem \ref{thm-2-main-temp}(a)) are approximated by Monte-Carlo simulation.

\item[(S.5)] {\bf Dense model and extremely sparse sampling:} We focus on the edge graph $H = \kl_2$. The model is dense and the sampling is extremely sparse ($\beta = 0$ and $\al=1$). The edge probabilities are $\pi_{N,1,1}  = 0.2$, $\pi_{N,1,2}=0.05$ and $\pi_{N,2,2}=0.10$. We use $N = 2000000$ (2 million) and $p_N = 0.000025 = 50/N$. The expected number of sampled nodes is 50. We focus on ego-centric network formation and the Poisson-based limit law of $T_{N,2}(\kl_2)$ is provided in Theorem \ref{thm-2-main-temp}(c). We found that the nominal 95\% two-sided asymptotic PI in \eqref{pred-int-1} (with true model parameters) provided an empirical coverage of $0.935$, with average width $0.057$.

\end{enumerate}

The results in Tables \ref{tab:sim-1}, \ref{tab:sim-2}, \ref{tab:sim-3} and \ref{tab:sim-4} and the results for simulation setup (S.5) show excellent empirical coverages for the proposed PI's in all the above five simulation scenarios, even when the model parameters are estimated on the basis of the sampled data. This strongly supports our theoretical results and also illustrates their applicability. The computations were carried out on a 64 core server with 128 GB RAM. The code was written in Rcpp and it was further optimized for increased speed through the use of generative AI. All code, data and outputs are available on the github page \href{https://github.com/AnirbanM13/subgraph-count-estimate}{https://github.com/AnirbanM13/subgraph-count-estimate}.






\begin{rem}[Evaluation of the variance expression]
\label{rem:var-comp}
For arbitrary choices of $H$, the expression for $\sigma^2_{N,l}(H)$ can involve a sum over numerous terms, each requiring the construction of an union graph formed by two overlapping isomorphic copies of $H$. It is nearly impossible to do this manually for even for relatively simple choices of $H$. But this can be achieved through the use of a computer program by using the following expressions. Fix any $t\in[R]$ and consider the set of vectors ${[R]}_{<,t}$ (cf. \eqref{n-sub-r}), and the collection $\mathcal{S}_t$ of all permutations of $\{1,\ldots,t\}$. For any $\kbb\in{[R]}_{<,t}$, let $B(\kbb,t)$ denote the $t$-vertex subgraph of $H$, induced by the vertices in $A(\kbb)$. Recall the definition of $f_2(p:H)$ (cf. \eqref{f2-def}), the node sampling probability $p_N$ (cf. \eqref{pN-def}), the class size proportions $\{\lambda_{N,k}:k\in[K]\}$ described in \eqref{Nkdef} and the edge probability matrix $\PiB_N=((\pi_{N,u,v}:1\leq u,v\leq K))$ defined in \eqref{pi-mat-def}. For any $t\in[R]$, any $\jbb,\kbb\in{[R]}_{<,t}$ and any permutation $\xi\in\mathcal{S}_t$, define the quantity,
\begin{align}
&\Theta_{N,t}(\jbb,\kbb,\xi:\PiB_N,\la_{N,1},\ldots,\la_{N,K},H) \notag \\
& = \sum_{\ub,\vb\in {[K]}^R}\prod_{\{r,s\}\in E(H)}\pi_{N,u_r,u_s}\prod_{\{r,s\}\in E(H)\setminus E(B(\kbb,t))}\pi_{N,v_r,v_s} \prod_{r=1}^R \la_{N,u_r}\prod_{s\in [R]\setminus A(\kbb)} \la_{N,v_s} \prod_{r=1}^t \mathbf{1}\left(u_{j_r} = v_{k_{\xi(r)}}\right).
\label{Theta-t-def}
\end{align}
Then, for any target subgraph $H$, we can write (as $N\rai$),
\begin{align}
\sigma^2_{N,1}(H) & \asymp \sum_{t=1}^R \left(p^{2R-t}_N-p^{2R}_N\right) N^{2R-t}\sum_{\jbb,\kbb\in {[R]}_{<,t}}\sum_{\xi\in\mathcal{S}_t}\Theta_{N,t}(\jbb,\kbb,\xi:\PiB_N,\la_{N,1},\ldots,\la_{N,K},H).
\label{var-exp-ind-gen-H}
\end{align}
Now, recall the definition of $\mathcal{M}_t(\jbb,\kbb,\xi)$ from \eqref{Mt-jk-xi-def}. Let $(\tilde{\sbb}_1,\tilde{\sbb}_2)\in\mathcal{M}_t(\jbb,\kbb,\xi)$, with $A(\tilde{\sbb}_1)=[R]$ and $A(\tilde{\sbb}_1)\cup A(\tilde{\sbb}_2) = \{1,\ldots,2R-t\}$. Consider the union graph $H(\tilde{\sbb}_1)\cup H(\tilde{\sbb}_2)$ and denote it by $H(\jbb,\kbb,\xi)$, where these graphs are created according to the method described above in \eqref{sig-temp-0}. Let $\mathcal{D}\left(H(\jbb,\kbb,\xi)\right)$ denote the collection of all vertex covers of $H(\jbb,\kbb,\xi)$. Then, for any target subgraph $H$ satisfying Assumption \ref{assump-5} we can write (as $N\rai$), 
\begin{align}
\sigma^2_{N,2}(H) \asymp \sum_{t=1}^R N^{2R-t} \sum_{\jbb,\kbb\in {[R]}_{<,t}}\sum_{\xi\in\mathcal{S}_t} & \left[\sum_{D\in\mathcal{D}(\jbb,\kbb,\xi)} p^{|D|}_N {(1-p_N)}^{2R-t-|D|} - f^2_2(p_N:H)\right]\notag\\
&\qquad \qquad \times \Theta_{N,t}(\jbb,\kbb,\xi:\PiB_N,\la_{N,1},\ldots,\la_{N,K},H).
\label{var-exp-ego-gen-H}
\end{align}
For the benefit of practitioners, the relevant code for computing these variance expressions for any choice of $H$ is provided on the github page \href{https://github.com/AnirbanM13/subgraph-count-estimate}{https://github.com/AnirbanM13/subgraph-count-estimate}. They require inputs of (either true or estimated) model parameters, $K$, $\PiB_N$, $\{\la_{N,k}\}$, the adjacency matrix of $H$ the value of $p_N$ and the choice of $\beta$. The code can be used for target graphs $H$ with $R=7$ vertices. The computational burden increases rapidly as $R$ increases. The github page also includes a code for verifying Assumption \ref{assump-5} for a given choice of $H$.
\end{rem}
\ctable[
  pos=htbp,
  caption={\footnotesize {\bf Simulation setup (S.1), with $\al=\beta = 0$ and $N=20000$}: Empirical coverage and average lengths (in parenthesis) of 95\% asymptotic PI's for various population summary statistics. Results are reported for both induced and ego-centric approaches and for both \textbf{known (true)} and \textbf{estimated} model parameters.},
  label={tab:sim-1},
  captionskip = -2ex,
  doinside=\scriptsize
]{c c 
  >{\centering\arraybackslash}p{1.25cm} >{\centering\arraybackslash}p{1.25cm}
  >{\centering\arraybackslash}p{1.25cm} >{\centering\arraybackslash}p{1.25cm}
  >{\centering\arraybackslash}p{1.25cm} >{\centering\arraybackslash}p{1.25cm}
  >{\centering\arraybackslash}p{1.25cm} >{\centering\arraybackslash}p{1.25cm}
}{
\setlength{\tabcolsep}{1.2pt}
\renewcommand{\arraystretch}{1.02}
}{
\toprule
 & & \multicolumn{8}{c}{Population level summary statistic} \\
 \cmidrule(lr){3-10}

        &  & 
\multicolumn{2}{c}{$N^{-2}S_N(\kl_2)$} &
\multicolumn{2}{c}{$N^{-3}S_N(\kl_{1,2})$} &
\multicolumn{2}{c}{$N^{-3}S_N(\kl_3)$} &
\multicolumn{2}{c}{$\Gamma_N$} \\

\cmidrule(lr){3-4} \cmidrule(lr){5-6} \cmidrule(lr){7-8} \cmidrule(lr){9-10}

{Method} & {$p_N$} 
& \textbf{True} & \textbf{Est.}
& \textbf{True} & \textbf{Est.}
& \textbf{True} & \textbf{Est.}
& \textbf{True} & \textbf{Est.} \\
\midrule

\textbf{Induced} 
& 0.01 
& 0.952 & 0.951 
& 0.959 & 0.949 
& 0.959 & 0.940 
& 0.961 & 0.932 \\[-1.5ex]

&       
& (0.0572) & (0.0572)
& (0.0097) & (0.0096)
& (0.00145) & (0.0014)
& (0.0348) & (0.0329) \\[0.5ex]

& 0.05 
& 0.951 & 0.951
& 0.953 & 0.952
& 0.954 & 0.951
& 0.950 & 0.944 \\[-1.5ex]

&       
& (0.0249) & (0.0250)
& (0.0042) & (0.0042)
& (0.00061) & (0.0006)
& (0.0127) & (0.0127) \\

\midrule

\textbf{Ego-centric} 
& 0.01 
& 0.954 & 0.950
& 0.950 & 0.941
& 0.954 & 0.934
& 0.950 & 0.940 \\[-1.5ex]

&       
& (0.0283) & (0.0283)
& (0.0033) & (0.0032)
& (0.00093) & (0.0009)
& (0.0482) & (0.0465) \\[0.5ex]

& 0.05 
& 0.952 & 0.952
& 0.955 & 0.948
& 0.961 & 0.952
& 0.950 & 0.949 \\[-1.5ex]

&       
& (0.0121) & (0.0121)
& (0.0015) & (0.0015)
& (0.00041) & (0.0003)
& (0.0195) & (0.0195) \\

\bottomrule
}



\ctable[
  pos=t,
  caption={\footnotesize {\bf Simulation setup (S.2), with $\al=0$, $\beta = 1/2$ and $N=100000$}: Empirical coverage and average lengths (in parenthesis) of 95\% asymptotic PI's for various population summary statistics. Results are reported for both induced and ego-centric approaches and for both \textbf{known (true)} and \textbf{estimated} model parameters.},
  label={tab:sim-2},
  captionskip = -2ex,
  doinside=\scriptsize
]{c c
  >{\centering\arraybackslash}p{1.15cm} >{\centering\arraybackslash}p{1.15cm}
  >{\centering\arraybackslash}p{1.15cm} >{\centering\arraybackslash}p{1.15cm}
  >{\centering\arraybackslash}p{1.15cm} >{\centering\arraybackslash}p{1.15cm}
  >{\centering\arraybackslash}p{1.15cm} >{\centering\arraybackslash}p{1.15cm}
}{
\setlength{\tabcolsep}{1.1pt}
\renewcommand{\arraystretch}{1.02}
\tnote[a]{{\scriptsize{Lengths of PI's for the edge, wedge and triangle densities should be multiplied by $10^{-4}$, $10^{-6}$ and $10^{-8}$, respectively.}}}
}
{
\toprule
 & & \multicolumn{8}{c}{Population level summary statistic} \\
\cmidrule(lr){3-10}

 &  &
\multicolumn{2}{c}{$N^{-2}S_N(\kl_2)$\tmark[a]}&
\multicolumn{2}{c}{$N^{-3}S_N(\kl_{1,2})$\tmark[a]} &
\multicolumn{2}{c}{$N^{-3}S_N(\kl_3)$\tmark[a]} &
\multicolumn{2}{c}{$\Gamma_N$} \\

\cmidrule(lr){3-4} \cmidrule(lr){5-6} \cmidrule(lr){7-8} \cmidrule(lr){9-10}

Method & {$p_N$}
& \textbf{True} & \textbf{Est.}
& \textbf{True} & \textbf{Est.}
& \textbf{True} & \textbf{Est.}
& \textbf{True} & \textbf{Est.} \\
\midrule

\textbf{Induced}
& 0.010
& 0.949 & 0.952
& 0.950 & 0.945
& 0.955 & 0.946
& 0.953 & 0.943 \\[-1.5ex]

&
& $(11.7758)$ & $(11.9296)$
& $(9.1028)$ & $(8.8997)$
& $(12.8037)$ & $(12.4201)$
& $(0.0051)$ & $(0.0049)$ \\[0.5ex]

& 0.050
& 0.948 & 0.948
& 0.949 & 0.949
& 0.952 & 0.949
& 0.954 & 0.953 \\[-1.5ex]

&
& $(5.0548)$ & $(5.0505)$
& $(3.7741)$ & $(3.7613)$
& $(2.7353)$ & $(2.7104)$
& $(0.0005)$ & $(0.0005)$ \\

\midrule

\textbf{Ego-centric}
& 0.010
& 0.954 & 0.956
& 0.950 & 0.940
& 0.953 & 0.921
& 0.950 & 0.900 \\[-1.5ex]

&
& $(5.7118)$ & $(5.7853)$
& $(2.8968)$ & $(2.7748)$
& $(3.9551)$ & $(3.5238)$
& $(0.0011)$ & $(0.0009)$ \\[0.5ex]

& 0.050
& 0.950 & 0.948
& 0.944 & 0.945
& 0.953 & 0.951
& 0.940 & 0.949 \\[-1.5ex]

&
& $(2.4505)$ & $(2.4483)$
& $(1.2680)$ & $(1.2631)$
& $(1.6284)$ & $(1.6121)$
& $(0.0004)$ & $(0.0004)$ \\

\bottomrule
}

\ctable[
  pos=ht,
  caption={\footnotesize {\bf Simulation setup (S.3), with $\al=1/4$, $\beta = 1/2$ and $N=200000$}: Empirical coverage and average lengths (in parenthesis) of 95\% asymptotic PI's for various population summary statistics. Results are reported for both induced and ego-centric approaches and only for \textbf{known (true)} model parameters. Here, we use $p_N = (0.00423)\cdot N^{-1/4} = 0.0002$.},
  label={tab:sim-3},
  captionskip = -2ex,
  doinside=\footnotesize
]{c c >{\centering\arraybackslash}p{2cm} >{\centering\arraybackslash}p{2cm} >{\centering\arraybackslash}p{2cm} >{\centering\arraybackslash}p{2cm}}{
\tnote[a]{{\scriptsize{Lengths of PI's for the edge, wedge and triangle densities should be multiplied by $10^{-4}$, $10^{-6}$ and $10^{-8}$, respectively.}}}
}{
\toprule
 & & \multicolumn{4}{c}{Population level summary statistic} \\
\cmidrule(lr){3-6}
Method & & $N^{-2}S_N(\kl_2)$\tmark[a] & $N^{-3}S_N(\kl_{1,2})$\tmark[a] & $N^{-3}S_N(\kl_3)$\tmark[a] & $\Gamma_N$ \\
\midrule

\textbf{Induced} 
&  
& 0.930 
& 0.939 
& 0.941
& 0.938\\[-1.5ex]
&        
& ($11.61$) 
& ($6.65924$) 
& ($26.3866$) 
& ($0.03812$)\\

\midrule

\textbf{Ego-centric} 
&  
& 0.949 
& 0.947 
& 0.943 
& 0.946\\[-1.5ex]
&        
& (5.2335) 
& ($1.72755$) 
& ($2.62651$) 
& (0.00279)\\

\bottomrule
}

\ctable[
  caption = {\footnotesize {\bf Simulation setup (S.4), with $\al=0$, $\beta = 2$ and $N=10^6$}: Empirical coverage and average lengths (in parenthesis) of 95\% asymptotic PI's for the \textbf{edge density} $N^{-2}S_N(\kl_2)$. Results are reported for both induced and ego-centric approaches and only for \textbf{known (true)} model parameters.},
  label   = {tab:sim-4},
  pos     = ht!,
  captionskip = -2ex,
  doinside = \footnotesize,
]{ccccc}{

}{
\FL
\multicolumn{2}{c}{$p_N = 0.01$} && \multicolumn{2}{c}{$p_N = 0.05$} \NN
\cmidrule(lr){1-2}\cmidrule(lr){4-5}
\textbf{Induced} & \textbf{Ego-centric} && \textbf{Induced} & \textbf{Ego-centric}\\
\midrule
$0.95$ & $0.93$ && $0.947$ & $0.93$ \\[-1.5ex]
$(1.1912\times 10^{-7})$ & $(8.1377\times 10^{-9})$ && $(2.3156\times 10^{-8})$ & $(3.5278\times 10^{-9})$
\LL}

\subsection{Analysis of a real data set}
\label{sec:real}

In this section, we apply our proposed methodology on a real data set. We focus on the \textit{2004 US Political blogs networks data}\footnote{\url{https://public.websites.umich.edu/~mejn/netdata/}}. The data set is compiled by \cite{adamic2005} and has been used earlier in the context of community detection (cf. \cite{le2019estimating}, \cite{Amini_2013}). Each node in the data set represents blogs focused on US politics, and the edges are the hyperlinks between these blogs. After removing isolated nodes from the original dataset, the population network consists of $N=1224$ nodes. The original data had directed edges. We transform the data to create an undirected population network in the following manner. If $Z_{i,j}$ denotes the $(i,j)$-th entry of the adjacency matrix of the original (directed) network, we create a new symmetric adjacency matrix with the $(i,j)$-th entry $Y_{i,j} = Y_{j,i} = \max\{Z_{i,j},Z_{j,i}\}$. Our data analysis will be based on this newly created adjacency matrix. Each blog is manually labeled {\it liberal} (denoted as class 1) or {\it conservative} (denoted as class 2). This manual labeling is available from this  \href{https://github.com/NiccoloBalestrieri/Political-Blog-2004-U.S.-Election-Analysis}{source}. Labeling {\it liberal} or {\it conservative} automatically defines two classes in the population, with $N_1 = 588$ and $N_2 = 636$. The average degree and edge density of the population network are $27.312$ and $0.02233$ respectively. Using the class label information for all $N$ nodes in the population, we can obtain the {\it true} edge connection probabilities, $\pi_{1,1} = 0.0423$, $\pi_{1,2} = 0.0042$ and $\pi_{2,2} = 0.0388$.

\ctable[
    caption = {Predicted values of subgraph counts and clustering coefficient and associated 95\% prediction intervals for the 2004 US Political blog data. We use $p_N = 0.1$. Estimated model parameters were used for constructing the PI's.},
    label   = {tab:real-data},
    pos = ht!,
    captionskip=-2ex,
    doinside = \footnotesize,
]{cccccc}{}{
\FL
  & & \multicolumn{2}{c}{Estimate} &\multicolumn{2}{c}{Prediction Interval}\NN
 \cmidrule(l){3-4}\cmidrule{5-6}
Summary statistic & True value & \textbf{Induced} & \textbf{Ego-centric} & \textbf{Induced} & \textbf{Ego-centric}\\
\midrule
$S_N(\kl_2)$ & $33430$ & $39600$ & $36389$ & $[26097.18,~53102.82]$ & $[30186.12,~42592.83]$\NN
$S_N(\kl_{1,2})$ & $2683050$ & $2842000$ & $2737266$ & $[1890248,~3793752]$ & $[2438309,~3036223]$\NN
$S_N(\kl_3)$ & $303129$ & $240000$ & $263357$ &$[79775.94,~400224.1]$ & $[178028.4,~348685.9]$ \NN
$\Gamma_N$ & $0.2259585$ & $0.1688951$ & $0.1924235$  &$[0.09685586,~0.2409344]$ & $[0.1473414,~0.2375056]$\LL
}

A {\it single} Bernoulli sample of nodes was drawn from the population network. We used $p_N = 0.1$, and the expected number of sampled nodes is $\approx 122$. We use $\al=0$ and $\beta = 0$. The following estimates of the model parameters were obtained from the induced sample, $\widehat{K}=2$, $\widehat{\la}_{N,1} = 0.2061$, $\widehat{\la}_{N,2} = 0.7938$, $\widehat{\pi}_{1,1} = 0.1908$, $\widehat{\pi}_{1,2}=0.0064$ and $\widehat{\pi}_{2,2}=0.0211$. Table \ref{tab:real-data} shows the true edge, wedge and triangle counts, and the true clustering coefficient. It also shows the corresponding sampling based estimates in the induced and ego-centric cases and the resulting 95\% PI's. The above found estimates of the model parameters were used to construct these PI's. The real data analysis shows that the proposed 95\% PI's contain the values of the true population summary statistics.



\section*{Conclusion}

This article provides a rigorous study of the effects of network sampling, especially when sampling is carried out under resource constraints, and the population network can be 
dense or sparse. The effects of both induced and ego-centric network formation are explored. A model based setup is used and the asymptotic theory for estimated subgraph counts is developed a joint model and design based framework. Our results provide deep insights into the effects of network sampling under resource constraints and its complex relation with the sparsity level in the population network and the structure of the target subgraph. We have also developed Poisson based limit laws for estimated subgraph counts at more extreme sparsity levels and showed that various types of Poisson based limit laws can arise. Most importantly, the criteria for obtaining a Gaussian limit law also turn out to be necessary. In addition, we have developed limit laws for the estimated clustering coefficient statistic under similar settings. This is a completely new and important finding in the context of the clustering coefficient statistic. Our results are potentially applicable in any situation where the population network exhibits a block structure and theory can be possibly extended for more complicated network models. We believe that this article makes an important and meaningful contribution in the area of network sampling based inference.



\bibliographystyle{apalike}
\bibliography{resplanbib-1}

@article{fienberg-12,
	author = {Fienberg, Stephen E.},
	doi = {10.1080/10618600.2012.738106},
	journal = {Journal of Computational and Graphical Statistics},
	number = {4},
	pages = {825--839},
	publisher = {Taylor \& Francis},
	title = {A Brief History of Statistical Models for Network Analysis and Open Challenges},
	volume = {21},
	year = {2012}
}

@book{izenman2023,
  title={Network models for data science},
  author={Izenman, Alan Julian},
  year={2023},
  publisher={Cambridge University Press}
}

@article {bhattacharya2022,
    AUTHOR = {Bhattacharya, Bhaswar B. and Das, Sayan and Mukherjee, Sumit},
     TITLE = {Motif estimation via subgraph sampling: the fourth-moment
              phenomenon},
   JOURNAL = {Ann. Statist.},
  FJOURNAL = {The Annals of Statistics},
    VOLUME = {50},
      YEAR = {2022},
    NUMBER = {2},
     PAGES = {987--1011},
      ISSN = {0090-5364,2168-8966},
   MRCLASS = {62G05 (05C30 62G20)},
  MRNUMBER = {4404926},
       DOI = {10.1214/21-aos2134},
       URL = {https://doi-org.libraryisikolkata.remotexs.in/10.1214/21-aos2134},
}

@book{bill-95,
	author = {Billingsley, Patrick},
	edition = {Third},
	isbn = {0-471-00710-2},
	mrclass = {60-01 (28-01)},
	mrnumber = {1324786 (95k:60001)},
	note = {A Wiley-Interscience Publication},
	pages = {xiv+593},
	publisher = {John Wiley \& Sons, Inc., New York},
	series = {Wiley Series in Probability and Mathematical Statistics},
	title = {Probability and measure},
	year = {1995}
}

@book {janson-rucinski-randomgraphs,
    AUTHOR = {Janson, Svante and \L uczak, Tomasz and Rucinski, Andrzej},
     TITLE = {Random graphs},
    SERIES = {Wiley-Interscience Series in Discrete Mathematics and
              Optimization},
 PUBLISHER = {Wiley-Interscience, New York},
      YEAR = {2000},
     PAGES = {xii+333},
      ISBN = {0-471-17541-2},
   MRCLASS = {05C80 (60C05 82B41)},
  MRNUMBER = {1782847},
MRREVIEWER = {Mark\ R.\ Jerrum},
       DOI = {10.1002/9781118032718},
       URL = {https://doi.org/10.1002/9781118032718},
}

@book {kolaczyk09,
    AUTHOR = {Kolaczyk, Eric D.},
     TITLE = {Statistical analysis of network data},
    SERIES = {Springer Series in Statistics},
      NOTE = {Methods and models},
 PUBLISHER = {Springer, New York},
      YEAR = {2009},
     PAGES = {xii+386},
      ISBN = {978-0-387-88145-4},
   MRCLASS = {05C82 (62-09 62M40 68M11 90B10 91D30)},
  MRNUMBER = {2724362},
MRREVIEWER = {Lenwood\ S.\ Heath},
       DOI = {10.1007/978-0-387-88146-1},
       URL = {https://doi-org.libraryisikolkata.remotexs.in/10.1007/978-0-387-88146-1},
}

@article{illenberger-12,
	author = {Illenberger, Johannes and Fl{\"o}tter{\"o}d, Gunnar},
	doi = {10.1016/j.socnet.2012.09.001},
	issn = {0378-8733},
	journal = {Social Networks },
	note = {},
	number = {4},
	pages = {701--711},
	title = {Estimating network properties from snowball sampled data },
	url = {http://www.sciencedirect.com/science/article/pii/S0378873312000548},
	volume = {34},
	year = {2012}
}

@article{thompson-frank-2000,
	author = {Thompson, Steven K and Frank, Ove},
	journal = {Survey Methodology},
	number = {1},
	pages = {87--98},
	title = {Model-based estimation with link-tracing sampling designs},
	volume = {26},
	year = {2000}
}

@article {gile2015,
    AUTHOR = {Gile, Krista J. and Handcock, Mark S.},
     TITLE = {Network model-assisted inference from respondent-driven
              sampling data},
   JOURNAL = {J. Roy. Statist. Soc. Ser. A},
  FJOURNAL = {Journal of the Royal Statistical Society. Series A. Statistics
              in Society},
    VOLUME = {178},
      YEAR = {2015},
    NUMBER = {3},
     PAGES = {619--639},
      ISSN = {0964-1998,1467-985X},
   MRCLASS = {99-01},
  MRNUMBER = {3348351},
       DOI = {10.1111/rssa.12091},
       URL = {https://doi-org.libraryisikolkata.remotexs.in/10.1111/rssa.12091},
}

@article{frank2011survey,
	author = {Frank, Ove},
	journal = {The Sage handbook of social network analysis},
	pages = {389--403},
	publisher = {Sage},
	title = {Survey sampling in networks},
	year = {2011}
}

@article {zhang2015,
    AUTHOR = {Zhang, Yaonan and Kolaczyk, Eric D. and Spencer, Bruce D.},
     TITLE = {Estimating network degree distributions under sampling: an
              inverse problem, with applications to monitoring social media
              networks},
   JOURNAL = {Ann. Appl. Stat.},
  FJOURNAL = {The Annals of Applied Statistics},
    VOLUME = {9},
      YEAR = {2015},
    NUMBER = {1},
     PAGES = {166--199},
      ISSN = {1932-6157,1941-7330},
   MRCLASS = {62D05 (05C82 62J07 62P25 91D30)},
  MRNUMBER = {3341112},
       DOI = {10.1214/14-AOAS800},
       URL = {https://doi-org.libraryisikolkata.remotexs.in/10.1214/14-AOAS800},
}

@inbook{frank2005,
	author = {Frank, Ove},
	booktitle = {Models and Methods in Social Network Analysis},
	collection = {Structural Analysis in the Social Sciences},
	doi = {10.1017/CBO9780511811395.003},
	pages = {31--56},
	place = {Cambridge},
	publisher = {Cambridge University Press},
	series = {Structural Analysis in the Social Sciences},
	title = {Network Sampling and Model Fitting},
	year = {2005}
}

@article{bliss-14,
	author = {Bliss, Catherine A. AND Danforth, Christopher M. AND Dodds, Peter Sheridan},
	doi = {10.1371/journal.pone.0108471},
	journal = {PLOS ONE},
	month = {10},
	number = {10},
	pages = {1--18},
	publisher = {Public Library of Science},
	title = {Estimation of Global Network Statistics from Incomplete Data},
	url = {https://doi.org/10.1371/journal.pone.0108471},
	volume = {9},
	year = {2014}
}

@inproceedings{rib-12,
	author = {Ribeiro, B. and Towsley, D.},
	booktitle = {Decision and Control (CDC), 2012 IEEE 51st Annual Conference on},
	doi = {10.1109/CDC.2012.6425857},
	issn = {0743-1546},
	month = {Dec},
	pages = {5240--5247},
	title = {On the estimation accuracy of degree distributions from graph sampling},
	year = {2012}
}

@article{cost-03,
	author = {Costenbader, Elizabeth and Valente, Thomas W},
	doi = {10.1016/S0378-8733(03)00012-1},
	issn = {0378-8733},
	journal = {Social Networks },
	note = {},
	number = {4},
	pages = {283--307},
	title = {The stability of centrality measures when networks are sampled },
	url = {http://www.sciencedirect.com/science/article/pii/S0378873303000121},
	volume = {25},
	year = {2003}
}

@article {handcock-gile10,
    AUTHOR = {Handcock, Mark S. and Gile, Krista J.},
     TITLE = {Modeling social networks from sampled data},
   JOURNAL = {Ann. Appl. Stat.},
  FJOURNAL = {The Annals of Applied Statistics},
    VOLUME = {4},
      YEAR = {2010},
    NUMBER = {1},
     PAGES = {5--25},
      ISSN = {1932-6157,1941-7330},
   MRCLASS = {99-01},
  MRNUMBER = {2758082},
       DOI = {10.1214/08-AOAS221},
       URL = {https://doi-org.libraryisikolkata.remotexs.in/10.1214/08-AOAS221},
}

@article {lei2015consistency,
    AUTHOR = {Lei, Jing and Rinaldo, Alessandro},
     TITLE = {Consistency of spectral clustering in stochastic block models},
   JOURNAL = {Ann. Statist.},
  FJOURNAL = {The Annals of Statistics},
    VOLUME = {43},
      YEAR = {2015},
    NUMBER = {1},
     PAGES = {215--237},
      ISSN = {0090-5364,2168-8966},
   MRCLASS = {62F12 (62H30)},
  MRNUMBER = {3285605},
MRREVIEWER = {Zden\v{e}k\ Hl\'{a}vka},
       DOI = {10.1214/14-AOS1274},
       URL = {https://doi-org.libraryisikolkata.remotexs.in/10.1214/14-AOS1274},
}

@article {celisse2012consistency,
    AUTHOR = {Celisse, Alain and Daudin, Jean-Jacques and Pierre, Laurent},
     TITLE = {Consistency of maximum-likelihood and variational estimators
              in the stochastic block model},
   JOURNAL = {Electron. J. Stat.},
  FJOURNAL = {Electronic Journal of Statistics},
    VOLUME = {6},
      YEAR = {2012},
     PAGES = {1847--1899},
      ISSN = {1935-7524},
   MRCLASS = {62G05 (05C80 62-09 62E17 62G20 62H30)},
  MRNUMBER = {2988467},
       DOI = {10.1214/12-EJS729},
       URL = {https://doi-org.libraryisikolkata.remotexs.in/10.1214/12-EJS729},
}

@article {rohe2011spectral,
    AUTHOR = {Rohe, Karl and Chatterjee, Sourav and Yu, Bin},
     TITLE = {Spectral clustering and the high-dimensional stochastic
              blockmodel},
   JOURNAL = {Ann. Statist.},
  FJOURNAL = {The Annals of Statistics},
    VOLUME = {39},
      YEAR = {2011},
    NUMBER = {4},
     PAGES = {1878--1915},
      ISSN = {0090-5364,2168-8966},
   MRCLASS = {62H30 (05C50 05C85 60B20 62G99 62H25 62P25)},
  MRNUMBER = {2893856},
MRREVIEWER = {Zden\v{e}k\ Hl\'{a}vka},
       DOI = {10.1214/11-AOS887},
       URL = {https://doi-org.libraryisikolkata.remotexs.in/10.1214/11-AOS887},
}

@article{holland1983,
	author = {Holland, Paul W and Laskey, Kathryn Blackmond and Leinhardt, Samuel},
	journal = {Social networks},
	number = {2},
	pages = {109--137},
	publisher = {Elsevier},
	title = {Stochastic blockmodels: First steps},
	volume = {5},
	year = {1983}
}

@article{anderson1992,
	author = {Anderson, Carolyn J and Wasserman, Stanley and Faust, Katherine},
	journal = {Social networks},
	number = {1-2},
	pages = {137--161},
	publisher = {Elsevier},
	title = {Building stochastic blockmodels},
	volume = {14},
	year = {1992}
}

@article{lee2006,
	author = {Lee, Sang Hoon and Kim, Pan-Jun and Jeong, Hawoong},
	doi = {10.1103/physreve.73.016102},
	issn = {1550-2376},
	journal = {Physical Review E},
	month = {Jan},
	number = {1},
	publisher = {American Physical Society (APS)},
	title = {Statistical properties of sampled networks},
	url = {http://dx.doi.org/10.1103/PhysRevE.73.016102},
	volume = {73},
	year = {2006}
}

@article{now-wierman-88,
title = {Subgraph counts in random graphs using incomplete u-statistics methods},
journal = {Discrete Mathematics},
volume = {72},
number = {1},
pages = {299-310},
year = {1988},
issn = {0012-365X},
doi = {https://doi.org/10.1016/0012-365X(88)90220-8},
url = {https://www.sciencedirect.com/science/article/pii/0012365X88902208},
author = {Krzysztof Nowicki and John C. Wierman},
}

@article{snijders1992estimation,
	author = {Snijders, Tom AB},
	journal = {Bulletin of Sociological Methodology/Bulletin de M{\'e}thodologie Sociologique},
	number = {1},
	pages = {59--70},
	publisher = {Sage Publications Sage CA: Thousand Oaks, CA},
	title = {Estimation on the basis of snowball samples: how to weight?},
	volume = {36},
	year = {1992}
}

@article {frank1978sampling,
    AUTHOR = {Frank, Ove},
     TITLE = {Sampling and estimation in large social networks},
   JOURNAL = {Social Networks},
  FJOURNAL = {Social Networks. An International Journal of Structural
              Analysis},
    VOLUME = {1},
      YEAR = {1978},
    NUMBER = {1},
     PAGES = {91--101},
      ISSN = {0378-8733},
   MRCLASS = {62D05 (62P25 68E10 94C15)},
  MRNUMBER = {506605},
       DOI = {10.1016/0378-8733(78)90015-1},
       URL = {https://doi-org.libraryisikolkata.remotexs.in/10.1016/0378-8733(78)90015-1},
}

@article {franksjs-78,
    AUTHOR = {Frank, Ove},
     TITLE = {Estimation of the number of connected components in a graph by
              using a sampled subgraph},
   JOURNAL = {Scand. J. Statist.},
  FJOURNAL = {Scandinavian Journal of Statistics. Theory and Applications},
    VOLUME = {5},
      YEAR = {1978},
    NUMBER = {4},
     PAGES = {177--188},
      ISSN = {0303-6898},
   MRCLASS = {05C99 (62D05)},
  MRNUMBER = {515656},
MRREVIEWER = {F.\ Harary},
}

@inproceedings{leskovec06,
	address = {New York, NY, USA},
	author = {Leskovec, Jure and Faloutsos, Christos},
	booktitle = {Proceedings of the 12th ACM SIGKDD International Conference on Knowledge Discovery and Data Mining},
	doi = {10.1145/1150402.1150479},
	isbn = {1595933395},
	location = {Philadelphia, PA, USA},
	numpages = {6},
	pages = {631--636},
	publisher = {Association for Computing Machinery},
	series = {KDD '06},
	title = {Sampling from Large Graphs},
	url = {https://doi.org/10.1145/1150402.1150479},
	year = {2006}
}

@inproceedings{ruggeri19,
	address = {Cham},
	author = {Ruggeri, Nicol{\`o} and {De Bacco}, Caterina},
	booktitle = {Complex Networks and Their Applications VIII},
	editor = {Cherifi, Hocine and Gaito, Sabrina and Mendes, Jos{\'e} Fernendo and Moro, Esteban and Rocha, Luis Mateus},
	isbn = {978-3-030-36687-2},
	pages = {90--101},
	publisher = {Springer International Publishing},
	title = {Sampling on Networks: Estimating Eigenvector Centrality on Incomplete Networks},
	year = {2020}
}

@book {crane18,
    AUTHOR = {Crane, Harry},
     TITLE = {Probabilistic foundations of statistical network analysis},
    SERIES = {Monographs on Statistics and Applied Probability},
    VOLUME = {157},
 PUBLISHER = {CRC Press, Boca Raton, FL},
      YEAR = {2018},
     PAGES = {xx+236},
      ISBN = {978-1-1386-3015-4; 978-1-1385-8599-7},
   MRCLASS = {62-02 (60C05 60G09 60J20 62A01 62D05 62Hxx)},
  MRNUMBER = {3791467},
MRREVIEWER = {Michael\ Stolz},
}

@techreport{agc-lewis-16,
	author = {Chandrasekhar, Arun G. and Lewis, Randall},
	note = {Preprint available at \href{http://stanford.edu/~arungc/CL.pdf}{http://stanford.edu/~arungc/CL.pdf}},
	title = {Econometrics of Sampled Networks},
    institution ={Stanford University}, 
	type = {Working Paper},
	year = {2016}
}

@article {erdos-renyi60,
    AUTHOR = {Erd{\H{o}}s, P. and R\'enyi, A.},
     TITLE = {On the evolution of random graphs},
   JOURNAL = {Magyar Tud. Akad. Mat. Kutat\'o{} Int. K\"ozl.},
    VOLUME = {5},
      YEAR = {1960},
     PAGES = {17--61},
      ISSN = {0541-9514},
   MRCLASS = {05.40},
  MRNUMBER = {125031},
MRREVIEWER = {John\ Riordan},
}

@book{newman2018,
	author = {Newman, Mark},
	publisher = {Oxford university press},
	title = {Networks},
	year = {2018}
}

@article{shi-2019,
  title={Model-based and design-based inference: reducing bias due to differential recruitment in respondent-driven sampling},
  author={Shi, Yongren and Cameron, Christopher J and Heckathorn, Douglas D},
  journal={Sociological Methods \& Research},
  volume={48},
  number={1},
  pages={3--33},
  year={2019},
  publisher={SAGE Publications Sage CA: Los Angeles, CA}
}

@article{otsu-24,
  title={Empirical likelihood for network data},
  author={Matsushita, Yukitoshi and Otsu, Taisuke},
  journal={Journal of the American Statistical Association},
  volume={119},
  number={547},
  pages={2117--2128},
  year={2024},
  publisher={Taylor \& Francis}
}

@article{coulson2016,
  title={Poisson approximation of subgraph counts in stochastic block models and a graphon model},
  author={Coulson, Matthew and Gaunt, Robert E and Reinert, Gesine},
  journal={ESAIM: Probability and Statistics},
  volume={20},
  pages={131--142},
  year={2016},
  publisher={EDP Sciences}
}

@inproceedings{ebbes2008,
	author = {Ebbes, Peter and Huang, Zan and Rangaswamy, Arvind and Thadakamalla, Hari P and Unit, ORGB},
	booktitle = {18th Annual Workshop on Information Technologies and Systems, Paris, France},
	organization = {Citeseer},
	pages = {102--104},
	title = {Sampling large-scale social networks: Insights from simulated networks},
	volume = {100},
	year = {2008}
}

@article{gonzal-bias-14,
	author = {Sandra González-Bailón and Ning Wang and Alejandro Rivero and Javier Borge-Holthoefer and Yamir Moreno},
	doi = {10.1016/j.socnet.2014.01.004},
	issn = {0378-8733},
	journal = {Social Networks},
	pages = {16–27},
	title = {Assessing the bias in samples of large online networks},
	url = {https://www.sciencedirect.com/science/article/pii/S0378873314000057},
	volume = {38},
	year = {2014}
}

@inproceedings{maiya-bias-11,
  title={Benefits of bias: Towards better characterization of network sampling},
  author={Maiya, Arun S and Berger-Wolf, Tanya Y},
  booktitle={Proceedings of the 17th ACM SIGKDD international conference on Knowledge discovery and data mining},
  pages={105--113},
  year={2011}
}

@article {rubin-2005,
    AUTHOR = {Rubin-Bleuer, Susana and Schiopu Kratina, Ioana},
     TITLE = {On the two-phase framework for joint model and design-based
              inference},
   JOURNAL = {Ann. Statist.},
  FJOURNAL = {The Annals of Statistics},
    VOLUME = {33},
      YEAR = {2005},
    NUMBER = {6},
     PAGES = {2789--2810},
      ISSN = {0090-5364,2168-8966},
   MRCLASS = {62D05 (62E20)},
  MRNUMBER = {2253102},
MRREVIEWER = {M.\ E.\ Thompson},
       DOI = {10.1214/009053605000000651},
       URL = {https://doi.org/10.1214/009053605000000651},
}

@article {ko-jecon24,
    AUTHOR = {Hsieh, Chih-Sheng and Hsu, Yu-Chin and Ko, Stanley I. M. and
              Kov\'{a}\v{r}\'{\i}k, Jarom\'{\i}r and Logan, Trevon D.},
     TITLE = {Non-representative sampled networks: estimation of network
              structural properties by weighting},
   JOURNAL = {J. Econometrics},
  FJOURNAL = {Journal of Econometrics},
    VOLUME = {240},
      YEAR = {2024},
    NUMBER = {1},
     PAGES = {Paper No. 105689, 20},
      ISSN = {0304-4076,1872-6895},
   MRCLASS = {99-01},
  MRNUMBER = {4703367},
       DOI = {10.1016/j.jeconom.2024.105689},
       URL = {https://doi-org.libraryisikolkata.remotexs.in/10.1016/j.jeconom.2024.105689},
}

@book{pfeffermann-2000,
  title={Handbook of Statistics\_29B: Sample Surveys: Inference and Analysis},
  author={Pfeffermann, Danny},
  volume={29},
  year={2000},
  publisher={Elsevier}
}

@phdthesis{sischka2023graphon,
  title={Graphon models for network data: estimation, extensions and applications},
  author={Sischka, Benjamin},
  year={2023},
  school = {LMU Munich},
  note={Thesis available at \href{https://edoc.ub.uni-muenchen.de/32197/1/Sischka_Benjamin.pdf}{this link}}
}

@inproceedings{airoldi2013stochastic,
	author = {Airoldi, Edo M and Costa, Thiago B and Chan, Stanley H},
	booktitle = {Advances in Neural Information Processing Systems},
	editor = {Burges, C.J. and Bottou, L. and Welling, M. and Ghahramani, Z. and Weinberger, K.Q.},
	pages = {},
	publisher = {Curran Associates, Inc.},
	title = {Stochastic blockmodel approximation of a graphon: Theory and consistent estimation},
	url = {https://proceedings.neurips.cc/paper_files/paper/2013/file/b7b16ecf8ca53723593894116071700c-Paper.pdf},
	volume = {26},
	year = {2013}
}

@article{olhede2014network,
  title={Network histograms and universality of blockmodel approximation},
  author={Olhede, Sofia C and Wolfe, Patrick J},
  journal={Proceedings of the National Academy of Sciences},
  volume={111},
  number={41},
  pages={14722--14727},
  year={2014},
  publisher={National Acad Sciences}
}

@article{matias2014modeling,
  title={Modeling heterogeneity in random graphs through latent space models: a selective review},
  author={Matias, Catherine and Robin, St{\'e}phane},
  journal={ESAIM: Proceedings and Surveys},
  volume={47},
  pages={55--74},
  year={2014},
  publisher={EDP Sciences}
}

@article{lee2019review,
  title={A review of stochastic block models and extensions for graph clustering},
  author={Lee, Clement and Wilkinson, Darren J},
  journal={Applied Network Science},
  volume={4},
  number={1},
  pages={1--50},
  year={2019},
  publisher={Springer}
}

@article{nicola-22,
author = {Giacomo De Nicola and Benjamin Sischka and Göran Kauermann},
title ={Mixture models and networks: The stochastic blockmodel},

journal = {Statistical Modelling},
volume = {22},
number = {1-2},
pages = {67-94},
year = {2022},
doi = {10.1177/1471082X211033169},

URL = { 
    
        https://doi.org/10.1177/1471082X211033169
    
    

},
eprint = { 
    
        https://doi.org/10.1177/1471082X211033169
    
    

},
}

@article{abbe2018community,
  title={Community detection and stochastic block models: recent developments},
  author={Abbe, Emmanuel},
  journal={J. Mach. Learn. Res.},
  volume={18},
  number={177},
  pages={1--86},
  year={2018}
}

@article {boistard2017,
    AUTHOR = {Boistard, H\'{e}l\`ene and Lopuha\"{a}, Hendrik P. and
              Ruiz-Gazen, Anne},
     TITLE = {Functional central limit theorems for single-stage sampling
              designs},
   JOURNAL = {Ann. Statist.},
  FJOURNAL = {The Annals of Statistics},
    VOLUME = {45},
      YEAR = {2017},
    NUMBER = {4},
     PAGES = {1728--1758},
      ISSN = {0090-5364,2168-8966},
   MRCLASS = {62D05 (60F17)},
  MRNUMBER = {3670194},
       DOI = {10.1214/16-AOS1507},
       URL = {https://doi-org.libraryisikolkata.remotexs.in/10.1214/16-AOS1507},
}

@article {athreya2018statistical,
    AUTHOR = {Athreya, Avanti and Fishkind, Donniell E. and Tang, Minh and
              Priebe, Carey E. and Park, Youngser and Vogelstein, Joshua T.
              and Levin, Keith and Lyzinski, Vince and Qin, Yichen and
              Sussman, Daniel L.},
     TITLE = {Statistical inference on random dot product graphs: a survey},
   JOURNAL = {J. Mach. Learn. Res.},
  FJOURNAL = {Journal of Machine Learning Research (JMLR)},
    VOLUME = {18},
      YEAR = {2018},
     PAGES = {1--92},
      ISSN = {1532-4435,1533-7928},
   MRCLASS = {05C80 (05-02 62H12 62H15 62H30 62H99 62M15)},
  MRNUMBER = {3827114},
}

@article{karrer2011stochastic,
  title={Stochastic blockmodels and community structure in networks},
  author={Karrer, Brian and Newman, Mark EJ},
  journal={Physical review E},
  volume={83},
  number={1},
  pages={016107},
  year={2011},
  publisher={APS}
}

@article {le2019estimating,
    AUTHOR = {Le, Can M. and Levina, Elizaveta},
     TITLE = {Estimating the number of communities by spectral methods},
   JOURNAL = {Electron. J. Stat.},
  FJOURNAL = {Electronic Journal of Statistics},
    VOLUME = {16},
      YEAR = {2022},
    NUMBER = {1},
     PAGES = {3315--3342},
      ISSN = {1935-7524},
   MRCLASS = {62H12 (62H30)},
  MRNUMBER = {4422967},
MRREVIEWER = {Yuehan\ Yang},
       DOI = {10.1214/21-ejs1971},
       URL = {https://doi-org.libraryisikolkata.remotexs.in/10.1214/21-ejs1971},
}

@inproceedings{saade2014spectral,
 author = {Saade, Alaa and Krzakala, Florent and Zdeborov\'{a}, Lenka},
 booktitle = {Advances in Neural Information Processing Systems},
 editor = {Z. Ghahramani and M. Welling and C. Cortes and N. Lawrence and K.Q. Weinberger},
 pages = {},
 publisher = {Curran Associates, Inc.},
 title = {Spectral Clustering of graphs with the Bethe Hessian},
 url = {https://proceedings.neurips.cc/paper_files/paper/2014/file/63923f49e5241343aa7acb6a06a751e7-Paper.pdf},
 volume = {27},
 year = {2014}
}

@book{van2017random,
  title={Random Graphs and Complex Networks},
  author={van der Hofstad, R.},
  number={v. 1},
  isbn={9781107172876},
  lccn={2016047808},
  series={Cambridge Series in Statistical and Probabilistic Mathematics},
  url={https://books.google.co.in/books?id=1baSDQAAQBAJ},
  year={2017},
  publisher={Cambridge University Press}
}

@inproceedings{adamic2005,
author = {Adamic, Lada A. and Glance, Natalie},
title = {The political blogosphere and the 2004 U.S. election: divided they blog},
year = {2005},
isbn = {1595932151},
publisher = {Association for Computing Machinery},
address = {New York, NY, USA},
url = {https://doi.org/10.1145/1134271.1134277},
doi = {10.1145/1134271.1134277},
booktitle = {Proceedings of the 3rd International Workshop on Link Discovery},
pages = {36–43},
numpages = {8},
location = {Chicago, Illinois},
series = {LinkKDD '05}
}

@article {Amini_2013,
    AUTHOR = {Amini, Arash A. and Chen, Aiyou and Bickel, Peter J. and
              Levina, Elizaveta},
     TITLE = {Pseudo-likelihood methods for community detection in large
              sparse networks},
   JOURNAL = {Ann. Statist.},
  FJOURNAL = {The Annals of Statistics},
    VOLUME = {41},
      YEAR = {2013},
    NUMBER = {4},
     PAGES = {2097--2122},
      ISSN = {0090-5364,2168-8966},
   MRCLASS = {62G20 (62H99)},
  MRNUMBER = {3127859},
       DOI = {10.1214/13-AOS1138},
       URL = {https://doi-org.libraryisikolkata.remotexs.in/10.1214/13-AOS1138},
}

@book{bianconi-2018,
  title={Multilayer {N}etworks: {S}tructure and {F}unction},
  author={Bianconi, Ginestra},
  year={2018},
  publisher={Oxford university press}
}

@book{diestel2018graph,
  title={Graph Theory},
  author={Diestel, R.},
  series={Graduate Texts in Mathematics},
  year={2018},
  publisher={Springer Berlin Heidelberg},
}

@book{kba-snl,
  title={Measure theory and probability theory},
  author={Athreya, Krishna B and Lahiri, Soumendra N},
  volume={19},
  year={2006},
  publisher={Springer}
}

@article{vincent-21,
    author = {Vincent, Kyle and Thompson, Steve},
    title = {Estimating the Size and Distribution of Networked Populations with Snowball Sampling},
    journal = {Journal of Survey Statistics and Methodology},
    volume = {10},
    number = {2},
    pages = {397-418},
    year = {2021},
    month = {07},
    issn = {2325-0984},
    doi = {10.1093/jssam/smaa042},
    url = {https://doi.org/10.1093/jssam/smaa042},
    eprint = {https://academic.oup.com/jssam/article-pdf/10/2/397/43390736/smaa042.pdf},
}

@article {bkr-89,
    AUTHOR = {Barbour, A. D. and Karo\'nski, Micha\l{} and Ruci\'nski,
              Andrzej},
     TITLE = {A central limit theorem for decomposable random variables with
              applications to random graphs},
   JOURNAL = {J. Combin. Theory Ser. B},
  FJOURNAL = {Journal of Combinatorial Theory. Series B},
    VOLUME = {47},
      YEAR = {1989},
    NUMBER = {2},
     PAGES = {125--145},
      ISSN = {0095-8956,1096-0902},
   MRCLASS = {60F05 (05C80 60C05)},
  MRNUMBER = {1047781},
MRREVIEWER = {John\ C.\ Wierman},
       DOI = {10.1016/0095-8956(89)90014-2},
       URL = {https://doi.org/10.1016/0095-8956(89)90014-2},
}

@article {sb-pjb-15,
    AUTHOR = {Bhattacharyya, Sharmodeep and Bickel, Peter J.},
     TITLE = {Subsampling bootstrap of count features of networks},
   JOURNAL = {Ann. Statist.},
  FJOURNAL = {The Annals of Statistics},
    VOLUME = {43},
      YEAR = {2015},
    NUMBER = {6},
     PAGES = {2384--2411},
      ISSN = {0090-5364,2168-8966},
   MRCLASS = {62F40 (62D05 62G09)},
  MRNUMBER = {3405598},
MRREVIEWER = {Wei\ Sun},
       DOI = {10.1214/15-AOS1338},
       URL = {https://doi.org/10.1214/15-AOS1338},
}

@article{pjb-chen-2011,
    AUTHOR = {Bickel, Peter J. and Chen, Aiyou and Levina, Elizaveta},
     TITLE = {The method of moments and degree distributions for network
              models},
   JOURNAL = {Ann. Statist.},
  FJOURNAL = {The Annals of Statistics},
    VOLUME = {39},
      YEAR = {2011},
    NUMBER = {5},
     PAGES = {2280--2301},
      ISSN = {0090-5364,2168-8966},
   MRCLASS = {62H99 (62E10 62E20 62G05)},
  MRNUMBER = {2906868},
MRREVIEWER = {Emanuel\ Ben-David},
       DOI = {10.1214/11-AOS904},
       URL = {https://doi.org/10.1214/11-AOS904},
}

@article{stewart-aos-2025,
      title={Pseudo-likelihood-based $M$-estimation of random graphs with dependent edges and parameter vectors of increasing dimension},
      author={Jonathan R. Stewart and Michael Schweinberger},
      year={2025},
      journal = {Ann. Statist.},
        FJOURNAL = {The Annals of Statistics},
      eprint={2012.07167},
      archivePrefix={arXiv},
      primaryClass={math.ST},
      note={Preprint available at https://arxiv.org/abs/2012.07167},
}

@misc{agterberg2025,
      title={An Overview of Asymptotic Normality in Stochastic Blockmodels: Cluster Analysis and Inference},
      author={Joshua Agterberg and Joshua Cape},
      year={2025},
      eprint={2305.06353},
      archivePrefix={arXiv},
      primaryClass={math.ST},
      note={\href{https://arxiv.org/abs/2305.06353}{arxiv:2305.06353}},
}

@inproceedings{iwasaki2018,
  title={Estimating the clustering coefficient of a social network by a non-backtracking random walk},
  author={Iwasaki, Kenta and Shudo, Kazuyuki},
  booktitle={2018 IEEE International Conference on Big Data and Smart Computing (BigComp)},
  pages={114--118},
  year={2018},
  organization={IEEE}
}

@inproceedings{hardiman2013,
  title={Estimating clustering coefficients and size of social networks via random walk},
  author={Hardiman, Stephen J and Katzir, Liran},
  booktitle={Proceedings of the 22nd international conference on World Wide Web},
  pages={539--550},
  year={2013}
}

@article{luo-15,
author = {Luo, Peng and Li, Yongli and Wu, Chong and Zhang, Guijie},
title = {Toward cost-efficient sampling methods},
journal = {International Journal of Modern Physics C},
volume = {26},
number = {05},
pages = {1550050},
year = {2015},
doi = {10.1142/S0129183115500503},
URL = {    https://doi.org/10.1142/S0129183115500503},
eprint = { 
https://doi.org/10.1142/S0129183115500503},
}

@article{kaan-19,
  author       = {Kaan Bing{\"{o}}l and
                  Bahaeddin Eravci and
                  {\c{C}}agri {\"{O}}zgen{\c{c}} Etemoglu and
                  Hakan Ferhatosmanoglu and
                  Bugra Gedik},
  title        = {Topic-Based Influence Computation in Social Networks under Resource
                  Constraints},
  journal      = {CoRR},
  volume       = {abs/1801.02198},
  year         = {2018},
  url          = {http://arxiv.org/abs/1801.02198},
  eprinttype   = {arXiv},
  eprint       = {1801.02198},
  bibsource    = {dblp computer science bibliography, https://dblp.org}
}

@inproceedings{xu2017,
  title={Challenging the limits: Sampling online social networks with cost constraints},
  author={Xu, Xin and Lee, Chul-Ho and others},
  booktitle={IEEE Infocom 2017-IEEE Conference on Computer Communications},
  pages={1--9},
  year={2017},
  organization={IEEE}
}

@misc{nualart2025,
      title={Multivariate Poisson approximation of joint subgraph counts in random graphs via size-biased couplings}, 
      author={Eulalia Nualart and Rui-Ray Zhang},
      year={2025},
      eprint={2501.11180},
      archivePrefix={arXiv},
      primaryClass={math.PR},
      note={\href{https://arxiv.org/abs/2501.11180}{arxiv:2501.11180}}, 
}

@article {pianoforte2023,
    AUTHOR = {Pianoforte, Federico and Turin, Riccardo},
     TITLE = {Multivariate {P}oisson and {P}oisson process approximations
              with applications to {B}ernoulli sums and {$U$}-statistics},
   JOURNAL = {J. Appl. Probab.},
  FJOURNAL = {Journal of Applied Probability},
    VOLUME = {60},
      YEAR = {2023},
    NUMBER = {1},
     PAGES = {223--240},
      ISSN = {0021-9002,1475-6072},
   MRCLASS = {60F05 (60G55 62E17)},
  MRNUMBER = {4546119},
       DOI = {10.1017/jpr.2022.33},
       URL = {https://doi.org/10.1017/jpr.2022.33},
}

@book {dudley-real,
    AUTHOR = {Dudley, R. M.},
     TITLE = {Real analysis and probability},
    SERIES = {Cambridge Studies in Advanced Mathematics},
    VOLUME = {74},
      NOTE = {Revised reprint of the 1989 original.},
 PUBLISHER = {Cambridge University Press, Cambridge},
      YEAR = {2002},
     PAGES = {x+555},
      ISBN = {0-521-00754-2},
   MRCLASS = {60-01 (00A05 28-01 46-01 54-01)},
  MRNUMBER = {1932358},
       DOI = {10.1017/CBO9780511755347},
       URL = {https://doi.org/10.1017/CBO9780511755347},
}

\newpage

\appendix

\section*{Appendix}


\renewcommand{\thesubsection}{\Alph{subsection}}
\newcommand{\currentappendix}{}
\renewcommand{\thetheorem}{\currentappendix.\arabic{theorem}}


\newcommand{\setappendix}[1]{%
  \renewcommand{\currentappendix}{#1}%
}
\newcommand{\appitem}[3]{%
\noindent Appendix #1: #2 \dotfill \pageref{#3}\par
}

\renewcommand{\theequation}{E.\arabic{equation}}
\setcounter{equation}{0}

\subsection*{Contents}

\appitem{A}{Notation and definitions}{sec:app-A}
\appitem{B}{Proofs of main results}{sec:app-B}
\appitem{C}{Results from other sources}{sec:app-C}
\appitem{D}{Supplementary results}{sec:app-D}
\appitem{E}{Variance expressions for some common target subgraphs}{sec:app-E}

\bigskip



This appendix contains the following materials. Appendix \ref{sec:app-A} contains some notation and definitions that have been used in our proofs. Appendix \ref{sec:app-B} contains proofs of our main results. Appendix \ref{sec:app-C} contains results from other sources that have been used in our proofs. Appendix \ref{sec:app-D} contains some supplementary results that are required to prove the main results. Appendix \ref{sec:app-E} contains the variance expressions $\sigma^2_{N,l}(H)$ for specific choices of $H$, like the edge, wedge and triangle graphs and also the expressions for any complete and star graph. 

\setappendix{A}
\refstepcounter{subsection}
\subsection*{Appendix A: Notation and definitions}
\addcontentsline{toc}{subsection}{Appendix A: Notation and definitions}
\phantomsection
\label{sec:app-A}

Firstly, we introduce some definitions and notation. Some of these definitions are repeated for the sake of completeness. For any positive integer $m$, we write $[m]=\{1,\ldots,m\}$ to denote the set of the first $m$ positive integers. For any two sequences of positive real numbers $\{a_N:N\geq 1\}$ and $\{b_N:N\geq 1\}$, we write $a_N\asymp b_N$, if $a_N/b_N\rightarrow 1$, as $N\rai$. Similarly, we write $a_N\sim b_N$, if there exist constants $0<m_1 < m_2<\infty$, and an integer $N_0\geq 1$, such that $m_1<a_N/b_N <m_2$, for all $N\geq N_0$. Define the following collections of $R$ dimensional vectors,
\begin{equation}
\begin{aligned}
&{[N]}^R = \left\{\mathbf{s}={\left(s_1,\ldots,s_{R}\right)}^\primet:s_i\in [N],\ \text{for all $i\in[R]$}\right\},\quad 
{[N]}_{R} = \left\{\mathbf{s}\in {[N]}^R:\text{$s_1,\ldots,s_{R}$ are distinct}\right\},\ \text{and}\\
&{[N]}_{<,R} = \left\{\mathbf{s}\in {[N]}_R:s_1<s_2<\cdots <s_R \right\}.
\label{n-sub-r-2}
\end{aligned}
\end{equation}
For any vector $\ub = {(u_1,\ldots,u_t)}^\primet\in\rl^t$, we write
\begin{equation}
\begin{aligned}
A(\ub) & = \{a\in \rl:\exists\ \text{some $i\in [t]$, such that, $u_i = a$}\} = \text{the set of unique components of $\ub$.}
\end{aligned}
\label{i-del}
\end{equation}
For example, consider any $\sbb = {(s_{1},\ldots, s_{R})}^\primet\in {[N]}_R$, where ${[N]}_R$ is defined in \eqref{n-sub-r}. Then $A(\sbb) = \{s_{1},\ldots,s_{R}\}$. If $\sbb_1 = {(1,2,3)}^\primet$ and $\sbb_2 = {(2,1,3)}^\primet$, then $A(\sbb_1) = A(\sbb_2)=\{1,2,3\}$.
Let $t$ be a positive integer and consider another integer $s\in\{1,\ldots,t\}$. Consider a $s$ dimensional vector $\mathbf{i} = {(i_1,\ldots,i_s)}^\primet\in {[t]}_{<,s}$. For any $\ub={(u_1,\ldots,u_t)}^\primet\in\rl^t$, we write the $s$-dimensional sub-vector of $\ub$, ${(\ub)}_{\mathbf{i}} = {\left(u_{i_1},\ldots,u_{i_s}\right)}^\primet$, containing components of $\ub$ indexed by $\mathbf{i}$. For example, if $\ub = {(1,3,2,6)}^\primet$ and $\mathbf{i} = {(1,4,3)}^\primet$, then $\ub_{\mathbf{i}} = {(1,6,2)}^\primet$. Following Assumption \ref{assump-2}, define the $K$ dimensional vector of class proportions and their limiting values,
\begin{align}
    &\lamb_N = {(\la_{N,1},\ldots,\la_{N,K})}^\primet\quad\text{and}\quad\lamb={(\la_{1},\ldots,\la_K)}^\primet.
    \label{lamb-N},
\end{align}
Recall that $\la_{N,i}$ converges to $\la_i$ for each $i\in [K]$ (see Section~\ref{sec:inf-frame}), for all $i\in[K]$. In addition, we write $\mathbb{N}_0=\{0,1,2,\ldots\}$ as the set of all non-negative integers.

Now, we define some distance metrics on the space of all probability measures. Following \cite{dudley-real} (see p. 420), the Wasserstein distance between two $d$ dimensional random vectors $\Xb$ and $\Yb$ is defined as 
\begin{align}\label{wass-dist}
    d_{{W}}(\Xb,\Yb) =\sup\bigg\{\left|\E(h(\Xb)) - \E(h(\Yb))\right|:h\in \mathcal{A}_W\bigg\},
\end{align}
where, $\mathcal{A}_W = \{h\mid h:\mathbb{R}^{d}\mapsto\mathbb{R},~\norm{h}{L} \leq 1\}$, $\norm{h}{L} = \sup\{|h(\xb) - h(\yb)|/\norm{\xb-\yb}{1}:\xb\neq\yb\in \mathbb{R}^d\}$ is the Lipschitz norm of the map $h$, and ${\|\xb\|}_1=\sum_{j=1}^d |x_j|$, is the $\ell_1$ norm for any $\xb\in\rl^d$. The bounded Wasserstein distance (also known as the bounded Lipschitz distance) between two $d$ dimensional random vectors $\Xb$ and $\Yb$ is defined as (see p. 394 \cite{dudley-real}),
\begin{align}\label{bdd-wass}
    d_{{BL}}\left(\Xb,\Yb\right) = \bigg\{\left|\E(h(\Xb)) - \E(h(\Yb))\right|:h\in \mathcal{A}_{BL}\bigg\},
\end{align}
where, $\mathcal{A}_{BL} = \{h\mid h:\mathbb{R}^d\mapsto \mathbb{R},~\norm{h}{L}+\norm{h}{\infty}\leq 1\}$, $\norm{h}{\infty} = \sup\{|h(\xb)|:\xb\in \mathbb{R}^d\}$  and $\norm{h}{L}$ has been defined earlier. Obviously, $\mathcal{A}_{BL}\subseteq \mathcal{A}_W$ and $d_{BL}(\Xb,\Yb)\leq d_W(\Xb,\Yb)$. Convergence in $d_W$ implies convergence in $d_{BL}$. The total variation distance between two real valued random variables $X$ and $Y$ is defined as 
\begin{align}\label{TV-dist}
    d_{{TV}}\left(X,Y\right) = \sup\bigg\{\left|\E(h(X)) - \E(h(Y))\right|:h\in \mathcal{A}_{TV}\bigg\},
\end{align}
where, $\mathcal{A}_{TV}=\{h_B\mid B\in\mathcal{B}(\rl)\}$, and $h_B(x) = \mathbf{1}(x\in B)$, for all $x\in\rl$, and $\mathcal{B}(\rl)$ is the Borel sigma-field in $\rl$. 

Various constant terms will arise within the proofs of our main and supplementary results. Using appropriate notation, we will keep track of how these constant terms depend on all or some of the following quantities: the underlying sparsity levels $\al$ and $\beta$, the underlying (limiting) model parameters $\Cb,\lamb$ (see \eqref{PiN-lim}, \eqref{lamb-N}), the constant term $c$ within $p_N$ (see \eqref{pN-def}), the target subgraph $H$ and also on the network formation mechanism $l$ (with $l=1$, for the induced case and $l=2$ for the ego-centric case). If a constant $a_1$ depends on the parameters $\alpha$, $\beta$, $H$, $\lamb$, $\Cb$, and $l$, we write, $a_1 = a_1(l,\alpha,\beta,H,\lamb,\Cb)$. If $a_1$ depends only on $\al,\beta$ and $\lamb$, we write $a_1 = a_1(\alpha,\beta,\lamb)$. 
We will keep track of how the various constants arising through our proofs depend on all or some of the underlying parameters. But, for notational simplicity we will suppress this dependency. 
Similarly, certain statements will hold if the population network size $N$ becomes large enough and exceeds a threshold $M_1$ (or $M_2,\ldots$). We will keep track of how each $M_i$ depends on $(l,\alpha,\beta,H,\lamb,\Cb)$ (or a subset of them), but we will suppress the dependency for notational convenience. Sometimes, the same symbol or variable names may appear in the proof of different results, but they may denote or define different quantities within the different proofs. The meaning of this sort of symbol or variable will only be confined in that proof.

\setappendix{B}
\refstepcounter{subsection}
\subsection*{Appendix B: Proofs of main results}
\label{sec:app-B}



\begin{proof}[Proof of Theorem \ref{thm-1-main}] 
For any $\sbb\in{[N]}_R$, define $Y_N(\sbb:H) = \prod_{\{i,j\}\in E(H)} Y_{\sbb_i,\sbb_j}$ and $\widetilde{W}_{N,l}(\sbb:H) = \left(W_{N,l}(\sbb:H)-f_l(p_N:H)\right)$, where $W_{N,l}(\sbb:H) = \prod_{\{i,j\}\in E(H)}h_l\left(W_{N,s_i},W_{N,s_j}\right)$, and $f_l(p:H)$ is defined in \eqref{fl-def} for $l=1,2$. We can write   
\begin{align}
\label{piv-1}
T_{N,l}(H) 
= \frac{1}{\sigma_{N,l}(H)}\sum_{\sbb\in [N]_R}\widetilde{W}_{N,l}(\sbb:H)\cdot Y_N(\sbb:H),
\end{align}
where, $\sigma^2_{N,l}(H) = \Var\left(\sum_{\sbb\in [N]_R}\widetilde{W}_{N,l}(\sbb:H)\cdot Y_N(\sbb:H)\right)$, $l=1,2$. This is a sum of dependent random variables. For any $\sbb\in{[N]}_R$, the corresponding component $\widetilde{W}_{N,l}(\sbb:H)\cdot Y_N(\sbb:H)$ in the above sum has its own dependency neighborhood $N(\sbb)=\{\sbb_1\in{[N]}_R:|A(\sbb)\cap A(\sbb_1)|\geq 1\}$. Let $Z\sim N(0,1)$. Following the arguments presented in \cite{bkr-89} (cf. Section 3), the existence of the dependency neighborhood structure for each $\sbb\in {[N]}_R$ in \eqref{piv-1} allows us to apply Theorem 6.32 of \cite{janson-rucinski-randomgraphs} (see Theorem \ref{thm-ruc}) on the summands involved in \eqref{piv-1}, and obtain the following upper bound,
\begin{align}
&d_1\left(T_{N,l}(H),Z\right) \leq \frac{1}{\sigma^3_{N,l}(H)}
\sum_{\sbb_1\in \nsr}\sum_{\sbb_2\in N(\sbb_1)}\sum_{\sbb_3\in N(\sbb_1)}\left[\E\left|\prod_{i=1}^3 \widetilde{W}_{N,l}(\sbb_i:H)\right|\cdot \E\left(\prod_{i=1}^3 Y_N(\sbb_i:H)\right)
\right.\notag\\
&\quad \left. {}+\E\left|\widetilde{W}_{N,l}(\sbb_1:H)\cdot \widetilde{W}_{N,l}(\sbb_2:H)\right|\cdot \E|\widetilde{W}_{N,l}(\sbb_3:H)|\cdot \E\left(Y_N(\sbb_1:H)\cdot Y_N(\sbb_2:H)\right)\cdot\E(Y_N(\sbb_3:H))\right]\notag\\
&\equiv \frac{1}{\sigma^3_{N,l}(H)}\left(A_{N,l,1} + A_{N,l,2}\right),\quad\text{for all $N\geq 1$ and for both $l=1,2$.}
\label{ANl-brk}
\end{align}
For any $k_1,k_2\in[R]$, we define the following collection of vectors,
\begin{align}
\Mc_{k_1,k_2}
&= \Bigl\{(\sbb_1,\sbb_2,\sbb_3) :
\sbb_1,\sbb_2,\sbb_3 \in [N]_R,\,
\lvert A(\sbb_2)\cap A(\sbb_1)\rvert = k_1,\,
\lvert A(\sbb_3)\cap (A(\sbb_1)\cup A(\sbb_2))\rvert = k_2
\Bigr\}.
\label{Mck-setdef}
\end{align}
From the definitions of the neighborhood $N(\sbb)$ and the set $\Mc_{k_1,k_2}$, it follows that
$$
\Bigl\{(\sbb_1,\sbb_2,\sbb_3):\sbb_1,\sbb_2,\sbb_3 \in [N]_R,\,
\sbb_2 \in N(\sbb_1),\,
\sbb_3 \in N(\sbb_1)
\Bigr\}
\subseteq
\bigcup_{k_1,k_2=1}^{R} \Mc_{k_1,k_2}.
$$
As a result, to obtain an upper bound on $A_{N,l,1}$, $l=1,2$ (cf. \eqref{ANl-brk}), which involves non-negative expected values, we can enlarge the sum using the above set inclusion, and write 
\begin{align}\label{A-1}
    A_{N,l,1} \leq \sum_{k_1=1}^R \sum_{k_2=1}^R \sum_{(\sbb_1,\sbb_2,\sbb_3)\in \Mc_{k_1,k_2}}\E\left|\prod_{i=1}^3 \widetilde{W}_{N,l}(\sbb_i:H)\right|\cdot \E\left(\prod_{i=1}^3 Y_N(\sbb_i:H)\right).
\end{align}
Similarly, we can obtain an upper bound for $A_{N,l,2}$ by extending the summation over the larger set $\cup_{k_1,k_2}\Mc_{k_1,k_2}$. Consider the expected value of the design based quantity in the r.h.s. of \eqref{A-1}. Note that $\E(W_{N,l}(\sbb:H)) = f_l(p_N:H)$, for any $\sbb\in \nsr$, for both $l=1,2$. Thus, for any $\sbb_1,\sbb_2,\sbb_3\in \nsr$,
\begin{align}
\label{exp-W}
    &\E\left|\prod_{i=1}^3\widetilde{W}_{N,l}(\sbb_i:H)\right|= \E\left|\prod_{i=1}^3\left(W_{N,l}(\sbb_i:H) - f_l(p_N:H)\right)\right|\notag\\
    &\leq \E\left(\prod_{i=1}^3 W_{N,l}(\sbb_i:H)\right)+3f_l(p_N:H)\cdot\E\left(W_{N,l}(\sbb_1:H)\cdot W_{N,l}(\sbb_2:H)\right)+4{(f_l(p_N:H))}^3,\quad\text{for $l=1,2$.}
\end{align}
Also note, if $(\sbb_1,\sbb_2,\sbb_3)\in \mathcal{M}_{k_1,k_2}$ (cf. \eqref{Mck-setdef}), then 
$$
|A(\sbb_1)\cup A( \sbb_2)\cup A(\sbb_3)| = |A(\sbb_1)| + |A(\sbb_2)\cap A(\sbb_1)^c| + \left|A(\sbb_3)\cap \left(A(\sbb_1)\cup A(\sbb_2)\right)^c\right| = (3R-k_1-k_2).
$$
Now consider the induced case ($l=1$) and note that $f_1(p_N:H) = p_N^R$. Then, for any $(\sbb_1,\sbb_2,\sbb_3)\in \mathcal{M}_{k_1,k_2}$, the expected values in the r.h.s. of \eqref{exp-W} can be written as,
\begin{equation}
\label{prod-W-ind}
\begin{aligned}
&\E\left(\prod_{i=1}^3 W_{N,1}(\sbb_i:H)\right) =p_N^{|A(\sbb_1)\cup A(\sbb_2)\cup A(\sbb_3)|} =p_N^{3R - k_1-k_2},\quad {(f_1(p_N:H))}^3 = p_N^{3R},\\
&\quad\text{and}\quad f_1(p_N:H)\cdot \E\left(W_{N,1}(\sbb_1:H)W_{N,l}(\sbb_2:H)\right) = p_N^{3R - k_1}.
\end{aligned}
\end{equation}
Adding up the quantities in the r.h.s. of \eqref{exp-W} (for $l=1$) and using the expressions obtained above in \eqref{prod-W-ind}, it now follows that for any $k_1,k_2\in[R]$ and $(\sbb_1,\sbb_2,\sbb_3)\in \Mc_{k_1,k_2}$,
\begin{align}
\E\left|\prod_{i=1}^3\widetilde{W}_{N,1}(\sbb_i:H)\right|\leq 8\E\left(\prod_{i=1}^3 W_{N,1}(\sbb_i:H)\right) = 8\cdot p_N^{3R - k_1-k_2}.
\label{bdd-imp-1}
\end{align}
A similar argument can be carried out for the components within $A_{N,1,2}$ (cf. \eqref{ANl-brk}) and using \eqref{prod-W-ind}, we can claim,
$$
\E\left|\widetilde{W}_{N,1}(\sbb_1:H)\cdot \widetilde{W}_{N,1}(\sbb_2:H)\right|\cdot \E|\widetilde{W}_{N,1}(\sbb_3:H)| \leq 8\cdot p_N^{3R - k_1 -k_2},\quad\text{if $(\sbb_1,\sbb_2,\sbb_3)\in\Mc_{k_1,k_2}$.}
$$
Next, we consider the ego-centric case ($l=2$). Consider any $(\sbb_1,\sbb_2,\sbb_3)\in \Mc_{k_1,k_2}$ (cf. \eqref{Mck-setdef}) and let $H(\sbb_i)$, $i=1,2,3$, denote an isomorphic copy of $H$ with vertex set $V(H(\sbb_i)) = A(\sbb_i)$ and edge set $E(H(\sbb_i)) = \{\{s_{i,r},s_{i,s}\}:\{r,s\}\in E(H)\}$, $i=1,2,3$. Let $\mathcal{D}(\cup_{i=1}^3 H(\sbb_i))$ denote the collection of all vertex covers of $\cup_{i=1}^3 H(\sbb_i)$, and let $\mathcal{D}_{min}(\cup_{i=1}^3 H(\sbb_i))$ denote the collection of all minimum vertex covers  within $\mathcal{D}(\cup_{i=1}^3 H(\sbb_i))$. Let $|\mathcal{D}(\cup_{i=1}^3 H(\sbb_i))|$ denote the total number of vertex covers of $\cup_{i=1}^3 H(\sbb_i)$. Note that, $|\mathcal{D}(\cup_{i=1}^3 H(\sbb_i))| \leq {|\mathcal{D}(H)|}^3$, where $|\mathcal{D}(H)|$ denotes the total number of vertex covers of $H$. Due to Lemma \ref{VC-lem-1} part (b), there exists a minimum vertex cover $D^\star(\sbb_1,\sbb_2,\sbb_3)\in \mathcal{D}_{min}(\cup_{i=1}^3 H(\sbb_i))$, such that 
\begin{align}\label{ver-cov-bdd}
    |D^\star(\sbb_1,\sbb_2,\sbb_3)| 
    &\geq 3\tau(H) - \min\{\tau(H), k_1\} - \min\{\tau(H), k_2\}.
\end{align}
Now, using \eqref{f2-def}, the inequalities in \eqref{ver-cov-bdd} and the minimum vertex cover $D^\star(\sbb_1,\sbb_2,\sbb_3)$ defined above, for any $(\sbb_1,\sbb_2,\sbb_3)\in \Mc_{k_1,k_2}$ we can write the first term on the r.h.s. of \eqref{exp-W} as 
\begin{align}
\label{prod-W-ego}
&\E\left(\prod_{r=1}^3 W_{N,2}(\sbb_r:H)\right)= \E\left(\prod_{\{i,j\}\in E(H(\sbb_1)\cup H(\sbb_2)\cup H(\sbb_3))}\max\{W_{N,i},W_{N,j}\}\right)\notag\\
& = \sum_{D\in \mathcal{D}({\cup_{i=1}^3H(\sbb_i)})}p_N^{|D|}\left(1-p_N\right)^{3R - |D|}\notag\\
&= \sum_{D\in \mathcal{D}_{\min}({\cup_{i=1}^3H(\sbb_i)})}p_N^{|D|}\cdot (1-p_N)^{3R-|D|} + \sum_{D\in \mathcal{D}(\cup_{i=1}^3 H(\sbb_i))\setminus\mathcal{D}_{\min}(\cup_{i=1}^3 H(\sbb_i))}p_N^{|D|}\cdot (1-p_N)^{3R-|D|}\notag\\
& \leq p_N^{|D^\star(\sbb_1,\sbb_2,\sbb_3)|}\left\{ (1-p_N)^{3R-|D^\star(\sbb_1,\sbb_2,\sbb_3)|}\cdot\left|\mathcal{D}_{\min}(\cup_{i=1}^3 H(\sbb_i))\right|\right.\notag\\
&\qquad \left. {}+ \sum_{D\in \mathcal{D}(\cup_{i=1}^3 H(\sbb_i))\setminus\mathcal{D}_{\min}(\cup_{i=1}^3 H(\sbb_i))}p_N^{|D| - |D^\star(\sbb_1,\sbb_2,\sbb_3)|}\cdot (1-p_N)^{3R-|D|}\right\}\notag\\
&\leq  p_N^{|D^\star(\sbb_1,\sbb_2,\sbb_3)|}\cdot\left|\mathcal{D}(\cup_{i=1}^3 H(\sbb_i))\right|\notag\\
& \leq  p_N^{{}3\tau(H) - \min\{\tau(H,k_1) - \min\{\tau(H),k_2\}\}}{\left|\mathcal{D}(H)\right|}^3.
\end{align}
Using Lemma \ref{VC-lem-1} part (a),
for any  $(\sbb_1,\sbb_2,\sbb_3)\in\Mc_{k_1,k_2}$, we can put the following upper bounds on the different expected values in r.h.s. of \eqref{exp-W},
\begin{equation}
\label{bdd-W-term}
\begin{aligned}
&{(f_2(p_N:H))}^3  \leq p^{3\tau(H)}_N{\left|\mathcal{D}(H)\right|}^3, \quad\text{and}\\
& f_2(p_N:H)\cdot \E(W_{N,2}(\sbb_1:H)\cdot W_{N,2}(\sbb_2:H)) \leq p_N^{\{3\tau(H) - \min\{\tau(H),k_1\}\}}{\left|\mathcal{D}(H)\right|}^3.
\end{aligned}
\end{equation}
We can now combine the bounds in \eqref{prod-W-ego}, \eqref{bdd-W-term} and use them in \eqref{exp-W} to claim that 
\begin{align}
    \E\left|\prod_{r=1}^3 \widetilde{W}_{N,2}(\sbb_r:H)\right|&\leq 8\cdot\E\left(\prod_{r=1}^3 W_{N,2}(\sbb_r:H)\right)\leq 8\cdot{\left|\mathcal{D}(H)\right|}^3\cdot p^{3\tau(H)-\min\{\tau(H),k_1\}-\min\{\tau(H),k_2\}}_N,
    \label{bdd-imp-2}
\end{align}
for any  $(\sbb_1,\sbb_2,\sbb_3)\in\Mc_{k_1,k_2}$. To obtain the above described exponent of $p_N$ (in the r.h.s. of \eqref{bdd-imp-2}), we have used the fact that $p_N<1$, and $k_1$, $k_2$, and $\tau(H)$ are strictly positive. Similarly to the induced case, we can get the same upper bound for the components within $A_{N,2,2}$ (cf. \eqref{ANl-brk}),
$$
\E\left|\widetilde{W}_{N,2}(\sbb_1:H)\cdot \widetilde{W}_{N,2}(\sbb_2:H)\right|\cdot \E|\widetilde{W}_{N,2}(\sbb_3:H)| \leq 8\cdot {\left|\mathcal{D}(H)\right|}^3\cdot p^{3\tau(H)-\min\{\tau(H),k_1\}-\min\{\tau(H),k_2\}}_N,
$$
Combining \eqref{bdd-imp-1} and \eqref{bdd-imp-2}, and defining $a_{1}=a_1(H)=\max\left\{1,\left|\mathcal{D}(H)\right|^3\right\}$, we have (from \eqref{ANl-brk}), the following bound for both $l=1$ and $l=2$,
\begin{align}
\label{num-dW-l-1}
& A_{N,l,1}\leq 8\cdot a_{1}\cdot  \sum_{k_1=1}^{R}\sum_{k_2=1}^{R} p_N^{\displaystyle e_l(k_1,k_2:H)\displaystyle}\sum_{(\sbb_1,\sbb_2,\sbb_3)\in \Mc_{k_1,k_2}} \E\left(\prod_{i=1}^3 Y_N(\sbb_i:H)\right),
\end{align}
where,
\begin{align}\label{def-el}
    e_l(k_1,k_2:H) = \begin{cases}
        3R -k_1-k_2& \text{if $l=1$,}\\
        3\tau(H) - \min(\tau(H),k_1) -\min(\tau(H),k_2) & \text{if $l=2$.}
    \end{cases}
\end{align}
In order to control the expected values for the model based terms in \eqref{A-1}, we define the following sub-collection of vectors within the collection $\Mc_{k_1,k_2}$ (cf. \eqref{Mck-setdef}),  
\begin{equation*}
\begin{aligned}
\mathcal{O}_{k_1,k_2}
= \Bigl\{(\sbb_1,\sbb_2,\sbb_3)\in \mathcal{M}_{k_1,k_2}:~
&|E(H(\sbb_1)\cap H(\sbb_2))| = m(k_1:H),\\
&|E(H(\sbb_3))\cap (E(H(\sbb_1)\cup H(\sbb_2)))| = m(k_2:H)
\Bigr\}.
\end{aligned}
\end{equation*}
Note that, the collection $\mathcal{O}_{k_1,k_2}$ is non-empty. Since there is a sub-collection $$\mathcal{O}_{0,k_1,k_2}=\{(\sbb_1,\sbb_2,\sbb_3)\in \Mc_{k_1,k_2}: |E(H(\sbb_1)\cap H(\sbb_2))| = m(k_1:H), |E(H(\sbb_1)\cap H(\sbb_3))| = m(k_2:H)\}\subseteq \mathcal{O}_{k_1,k_2},$$ and $\mathcal{O}_{0,k_1,k_2}$ is non-empty.
Thus, if $(\sbb_1,\sbb_2,\sbb_3)\in \mathcal{O}_{k_1,k_2}$, then for these choices of $(\sbb_1,\sbb_2,\sbb_3)$, the graphs $H(\sbb_1)\cap H(\sbb_2)$ and $H(\sbb_3)\cap \left(H(\sbb_1)\cup H(\sbb_2)\right)$ achieve their highest possible edge count sizes (see Lemma \ref{lem-1}). As a result, if $(\sbb_1,\sbb_2,\sbb_3)\in \mathcal{O}_{k_1,k_2}$, then splitting the edges within $\cup_{i=1}^3 H(\sbb_i)$ into disjoint collections, we can write
\begin{align}\label{E-set-bdd}
&|E(H(\sbb_1)\cup H(\sbb_2)\cup H(\sbb_3))| \notag\\
& = |E(H(\sbb_1))| + |E(H(\sbb_2))| - |E(H(\sbb_1)\cap E(H(\sbb_2))| + |E(H(\sbb_3))| - |E(H(\sbb_3)\cap \left(H(\sbb_1)\cup H(\sbb_2)\right)|\notag\\
& = T+\left(T-m(k_1:H)\right) + \left(T-m(k_2:H)\right) = 3T - m(k_1:H)-m(k_2:H).
\end{align}
And if $(\sbb_1,\sbb_2,\sbb_3)\in \Mc_{k_1,k_2}\setminus\mathcal{O}_{k_1,k_2}$, then the number of edges within these graph intersections must be smaller than their maximum possible values, {\it i.e.}, 
\begin{align*}
|E(H(\sbb_1)\cap E(H(\sbb_2))|& =m(k_1:H)-\delta_1(\sbb_1,\sbb_2,\sbb_3)\quad\text{and}\\
|E(H(\sbb_3)\cap \left(H(\sbb_1)\cup H(\sbb_2)\right)| & = m(k_2:H)-\delta_2(\sbb_1,\sbb_2,\sbb_3),
\end{align*}
where, $\delta_1(\sbb_1,\sbb_2,\sbb_3)$ and $\delta_2(\sbb_1,\sbb_2,\sbb_3)$ are some non-negative integers, and at least one of them is strictly positive. This implies, if $(\sbb_1,\sbb_2,\sbb_3)\in \Mc_{k_1,k_2}\setminus\mathcal{O}_{k_1,k_2}$, then 
\begin{align}\label{E-set-bdd-1}
&|E(H(\sbb_1)\cup H(\sbb_2)\cup H(\sbb_3))| = T + \left(T-\{m(k_1:H)-\delta_1(\sbb_1,\sbb_2,\sbb_3)\}\right) + \left(T- \{m(k_2:H)-\delta_2(\sbb_1,\sbb_2,\sbb_3)\}\right)\notag\\
& = 3T - m(k_1:H)-m(k_2:H) + \delta_1(\sbb_1,\sbb_2,\sbb_3)+\delta_2(\sbb_1,\sbb_2,\sbb_3)\notag\\
& \leq 3T - m(k_1:H)-m(k_2:H) + \delta_{k_1,k_2},
\end{align}
where, $\delta_{k_1,k_2} = 2\cdot \max_{(\sbb_1,\sbb_2,\sbb_3)\in \mathcal{M}_{k_1,k_2}}\left\{\delta_1(\sbb_1,\sbb_2,\sbb_3),\delta_2(\sbb_1,\sbb_2,\sbb_3)\right\}$, is a non-negative integer that depends on $k_1$ and $k_2$, and it is bounded above by $\max_{1\leq t\leq R}{m(t:H)} = T$, which is fixed and finite. Also note that from Assumption \ref{assump-3} in \eqref{PiN-lim}, we can claim that there exists a positive integer $M_1$ (possibly depending on $\beta$), such that for all $N\geq M_1$,
$$
N^\beta \max_{1\leq i,j\leq K}\pi_{N,i,j} \leq 2\cdot \max_{1\leq i,j\leq K} c_{i,j} = 2\cdot c_{\max},\quad\text{(say).}
$$
For $N\geq M_1$, using \eqref{E-set-bdd} and \eqref{E-set-bdd-1} the sum over $\Mc_{k_1,k_2}$ in the r.h.s. of \eqref{num-dW-l-1} can be written as 
\begin{align}\label{exp-Y}
&\sum_{(\sbb_1,\sbb_2,\sbb_3)\in \Mc_{k_1,k_2}} \E\left(\prod_{i=1}^3 Y_N(\sbb_i:H)\right)= \sum_{(\sbb_1,\sbb_2,\sbb_3)\in \mathcal{O}_{k_2,k_2}} \E\left(\prod_{i=1}^3 Y_N(\sbb_i:H)\right)+ \sum_{(\sbb_1,\sbb_2,\sbb_3)\in  \mathcal{O}^c_{k_2,k_2}} \E\left(\prod_{i=1}^3 Y_N(\sbb_i:H)\right)\notag\\
    &= \sum_{(\sbb_1,\sbb_2,\sbb_3)\in  \mathcal{O}_{k_2,k_2}} \prod_{\{i,j\}\in E\left(\cup_{i=1}^3 H(\sbb_i)\right)}\pi_{N,\al_i,\al_j}+ \sum_{(\sbb_1,\sbb_2,\sbb_3)\in \mathcal{O}^c_{k_2,k_2}} \prod_{\{i,j\}\in E\left(\cup_{i=1}^3 H(\sbb_i)\right)}\pi_{N,\al_i,\al_j}\notag\\
    &\leq\sum_{(\sbb_1,\sbb_2,\sbb_3)\in  \mathcal{O}_{k_2,k_2}} N^{-\beta(3T - m(k_1:H) -m(k_2:H))}{\left(N^{\beta}\max_{1\leq i,j\leq K}\pi_{N,i,j}\right)}^{3T - m(k_1:H) -m(k_2:H)}\notag\\
    &\qquad{}+ \sum_{(\sbb_1,\sbb_2,\sbb_3)\in  \mathcal{O}^c_{k_2,k_2}} N^{-\beta(3T - m(k_1:H) -m(k_2:H))}N^{-\beta\cdot\delta_{k_1,k_2}}\left(N^{\beta}\max_{1\leq i,j\leq K}\pi_{N,i,j}\right)^{{}3T - m(k_1:H) -m(k_2:H) +\delta_{k_1,k_2} }\notag\\
    &\leq {\left(2c_{\max}\right)}^{3T - m(k_1:H)-m(k_1:H) +\delta_{k_1,k_2}}\sum_{(\sbb_1,\sbb_2,\sbb_3)\in \Mc_{k_1,k_2}} N^{-\beta(3T - m(k_1:H) -m(k_2:H))}\notag\\
    &=a_2\cdot N^{3R-k_1-k_2-\beta(3T - m(k_1:H) -m(k_2:H))},
\end{align}
where $a_2 = a_2(\beta,H,\Cb) = \max_{k_1,k_2\in[R] }\left\{{\left(2c_{\max}\right)}^{3T - m(k_1:H)-m(k_1:H) +\delta_{k_1,k_2}} \right\}$, is a positive constant depending on $\beta$, the graph $H$, and $\Cb$. Using similar arguments as used earlier, the same upper bound can be used for the term $\E(Y_N(\sbb_1:H)\cdot Y_N(\sbb_2:H))\cdot \E(Y_N(\sbb_3:H))$, within the term $A_{N,l,2}$ in \eqref{ANl-brk}. Then, combining \eqref{num-dW-l-1}, and a similar upper bound for $A_{N,l,2}$ and using the definition of $g_{l,\al,\beta,H}(t)$, $t_l(\al,\beta;H)$ and $e_l(k_1,k_2:H)$, $l=1,2$, defined in \eqref{delta-bdd-temp} and \eqref{def-el}, we can claim that there exists a positive integer $M_2$, depending on    $l,\al,\beta$ and $H$, such that for all $N\geq M_2$, 
\begin{align}\label{dW-num}
     &A_{N,l,1}\leq a_{1}\cdot a_2\cdot \sum_{k_1,k_2=1}^R N^{3R -k_1-k_2-\al\cdot e_l(k_1,k_2) - \beta(2T- m(k_1:H) - m(k_2:H))} \notag\\
     &=a_{1}\cdot a_{2}\cdot N^{-3\cdot g_{l,\al,\beta,H}(R)}\sum_{k=1}^RN^{2\cdot g_{l,\al,\beta,H}(k)}
    \leq  a_{1}\cdot a_{2}\cdot N^{{}-3g_{l,\al,\beta,H}\left(R\right)}N^{ 2g_{l,\al,\beta,H}\left(t_l(\al,\beta;H)\right)}\{1+o(1)\}.
\end{align}
From Lemma \ref{lem-var}, there exists a positive integer $M_3$, depending on $l,\al,\beta$ and the graph $H$, such that for all $N\geq M_3$
\begin{align}\label{dW-denom}
    & a_{3}\leq \frac{\sigma^3_{N,l}(H)}{\left[N^{-2g_{l,\al,\beta,H}(R)} \cdot N^{g_{l,\al,\beta,H}(t_l(\al,\beta;H))}\right]^{3/2}}\leq a_{4},
\end{align}
where $a_{3} = \left[L(l,\al,\beta,H,\lamb, \Cb)\right]^{3/2}$ and $a_{4} = \left[U(l,\al,\beta,H,\lamb, \Cb)\right]^{3/2}$, are finite positive constants defined in Lemma \ref{lem-var}. Recall that Assumption \ref{assump-5} is used obtained \eqref{dW-denom} in case of $l=2$ or egocentric network formation. Then, combining \eqref{dW-num} and the lower bound of \eqref{dW-denom}, for $N\geq \max\{M_1,M_2,M_3\}$, we get the following bound
\begin{align*}
    A_{N,l,1} \leq  \frac{a_{1}\cdot a_{2}}{a_{3}}\cdot \displaystyle N^{~\displaystyle \left\{g_{l,\al,\beta,H}\left(t_l(\al,\beta;H)\right)/2\right\}\displaystyle}.
\end{align*}
A similar upper bound is also available for the term $A_{N,l,2}$ using the previously stated arguments. As a result, 
\begin{align*}
    d_1\!\left(T_{N,l}(H),~Z\right)
    \leq b_l(H)\,
    N^{~\displaystyle \left\{g_{l,\al,\beta,H}\left(t_l(\al,\beta;H)\right)/2\right\}\displaystyle},\quad\text{for $l=1,2$ and $N\geq \max\{M_1,M_2,M_3\}$,}
\end{align*}
where, $b_l(H) = \frac{a_1\cdot a_2}{a_3}$, $l=1,2$, depend on $H$, the model parameters, and the sparsity levels $(\al,\beta)$. 
This completes the proof. 
\end{proof}
\bigskip


\begin{proof}[Proof of Corollary \ref{cor-clt-all}]
For part (a), we have already established this statement in the ego-centric case (for $l=2$) in \eqref{g2-bdd-ego}. The  induced case (for $l=1$) follows similarly. To prove part (b), note that from Theorem \eqref{thm-1-main}, we know
\begin{align*}
d_1\left(T_{N,l}(H),Z\right) \leq b_l(H)\cdot \displaystyle N^{~\displaystyle \left\{g_{l,\al,\beta,H}\left(t_l(\al,\beta;H)\right)/2\right\}\displaystyle},
\end{align*}
for large enough $N$. When $(\al,\beta)\in C_l(H)$ (cf. \eqref{c-set-ind}, \eqref{c-set-ego}), the exponent of $N$ in the above bound is negative, which implies that stated order relation holds for $l=1,2$. Part (c) is trivial consequence of the fact that convergence in $d_1$ distance implies convergence in distribution. 
\end{proof}
\bigskip


\begin{proof}[Proof of Theorem \ref{thm-2-main-temp}(a): Dense sampling]
In this case, $(\al,\beta)=(0,R/T)$. As a result $p_N= c/N^{\al} = c\in (0,1)$, for all $N\geq 1$. For $l=1,2$, define the random vectors 
\begin{equation}
\begin{aligned}
\label{U-vec-def}
\mathbf{U}_{N,l}
&= (U_{N,l,1},U_{N,l,2})^\primet
 = \bigl(\widehat{S}_{N,l}(H),\, S_N(H)-\widehat{S}_{N,l}(H)\bigr)^\primet = \left(\sum_{\sbb\in\nsr} X_l^1(\sbb),\, \sum_{\sbb\in\nsr} X_l^2(\sbb)\right)^\primet,~\text{where,}\\
X_l^{1}(\sbb)
&= \prod_{\{i,j\}\in E(H)} h_l(W_{N,s_i},W_{N,s_j})\, Y_{s_i,s_j},
\qquad
X_l^{2}(\sbb)
 = \left(1-\prod_{\{i,j\}\in E(H)} h_l(W_{N,s_i},W_{N,s_j})\right)
   \prod_{\{i,j\}\in E(H)} Y_{s_i,s_j}.
\end{aligned}    
\end{equation}
We develop a multivariate Poisson approximation for the vector $\mathbf{U}_{N,l}$ using Theorem~2.6 of~\cite{nualart2025}
(see Theorem~\ref{thm-mult-SB} for details). The first step is to construct a size-biased coupling for $\mathbf{U}_{N,l}$, for $l=1,2$. For each $\rb \in \nsr$, we define the following random vectors,
\begin{align}\label{new-U-vec}
    &\mathbf{U}_{N,l}^{1,\rb} = \left(\left(X_l^{1}({\sbb})\right)_{\sbb\in \nsr, \sbb\neq \rb}\right)^\primet~\text{and}~\mathbf{U}_{N,l}^{2,\rb} = \left(\left(X_l^{1}({\sbb})\right)_{\sbb\in \nsr}, \left(X_l^{2}({\sbb})\right)_{\sbb\in \nsr,\sbb\neq \rb}\right)^\primet,~\text{for $l=1,2$.}
\end{align}
Next, we define the following random vectors on the same probability space,
\begin{align}
\label{new-U-vec-1}
\mathbf{U}_{N,l}^{1,(1,\rb)}
= \Bigl(\bigl(X_{l}^{1,(1,\rb)}(\sbb)\bigr)_{\sbb\in\nsr,\ \sbb\neq \rb}\Bigr)^\primet~\text{and}\quad\mathbf{U}_{N,l}^{2,(2,\rb)}
= \Bigl(\bigl(X_{l}^{1,(2,\rb)}(\sbb)\bigr)_{\sbb\in\nsr},\,
          \bigl(X_{l}^{2,(2,\rb)}(\sbb)\bigr)_{\sbb\in\nsr,\ \sbb\neq \rb}\Bigr)^\primet,
\end{align}
such that the following relation holds,
\begin{align}\label{cond-dist}
        \mathbf{U}_{N,l}^{1,(1,\rb)} \eqd \left[\mathbf{U}_{N,l}^{1,\rb}~|~ X_l^1({\rb})=1\right]~\text{and}~\mathbf{U}_{N,l}^{2,(2,\rb)} \eqd \left[\mathbf{U}_{N,l}^{2,\rb}\,| X_l^2(\rb)=1\right].
\end{align}
We choose an index $\mathbf{I}_i\in \nsr$ for $i=1,2$, with the following probability distribution, 
\begin{align}\label{rand-ind}
        \pr\left(\mathbf{I}_i = \rb\right) = \frac{\E(X_{l}^i({\rb}))}{\E(U_{N,l,i})},\quad \rb\in \nsr~\text{and}~\sum_{\rb\in \nsr}\pr\left(\mathbf{I}_i = \rb\right)=1,
\end{align}
where the choice of $\mathbf{I}_i$, $i=1,2$, is independent of $\mathbf{U}_{N,l}^{j,\rb}$ and $\mathbf{U}_{N,l}^{j,(j,\rb)}$ for all $j=1,2$, see \eqref{new-U-vec-1}. The size biased coupling of $\mathbf{U}_{N,l}$, $l=1,2$, (cf. \eqref{U-vec-def}) is denoted as $\widetilde{\mathbf{U}}_{N,l} = {\left(\widetilde{\mathbf{U}}_{N,l}^1,\widetilde{\mathbf{U}}_{N,l}^2\right)}^\primet$, and it is defined as
\begin{equation}
\label{SB-U}
\begin{aligned}
&\widetilde{\mathbf{U}}_{N,l}^1 = \left(\sum_{\sbb\in \nsr,\sbb \neq \mathbf{I}_1}X_l^{1,(1,\mathbf{I}_1)}({\sbb}) + 1\right) = \widetilde{U}^1_{N,l,1},\notag\\
&\widetilde{\mathbf{U}}_{N,l}^2 = {\left(\sum_{\sbb\in \nsr}X_l^{1,(2,\mathbf{I}_2)}({\sbb}),\sum_{\sbb\in \nsr,\sbb \neq \mathbf{I}_2}X_l^{2,(2,\mathbf{I}_2)}({\sbb}) + 1\right)}^\primet = {\left(\widetilde{U}^2_{N,l,1},\widetilde{U}^2_{N,l,2}\right)}^\primet,
\end{aligned}\quad \text{for $l=1,2$.}
\end{equation}
Firstly we show that $\widetilde{\mathbf{U}}_{N,l}={\left(\widetilde{\mathbf{U}}_{N,l}^1,\widetilde{\mathbf{U}}_{N,l}^2\right)}^\primet$ (cf. \eqref{SB-U}) is indeed a size biased coupling of $\Ub_{N,l}$, for $l=1,2$. Using \eqref{cond-dist} and \eqref{rand-ind}, we note that for any suitable value of $u$,
\begin{align}\label{SB-P}
    &\E(U_{N,l,1})\cdot\pr\left(\widetilde{U}^1_{N,l,1} = u\right) = \E(U_{N,l,1})\cdot\pr\left(\sum_{\sbb\in \nsr,\sbb \neq \mathbf{I}_1}X_l^{1,(1,\mathbf{I}_1)}({\sbb}) + 1 = u\right)\notag\\
    &=\sum_{\rb\in\nsr }\pr(\mathbf{I}_1 = \rb)\cdot\pr\left(\sum_{\sbb\in \nsr,\sbb \neq \rb}X_l^{1,(1,\rb)}({\sbb}) + 1 = u\right)\notag\\
    &= \sum_{\rb\in \nsr}\pr(X^1_{l}(\rb) =1)\cdot\pr\left(\sum_{\sbb\in \nsr}X_l^{1}({\sbb}) - X^1_{l}(\rb)+ 1 = u~|~X^1_{l}(\rb) =1\right)\notag\\
    &= \sum_{\rb\in \nsr}\pr\left(\sum_{\sbb\in \nsr}X_l^{1}({\sbb}) = u,~X^1_{l}(\rb) =1\right) = \sum_{\rb\in \nsr}\pr\left(U_{N,l,1} = u,~X^1_{l}(\rb) =1\right)\notag\\
    &= \sum_{\rb\in \nsr}\E\left(X^1_{l}(\rb)\cdot\mathbf{1}( U_{N,l,1}=u)\right) = \E\left(\sum_{\rb\in \nsr}X^1_{l}(\rb)\cdot\mathbf{1}( U_{N,l,1}=u) \right) = u\cdot \pr\left(U_{N,l,1} = u\right)\notag\\
    &\Rightarrow ~ \pr\left(\widetilde{U}^1_{N,l,1} = u\right) = \frac{u}{\E\left(U_{N,l,1}\right)}\cdot\pr\left(U_{N,l,1}=u\right).
\end{align} 
A similar argument can be used to show that 
\begin{align}\label{SB-P-1}
    \pr\left((\widetilde{U}^2_{N,l,1}, \widetilde{U}^2_{N,l,2})^\primet = (u_1,u_2)^\primet\right) = \frac{u_2}{\E(U_{N,l,2})}\cdot \pr\left((U_{N,l,1},U_{N,l,2})^\primet = (u_1,u_2)^\primet\right).
\end{align}
This ensures that $\widetilde{\mathbf{U}}_{N,l}$ is a size biased version of $\mathbf{U}_{N,l}$, for $l=1,2$. Now, define the random vectors $\mathbf{Z}_{N,l} = {(Z_{N,l,1},\, Z_{N,l,2})}^\primet$, for $l=1,2$, where $Z_{N,l,1}\sim \mathrm{Poisson}\bigl(\E(U_{N,l,1})\bigr)$ and $Z_{N,l,2}\sim \mathrm{Poisson}\bigl(\E(U_{N,l,2})\bigr)$, for $l=1,2$, and $Z_{N,l,1}$ and $Z_{N,l,2}$ are independent. We apply Theorem 2.6 \cite{nualart2025} (see Theorem \ref{thm-mult-SB}) to obtain the following upper bound on the Wasserstein distance $d_W$ between the random vectors $\Ub_{N,l}$ and $\mathbf{Z}_{N,l}$, for $l=1,2$,
\begin{align}\label{dW-bdd-poi}
    d_{{W}}\left(\mathbf{U}_{N,l},~ \mathbf{Z}_{N,l}\right)&\leq \sum_{i=1}^2\min\{1,E(U_{N,l,i})\}\cdot \E\left|\widetilde{U}_{N,l,i}^i-1-U_{N,l,i}\right|+ 2\cdot \E(U_{N,l,2}) \cdot\E\left|\widetilde{U}_{N,l,1}^2-U_{N,l,1}\right|\notag\\
    &= \sum_{i=1}^2\min\{1,E(U_{N,l,i})\}\cdot B_{N,l,i} +2\cdot \E(U_{N,l,2})\cdot B_{N,l,3},~\text{(say)},
\end{align}
where, $\widetilde{U}^1_{N,l,1}, \widetilde{U}^2_{N,l,1}$ and $\widetilde{U}^2_{N,l,1}$, $l=1,2$, are defined in \eqref{SB-U}.  Using \eqref{U-vec-def}, \eqref{SB-U} and \eqref{rand-ind}, for $i=1,2$, we can write
\begin{align}\label{B-1}
    &B_{N,l,i} = \E\left|\widetilde{U}^i_{N,l,i} -1 -U_{N,l,i}\right|=\E\left|\sum_{\sbb\in \nsr,\sbb\neq \mathbf{I}_i}X_l^{i,(i,\mathbf{I}_i)}(\sbb) -\sum_{\sbb\in \nsr}X^i_l(\sbb)\right|\notag\\
    &= \sum_{\rb\in \nsr}\pr\left(\mathbf{I}_i=\rb\right)\E\left|\sum_{\sbb\in \nsr,\sbb\neq \rb}\left(X^{i,(i,\rb)}_l(\sbb) - X^i_l(\sbb)\right) -X^i_l(\rb)\right|\notag\\
    &= \sum_{\rb\in \nsr}\pr\left(\mathbf{I}_i=\rb\right)\E\left|\sum_{\substack{\sbb \in \nsr,\sbb\neq \rb \\ A(\sbb)\cap A(\rb) \neq\varnothing}}\left(X^{i,(i,\rb)}_l(\sbb) - X^i_l(\sbb)\right) -X^i_l(\rb)\right|\notag\\
    &\leq \sum_{\rb\in \nsr}\pr\left(\mathbf{I}_i=\rb\right)\E\left(\sum_{\substack{\sbb \in \nsr ,\sbb\neq \rb\\ A(\sbb)\cap A(\rb) \neq\varnothing}}\left(X^{i,(i,\rb)}_l(\sbb) + X^i_l(\sbb)\right) +X^i_l(\rb)\right)\notag\\
    &=\sum_{\rb\in \nsr}\pr\left(\mathbf{I}_i=\rb\right)\E\left(\sum_{\substack{\sbb \in \nsr,\sbb\neq \rb \\ A(\sbb)\cap A(\rb) \neq\varnothing}}\left(X^{i,(i,\rb)}_l(\sbb) - X^i_l(\sbb)\right) +2\sum_{\substack{\sbb \in \nsr,\sbb\neq\rb \\ A(\sbb)\cap A(\rb) \neq\varnothing}}X^i_l(\sbb) +X^i_l(\rb)\right)\notag\\
    &=\sum_{\rb\in \nsr}\pr\left(\mathbf{I}_i=\rb\right)\E\left(\sum_{\sbb \in \nsr,\sbb\neq \rb}\left(X^{i,(i,\rb)}_l(\sbb) - X^i_l(\sbb)\right) +2\sum_{\substack{\sbb \in \nsr,\sbb\neq\rb \\ A(\sbb)\cap A(\rb) \neq\varnothing}}X^i_l(\sbb) +X^i_l(\rb)\right)\notag\\
    &=\sum_{\rb\in \nsr}\pr\left(\mathbf{I}_i=\rb\right)\E\left(\left\{\sum_{\sbb \in \nsr}X^{i,(i,\rb)}_l(\sbb)-1 - \sum_{\sbb \in \nsr}X^i_l(\sbb)\right\} +2\sum_{\substack{\sbb \in \nsr \\ A(\sbb)\cap A(\rb) \neq\varnothing}}X^i_l(\sbb)\right)\notag\\
    &=\E\left(\widetilde{U}_{N,l,i}^i -1- U_{N,l,i} \right) +2\sum_{\rb\in \nsr}\pr\left(\mathbf{I}_i=\rb\right)\sum_{\sbb\in N(\rb)}\E\left(X^i_l(\sbb)\right)\notag\\
    &= \sum_{x_i\in \mathbb{N}_0}x_i\pr(\widetilde{U}_{N,l,i}^i = x_i) - 1 -\E(U_{N,l,i}) + \frac{2}{\E(U_{N,l,i})}\sum_{\rb\in \nsr}\sum_{\sbb\in N(\rb)}\E(X_l^i(\sbb))\E(X_l^i(\rb))\notag\\
     &= \sum_{x_1,\ldots,x_i\in \mathbb{N}_0}x_i\cdot\pr(\widetilde{U}^i_{N,l,1} =x_1,\ldots,\widehat{U}^i_{N,l,i-1}=x_{i-1},\widetilde{U}_{N,l,i}^i = x_i) \notag\\
     &\qquad{}\qquad{}- 1 -\E(U_{N,l,i})+ \frac{2}{\E(U_{N,l,i})}\sum_{\rb\in \nsr}\sum_{\sbb\in N(\rb)}\E(X_l^i(\sbb))\E(X_l^i(\rb))\notag\\
     &=\frac{1}{\E(U_{N,l,i})} \sum_{x_1,\ldots,x_i\in \mathbb{N}_0}x_i^2\cdot\pr({U}_{N,l,1} =x_1,\ldots,{U}_{N,l,i-1}=x_{i-1},{U}_{N,l,i} = x_i) \notag\\
     &\qquad{}\qquad{}- 1 -\E(U_{N,l,i})+ \frac{2}{\E(U_{N,l,i})}\sum_{\rb\in \nsr}\sum_{\sbb\in N(\rb)}\E(X_l^i(\sbb))\E(X_l^i(\rb))\notag\\
     &=\frac{1}{\E(U_{N,l,i})} \sum_{x_i\in \mathbb{N}_0}x_i^2\cdot\pr({U}_{N,l,i} = x_i) - 1 -\E(U_{N,l,i})\notag\\
     &\qquad {}+ \frac{2}{\E(U_{N,l,i})}\sum_{\rb\in \nsr}\sum_{\sbb\in N(\rb)}\E(X_l^i(\sbb))\E(X_l^i(\rb))\notag\\
    &= \frac{\E([U_{N,l,i}]^2)}{\E(U_{N,l,i})} - 1 -\E(U_{N,l,i}) + \frac{2}{\E(U_{N,l,i})}\sum_{\rb\in \nsr}\sum_{\sbb\in N(\rb)}\E(X_l^i(\sbb))\E(X_l^i(\rb))\notag\\
    &=\frac{1}{\E(U_{N,l,i})}\left(\Var\left(U_{N,l,i}\right) - \E\left(U_{N,l,i}\right)\right) +\frac{2}{\E(U_{N,l,i})}\sum_{\rb\in \nsr}\sum_{\sbb\in N(\rb)}\E(X_l^i(\sbb))\E(X_l^i(\rb)).
    \end{align}
Next, we analyze the term $B_{N,l,3}$ from the expression \eqref{dW-bdd-poi}, that is,
\begin{align}\label{B-3}
    &B_{N,l,3}= \E\left|\widetilde{U}^2_{N,l,1} - U_{N,l,1}\right|=\sum_{\rb\in \nsr}\pr(\mathbf{I}_2 = \rb)\cdot \E\left|\sum_{\sbb\in \nsr}X_l^{1,(2,\rb)}(\sbb)- \sum_{\sbb\in \nsr}X_l^1(\sbb)\right|\notag\\
    &= \sum_{\rb\in \nsr}\pr(\mathbf{I}_2 = \rb)\cdot \E\left|\sum_{\substack{\sbb\in \nsr\\ |A(\sbb)\cap A(\rb)|\neq \varnothing}}\left(X_l^{1,(2,\rb)}(\sbb)- X_l^1(\sbb)\right)\right|\notag\\
    &\leq \sum_{\rb\in \nsr}\pr(\mathbf{I}_2 = \rb)\cdot \E\left(\sum_{\substack{\sbb\in \nsr\\ |A(\sbb)\cap A(\rb)|\neq \varnothing}}\left(X_l^{1,(2,\rb)}(\sbb)- X_l^1(\sbb)\right)\right)\notag\\
    &\qquad {}+ 2\sum_{\rb\in \nsr}\pr(\mathbf{I}_2 = \rb)\sum_{\substack{\sbb\in \nsr\\ |A(\sbb)\cap A(\rb)|\neq \varnothing}}\E(X_l^1(\sbb))\notag\\
    &= \sum_{\rb\in \nsr}\pr(\mathbf{I}_2 = \rb)\cdot \E\left(\sum_{\sbb\in \nsr}X_l^{1,(2,\rb)}(\sbb)- \sum_{\sbb\in \nsr}X_l^1(\sbb)\right)+ 2\sum_{\rb\in \nsr}\sum_{\sbb\in N(\rb)}\pr(\mathbf{I}_2 = \rb)\E(X_l^1(\sbb))\notag\\
    &= \E\left(\widetilde{U}^2_{N,l,1} - U_{N,l,1}\right) + \frac{2}{\E(U_{N,l,2})}\sum_{\rb\in \nsr}\sum_{\sbb\in N(\rb)}\E(X^2_l(\rb))\E(X_l^1(\sbb))\notag\\
    &= \E\left(\widetilde{U}^2_{N,l,1}\right) - \E(U_{N,l,1}) +\frac{2}{\E(U_{N,l,2})}\sum_{\rb\in \nsr}\sum_{\sbb\in N(\rb)}\E(X^2_l(\rb))\E(X_l^1(\sbb))\notag\\
    &= \sum_{x\in \mathbb{N}_0}x\pr\left(\widetilde{U}^2_{N,l,1} = x\right)- \E(U_{N,l,1}) +\frac{2}{\E(U_{N,l,2})}\sum_{\rb\in \nsr}\sum_{\sbb\in N(\rb)}\E(X^2_l(\rb))\E(X_l^1(\sbb))\notag\\
    &= \sum_{x\in \mathbb{N}_0} \sum_{y\in \mathbb{N}_0}x\pr\left(\widetilde{U}^2_{N,l,1} = x, \widetilde{U}^2_{N,l,2} =y\right) - \E(U_{N,l,1}) +\frac{2}{\E(U_{N,l,2})}\sum_{\rb\in \nsr}\sum_{\sbb\in N(\rb)}\E(X^2_l(\rb))\E(X_l^1(\sbb))\notag\\
    &=\sum_{x\in \mathbb{N}_0} \sum_{y\in \mathbb{N}_0}\frac{xy\cdot \pr\left(U_{N,l,1} = x, {U}_{N,l,2} =y\right)}{E(U_{N,l,2})}- \E(U_{N,l,1}) +\frac{2}{\E(U_{N,l,2})}\sum_{\rb\in \nsr}\sum_{\sbb\in N(\rb)}\E(X^2_l(\rb))\E(X_l^1(\sbb))\notag\\
    &= \frac{1}{\E(U_{N,l,2})}\left(\E(U_{N,l,1}\cdot U_{N,l,2}) - \E(U_{N,l,1})\E(U_{N,l,2})\right) +\frac{2}{\E(U_{N,l,2})}\sum_{\rb\in \nsr}\sum_{\sbb\in N(\rb)}\E(X^2_l(\rb))\E(X_l^1(\sbb))\notag\\
    &= \frac{\cov(U_{N,l,1}, U_{N,l,2})}{\E(U_{N,l,2})} +\frac{2}{\E(U_{N,l,2})}\sum_{\rb\in \nsr}\sum_{\sbb\in N(\rb)}\E(X^2_l(\rb))\E(X_l^1(\sbb)).
\end{align}
Therefore, combining \eqref{B-1} and \eqref{B-3}, the upper bound in \eqref{dW-bdd-poi} reduces to
\begin{align}\label{dW-bdd-poi-1}
&d_{{W}}(\Ub_{N,l}, \mathbf{Z}_{N,l})\leq\sum_{i=1}^2 \min\left\{1,\frac{1}{\E(U_{N,l,i})}\right\}\left\{\Var\left(U_{N,l,i}\right) - \E\left(U_{N,l,i}\right) +2\sum_{\rb\in \nsr}\sum_{\sbb\in N(\rb)}\E\left(X_l^i(\sbb)\right)\cdot\E\left(X_l^i(\rb)\right)\right\}\notag\\
    &\qquad{}\qquad{}+2\cdot\cov(U_{N,l,1}, U_{N,l,2}) +4\sum_{\rb\in \nsr}\sum_{\sbb\in N(\rb)}\E\left(X^2_l(\rb)\right)\cdot\E\left(X_l^1(\sbb)\right),\ \text{for $l=1,2$,}
\end{align}
where, $N(\rb) = \{\sbb: |A(\sbb)\cap A(\rb)|\geq 1\}$. 
Note that, for $l=1,2$, using the definitions provided in \eqref{U-vec-def} we can write
\begin{align}\label{var-exp-diff}
    &\left|\Var(U_{N,l,1}) - \E\left(U_{N,l,1}\right)\right| = \left|\Var\left(\widehat{S}_{N,l}(H)\right) - \E\left(\widehat{S}_{N,l}(H)\right)\right|\notag\\
    &= \left|\sum_{t=1}^R\sum_{(\sbb_1,\sbb_2)\in \mathcal{M}_t}\cov\left(W_{N,l}(\sbb_1:H)Y_N(\sbb_1:H), W_{N,l}(\sbb_2:H)Y_N(\sbb_2:H)\right) - \sum_{\sbb\in \nsr}\E(W_{N,l}(\sbb:H)\cdot Y_N(\sbb:H))\right|\notag\\
    &= \left|\sum_{t=1}^{R-1}\sum_{(\sbb_1,\sbb_2)\in \mathcal{M}_t}\cov\left(W_{N,l}(\sbb_1:H)Y_N(\sbb_1:H), W_{N,l}(\sbb_2:H)Y_N(\sbb_2:H)\right) \right.\notag\\
    &\left.{}+ \sum_{\sbb\in \nsr}\Var\left(W_{N,l}(\sbb:H)\cdot Y_N(\sbb:H)\right)- \sum_{\sbb\in \nsr}\E(W_{N,l}(\sbb:H)\cdot Y_N(\sbb:H))\right|\notag\\
    &=\left| \sum_{t=1}^{R-1}\sum_{(\sbb_1,\sbb_2)\in \mathcal{M}_t}\cov\left(W_{N,l}(\sbb_1:H)Y_N(\sbb_1:H), W_{N,l}(\sbb_2:H)Y_N(\sbb_2:H)\right)+\right.\notag\\
    &\left.+ \sum_{\sbb\in \nsr}\left(\E\left(W_{N,l}(\sbb:H)Y_N(\sbb:H)\right) - \left[\E\left(W_{N,l}(\sbb:H)Y_N(\sbb:H)\right)\right]^2\right)- \sum_{\sbb\in \nsr}\E(W_{N,l}(\sbb:H)\cdot Y_N(\sbb:H))\right|\notag\\
    &=\left| \sum_{t=1}^{R-1}\sum_{(\sbb_1,\sbb_2)\in \mathcal{M}_t}\cov\left(W_{N,l}(\sbb_1:H)Y_N(\sbb_1:H), W_{N,l}(\sbb_2:H)Y_N(\sbb_2:H)\right)\right.\notag\\
    &\qquad \left. {}-\sum_{\sbb\in \nsr}\left[\E\left(W_{N,l}(\sbb:H)Y_N(\sbb:H)\right)\right]^2\right|\notag\\
    &= \left|D^{(1)}_{N,l,1} -D^{(1)}_{N,l,2}\right|,\quad\text{(say),}\notag\\
    & \leq |D^{(1)}_{N,l,1}| +|D^{(1)}_{N,l,2}|.
\end{align}
From the expression in \eqref{var-exp-diff}, using the fact that $H$ is strictly balanced, $(\al,\beta) = (0,R/T)$ and using Lemma \ref{lem-1}, there exists a positive integer $M_4$, dependent on the graph $H$ and the model parameters, such that for $N\geq M_4$ we can claim,
\begin{align}\label{D-1}
    &|D^{(1)}_{N,l,1}| = \left|\sum_{t=1}^{R-1}\sum_{(\sbb_1,\sbb_2)\in \mathcal{M}_t}\cov\left(W_{N,l}(\sbb_1:H)Y_N(\sbb_1:H), W_{N,l}(\sbb_2:H)Y_N(\sbb_2:H)\right)\right|\notag\\
    &= \left|\sum_{t=1}^{R-1}\sum_{(\sbb_1,\sbb_2)\in \mathcal{M}_t}\E(W_{N,l}(\sbb_1:H)W_{N,l}(\sbb_2:H))\cdot\E(Y_N(\sbb_1:H)Y_N(\sbb_2:H))\right.\notag\\
    &\left.\qquad{}\qquad{}- \sum_{t=1}^{R-1}\sum_{(\sbb_1,\sbb_2)\in \mathcal{M}_t}\E(W_{N,l}(\sbb_1:H))\E(W_{N,l}(\sbb_2:H))\cdot \E(Y_N(\sbb_1:H))\E(Y_N(\sbb_2:H))\right|\notag\\
    &\leq \left|\sum_{t=1}^{R-1}\sum_{(\sbb_1,\sbb_2)\in \mathcal{M}_t}\E(Y_{N}(\sbb_1,\sbb_2)Y_{N}(\sbb_2:H)) - \sum_{t=1}^{R-1}\sum_{(\sbb_1,\sbb_2)\in \mathcal{M}_t}f_l(c:H)^2\prod_{\{i,j\}\in E(H)}\pi_{N,\al_{s_{1,i}},\al_{s_{1,j}}}\pi_{N,\al_{s_{2,i}},\al_{s_{2,j}}}\right|\notag\\
    &\leq \left|\sum_{t=1}^{R-1}N^{2R-2T\beta -t + \beta m(t;H)}\left(N^{\beta}\max_{i,j\in [K]}\pi_{N,i,j}\right)^{2T -m(t:H)}\right|\notag\\
    &\qquad{}\qquad{}+\left|\sum_{t=1}^{R-1}N^{2R-t - 2T\beta}\sum_{\ub,\vb\in [K]^R}f_l(c:H)^2\prod_{\{i,j\}\in E(H)}N^{2\beta}\pi_{N,u_i,u_j}\pi_{N,v_i,v_j}\left(\prod_{i=1}^R\la_{N,u_i}\la_{N,v_i}+O\left(\frac{1}{N}\right)\right)\right|\notag\\
    &= \left(N^{\beta}\max_{i,j\in [K]}\pi_{N,i,j}\right)^{2T}\sum_{t=1}^{R-1}N^{-t + \frac{R}{T} m(t;H)}\notag\\
    &\qquad{}\qquad{}+\sum_{t=1}^{R-1}N^{-t}\sum_{\ub\in [K]^R}f_l(c:H)^2\prod_{\{i,j\}\in E(H)}N^{2\beta}\pi_{N,u_i,u_j}\pi_{N,v_i,v_j}\left(\prod_{i=1}^R\la_{N,u_i}\la_{N,v_i}+O\left(\frac{1}{N}\right)\right).
\end{align}
In the expression \eqref{D-1}, the exponent in the first term is negative since H is strictly balanced. This implies $\frac{t}{m(t:H)}> \frac{R}{T}$ for all $t\in [R-1]$. Hence, using Assumptions \ref{assump-2} and \ref{assump-3}, it follows that $D^{(1)}_{N,l,1}\to 0$, as $N\to\infty$. Next, consider the term $|D^{(1)}_{N,l,2}|$ in \eqref{var-exp-diff}. Write 
\begin{align}\label{D-2}
    &|D^{(1)}_{N,l,2}| = \sum_{\sbb\in \nsr}\left[\E\left(W_{N,l}(\sbb:H)Y_N(\sbb:H)\right)\right]^2 = \sum_{\sbb\in \nsr}f_l^2(c:H)\prod_{\{i,j\}\in E(H)}\pi_{N,\al_{s_i},\al_{s_j}}\notag\\
    &= N^{R-2T\beta}f_l^2(c:H)\sum_{\ub\in [K]^R}\prod_{\{i,j\}\in E(H)}N^{2\beta}\pi^2_{N,u_i,u_j}\left(\prod_{i=1}^R\la_{N,s_i}+O\left(\frac{1}{N}\right)\right)\notag\\
    &= N^{-T\beta}f_l^2(c:H)\sum_{\ub\in [K]^R}\prod_{\{i,j\}\in E(H)}N^{2\beta}\pi^2_{N,u_i,u_j}\left(\prod_{i=1}^R\la_{N,s_i}+O\left(\frac{1}{N}\right)\right).
\end{align}
Therefore, using Assumptions \ref{assump-2} and \ref{assump-3}, $D_{2,N}\to 0$ as $N\to\infty$. Hence, combining \eqref{D-1} and \eqref{D-2}, we can say that $\left\{\Var(U_{N,l,1}) - \E\left(U_{N,l,1}\right)\right\}\to 0$ as $N\to\infty$. Using a similar argument as the one used in \eqref{var-exp-diff} above, one can show that $|\Var(U_{N,l,2}) - \E(U_{N,l,2})|\leq |D^{(2)}_{N,l,1}| + |D^{(2)}_{N,l,2}|$, where 
\begin{align*}
    &\left|D^{(2)}_{N,l,1}\right|\leq \left(N^{\beta}\max_{i,j\in [K]}\pi_{N,i,j}\right)^{2T}\sum_{t=1}^{R-1}N^{-t + \frac{R}{T} m(t;H)}\notag\\
    &\qquad{}\qquad{}+\sum_{t=1}^{R-1}N^{-t}\sum_{\ub\in [K]^R}(1-f_l(c:H))^2\prod_{\{i,j\}\in E(H)}N^{2\beta}\pi^2_{N,u_i,u_j}\left(\prod_{i=1}^R\la_{N,u_i}\la_{N,v_i}+O\left(\frac{1}{N}\right)\right),\quad\text{and}\notag\\
    &|D^{(2)}_{N,l,2}| = N^{-T\beta}(1-f_l(c:H))^2\sum_{\ub\in [K]^R}\prod_{\{i,j\}\in E(H)}N^{2\beta}\pi^2_{N,u_i,u_j}\left(\prod_{i=1}^R\la_{N,s_i}+O\left(\frac{1}{N}\right)\right),
\end{align*}
using arguments similar to those used in \eqref{D-1} and \eqref{D-2}. Thus, using Assumption \ref{assump-2}, \ref{assump-3} and the fact that $\beta=R/T$, we can say that $\left|D^{(2)}_{N,l,1}\right|+\left|D^{(2)}_{N,l,2}\right|\to 0$ as $N\to\infty$. Hence $\{\Var(U_{N,l,2}) -\E(U_{N,l,2})\}\to 0$ as $N\to \infty$. Next, consider the term $\cov(U_{N,l,1},U_{N,l,2})$ in the r.h.s. of \eqref{dW-bdd-poi-1}. Using similar arguments as \eqref{D-1} we write,
\begin{align}\label{cov-term}
    &\left|\cov\left(U_{N,l,1},U_{N,l,2}\right)\right|\notag\\
    &=\left|\sum_{t=1}^{R}\sum_{(\sbb_1,\sbb_2)\in \mathcal{M}_t}\cov\left(W_{N,l}(\sbb_1:H)Y_N(\sbb:H),(1-W_{N,l})Y_N(\sbb_2:H)\right)\right|\notag\\
    &\leq \left|\sum_{t=1}^{R-1}\sum_{(\sbb_1,\sbb_2)\in \mathcal{M}_t}\cov(W_{N,l}(\sbb_1:H) Y_{N}(\sbb_1:H), (1-W_{N,l}(\sbb_2:H)) Y_{N}(\sbb_2:H))\right|\notag\\
    &\qquad{}+ \left|\sum_{(\sbb_1,\sbb_2)\in \mathcal{M}_R}\cov(W_{N,l}(\sbb_1:H) Y_{N}(\sbb_1:H),(1- W_{N,l}(\sbb_2:H)) Y_{N}(\sbb_2:H))\right|\notag\\
    &\leq \left|\sum_{t=1}^{R-1}\sum_{(\sbb_1,\sbb_2)\in \mathcal{M}_t}\left(\E(Y_N(\sbb_1:H)Y_N(\sbb_2:H)) -[((1-f_l(c,H))f_l(c:H))\cdot\E(Y_N(\sbb_1:H))]^2\right)\right|\notag\\
    &\qquad{}+ \left|\sum_{(\sbb_1,\sbb_2)\in \mathcal{M}_R}\cov(W_{N,l}(\sbb_1:H) Y_{N}(\sbb_1:H),(1- W_{N,l}(\sbb_2:H)) Y_{N}(\sbb_2:H))\right|\notag\\
    &\leq \left|\left(N^{\beta}\max_{i,j\in [K]}\pi_{N,i,j}\right)^{2T}\sum_{t=1}^{R-1}N^{2R-2T\beta -t + \beta m(t;H)}\right|\notag\\
    &\qquad{}+\left|\sum_{t=1}^{R-1}N^{2R-t - 2T\beta}\sum_{\ub\in [K]^R}f_l(c:H)(1-f_l(c:H))\prod_{\{i,j\}\in E(H)}N^{2\beta}\pi^2_{N,u_i,u_j}\left(\prod_{i=1}^R\la_{N,u_i}+O\left(\frac{1}{N}\right)\right)\right|\notag\\
    &\qquad{}\qquad{}+\left|f_l(c:H)(1-f_l(c:H))\sum_{\sbb\in \nsr}\prod_{\{i,j\}\in E(H)}\pi^2_{N,\al_{s_i},\al_{s_j}}\right|\notag\\
    &=\left|\left(N^{\beta}\max_{i,j\in [K]}\pi_{N,i,j}\right)^{2T}\sum_{t=1}^{R-1}N^{2R-2T\beta -t + \beta m(t;H)}\right|\notag\\
    &\qquad{}+\left|\sum_{t=1}^{R-1}N^{2R-t - 2T\beta}\sum_{\ub\in [K]^R}f_l(c:H)(1-f_l(c:H))\prod_{\{i,j\}\in E(H)}N^{2\beta}\pi^2_{N,u_i,u_j}\left(\prod_{i=1}^R\la_{N,u_i}+O\left(\frac{1}{N}\right)\right)\right|\notag\\
    &\qquad{}\qquad{}+\left|N^{R-2T\beta}f_l(c:H)(1-f_l(c:H))\sum_{\ub\in [K]^R}N^{2\beta}\prod_{\{i,j\}\in E(H)}\pi^2_{N,u_i,u_j}\left(\prod_{i=1}^R\la_{N,u_i} +O\left(\frac{1}{N}\right)\right)\right|\notag\\
    &=\left|\left(N^{\beta}\max_{i,j\in [K]}\pi_{N,i,j}\right)^{2T}\sum_{t=1}^{R-1}N^{ -t + (R/T) m(t;H)}\right|\notag\\
    &+\left|\sum_{t=1}^{R-1}N^{-t }\sum_{\ub\in [K]^R}f_l(c:H)(1-f_l(c:H))\prod_{\{i,j\}\in E(H)}N^{2\beta}\pi^2_{N,u_i,u_j}\left(\prod_{i=1}^R\la_{N,u_i}+O\left(\frac{1}{N}\right)\right)\right|\notag\\
    &+\left|N^{T\beta}f_l(c:H)(1-f_l(c:H))\sum_{\ub\in [K]^R}N^{2\beta}\prod_{\{i,j\}\in E(H)}\pi^2_{N,u_i,u_j}\left(\prod_{i=1}^R\la_{N,u_i} +O\left(\frac{1}{N}\right)\right)\right|.
\end{align}
Therefore, using the fact that $H$ is strictly balanced, that is, $t/m(t:H)<R/T$, for all $t\in[R-1]$, $\beta=R/T$ and using Assumptions \ref{assump-2} and \ref{assump-3}, we can say that $\cov\left(U_{N,l,1},U_{N,l,2}\right)\to 0$, as $N\rai$.
We show that the remaining terms on the r.h.s. of \eqref{dW-bdd-poi-1} also converge to zero as $N \rai$. Using Assumption \ref{assump-3} and the fact that $\beta = R/T$, we obtain the following
\begin{align}
    &\sum_{\rb\in \nsr}\sum_{\sbb\in N(\rb)}\E\left(X_l^i(\rb)\right)\cdot \E\left(X_l^i(\sbb)\right) =\sum_{\rb\in \nsr}\sum_{\sbb\in N(\rb)}\left[f_l(c:H)\right]^2 \prod_{\{i,j\}\in E(H)}\pi_{N,\al_{s_i},\al_{s_j}}\cdot \pi_{N,\al_{r_i},\al_{r_j}}\notag\\
    &\leq \sum_{\rb\in \nsr}\sum_{\sbb\in N(\rb)} N^{-2R}\left(N^{2(R/T)}\max_{1\leq i,j\leq K}\pi^2_{N,i,j}\right)^T\leq R\cdot \frac{1}{N} \left(N^{2(R/T)}\max_{1\leq i,j\leq K}\pi^2_{N,i,j}\right)^T \to 0,~\text{and}\notag\\
    &\sum_{\rb\in \nsr}\sum_{\sbb\in N(\rb)}\E\left(X_l^2(\rb)\right)\cdot \E\left(X_l^1(\sbb)\right) \leq R\cdot \frac{1}{N} \left(N^{2(R/T)}\max_{1\leq i,j\leq K}\pi^2_{N,i,j}\right)^T \to 0,~\text{as}~N\to\infty.
\end{align}
Combining all the above steps, it follows that  $d_{{W}}(\mathbf{U}_{N,l},\mathbf{Z}_{N,l})\rightarrow 0$ (see \eqref{dW-bdd-poi-1}), for both $l=1,2$, as $N\rai$. Convergence in $d_{{W}}$ implies convergence in $d_{{BL}}$ from the definition (see. \eqref{bdd-wass}). Hence, we can write 
\begin{equation}
d_{{BL}}(\mathbf{U}_{N,l},\mathbf{Z}_{N,l})\rightarrow 0,\quad\text{as $N\rai$, for $l=1,2$.}
\label{BL-conv-1}
\end{equation}
Next, using Assumptions \ref{assump-2} and \ref{assump-3} and $\beta = R/T$, we can write
\begin{align*}
    & \E\left(U_{N,l,1}\right) = f_l(c:H)\cdot \kappa_N(\PiB_N, \lamb_N:H),~\text{and} ~\E\left(U_{N,l,2}\right) = (1-f_l(c:H))\cdot \kappa_N(\PiB_N, \lamb_N:H),~\text{where},\notag\\
    &\kappa_N(\PiB_N, \lamb_N:H) = \sum_{\sbb\in \nsr}\E\left(Y_N(\sbb:H)\right) = \sum_{\sbb\in \nsr}\prod_{\{i,j\}\in E(H)} \pi_{N,\al_{s_i},\al_{s_j}}\notag\\
    &= N^{R-T\beta }\sum_{\ub\in [K]^R}\prod_{\{i,j\}\in E(H)}N^{\beta}\pi_{N,u_i,u_j}\left(\prod_{i=1}^R\la_{N,u_i}+O\left(\frac{1}{N}\right)\right)\notag\\
    &\to \kappa(\Cb,\lamb:H), 
    \quad \text{(cf. \eqref{kap-def}).}
\end{align*}
Define the random vector $\mathbf{Z}_l = {(Z_{l,1},Z_{l,2})}^\primet$, for $l=1,2$, where $Z_{l,1}$ and $Z_{l,2}$ are independent,  $Z_{l,1}\sim \text{Poisson}\bigg(f_l(c:H)\cdot\kappa\left(\Cb,\lamb:H\right)\bigg)$,  $Z_{l,2}\sim \text{Poisson}\bigg((1-f_l(c:H))\cdot\kappa\left(\Cb,\lamb\right)\bigg)$, and $\kappa(\Cb,\lamb:H)$ is defined in \eqref{kap-def}. From the definition of $\mathbf{Z}_{N,l}$ and $\mathbf{Z}_l$, $l=1,2$, and using the fact that $\kappa_N(\PiB_N,\lamb_N:H)\rightarrow \kappa(\Cb,\lamb:H)$ (see \eqref{kap-def}), it implies that the characteristic function of $\mathbf{Z}_{N,l}$ converges to the characteristic function of $\mathbf{Z}_l$, for $l=1,2$. Thus, 
$\mathbf{Z}_{N,l}\darw \mathbf{Z}_l$, for $l=1,2$. Using Theorem 11.3.3 of \cite{dudley-real} we can claim that $d_{{BL}}(\mathbf{Z}_{N,l},\mathbf{Z}_l)\to 0$, as $N\rai$, for $l=1,2$. By triangle inequality and \eqref{BL-conv-1},
$$
d_{{BL}}\left(\mathbf{U}_{N,l},~\mathbf{Z}_{l} \right)\leq d_{{BL}}\left(\mathbf{U}_{N,l},\mathbf{Z}_{N,l}\right)+d_{{BL}}\left(\mathbf{Z}_{N,l},\mathbf{Z}_l\right)\to 0,\quad\text{as $N\rai$, for $l=1,2$.}
$$
Again, using Theorem 11.3.3 from \cite{dudley-real}, we can say that 
\begin{align}
\mathbf{U}_{N,l}\darw\mathbf{Z}_{l}, \quad\text{for $l=1,2$.}
\label{conv-main-1}
\end{align}
Using Lemma \ref{var-exp-bdd}, we can claim that $\sigma^2_{N,l}(H)\rightarrow f_l(c:H)(1-f_l(c:H))\cdot\kappa(\Cb,\lamb:H)$, for $l=1,2$, where $\kappa(\Cb,\lamb:H)$ is defined in \eqref{kap-def}. From the definition of the pivotal quantity in \eqref{TN-def-12} we can write (using \eqref{conv-main-1} and the continuous mapping theorem)
\begin{align*}
    T_{N,l}(H) &= \frac{\widehat{S}_{N,l}(H) - f_l(c:H)\cdot S_{N}(H)}{\sigma_{N,l}(H)}\notag\\
    &= \frac{(1-f_l(c:H))\cdot \widehat{S}_{N,l}(H) - f_l(c:H)\cdot (S_N(H) - \widehat{S}_{N,l}(H))}{\sigma_{N,l}(H)} \notag\\
    & = \frac{(1-f_l(c:H))\cdot U_{N,l,1} - f_l(c:H)\cdot U_{N,l,2}}{\sigma_{N,l}(H)}\\
    & \darw \frac{(1-f_l(c:H))\cdot Z_{l,1} - f_l(c:H)\cdot Z_{l,2}}{\left[f_l(c:H)(1-f_l(c:H))\cdot\kappa(\Cb,\lamb:H)\right]^{1/2}},\quad\text{for $l=1,2$.}
\end{align*}
This completes the proof.

\end{proof}


\begin{proof}[Proof of Theorem \ref{thm-2-main-temp}(b): Intermediately sparse sampling.]
In this case $\al\in (0,1)$. In the induced case, it is assumed that the sparsity levels $(\al,\beta)$ are within the boundary region
$$
F_{1}(H) = \left\{(\al,\beta): 0<\al<1,~\beta = \frac{R\cdot(1-\al)}{T}\right\},\quad\text{(see \eqref{F1F2-set}).}
$$
In the ego-centric case, $(\al,\beta)$ are within the boundary region 
\begin{align*}
F_2(H) &= \left\{(\al,\beta): 0 <\al < \widetilde{\al}_H,~\beta = \left(1-\frac{\tau(H)\cdot\al}{R}\right)\cdot\frac{R}{T}\right\},
\end{align*} 
where $\widetilde{\al}_H$ is defined \eqref{al-tild-def} and $H$ satisfies the Assumption \ref{assump-5}. Write
\begin{align}
    \widehat{S}_{N,l}(H) - f_l(p_N:H)\cdot S_N(H) &= \sum_{\sbb\in \nsr}W_{N,l}(\sbb:H)\cdot Y_N(\sbb:H) - f_l(p_N:H)\sum_{\sbb\in \nsr}Y_N(\sbb:H)\notag\\
    &\equiv E_{N,l,1} - E_{N,l,2},\quad\text{(say), for $l=1,2$.}
    \label{ENl1-2}
\end{align}
Recall (see~\eqref{pN-def}) that $p_N = c N^{-\alpha}$ for some
$\alpha \in (0,1)$ and $c>0$. When $(\alpha,\beta) \in F_l(H)$, $l=1,2$, we can use Assumptions \ref{assump-2} and \ref{assump-3} to write
\begin{align}\label{exp-B-1-B-2}
    &\E\left(E_{N,l,1}\right) = \E\left(E_{N,l,2}\right)=f_l(p_N:H)\sum_{\sbb\in\nsr}\prod_{\{i,j\}\in E(H)} \pi_{N,\al_{s_i},\al_{s_j}}\notag\\
    &=f_l(p_N:H)\cdot N^{R-T\beta}\sum_{\ub\in [K]^R}\prod_{\{i,j\}\in E(H)} N^{\beta}\pi_{N,u_i,u_j}\left(\prod_{i=1}^R\lambda_{N,u_i}+O\left(\frac{1}{N}\right)\right)\notag\\
    &\to \phi_l(c:H)\sum_{\ub\in [K]^R}\prod_{\{i,j\}\in E(H)}c_{u_i,u_j}\prod_{i=1}^R\lambda_{u_i} = \phi_l(c:H)\cdot \kappa(\Cb, \lamb:H),~\text{as $N\rai$, for $l=1,2$,}
\end{align}
where, $\phi_1(c:H) = c^R$, $\phi_2(c:H) = \lvert \mathcal{D}_{\min}(H) \rvert \cdot c^{\tau(H)}$, and $\kappa(\Cb,\lamb:H)$ is defined in \eqref{kap-def} and $|\mathcal{D}_{\min}(H)|$ is the size of the collection of all  minimum vertex cover sets for $H$. Note, $|\mathcal{D}_{\min}(H)|$ is a finite quantity depending on $R$ and $\tau(H)$. We will show that $E_{N,l,2} \parw \phi_l(c:H)\cdot \kappa(\Cb,\lamb:H)$ (cf. \eqref{ENl1-2}). As $H$ is strictly balanced and for $l=2$, $H$ satisfies Assumption \ref{assump-5}. Then, for $(\al,\beta)\in F_l(H)$, we can write (for $l=1$ and $l=2$)
\begin{align}\label{V-2}
    &\Var\left(E_{N,l,2}\right) = [f_l(p_N:H)]^2\cdot\Var\left(\sum_{\sbb\in \nsr}Y_N(\sbb:H)\right)\notag\\
    &= [f_l(p_N:H)]^2\sum_{t=2}^R\sum_{(\sbb_1,\sbb_2)\in \Mc_t}\cov\left(Y_N(\sbb_1:H),Y_N(\sbb_2:H)\right)\notag\\
    &= [f_l(p_N:H)]^2 \sum_{t=2}^R\sum_{(\sbb_1,\sbb_2)\in \Mc_t}\left[\E\left\{Y_N(\sbb_1:H)Y_N(\sbb_2:H)\right\}-\right\{\E\left(Y_N(\sbb:H)\right)\left\}^2\right]\notag\\
    &\leq  \phi^2_l(p_N:H)\cdot a_{\max}(H)\sum_{t=2}^RN^{2R- t - \beta(2T - m(t:H))}\notag\\
    &=\begin{cases}
        a_{\max}(H) c^{2R} \cdot N^{2R - 2R\al -2T\beta}\sum_{t=2}^RN^{- t + \beta m(t:H)}&\quad \text{if $l=1$,}\\
        |\mathcal{D}_{\min}(H)|^2c^{2\tau(H)}a_{\max}(H)\cdot N^{2R - 2\tau(H)\al - 2T\beta}\sum_{t=2}^RN^{- t + \beta m(t:H)}&\quad\text{if $l=2$,}
    \end{cases}\notag\\
    &=\begin{cases}
        a_{\max}(H)c^{2R}\sum_{t=2}^RN^{- t + \left(\frac{R}{T}(1-\al)\right) m(t:H)}&\qquad \quad\text{if $l=1$,}\\
        |\mathcal{D}_{\min}(H)|^2c^{2\tau(H)}a_{\max}(H)\cdot\sum_{t=2}^RN^{- t + \left(\frac{R}{T}\left(1-\frac{\tau(H)}{R}\al\right)\right) m(t:H)}&\qquad \quad\text{if $l=2$,}
    \end{cases}\notag\\
    &=\begin{cases}
        a_{\max}(H)c^{2R}\sum_{t=2}^RN^{m(t:H)\left\{- \frac{t}{m(t:H)} + \frac{R}{T}(1-\al)\right\} }&\quad \text{if $l=1$,}\\
        |\mathcal{D}_{\min}(H)|^2c^{2\tau(H)}a_{\max}(H)\cdot\sum_{t=2}^RN^{m(t:H)\left\{- \frac{t}{m(t:H)} + \frac{R}{T}\left(1-\frac{\tau(H)}{R}\al\right)\right\}}&\quad \text{if $l=2$,}
    \end{cases}
\end{align}
where, $a_{\max}(H) = \left(N^{\beta}\max_{i,j\in [K]}\pi_{N,i,j}\right)^{2T}$ converges to a finite positive constant (due to Assumption \ref{assump-3}), depending only on $H$ and the model parameters. As $H$ is strictly balanced, $\frac{t}{m(t:H)} > \frac{R}{T}$ for all $t \in\{2,\ldots,R-1\}$, and at $t=R$, $\frac{R}{T} > \frac{(1-\delta)\cdot R}{T}$ for any $\delta \in (0,1)$. This implies
$$
-\frac{t}{m(t:H)} + \frac{R}{T}\cdot (1-\delta) <0,\quad\text{for all $t=2,\ldots,R$,}
$$
where $\delta = \al$, if $l=1$, and $\delta = \tau(H)\al/R$, if $l=2$. Hence, the exponent of $N$ (for both $l=1$ and $l=2$) in the upper bound provided in \eqref{V-2} will be strictly negative. Therefore, $\Var\left(E_{N,l,2}\right) \to 0$ as $N \to \infty$, for $l=1,2$. Hence, $E_{N,l,2} - \E\left(E_{N,l,2}\right) = o_P(1)$, for $l=1,2$ (see \eqref{ENl1-2}).  
Next, applying Theorem 6.23 of \cite{janson-rucinski-randomgraphs} (see Theorem \ref{thm-ruc-1}) on $E_{N,l,1}$ (cf. \eqref{ENl1-2}) for both choices of $l$, we obtain the following bound on the total variation distance between the random variable $E_{N,l,1}$ and $V_{N,l}$, where $V_{N,l}\sim \text{Poisson}\left(f_l(p_N:H)\cdot \kappa_N(\PiB_N, \lamb_N:H)\right)$, (for both $l=1,2$), 
\begin{align}\label{poi-TV-bdd}
    &d_{{TV}}\left(E_{N,l,1},V_{N,l}\right)\leq \min\left\{1, \frac{1}{f_l(p_N:H)\kappa_N(\PiB_N, \lamb_N:H)}\right\}\bigg(\left\{\Var\left(E_{N,l,1}\right) - \E(E_{N,l,1})\right\} \notag\\
    &\qquad{} +2\sum_{\sbb_1\in \nsr}\sum_{\sbb_2\in N(\sbb_1)}\eta_{N,l}\left(\sbb_1: p_N,\PiB_N,H\right)\eta_{N,l}\left(\sbb_2: p_N,\PiB_N,H\right)+ 2\sum_{\sbb\in \nsr}\left[\eta_{N,l}\left(\sbb: p_N,\PiB_N,H\right)\right]^2\bigg)\notag\\
    &= \min\left\{1, \frac{1}{f_l(p_N:H)\kappa_N(\PiB_N, \lamb_N:H)}\right\}\bigg(\left\{\Var\left(E_{N,l,1}\right) - \E(E_{N,l,1})\right\} +2\cdot G_{N,l,1}+ 2\cdot G_{N,l,2}\bigg),~\text{(say)},
\end{align}
where, $\eta_{N,l}\left(\sbb: p_N,\PiB_N,H\right) = \E(W_{N,l}(\sbb:H)Y_N(\sbb:H))$, for $l=1,2$. It remains to show that the upper bound in~\eqref{poi-TV-bdd} converges to zero for $l=1$ and $2$, as $N \to \infty$. First, in case of induced network formation ($l=1$), for all $(\alpha,\beta) \in F_1(H)$, we study the quantity $|\Var(E_{N,1,1}) - \E(E_{N,1,1})|$ from \eqref{poi-TV-bdd} in the following way. 
\begin{align}\label{1st-term-TV}
    &\lvert \Var(E_{N,1,1}) - \E(E_{N,1,1}) \rvert\notag\\
    &=\left| \sum_{t=1}^R\sum_{(\sbb_1,\sbb_2)\in \Mc_t}\cov\left(W_{N,1}(\sbb_1:H)\cdot Y_N(\sbb_1:H),~W_{N,1}(\sbb_2:H)\cdot Y_N(\sbb_2:H)\right)-\E(E_{N,1,1})\right|\notag\\
    &= \left|\sum_{t=1}^{R-1}\sum_{(\sbb_1,\sbb_2)\in \Mc_t}\cov\left(W_{N,1}(\sbb_1:H)\cdot Y_N(\sbb_1:H),~W_{N,1}(\sbb_2:H)\cdot Y_N(\sbb_2:H)\right)\right.\notag\\
    &\left.\quad{}+\sum_{(\sbb_1,\sbb_2)\in \Mc_R}\cov\left(W_{N,1}(\sbb_1:H)\cdot Y_N(\sbb_1:H),~W_{N,1}(\sbb_2:H)\cdot Y_N(\sbb_2:H)\right)-\E(E_{N,1,1})\right|\notag\\
    &\leq \left|\sum_{t=1}^{R-1}\sum_{(\sbb_1,\sbb_2)\in \Mc_t}\bigg\{\E\left(W_{N,1}(\sbb_1:H)W_{N,1}(\sbb_2:H))\E( Y_N(\sbb_1:H)Y_N(\sbb_2:H)\right)\right.\notag\\
    &\left.\qquad{}- \E\left(W_{N,1}(\sbb_1:H)\right)\E\left(Y_N(\sbb_1:H)\right)\E\left(W_{N,1}(\sbb_2:H)\right)\E\left(Y_N(\sbb_2:H)\right)\bigg\}\right|\notag\\
    &\qquad{}\qquad{}+\left|\sum_{\sbb\in \nsr}\Var\left(W_{N,1}(\sbb:H)\cdot Y_N(\sbb:H)\right)-\E(E_{N,1,1})\right|\notag\\
    &\leq \left|\sum_{t=1}^{R-1}\sum_{(\sbb_1,\sbb_2)\in \Mc_t}\bigg\{\E\left(W_{N,1}(\sbb_1:H)W_{N,1}(\sbb_2:H))\cdot\E( Y_N(\sbb_1:H)Y_N(\sbb_2:H)\right)\bigg\}\right|\notag\\
    &\qquad{}+\left|\sum_{\sbb\in \nsr}\Var\left(W_{N,1}(\sbb:H)\cdot Y_N(\sbb:H)\right)-\E(E_{N,1,1})\right|\notag\\
    &=\left|\sum_{t=1}^{R-1}\left(\frac{c}{N^{\al}}\right)^{2R-t}\left(1-\left(\frac{c}{N^{\al}}\right)^{-t}\right)\sum_{(\sbb_1,\sbb_2)\in \Mc_t}\E\left( Y_N(\sbb_1:H)Y_N(\sbb_2:H)\right)\right|\notag\\
    &\qquad{}+\left|\sum_{\sbb\in \nsr}\Var\left(W_{N,1}(\sbb:H)\cdot Y_N(\sbb:H)\right)-\E(E_{N,1,1})\right|.
    \end{align}
Then under Assumptions \ref{assump-2} and \ref{assump-3}, using Lemma \ref{lem-1} in the first term of \eqref{1st-term-TV}, we can say that there exists a positive integer $M_5$, depending only on the graph $H$ and the model parameters, such that for all $N \ge M_5$,     
    \begin{align}
   \text{r.h.s. of \eqref{1st-term-TV}} &\leq \left|a_{\max}(H)\cdot c^{2R}\cdot N^{-g_{1,\al,\beta,H}(R)}\sum_{t=1}^{R-1}N^{{}g_{1,\al,\beta,H}(t)} \right| \notag\\
    &\qquad{}+\left|\sum_{\sbb\in \nsr}\E(W_{N,1}(\sbb:H)Y_{N}(\sbb:H)) - \sum_{\sbb\in \nsr}\bigg(\E(W_{N,1}(\sbb:H)Y_{N}(\sbb:H))\bigg)^2 - \E(E_{N,1,1})\right|\notag\\
    &\leq \left|a_{\max}(H)\cdot c^{2R}\cdot  N^{-g_{1,\al,\beta,H}(R)}\sum_{t=1}^{R-1}N^{{}g_{1,\al,\beta,H}(t)} \right|+ \left| \sum_{\sbb\in \nsr}\bigg(\E(W_{N,1}(\sbb:H)Y_{N}(\sbb:H))\bigg)^2\right|\label{G-1}\\ 
    & =a_{\max}(H)\cdot c^{2R}\cdot \left(N^{-g_{1,\al,\beta,H}(R)}\sum_{t=1}^{R-1}N^{{}g_{1,\al,\beta,H}(t)} +  N^{-R+2g_{1,\al,\beta,H}(R)}\right),
\label{var-exp-diff-1}
\end{align}
where, $a_{\max}(H)$ is a finite positive constant (defined in \eqref{V-2}). Next, we study the quantity $|\Var(E_{N,l,1}) - \E(E_{N,l,1})|$ from \eqref{1st-term-TV} in case of ego-centric network formation ($l=2$), that is 
\begin{align}\label{1st-term-TV-1}
    &\lvert \Var(E_{N,2,1}) - \E(E_{N,2,1}) \rvert\notag\\
    &=\left| \sum_{t=1}^R\sum_{(\sbb_1,\sbb_2)\in \Mc_t}\cov\left(W_{N,2}(\sbb_1:H)\cdot Y_N(\sbb_1:H),~W_{N,2}(\sbb_2:H)\cdot Y_N(\sbb_2:H)\right)-\E(E_{N,2,1})\right|\notag\\
    &= \left|\sum_{t=1}^{R-1}\sum_{(\sbb_1,\sbb_2)\in \Mc_t}\cov\left(W_{N,2}(\sbb_1:H)\cdot Y_N(\sbb_1:H),~W_{N,2}(\sbb_2:H)\cdot Y_N(\sbb_2:H)\right)\right.\notag\\
    &\left.\quad{}+\sum_{(\sbb_1,\sbb_2)\in \Mc_R}\cov\left(W_{N,2}(\sbb_1:H)\cdot Y_N(\sbb_1:H),~W_{N,2}(\sbb_2:H)\cdot Y_N(\sbb_2:H)\right)-\E(E_{N,2,1})\right|\notag\\
    &\leq \left|\sum_{t=1}^{R-1}\sum_{(\sbb_1,\sbb_2)\in \Mc_t}\bigg\{\E\left(W_{N,2}(\sbb_1:H)W_{N,2}(\sbb_2:H))\E( Y_N(\sbb_1:H)Y_N(\sbb_2:H)\right)\right.\notag\\
    &\left.\qquad{}- \E\left(W_{N,2}(\sbb_1:H)\right)\E\left(Y_N(\sbb_1:H)\right)\E\left(W_{N,2}(\sbb_2:H)\right)\E\left(Y_N(\sbb_2:H)\right)\bigg\}\right|\notag\\
    &\qquad{}\qquad{}+\left|\sum_{\sbb\in \nsr}\Var\left(W_{N,2}(\sbb:H)\cdot Y_N(\sbb:H)\right)-\E(E_{N,2,1})\right|\notag\\
    &\leq \left|\sum_{t=1}^{R-1}\sum_{(\sbb_1,\sbb_2)\in \Mc_t}\bigg\{\E\left(W_{N,2}(\sbb_1:H)W_{N,2}(\sbb_2:H))\cdot\E( Y_N(\sbb_1:H)Y_N(\sbb_2:H)\right)\bigg\}\right|\notag\\
    &\qquad{}+\left|\sum_{\sbb\in \nsr}\Var\left(W_{N,2}(\sbb:H)\cdot Y_N(\sbb:H)\right)-\E(E_{N,2,1})\right|\notag\\
    &\leq \left|a_{\max}(H)\cdot c^{2\tau(H)}\left|\mathcal{D}(H)\right|^2 N^{-g_{2,\al,\beta,H}(R)}\sum_{t=1}^{R-1}N^{{}g_{2,\al,\beta,H}(t)} \right| \notag\\
    &\qquad{}+\left|\sum_{\sbb\in \nsr}\E(W_{N,2}(\sbb:H)Y_{N}(\sbb:H)) - \sum_{\sbb\in \nsr}\bigg(\E(W_{N,2}(\sbb:H)Y_{N}(\sbb:H))\bigg)^2 - \E(E_{N,2,1})\right|.
    \end{align}
    Then under Assumption \ref{assump-2} and \ref{assump-3}, along with the Assumption \ref{assump-5} over the graph $H$, using Lemma \ref{lem-var-ego-1} in the first term of \eqref{1st-term-TV-1}, we can say that there exists a positive integer $M_6$, depending on the graph and model parameters, such that for all $N\geq M_6$,
    \begin{align}
    &\text{r.h.s. of \eqref{1st-term-TV-1}}\leq \left|a_{\max}(H)\cdot c^{2\tau(H)}\left|\mathcal{D}(H)\right|^2 N^{-g_{2,\al,\beta,H}(R)}\sum_{t=1}^{R-1}N^{{}g_{2,\al,\beta,H}(t)} \right|\notag\\
    &\qquad{}\qquad{}+ \left| \sum_{\sbb\in \nsr}\bigg(\E(W_{N,2}(\sbb:H)Y_{N}(\sbb:H))\bigg)^2\right|\label{G-1-ego}\\ 
    & =a_{\max}(H)\cdot c^{2\tau(H)}\left|\mathcal{D}(H)\right|^2\left(N^{-g_{2,\al,\beta,H}(R)}\sum_{t=1}^{R-1}N^{{}g_{2,\al,\beta,H}(t)} +  N^{-R+2g_{2,\al,\beta,H}(R)}\right),
\label{var-exp-diff-2}
\end{align}
where, $|\mathcal{D}(H)|$ is the size of the collection of all possible vertex covers of $H$, which is a finite quantity. Next, when $(\alpha,\beta) \in F_l(H)$, we have $g_{l,\alpha,\beta,H}(R)=0$ for $l=1,2$. Moreover, for $(\alpha,\beta) \in F_l(H)$, $l=1,2$, and $t \in \{1,\ldots,R-1\}$, writing $\beta$ in terms of $\al$ we have the following.
\begin{align*}
    g_{1,\alpha,\beta,H}(t)
    &= (1-\alpha)\left(-\frac{t}{m(t:H)} + \frac{R}{T}\right)m(t:H),\quad\text{and}\\
    g_{2,\alpha,\beta,H}(t)
    &= m(t:H)\left\{\left(-\frac{t}{m(t:H)} + \frac{R}{T}\right)
    + \left(\frac{\min(\tau(H),t)}{m(t:H)} - \frac{\tau(H)}{T}\right)\alpha\right\}.
\end{align*}
Since $\alpha \in (0,1)$ in the induced case and
$\alpha \in (0,\widetilde{\alpha}_H)$ in the egocentric case, it follows that $g_{l,\alpha,\beta,H}(t) < 0$ for all $t \in \{1,\ldots,R-1\}$. Therefore, $\{\Var(E_{N,l,1}) - \E(E_{N,l,1})\}\to 0$ as $N\rai$.
\begin{align*}
    \{\Var(E_{N,l,1}) - \E(E_{N,l,1})\}\to 0,\quad\text{as $N\to\infty$ for $l=1,2$.}
\end{align*}
The expression for the term $G_{N,l,2},~l=1,2$ (see \eqref{poi-TV-bdd}) is the same as the second term in \eqref{G-1} and \eqref{G-1-ego} respectively; hence $G_{N,l,2} \to 0$, under the same assumptions. Using Assumptions \ref{assump-2} and \ref{assump-3} and the fact that $(\alpha,\beta) \in F_l(H),~l=1,2$, we obtain the following
\begin{align*}
    &G_{N,l,1} = \sum_{\sbb_1\in \nsr}\sum_{\sbb_2\in N(\sbb_1)}\eta_N\left(\sbb_1: p_N,\PiB_N,H\right)\eta_N\left(\sbb_2: p_N,\PiB_N,H\right)\notag\\
    &\leq R\sum_{(\sbb_1,\sbb_2)\in \mathcal{M}_1}\E(W_{N,l}(\sbb_1:H)Y_{N,l}(\sbb_1))\cdot \E(W_{N,l}(\sbb_2:H)Y_N(\sbb_2:H) )\notag\\
    &= R\sum_{(\sbb_1,\sbb_2)\in \mathcal{M}_1}f^2_l(p_N:H)\prod_{\{i,j\}\in E(H)}\pi_{\al_{s_{1,i}},\al_{s_{1,j}}}\pi_{\al_{s_{2,i}},\al_{s_{2,j}}}\leq R f^2_l(p_N:H)\cdot\left(N^{\beta}\max_{i,j\in [K]}\pi_{N,i,j}\right)^{2T} N^{2R-1 -2T\beta}\notag\\
    &\leq R\cdot \max\left\{c^{2R},c^{2\tau(H)}|\mathcal{D}(H)|^2\right\}\cdot\left(N^{\beta}\max_{i,j\in [K]}\pi_{N,i,j}\right)^{2T} N^{2g_{l,\al,\beta,H}(R)-1} \notag\\
    &= R \cdot\max\{c^{2R},|\mathcal{D}(H)|^2\}\cdot\left(N^{2\beta}\max_{i,j\in [K]}\pi^2_{i,j}\right) N^{-1}\rightarrow 0,\quad\text{as $N\rai$.}
\end{align*}
Thus, all terms in the r.h.s. of \eqref{poi-TV-bdd} go to zero as $N\rai$, and we can claim 
\begin{align}\label{TV-conv-2}
    d_{TV}\left(E_{N,l,1}, V_{N,l}\right)\to 0,\quad\text{as $N\to \infty$, for $l=1,2$.}
\end{align}
Define two random variables $V_1$ and $V_2$, where $V_l \sim{\text{Poisson}}\left(\phi_l(c:H)\cdot\kappa(\Cb,\lamb:H)\right)$, for $l=1,2$ (cf. \eqref{exp-B-1-B-2}). As $\E\left(V_{N,l}\right)\to \E(V_l)$, for $l=1,2$, it implies $V_{N,l} \xrightarrow{d} V_l$, for $l=1,2$. Now, using Lemma \ref{WL-TV} we can claim that $d_{TV}(V_{N,l},V_l)\to 0$, as $N\to\infty$, for $l=1,2$. By triangle inequality and \eqref{TV-conv-2}, we can write
\begin{align}\label{conv-main-2}
    d_{{TV}}\left(E_{N,l,1}, V_l\right)
    &\leq d_{{TV}}\left(E_{N,l,1}, V_{N,l}\right)
      + d_{{TV}}\left(V_{N,l}, V_l\right)
      \to 0,\quad\text{as $N\rai$, for $l=1,2$.}
\end{align}
Using Lemma \ref{var-exp-bdd}, we can claim that $\sigma^2_{N,l}(H) \to \phi_l(c:H)\cdot \kappa(\Cb,\lamb:H)$ for $l=1,2$, where $\kappa(\Cb,\la:H)$ is defined in \eqref{kap-def} and $\phi_l(c:H)$ is defined in \eqref{exp-B-1-B-2}, for $l=1,2$. 
Combining this with the fact that $E_{N,l,2} \parw \phi_l(c:H)\cdot\kappa(\Cb,\lamb:H)$ (see \eqref{exp-B-1-B-2}) and using \eqref{conv-main-2}, we can write
\begin{equation*}
    \begin{aligned}
      &T_{N,l}(H) \darw \frac{V_{l} - \phi_l(c:H)\cdot\kappa(\mathbf{C}, \boldsymbol{\lambda}:H)}{\left[\phi_l(c:H)\cdot\kappa(\mathbf{C}, \boldsymbol{\lambda}:H)\right]^{1/2}},\quad\text{for $l=1,2$.}
    \end{aligned}
\end{equation*}
This completes the proof.
\end{proof}

\begin{proof}[Proof of Theorem \ref{thm-2-main-temp} part (c): Extremely sparse sampling in case of $\tau(H)=1$]
As $\tau(H)= 1$, $H=\kl_{1,R-1}$ is a star graph on $R$ vertices (see Lemma \ref{lem-4} for a proof of this fact). We focus only on the case of ego-centric network formation and the sampling and model sparsity levels are $\al=1$ and $\beta\in [0,1)$. This implies, $p_N = c/N$, for all $N\geq 1$. Note that due to the edge structure of the star graph, we have $f_2(p_N:\kl_{1,R-1}) = p_N + {(1-p_N)}\cdot p^{R-1}_N$ (cf. \eqref{fl-def}) and for any $\sbb\in \nsr$ we can write (see \eqref{piv-1}), 
\begin{align*}
W_{N,2}(\sbb:\kl_{1,R-1})& =\prod_{\{i,j\}\in E(\kl_{1,R-1})} \max\{W_{N,s_i},W_{N,s_j}\} = W_{N,s_1} + (1-W_{N,s_1})\cdot \prod_{j=2}^R W_{N,s_j},\quad\text{and}\\
Y_N(\sbb:\kl_{1,R-1}) & = \prod_{j=2}^R Y_{s_1,s_j}.
\end{align*}
As a result (see \eqref{piv-1}),
\begin{align}
    &T_{N,2}(\kl_{1,R-1})= \frac{\widehat{S}_{N,2}(\kl_{1,R-1}) - f_2(p_N:\kl_{1,R-1})\cdot S_N(\kl_{1,R-1})}{\sigma_{N,2}(\kl_{1,R-1})}\notag\\
    &= \frac{1}{\sigma_{N,2}(\kl_{1,R-1})}\sum_{\sbb\in \nsr}\left(W_{N,2}(\sbb:\kl_{1,R-1}) - (p_N+(1-p_N)p_N^{R-1})\right)Y_N(\sbb:\kl_{1,R-1})\notag\\
    &= \frac{1}{\sigma_{N,2}(\kl_{1,R-1})}\sum_{\sbb\in \nsr}\left(W_{N,s_1} - p_N\right)Y_{N}(\sbb:\kl_{1,R-1})\notag\\
    &+ \frac{1}{\sigma_{N,2}(\kl_{1,R-1})}\sum_{\sbb\in \nsr}\left[(1-W_{N,s_1})W_{N,s_2}\cdots W_{N,s_R} - (1-p_N)p_N^{R-1}\right]Y_N(\sbb:\kl_{1,R-1})\notag\\
    &= \frac{G_{N,1}}{\sigma_{N,2}(\kl_{1,R-1})} + \frac{G_{N,2}}{\sigma_{N,2}(\kl_{1,R-1})},\quad\text{(say).}
\label{split-case-c}    
\end{align}
Using Lemma \ref{var-exp-bdd}, we can claim
\begin{align}\label{var-star}
    &\sigma_{N,2}^2(\kl_{1,R-1})=\Var\left(\widehat{S}_{N,2}(\kl_{1,R-1}) - f_2(p_N:\kl_{1,R-1})\cdot S_N(\kl_{1,R-1})\right)\notag\\
    &= N^{2(R-1)(1-\beta)}\left\{c\sum_{\ub\in [K]^R,(v_2,\ldots,v_R)\in [K]^{R-1}}\prod_{j=2}^RN^{2\beta}\pi_{u_1,u_j}\pi_{u_1,v_j}\prod_{j=1}^R\la_{N,u_j}\prod_{j=2}^R\la_{N,v_j}\left(1 +O\left(\frac{1}{N}\right)\right) + O\left(\frac{1}{N}\right)\right\},\notag\\
    &\Rightarrow \frac{\sigma^2_{N,2}(\kl_{1,R-1})}{N^{2(R-1)(1-\beta)}} \to \nu(\Cb,\lamb),\quad\text{as $N\rai$,}
\end{align}
where, $\nu(\Cb,\lamb)$ is defined in \eqref{nu-def}. 
Note that $\E(G_{N,1}) = \E(G_{N,2}) = 0$, and 
\begin{align}\label{var-G2}
    &\Var\left(G_{N,2}\right)
    = \Var\left(\sum_{\sbb\in \nsr}\left[(1-W_{N,s_1})W_{N,s_2}\cdots W_{N,s_R} - (1-p_N)p_N^{R-1}\right]Y_N(\sbb:\kl_{1,R-1})\right)\notag\\
    &\leq \sum_{t=1}^R\sum_{(\sbb_1,\sbb_2)\in \mathcal{M}_t}\bigg|\cov\left((1-W_{N,s_{1,1}})W_{N,s_{1,2}}\cdots W_{N,s_{1,R}}, (1-W_{N,s_{2,1}})W_{N,s_{2,2}}\cdots W_{N,s_{2,R}}\right)\bigg|\notag\\
    &\qquad{}\qquad{}\times\E(Y_N(\sbb_1:\kl_{1,R-1}))Y_N(\sbb_2:\kl_{1,R-1})).
\end{align}
Using the Cauchy-Schwarz inequality, we can write
\begin{align*}
    &\bigg|\cov\left((1-W_{N,s_{1,1}})W_{N,s_{1,2}}\cdots W_{N,s_{1,R}}, (1-W_{N,s_{2,1}})W_{N,s_{2,2}}\cdots W_{N,s_{2,R}}\right)\bigg|\notag\\
    &\leq \Var\left((1-W_{N,s_{1,1}})W_{N,s_{1,2}}\cdots W_{N,s_{1,R}}\right)= N^{-(R-1)} c^{R-1}\left(1-\frac{c}{N}\right)\left(1- \left(\frac{c}{N}\right)^{R-1}\left(1-\frac{c}{N}\right)\right).
\end{align*}
Then, using \eqref{var-G2} and dividing by the variance term $\sigma^2_{N,2}(\kl_{1,R-1})$ in \eqref{var-star}, and using the fact that $\beta\in[0,1)$, we can claim
\begin{align*}
    &\Var\left(G_{N,2}\right) \leq \sum_{t=1}^R N^{2R-t} N^{-(R-1)} N^{-\beta(2(R-1) -(t-1))} = N^{-(R-1)}N^{2R - 2(R-1)\beta} \sum_{t=1}^RN^{-t +\beta m(t:H)} \notag\\
    &= N^{-(R-1)}N^{2R - 2(R-1)\beta} \sum_{t=1}^RN^{m(t:H)\left(-\frac{t}{m(t:H)} +\beta \right)} = N^{-(R-1)}N^{2R - 2(R-1)\beta} N^{-R + \beta(R-1)}= N^{1-(R-1)\beta}\notag\\
    &\Rightarrow\frac{\Var\left(G_{N,2}\right)}{\sigma^2_{N,2}(\kl_{1,R-1})}\leq \frac{N^{1-(R-1)\beta}}{N^{2(R-1)(1-\beta)}} = N^{-(2R-3) + (R-1)\beta}\rightarrow 0,\quad\text{as $N\rai$.} 
\end{align*}
Returning to \eqref{split-case-c}, we can now claim that
\begin{align}\label{G-2}
    \frac{G_{N,2}}{\sigma_{N,2}(\kl_{1,R-1})} = o_P(1).
\end{align}
We split the $G_{N,1}$ term in \eqref{split-case-c} as follows
\begin{align}
    &\frac{G_{N,1}}{\sigma_{N,2}(\kl_{1,R-1})}= \frac{1}{\sigma_{N,2}(\kl_{1,R-1})}\left(\sum_{\sbb\in \nsr}W_{N,s_{1}}\E\left(Y_N(\sbb:\kl_{1,R-1})\right) \right.\notag\\
    &\left.+\sum_{\sbb\in \nsr}W_{N,s_{1}}\bigg(Y_{N}(\sbb:\kl_{R-1})-\E\left(Y_{N}(\sbb:\kl_{R-1})\right)\bigg)- p_N\sum_{\sbb\in \nsr}Y_N(\sbb:\kl_{1,R-1})\right)\notag\\
    &= \frac{1}{\sigma_{N,2}(\kl_{1,R-1})}\left(L_{N,1,1} +L_{N,1,2} +L_{N,1,3} \right),~\text{(say)}.
    \label{GN1-split}
\end{align}
Note that, $\E(L_{N,1,2}) =0$. Using the definition of $\Mc_t$ (see \eqref{Etset}) and using Assumptions \ref{assump-2} and \ref{assump-3}, we can write
\begin{align*}
    &\frac{\Var\left(L_{N,1,2}\right)}{\sigma^2_{N,2}(\kl_{1,R-1})}= \frac{1}{\sigma^2_{N,2}(\kl_{1,R-1})}\Var\left(\sum_{\sbb\in \nsr}W_{N,s_{1}}\bigg(Y_{N}(\sbb:\kl_{1,R-1})-\E\left(Y_{N}(\sbb:\kl_{R-1})\right)\bigg)\right)\notag\\
    &= \frac{1}{\sigma^2_{N,2}(\kl_{1,R-1})}\sum_{t=1}^R\sum_{(\sbb_1,\sbb_2)\Mc_t}\E(W_{N,s_{1,1}}W_{N,s_{2,1}})\cdot\cov\left(Y_N(\sbb_1:\kl_{1,R-1}),Y_N(\sbb_2:\kl_{1,R-1})\right)\notag\\
    &=\frac{1}{\sigma^2_{N,2}(\kl_{1,R-1})}\sum_{t=2}^R\sum_{(\sbb_1,\sbb_2)\Mc_t}\E(W_{N,s_{1,1}}W_{N,s_{2,1}})\cdot\cov\left(Y_N(\sbb_1:\kl_{1,R-1}),Y_N(\sbb_2:\kl_{1,R-1})\right)\notag\\
    &\leq \frac{d_{\max}(H)}{N^{2(R-1)(1-\beta)}}\sum_{t=2}^RN^{2R-t-1}N^{-\beta(2(R-1) -(t-1))}\notag\\
    &= \frac{d_{\max}(H)}{N^{2(R-1)(1-\beta)}}\sum_{t=2}^RN^{(2R-t-1)(1-\beta)}\notag\\
    &=d_{\max}(H)\cdot \frac{N^{(2R-3)(1-\beta)}}{N^{2(R-1)(1-\beta)}} \to 0,\quad\text{as $N\rai$,}
\end{align*}
where, $d_{\max}(H) = c^2\left(N^{\beta}\max_{i,j\in [K]}\pi_{N,i,j}\right)^{2(R-1)}$ is a finite positive constant that depends on $H$ and the underlying model parameters. Hence, 
\begin{align}\label{G-1-2}
    \frac{L_{N,1,2}}{\sigma_{N,2}(\kl_{1,R-1})} = o_P(1).
\end{align}
Next, we analyze the term $L_{N,1,3}/\sigma_{N,2}(\kl_{1,R-1})$ in the expression \eqref{GN1-split}. Using Assumptions \ref{assump-2}, \ref{assump-3}
and the variance expression written in \eqref{var-star}, we can claim the following
\begin{align*}
    &\E\left(\frac{L_{N,1,3}}{\sigma_{N,2}(\kl_{1,R-1})}\right) \to \frac{c\cdot\kappa(\Cb,\lamb:\kl_{1,R-1})}{\left[\nu(\Cb,\lamb)\right]^{1/2}},\quad\text{and}\notag\\
    &\Var\left(\frac{L_{N,1,3}}{\sigma_{N,2}(\kl_{1,R-1})}\right) =O\left(\frac{N^{-2}\sum_{t=2}^RN^{-\beta(2R-t-1)}}{N^{2(R-1)(1-\beta)}}\right) = O\left(\frac{1}{N^{(R-1)(1-\beta)}N^{R+1}}\right) \to 0,\quad\text{as $N\to\infty$,}
\end{align*}
where $\kappa(\Cb,\lamb:H)$ defined in \eqref{kap-def}. Hence, 
\begin{align}\label{G-1-3}
    \frac{L_{N,1,3}}{\sigma_{N,2}(\kl_{1,R-1})} \parw \frac{c\cdot\kappa(\Cb,\lamb:\kl_{1,R-1})}{\left[\nu(\Cb,\lamb)\right]^{1/2}}.
\end{align}
Next, it remains to analyze the term $L_{N,1,1}/\sigma_{N,2}(\kl_{1,R-1})$. Using the convergence result for $\sigma_{N,2}(\kl_{1,R-1})$ in \eqref{var-star} we can write 
\begin{align*}
    &\frac{L_{N,1,1}}{N^{(R-1)(1-\beta)}}=\frac{1}{N^{(1-\beta)(R-1)}}\sum_{\sbb\in \nsr}W_{N,s_1}\prod_{j=2}^R\pi_{\al_{s_1}, \al_{s_j}}\notag\\
    &=\frac{N^{(R-1) -\beta(R-1)}}{N^{(1-\beta)(R-1)}} \left(\sum_{u_1=1}^K\sum_{s_1=1}^NW_{N,s_1}\mathbf{1}(\al_{s_1} =u_1)\left(\sum_{u_2,\ldots,u_R=1}^R\prod_{j=2}^RN^{\beta}\pi_{u_1,u_j}\left(\prod_{j=2}^R\la_{N,u_j} + O\left(\frac{1}{N}\right)\right)\right)\right) \notag\\
    &= \sum_{u_1=1}^K\sum_{s_1=1}^NW_{N,s_1}\mathbf{1}(\al_{s_1} =u_1)\left(\sum_{u_2,\ldots,u_R=1}^R\prod_{j=2}^RN^{\beta}\pi_{u_1,u_j}\left(\prod_{j=2}^R\la_{N,u_j} + O\left(\frac{1}{N}\right)\right)\right)\notag\\
    &= \sum_{u_1=1}^K\sum_{s_1=1}^NW_{N,s_1}\mathbf{1}(\al_{s_1} =u_1)\cdot a_{N,u_1},~\text{(say)}\notag\\
    &= \sum_{s_1=1}^NW_{N,s_1}\mathbf{1}(\al_{s_1} =1)a_{N,1}+\cdots+\sum_{s_1=1}^NW_{N,s_1}\mathbf{1}(\al_{s_1} =K)a_{N,K}\notag\\
    &= \sum_{s_1=1}^NW_{N,s_1}\left(\mathbf{1}(\al_{s_1} = 1)a_{N,1}+\cdots + \mathbf{1}(\al_{s_1} = K)\right)a_{N,K} \equiv J_N,~\text{(say)}.
\end{align*}
where, for fixed $u_1\in [K]^R$, using Assumptions \ref{assump-2} and \ref{assump-3}, we can write
\begin{align}\label{au-def}
    a_{N,u_1} = \sum_{u_2,\ldots,u_R=1}^R\prod_{j=2}^RN^{\beta}\pi_{u_1,u_j}\left(\prod_{j=2}^R\la_{N,u_j} + O\left(\frac{1}{N}\right)\right)\to a_{u_1}\equiv \sum_{u_2,\ldots,u_R=1}^R\prod_{j=2}^Rc_{u_1,u_j}\prod_{j=2}^R\la_{u_j}.
\end{align}
Now, using Assumptions \ref{assump-2} and \ref{assump-3}, the characteristic function of $J_N$ can be expressed as,
\begin{align}
    &\E(\exp(itJ_N)) = \E\left(\exp\left(it\sum_{s_1=1}^NW_{N,s_1}\left(\mathbf{1}(\al_{s_1} = 1)a_{N,1} +\cdots + \mathbf{1}(\al_{s_1} = K)\right)a_{N,K}\right)\right)\notag\\
    &= \prod_{s_1=1}^N\E\left(\exp\left(itW_{N,s_1}\left(\mathbf{1}(\al_{s_1} = 1)a_{N,1} +\cdots + \mathbf{1}(\al_{s_1} = K)a_{N,K}\right)\right)\right)\notag\\
    &=\prod_{s_1=1}^N \left(\exp\left(it[\mathbf{1}(\al_{s_1} = 1)a_{N,1} +\cdots + \mathbf{1}(\al_{s_1} = K)a_{N,K}]\right)\frac{c}{N} + \left(1-\frac{c}{N}\right)\right)\notag\\
    &= \prod_{u=1}^K\left(\left(\frac{c\cdot(\exp\left(ita_{N,u}\right)-1)}{N} + 1\right)^N\right)^{N_u/N}\notag\\
    &\to \exp\left(c\sum_{u_1=1}^K\lambda_{u_1}\left(e^{ita_{u_1}}-1\right)\right),~\text{as $N\rai$, for all $t\in\rl$.}
    \label{cf-Jlim}
\end{align} 
The above expression (in \eqref{cf-Jlim}) is the characteristic function of the limiting random variable $\sum_{u_1=1}^K a_{u_1}\cdot Z_{u_1}$, where $Z_{1},\ldots,Z_K$ are independent and $Z_{u_1} \sim \text{Poisson}(c\cdot\la_{u_1})$ for all $u_1\in [K]$. Therefore, using \eqref{var-star} we can claim
\begin{align}\label{G-1-1}
    \frac{L_{N,1,1}}{\sigma_{N,2}(\kl_{1,R-1})} \darw \frac{1}{[\nu(\Cb,\lamb:H)]^{1/2}}\sum_{u_1=1}^Ka_{u_1} Z_{u_1}. 
\end{align}
Hence, combining \eqref{G-2}, \eqref{G-1-2},\eqref{G-1-3} and \eqref{G-1-1}, also using the expression of $\kappa(\Cb,\lamb:\kl_{1,R-1})$ (cf. \eqref{kap-def}), we can write
\begin{align*}
    T_{N,2}(\kl_{1,R-1}) \darw \frac{1}{[\nu(\Cb,\lamb:H)]^{1/2}} \left(\sum_{u_1=1}^Ka_{u_1}(Z_{u_1} - c\cdot\la_{u_1})\right).
\end{align*}
where $a_{u_1}$ defined in \eqref{au-def} and $\nu(\Cb,\lamb:H)$ defined in \eqref{var-star}. This completes the proof.

\end{proof}


\begin{proof}[Proof of Theorem \ref{thm-cc-ind}] Even though this proof considers the estimated clustering coefficient in the induced case, it includes arguments which are usable in both the induced and the ego-centric cases. The clustering coefficient of the population network $G_N$ is defined as
\begin{align*}
\Gamma_N &= \begin{cases}
        \displaystyle\frac{S_N(\kl_3)}{S_N(\kl_{1,2})} \displaystyle& \text{if $S_N(\kl_{1,2})> 0$,}\\
         0 & \text{otherwise.}
    \end{cases}      
\end{align*}
The sample based estimated clustering coefficient is defined as 
\begin{align*}
    \widehat{\Gamma}_{N,l} = \begin{cases}
        \displaystyle\frac{\widehat{S}_{N,l}(\kl_3)/f_l(p_N,\kl_3)}{\widehat{S}_{N,l}(\kl_{1,2})/f_l(p_N:\kl_{1,2})} \displaystyle &\text{if $\widehat{S}_{N,l}(\kl_{1,2})>0$,}\\
         0,&\text{otherwise.}
    \end{cases},\quad\text{for $l=1,2$,}
\end{align*}
where, $\widehat{S}_{N,l}(\kl_{1,2})$ and $\widehat{S}_{N,l}(\kl_3)$, $l=1,2$, are defined in \eqref{s-all-main}, with
\begin{equation}
\label{fl-wedge-tri-def}
\begin{aligned}
&f_l(p_N:\kl_{1,2}) =\begin{cases}
    p_N^3 &\text{if $l=1$,}\\
    p_N^2(1-p_N) + p_N&\text{if $l=2$,}
\end{cases}\quad
\text{and}\quad\ f_l(p_N:\kl_3) = \begin{cases}
    p_N^3 &\text{if $l=1$,}\\
    3p_N^2(1-p_N) + p^3_N&\text{if $l=2$.}
\end{cases}
\end{aligned}
\end{equation}
From the definition provided above, one can write (for both $l=1$ and $l=2$),
\begin{align*}
    &\widehat{\Gamma}_{N,l} - \Gamma_N
    =\begin{cases}
    \displaystyle\frac{\widehat{S}_{N,l}(\kl_3)}{\widehat{S}_{N,l}(\kl_{1,2})}\frac{1}{b_l(p_N)} - \frac{S_N(\mathbb{K}_{3})}{S_N(\mathbb{K}_{1,2})} \displaystyle &\quad \text{if $S_N(\mathbb{K}_{1,2})>0$ and $\widehat{S}_{N,l}(\mathbb{K}_{1,2})>0$},\\
    \qquad {} 0&\quad \text{if $S_N(\mathbb{K}_{1,2})=0$},\\
     \displaystyle{}-\frac{S_N(\mathbb{K}_{3})}{S_N(\mathbb{K}_{1,2})}\displaystyle &\quad \text{if $S_N(\mathbb{K}_{1,2})>0$ and $\widehat{S}_{N,l}(\mathbb{K}_{1,2})=0$,}
    \end{cases},\\
\text{where,}\quad &b_l(p_N) = \frac{f_l(p_N:\kl_3)}{f_l(p_N:\kl_{1,2})},\quad\text{for $l=1,2$.}
\end{align*}
Define the sequences of events 
\begin{align}
D_{1,N} = \left[{S}_N(\mathbb{K}_{1,2})=0\right]\quad\text{and}\quad D_{l,2,N} = \left[\widehat{S}_{N,l}(\mathbb{K}_{1,2})=0\right],\quad\text{for each $N\geq 1$, and for $l=1,2$.}
\label{Dset-defs}
\end{align}
Then, for any $t\in\rl$ and the scaling sequence $\{r_{N,l}(\al,\beta):N\geq 1\}$, $l=1,2$, defined in \eqref{ll-cc-ind} and \eqref{ll-cc-ego}, we can write
\begin{align}
&\pr\left(r_{N,l}(\al,\beta)\left(\widehat{\Gamma}_{N,l} - \Gamma_N\right)\leq t\right) = \pr\left(r_{N,l}(\al,\beta)\left(\widehat{\Gamma}_{N,l} - \Gamma_N\right)\leq t\mid {\left(D_{l,2,N}\right)}^c\cap D^c_{1,N}\right)\pr\left({\left(D_{l,2,N}\right)}^c\cap D^c_{1,N}\right) \notag\\
    & \quad {}+  \pr\left(r_{N,l}(\al,\beta)\left(\widehat{\Gamma}_{N,l} - \Gamma_N\right)\leq t\mid D_{1,N}\right)\pr\left(D_{1,N}\right)\notag\\
    &\qquad {}+  \pr\left(r_{N,l}(\al,\beta)\left(\widehat{\Gamma}_{N,l} - \Gamma_N\right)\leq t\mid D_{1,N}^c\cap D_{l,2,N}\right)\pr\left(D_{1,N}^c\cap D_{l,2,N}\right)\notag\\
    & \equiv x_{1,N}(t) + x_{2,N}(t) + x_{3,N}(t),\quad\text{(say).}
\label{split-cc-1}    
\end{align}
Consider the second and third terms in the r.h.s. of \eqref{split-cc-1}. Note that $D_{1,N}\subseteq D_{l,2,N}$, for both $l=1,2$, since zero wedges in the population network $G_N$ will imply that there will be zero wedges in the observed  induced (or ego-centric) subgraph. This implies $\pr(D_{1,N}) \leq \pr(D_{l,2,N})$, for $l=1,2$. Note that for $l=1,2$, we can write
\begin{align*}
& \E\left(\widehat{S}_{N,1}(\kl_{1,2})\right) = 
        p_N^{3}\sum_{\sbb\in [N]_3}  \pi_{N,\al_{s_1},\al_{s_2}}\pi_{N,\al_{s_1},\al_{s_3}}\notag\\
  &=    N^{3-3\al-2\beta} c^3 \sum_{\ub\in [K]^3}N^{2\beta}\pi_{N,u_1,u_2}\pi_{N,u_1,u_3}\left(\la_{N,u_1}\la_{N,u_2}\la_{N,u_3} + O\left(\frac{1}{N}\right)\right),\quad\text{and}\notag\\
 & \E\left(\widehat{S}_{N,2}(\kl_{1,2})\right) = 
        \left(p_N + (1-p_N)p_N^2\right)\sum_{\sbb\in [N]_3}  \pi_{N,\al_{s_1},\al_{s_2}}\pi_{N,\al_{s_1},\al_{s_3}}\notag\\
  &=  N^{3-\al-2\beta}   c\left(1 + \left(1-\frac{c}{N^{\al}}\right)\frac{c}{N^{\al}}\right) \sum_{\ub\in [K]^3}N^{2\beta}\pi_{N,u_1,u_2}\pi_{N,u_1,u_3}\left(\la_{N,u_1}\la_{N,u_2}\la_{N,u_3} + O\left(\frac{1}{N}\right)\right).
\end{align*}
Also, from the definition of $g_{l,\al,\beta,H}(\cdot)$ in \eqref{delta-bdd-temp}, $g_{1,\al,\beta,\kl_{1,2}}(3) = -3 +3\al + 2\beta$, and $g_{2,\al,\beta,\kl_{1,2}}(3) = -3 +\al+2\beta
$. Using this, we can say that 
\begin{align*}
    &\frac{\E\left(\widehat{S}_{N,1}(\kl_{1,2})\right)}{N^{g_{1,\al,\beta,\kl_{1,2}}(3)}} = c^3\sum_{\ub\in [K]^3}c_{u_1,u_2}c_{u_1,u_3}\cdot\la_{u_1}\la_{u_2}\la_{u_3} + R_{N,1}(\kl_{1,2}) =\mu_1(\kl_{1,2}) + R_{N,1}(\kl_{1,2}) ,~\text{and},\notag\\
    &\frac{\E\left(\widehat{S}_{N,2}(\kl_{1,2})\right)}{N^{g_{2,\al,\beta,\kl_{1,2}}(3)}} =  c\sum_{\ub\in [K]^3}c_{u_1,u_2}c_{u_1,u_3}\cdot\la_{u_1}\la_{u_2}\la_{u_3}+ R_{N,2}(\kl_{1,2}) =\mu_2(\kl_{1,2}) + R_{N,2}(\kl_{1,2}), \quad\text{(say),}
\end{align*}
where,
\begin{align*}
    R_{N,1}(\kl_{1,2}) &= \left\{c^3 \sum_{\ub\in [K]^3}N^{2\beta}\pi_{N,u_1,u_2}\pi_{N,u_1,u_3}\left(\la_{N,u_1}\la_{N,u_2}\la_{N,u_3} + O\left(\frac{1}{N}\right)\right)\right.\notag\\
    &\qquad\qquad{}\left.{}- c^3\sum_{\ub\in [K]^3}c_{u_1,u_2}c_{u_1,u_3}\cdot\la_{u_1}\la_{u_2}\la_{u_3} \right\},~\text{and,}\notag\\
    R_{N,2}(\kl_{1,2}) &= \left\{c\left(1 + \left(1-\frac{c}{N^{\al}}\right)\frac{c}{N^{\al}}\right) \sum_{\ub\in [K]^3}N^{2\beta}\pi_{N,u_1,u_2}\pi_{N,u_1,u_3}\left(\la_{N,u_1}\la_{N,u_2}\la_{N,u_3} + O\left(\frac{1}{N}\right)\right)\right.\notag\\
    &\qquad\qquad{}\left.{}-c\sum_{\ub\in [K]^3}c_{u_1,u_2}c_{u_1,u_3}\cdot\la_{u_1}\la_{u_2}\la_{u_3}\right\}.
\end{align*}
Using Assumptions \ref{assump-2} and \ref{assump-3}, we can say that $|R_{N,l}(\kl_{1,2})|\to 0$ as $N\rai$ for $l=1,2$. Due to Lemma \ref{VC-lem-3}, the triangle and wedge satisfy the Assumption \ref{assump-5}. Now, using Proposition \ref{samp-CLT}, we can write,
\begin{align}
    b_{N,l}(\al,\beta)\left(\frac{\widehat{S}_{N,l}(\kl_{1,2}) - \E\left(\widehat{S}_{N,l}(\kl_{1,2})\right)}{N^{\frac{1}{2}g_{l,\al,\beta,\kl_{1,2}}(3)}}\right) \darw N(0,\psi^2_l(\kl_{1,2})),\quad\text{for $l=1,2$,}
\end{align}
where,
\begin{align*}
    b_{N,1}(\al,\beta) =\begin{cases}
        N^{\frac{1}{2}(1-\al)}&~\text{if $(\al,\beta)\in C_{1,1}(\kl_{1,2})$,}\\
        N^{\frac{3}{2}(1-\al) -\beta}&~\text{if $(\al,\beta)\in C_{1,3}(\kl_{1,2})$,}
    \end{cases}~\text{and}~
        b_{N,2}(\al,\beta) =\begin{cases}
        N^{\frac{1}{2}(1-\al)}&~\text{if $(\al,\beta)\in C_{2,1}(\kl_{1,2})$,}\\
        N^{\frac{1}{2}(3-\al-2\beta)}&~\text{if $(\al,\beta)\in C_{2,3}(\kl_{1,2})$.}
    \end{cases}
\end{align*}
Then from the expression \eqref{split-cc-1}, $\pr\left(D_{l,2,N}\right)$ has the following expression,
\begin{align*}
&\pr\left(D_{l,2,N}\right) = \pr\left(b_{N,l}(\al,\beta)\cdot\frac{\widehat{S}_{N,l}(\kl_{1,2})-\E\left(\widehat{S}_{N,l}(\kl_{1,2})\right)}{N^{\frac{1}{2}g_{l,\al,\beta,\kl_{1,2}}(3)}} = {}-b_{N,l}(\al,\beta)\cdot \frac{\E\left(\widehat{S}_{N,l}(\kl_{1,2})\right)}{N^{\frac{1}{2}g_{l,\al,\beta,\kl_{1,2}}(3)}}\right)\notag\\
& = \pr\left(b_{N,l}(\al,\beta)\cdot\frac{\widehat{S}_{N,l}(\kl_{1,2})-\E\left(\widehat{S}_{N,l}(\kl_{1,2})\right)}{N^{\frac{1}{2}g_{l,\al,\beta,\kl_{1,2}}(3)}} = {}-b_{N,l}(\al,\beta)\cdot\left(\mu_{l}(\kl_{1,2}) + R_{N,l}(\kl_{1,2})\right)\right)\\
& = \left[\pr\left(b_{N,l}(\al,\beta)\cdot\frac{\widehat{S}_{N,l}(\kl_{1,2})-\E\left(\widehat{S}_{N,l}(\kl_{1,2})\right)}{N^{\frac{1}{2}g_{l,\al,\beta,\kl_{1,2}}(3)}} \leq {}-b_{N,l}(\al,\beta)\cdot\left(\mu_{l}(\kl_{1,2}) + R_{N,l}(\kl_{1,2})\right)\right)\right.\notag\\
&\quad \left. {}- \Phi\left({}-\frac{b_{N,l}(\al,\beta)}{\psi_l(\kl_{1,2})}\cdot\left(\mu_{l}(\kl_{1,2}) + R_{N,l}(\kl_{1,2})\right)\right)\right]\\
&\quad {}-\left[ \pr\left(b_{N,l}(\al,\beta)\cdot\frac{\widehat{S}_{N,l}(\kl_{1,2})-\E\left(\widehat{S}_{N,l}(\kl_{1,2})\right)}{N^{\frac{1}{2}g_{l,\al,\beta,\kl_{1,2}}(3)}} < {}-b_{N,l}(\al,\beta)\cdot\left(\mu_{l}(\kl_{1,2}) + R_{N,l}(\kl_{1,2})\right)\right)\right.\\
&\quad \left. {}-\Phi\left({}-\frac{b_{N,l}(\al,\beta)}{\psi_l(\kl_{1,2})}\cdot\left(\mu_{l}(\kl_{1,2}) + R_{N,l}(\kl_{1,2})\right)\right)\right]\\
&\equiv e_{l,1,N}-e_{l,2,N},\quad\text{(say), for $l=1,2$.}
\end{align*}
Following the CLT result in Proposition \ref{samp-CLT}, we can use Polya's theorem (see Lemma 8.2.6 of \cite{kba-snl}) to write
$$
|e_{l,1,N}|\leq \sup_{t\in\rl}\left|\pr\left(b_{N,l}(\al,\beta)\cdot\frac{\widehat{S}_{N,l}(\kl_{1,2})-\E\left(\widehat{S}_{N,l}(\kl_{1,2})\right)}{N^{\frac{1}{2}g_{l,\al,\beta,\kl_{1,2}}(3)}}\leq t\right)-\Phi\left(\frac{t}{\psi_l(\kl_{1,2})}\right)\right|\rightarrow 0,\quad\text{as $N\rai$,}
$$
for $l=1,2$. In order to handle $e_{l,2,N}$, we use Lemma \ref{lem-sup-eq} (with $G$ being the standard normal cdf). This implies 
\begin{align*}
|e_{l,2,N}|&\leq \sup_{t\in\rl}\left|\pr\left(b_{N,l}(\al,\beta)\cdot\frac{\widehat{S}_{N,l}(\kl_{1,2})-\E\left(\widehat{S}_{N,l}(\kl_{1,2})\right)}{N^{\frac{1}{2}g_{l,\al,\beta,\kl_{1,2}}(3)}}< t\right)-\Phi\left(\frac{t}{\psi_l(\kl_{1,2})}\right)\right|\\
& = \sup_{t\in\rl}\left|\pr\left(b_{N,l}(\al,\beta)\cdot\frac{\widehat{S}_{N,l}(\kl_{1,2})-\E\left(\widehat{S}_{N,l}(\kl_{1,2})\right)}{N^{\frac{1}{2}g_{l,\al,\beta,\kl_{1,2}}(3)}}\leq t\right)-\Phi\left(\frac{t}{\psi_l(\kl_{1,2})}\right)\right| \rightarrow 0,\quad\text{as $N\rai$,}
\end{align*}
for $l=1,2$. This shows that $\pr(D_{1,N})\leq \pr\left(D_{l,2,N}\right)\rightarrow 0$, as $N\rai$, for each $l=1,2$. This implies $x_{2,N}(t)\rightarrow 0$, as $N\rai$, for each $t\in\rl$ (cf. \eqref{split-cc-1}). It also shows that $\pr\left(D^c_{1,N}\cap D_{l,2,N}\right)\leq \pr\left(D_{l,2,N}\right)\rightarrow 0$, in the third term of \eqref{split-cc-1}, thus ensuring that $x_{3,N}(t)\rightarrow 0$, for each $t\in\rl$. Continuing from \eqref{split-cc-1}, now we can write (using previous arguments)
\begin{align}
\label{r1N}
&\pr\left(r_{N,l}(\al,\beta)\left(\widehat{\Gamma}_{N,l}-\Gamma_N\right)\leq t\right) =  r_{1,N}(t) + o(1) \notag\\ & =\pr\left(r_{N,l}(\al,\beta)\left(\widehat{\Gamma}_{N,l}-\Gamma_N\right)\leq t, \widehat{S}_{N,l}(\kl_{1,2})>0,S_N(\kl_{1,2})>0\right)+o(1).
\end{align}
Now, define the following random vectors for all $N\geq 1$, 
\begin{equation}
\begin{aligned}
    &\widehat{\Xb}_{N,l} =\left(\widehat{X}_{N,l,1},\widehat{X}_{N,l,2}\right)^\primet= \left(\frac{\widehat{S}_{N,l}(\kl_{1,2})}{N^{3-2\beta}f_l(p_N:\kl_{1,2})}, \frac{\widehat{S}_{N,l}(\kl_3)}{N^{3-3\beta}f_l(p_N:\kl_3)}\right)^\primet,\quad\text{for $l=1,2$, and}\\
    &\Xb_{N} =(X_{N,1},X_{N,2})^\primet= \left(\frac{S_N(\kl_{1,2})}{N^{3-2\beta}}, \frac{S_{N}(\kl_3)}{N^{3-3\beta}}\right)^\primet.
\end{aligned}
\label{Xhat-Nldef}
\end{equation}
Note that, from Assumptions \ref{assump-2} and \ref{assump-3} we can write 
$\thb_N \equiv  \E\left(\widehat{\Xb}_{N,l}\right) =\E\left(\Xb_N\right) = \left(\theta_{1,N}, \theta_{2,N}\right)^\primet$, where,
\begin{equation}
\label{theta-def}
\begin{aligned}
&\theta_{1,N} = \frac{1}{N^{3-2\beta}}\sum_{i\neq j\neq k}\pi_{\al_i,\al_j}\pi_{\al_i,\al_k} \to \theta_1\equiv \sum_{u,v,w=1}^Kc_{u,v}c_{u,w}\cdot \la_u\la_v\la_w \in (0,\infty), \quad\text{and}\\
&\theta_{2,N} = \frac{1}{N^{3-3\beta}}\sum_{i\neq j\neq k}\pi_{\al_i,\al_j}\pi_{\al_i,\al_k}\pi_{\al_j,\al_k} \to \theta_2\equiv \sum_{u,v,w=1}^Kc_{u,v}c_{u,w}c_{v,w}\cdot \la_u\la_v\la_w\in (0,\infty).
\end{aligned}
\end{equation}
Define the map, 
$$
f(x,y) = \begin{cases}
        \frac{y}{x} & \text{for all $x\in \mathbb{R}\setminus\{0\}$ and $y\in\rl$,}\\
        0 &  \text{otherwise.}
        \end{cases}
$$
Note that $f(x,y)$ is continuous and differentiable in $\mathcal{D} = \{(x,y) : x\in \mathbb{R}\setminus\{0\},~y\in\mathbb{R}\}$. 
From \eqref{ll-cc-ind} and \eqref{ll-cc-ego}, write, 
$$
q_{N,l}(\al,\beta) = \frac{r_{N,l}(\al,\beta)}{N^\beta},\quad\text{for $l=1,2$, for all $N\geq 1$.}
$$
Note that using the above relation and \eqref{Xhat-Nldef}, we can write
\begin{align*}
    &r_{N,l}(\al,\beta)\left(\widehat{\Gamma}_{N,l}-\Gamma_N\right) = q_{N,l}(\al,\beta)\cdot N^{\beta}\cdot \left(\widehat{\Gamma}_{N,l}-\Gamma_N\right)=q_{N,l}(\al,\beta)N^{\beta}\left(\displaystyle\frac{\widehat{S}_{N,l}(\kl_3)}{\widehat{S}_{N,l}(\kl_{1,2})}\frac{1}{b_l(p_N)} - \frac{S_N(\mathbb{K}_{3})}{S_N(\mathbb{K}_{1,2})} \displaystyle\right)\notag\\
    &= q_{N,l}(\al,\beta)\left(\displaystyle\frac{\widehat{S}_{N,l}(\kl_3)/\{N^{3-3\beta}f_{l}(p_N:\kl_3)\}}{\widehat{S}_{N,l}(\kl_{1,2})/\{N^{3-2\beta}f_l(p_N:\kl_{1,2})\}} - \frac{S_N(\mathbb{K}_{3})/N^{3-3\beta}}{S_N(\mathbb{K}_{1,2})/N^{3-2\beta}} \displaystyle\right) = q_{N,l}(\al,\beta)\left[f\left(\widehat{\Xb}_{N,l}\right)  -f(\Xb_N)\right]. 
\end{align*}
From the expression \eqref{r1N}, using \eqref{Xhat-Nldef} and \eqref{theta-def}, we can write the following for any $t\in \mathbb{R}$,
\begin{align}\label{p-1}
     & \pr\left(r_{N,l}(\al,\beta)\left(\widehat{\Gamma}_{N,l}-\Gamma_N\right)\leq t,~\widehat{S}_{N,l}(\kl_{1,2})>0,~S_N(\kl_{1,2})>0\right)\notag\\
     &=\pr \left(q_{N,l}(\al,\beta)\left[f\left(\widehat{\Xb}_{N,l}\right)  -f(\Xb_N) \right]\leq t,~ \widehat{S}_{N,l}(\kl_{1,2})>0,S_N(\kl_2)>0\right)\notag\\
    &=\pr \left(q_{N,l}(\al,\beta)\left[f\left(\widehat{\Xb}_{N,l}\right) - f(\thb_N) -\left\{f(\Xb_N) -f(\thb_N)\right\}\right]\leq t,~ \widehat{S}_{N,l}(\kl_{1,2})>0,S_N(\kl_2)>0\right)\notag\\
    &= \pr\left(q_{N,l}(\al,\beta)\left[-\frac{\theta_{2,N}}{(\theta_{1,N})^2}\left(\widehat{X}_{N,l,1} - X_{N,1}\right)+\frac{1}{\theta_{1,N}}\left(\widehat{X}_{N,l,2} - X_{N,2}\right) + \frac{1}{2}\widehat{\Upsilon}_{N,l} +\frac{1}{2}\Upsilon_N\right]\leq t, \right.\notag\\
    &\qquad \qquad \left. ~\widehat{X}_{N,l,1}>0,~X_{N,1}>0\right),
\end{align}
where,
\begin{equation}
\begin{aligned}
\label{rem-term}
    &\widehat{\Upsilon}_{N,l} =\widehat{\mathbf{Z}}^\primet_{N,l}\mathbf{H}_f\left(\widehat{\xi}_{N,l},\widehat{\nu}_{N,l}\right)\widehat{\mathbf{Z}}_{N,l},\quad \Upsilon_N =\mathbf{Z}^\primet_{N}\mathbf{H}_f\left({\xi}_{N},{\nu}_{N}\right)\mathbf{Z}_{N},\\
    &\widehat{\mathbf{Z}}_{N,l} = {\left(\widehat{X}_{N,l,1}-\theta_{1,N},\widehat{X}_{N,l,2}-\theta_{2,N}\right)}^\primet,\quad \mathbf{Z}_{N,l} = {\left(X_{N,1} -\theta_{1,N},X_{N,2} -\theta_{2,N}\right)}^\primet,\\
    &{\left(\widehat{\xi}_{N,1}, \widehat{\nu}_{N,1}\right)}^\primet = \left(\theta_{1,N} +\delta_1 \left(\widehat{X}_{N,l,1} - \theta_{1,N}\right), \theta_{2,N} + \delta_1\left(\widehat{X}_{N,l,2}-\theta_{2,N}\right)\right)^\primet,\ \text{for some $\delta_1\in (0,1)$,}\\
    &{\left({\xi}_{N}, {\nu}_{N}\right)}^\primet = {\left(\theta_{1,N} +\delta_2 \left({X}_{N,1} - \theta_{1,N}\right), \theta_{2,N} + \delta_2\left({X}_{N,2}-\theta_{2,N}\right)\right)}^\primet,\quad\text{for some $\delta_2\in (0,1)$, and}\\
    & \mathbf{H}_f(x,y) = \begin{pmatrix}
        -\frac{2y}{x^3}&\frac{1}{x^2}\\
        \frac{1}{x^2} & 0
    \end{pmatrix}.
\end{aligned}
\end{equation}
Next, let us define the random variables (cf. \eqref{Xhat-Nldef}),
\begin{align}\label{piv-w-t}
    T_{N,l}\left(\kl_{1,2},\kl_3\right) = -\frac{\theta_{2,N}}{(\theta_{1,N})^2}\left(\widehat{X}_{N,l,1} - X_{N,1}\right)+\frac{1}{\theta_{1,N}}\left(\widehat{X}_{N,l,2} - X_{N,2}\right),\quad\text{for all $N\geq 1$, and $l=1,2$.}
\end{align}
For any $\epsilon>0$, consider events, $A_{l,1,N} = \left\{\left|\widehat{X}_{N,l,1} - \theta_{1,N}\right|\leq \epsilon\right\}$, $1,2$, and $B_{1,N} = \left\{|X_{N,1} - \theta_{1,N}|\leq  \epsilon\right\}$, and $C_{1,N} = \left\{\widehat{X}_{N,l,1}>0,X_{N,1}>0\right\}$. Thus, by construction $A_{l,1,N}\cap B_{1,N}\subseteq C_{1,N}$. Note that, using Lemma \ref{SN-hat-tght}, we can say that for $l=1,2$,
\begin{align*}
    \pr\left(A^c_{l,1,N}\right) = \pr\left(\left|\widehat{X}_{N,l,1} - \theta_{1,N}\right|> \epsilon\right) = o(1)\quad\text{and}\quad\pr\left(B^c_{N,1}\right) = \pr\left(|{X}_{N,1} - \theta_{1,N}|> \epsilon\right) = o(1).
\end{align*}
Then, from the expression \eqref{p-1}, we can write
\begin{align}\label{piv-cc-1}
    &\pr\left(q_{N,l}(\al,\beta)\left[T_{N,l}\left(\kl_{1,2},\kl_3\right) + \frac{1}{2}\widehat{\Upsilon}_{N,l} +\frac{1}{2}\Upsilon_N\right]\leq t, \widehat{X}_{N,l,1}>0,X_{N,1}>0\right)\notag\\
    &=\pr\left(q_{N,l}(\al,\beta)\left[T_{N,l}\left(\kl_{1,2},\kl_3\right) + \frac{1}{2}\widehat{\Upsilon}_{N,l} +\frac{1}{2}\Upsilon_N\right]\leq t, A_{l,1,N}\cap B_{1,N}\right)\notag\\
    & \qquad{}+\pr\left(q_{N,l}(\al,\beta)\left[T_{N,l}\left(\kl_{1,2},\kl_3\right) + \frac{1}{2}\widehat{\Upsilon}_{N,l} +\frac{1}{2}\Upsilon_N\right]\leq t, B_{1,N}\cap C_{1,N}~|~A_{l,1,N}^c\right)\pr(A_{l,1,N}^c)\notag\\
    &\qquad{}\qquad{}+ \pr\left(q_{N,l}(\al,\beta)\left[T_{N,l}\left(\kl_{1,2},\kl_3\right) + \frac{1}{2}\widehat{\Upsilon}_{N,l} +\frac{1}{2}\Upsilon_N\right]\leq t, A_{l,1,N}\cap C_{1,N}~|~B_{1,N}^c\right)\pr(B_{1,N}^c)\notag\\
    &\qquad{}\qquad{}\qquad{}+ \pr\left(q_{N,l}(\al,\beta)\left[T_{N,l}\left(\kl_{1,2},\kl_3\right) + \frac{1}{2}\widehat{\Upsilon}_{N,l} +\frac{1}{2}\Upsilon_N\right]\leq t, A^c_{l,1,N}\cap C_{1,N}~|~B_{1,N}^c\right)\pr(B_{1,N}^c)\notag\\
    &=\pr\left(q_{N,l}(\al,\beta)\cdot T_{N,l}\left(\kl_{1,2},\kl_3\right)  + \frac{q_{N,l}(\al,\beta)}{2}\left\{\widehat{\Upsilon}_{N,l} +\Upsilon_N\right\}\leq t,~ A_{l,1,N}\cap B_{1,N}\right) + o(1). 
\end{align}
From \eqref{theta-def}, we know $\theta_{1,N} \to \theta_1$. Fix $\eps = \theta_1/2$. Then there exists a $M\geq 1$, such that for all $N>M$, when the event $A_{l,1,N} \cap B_{1,N}$ happens, then 
\begin{align}\label{d-bdd}
  \frac{\theta_1}{2} < \theta_{1,N} + \delta_1\bigl(\widehat{X}_{N,l,1} - \theta_{1,N}\bigr) <\frac{3\theta_1}{2} \quad \text{and} \quad  \frac{\theta_1}{2} <\theta_{1,N} + \delta_2\bigl({X}_{N,1} - \theta_{1,N}\bigr) <\frac{3\theta_1}{2}. 
\end{align}
Now consider the case of induced network formation ($l=1$).  
Recall the definition of $R_{1,j}$, $j=1,\ldots,5$, from \eqref{rsets-def}. Using Lemma \ref{lem-var}, there exists an $M_5$ such that for all $N>M_5$, we can write
\begin{align}\label{var-rate-ind}
    &a_4\cdot N^{5(1-\al)-4\beta}\leq \sigma^2_{N,1}(\kl_{1,2}) \leq a_5\cdot N^{5(1-\al)-4\beta}\quad\text{when $(\al,\beta)\in \bigcup_{i=1}^5R_{1,i}$}\quad~\text{and}\notag\\
    & a_6\cdot N^{5(1-\al)-6\beta}\leq \sigma^2_{N,1}(\kl_{3}) \leq a_7\cdot N^{5(1-\al)-6\beta}\quad\text{when $(\al,\beta)\in R_{1,1}\cup R_{1,2}$, and}\\
    &\sigma^2_{N,1}(\kl_{3}) \geq \begin{cases}
        a_8\cdot N^{3(1-\al) - 3\beta}\quad~&\text{when $(\al,\beta)\in R_{1,3}$},\\
        a_9\quad~&\text{when $(\al,\beta)\in R_{1,4}\cup R_{1,5}$},
    \end{cases}\notag
\end{align}
where, $a_4,a_5,a_6,a_7,a_8,a_9$, are finite positive constants arising from Lemma \ref{lem-var}. Therefore, using the definition of $f_1(p_N:\kl_{1,2})$ (cf. \eqref{fl-wedge-tri-def}), we can write
\begin{align}\label{var-Xhat}
    \widetilde{\sigma}^2_{N,1}(\kl_{1,2}) &\equiv  \Var\left(\widehat{X}_{N,1,1}-X_{N,1}\right) = \frac{\Var\left(\widehat{S}_{N,1}(\kl_{1,2}) - f_1(p_N:\kl_{1,2})S_N(\kl_{1,2})\right)}{N^{6-4\beta}f^2_1(p_N:\kl_{1,2})} = \frac{\Var\left(T_{N,1}(\kl_{1,2})\right)}{N^{6(1-\al)-4\beta}c^6}\notag\\
    &=\frac{\sigma^2_{N,1}(\kl_{1,2})}{c^6\cdot N^{6(1-\al)-4\beta}},\quad\text{and}\notag\\
    \widetilde{\sigma}^2_{N,1}(\kl_{3})&\equiv \Var\left(\widehat{X}_{N,1,2}-X_{N,2}\right) = \frac{\Var\left(\widehat{S}_{N,1}(\kl_{2}) - f_1(p_N:\kl_{3})S_N(\kl_{3})\right)}{N^{6-6\beta}f^2_1(p_N:\kl_{3})} = \frac{\Var\left(T_{N,1}(\kl_{3})\right)}{N^{6(1-\al)-6\beta}c^6}\notag\\
    &=\frac{\sigma^2_{N,1}(\kl_{3})}{c^6\cdot N^{6(1-\al)-6\beta}}.
\end{align}
Using \eqref{var-rate-ind} and \eqref{var-Xhat}, we can claim that for all $N>M_5$,
\begin{align}\label{var-ratio-ind}
& \frac{a_4}{a_7}\leq \frac{\widetilde{\sigma}^2_{N,1}(\kl_{1,2})}{\widetilde{\sigma}^2_{N,1}(\kl_{3})}\leq \frac{a_5}{a_6},\quad\text{if $(\al,\beta)\in R_{1,1}\cup R_{1,2}$, and}\notag\\
& \frac{\widetilde{\sigma}^2_{N,1}(\kl_{1,2})}{\widetilde{\sigma}^2_{N,1}(\kl_{3})} \leq \begin{cases}
         \frac{a_5}{a_8}\cdot N^{2\left(1-\al -\frac{\beta}{2/3}\right)}~&\text{when $(\al,\beta) \in {R}_{1,3}$,}\\
        \frac{a_5}{a_9}\frac{1}{N^{\beta}}~&\text{when $(\al,\beta) \in {R}_{1,4}$,}\\
        \frac{a_5}{a_9}\frac{1}{N}&\text{when $(\al,\beta) \in {R}_{1,5}$.}
    \end{cases}
\end{align}
This implies that the limiting behavior of $T_{N,1}(\kl_{1,2},\kl_{1,3})$ (see \eqref{piv-w-t}) is jointly influenced by the wedge- and triangle count based terms $(\widehat{X}_{N,1,1}-X_{N,1})$ and $(\widehat{X}_{N,1,2}-X_{N,2})$, when $(\al,\beta)\in R_{1,1}\cup R_{1,2}$, and is influenced only by the triangle based term $(\widehat{X}_{N,1,2}-X_{N,2})$, when $(\al,\beta)\in R_{1,3}\cup R_{1,4}\cup R_{1,5}$. Thus, we need to analyze separately the limiting behavior of $T_{N,1}(\kl_{1,2},\kl_3)$ in these regions. Due to \eqref{piv-cc-1}, firstly we obtain the asymptotic distribution of
\begin{align}\label{piv-TN-ind}
   q_{N,1}(\al,\beta)(\al,\beta)\cdot T_{N,1}(\kl_{1,2},\kl_3) = q_{N,1}(\al,\beta)\left[-\frac{\theta_{2,N}}{(\theta_{1,N})^2}\left(\widehat{X}_{N,1,1} - X_{N,1}\right)+\frac{1}{\theta_{1,N}}\left(\widehat{X}_{N,1,2} - X_{N,2}\right)\right]. 
\end{align}
In case of $(\al,\beta)\in R_{1,1}\cup R_{1,2}$, using Proposition \ref{prop-joint}, and using \eqref{theta-def}, we can write,
 \begin{align*}
N^{\gamma(\al,\beta)}
\Var\!\bigl(T_{N,1}(\kl_{1,2},\kl_3)\bigr)
\;\to\;\rho^2_{1,j}(\kl_{1,2},\kl_3),
\quad \text{if $(\al,\beta)\in R_{1,j}$, for $j=1,2$,}
\end{align*}
where
$$
\gamma(\al,\beta)= \begin{cases}
1 & \text{if $(\al,\beta)\in R_{1,1}$,}\\
(1-\al) & \text{if $(\al,\beta)\in R_{1,2}$,}
\end{cases}
$$
and $\rho^2_{1,j}(\kl_{1,2},\kl_3)$, $j=1,2$, are defined in \eqref{var-rho-1-1} and \eqref{var-rho-1-2} respectively. In addition, using Proposition \ref{prop-joint} we can say that 
$$
N^{\gamma(\al,\beta)/2}\, T_{N,1}(\kl_{1,2},\kl_3) \darw N\left(0,\rho^2_{1,j}(\kl_{1,2},\kl_3)\right),
\qquad \text{if $(\al,\beta)\in R_{1,j}$, for $j=1,2$.}
$$
Thus, when $(\al,\beta)\in R_{1,1}\cup R_{1,2}$, we will have $q_{N,1}(\al,\beta) = N^{(1-\al)/2}$, for all $N\geq 1$, for $\al\in [0,1)$. Now consider the case where $(\al,\beta)\in R_{1,3}\cup R_{1,4}\cup R_{1,5}$. From \eqref{piv-TN-ind} we can write
\begin{align*}
    \frac{T_{N,1}(\kl_{1,2},\kl_3)}{\widetilde{\sigma}_{N,1}(\kl_3)} &= -\frac{\theta_{2,N}}{(\theta_{1,N})^2}\frac{\widetilde{\sigma}_{N,1}(\kl_{1,2})}{\widetilde{\sigma}_{N,1}(\kl_3)}\frac{\left(\widehat{X}_{N,1,1} - X_{N,1}\right)}{\widetilde{\sigma}_{N,1}(\kl_{1,2})}+\frac{1}{\theta_{1,N}}\frac{1}{\widetilde{\sigma}_{N,1}(\kl_3)}\left(\widehat{X}_{N,1,2} - X_{N,2}\right)\notag\\
    &= -\frac{\theta_{2,N}}{(\theta_{1,N})^2}\frac{\widetilde{\sigma}_{N,1}(\kl_{1,2})}{\widetilde{\sigma}_{N,1}(\kl_3)}T_{N,1}(\kl_{1,2})+\frac{1}{\theta_{1,N}}T_{N,1}(\kl_3).
\end{align*}
Now using \eqref{var-ratio-ind} and Corollary \ref{cor-clt-all}, we can write 
\begin{align}\label{piv-TN-imp}
    \frac{T_{N,1}(\kl_{1,2},\kl_3)}{\widetilde{\sigma}_{N,1}(\kl_3)} &= -\frac{\theta_{2,N}}{(\theta_{1,N})^2}\cdot O\left(\frac{1}{N}\right)\cdot O_P(1)+\frac{1}{\theta_{1,N}}T_{N,1}(\kl_3)\notag\\
    \Rightarrow ~q_{N,1}(\al,\beta)\cdot T_{N,1}(\kl_{1,2},\kl_3) &= -\frac{\theta_{2,N}}{(\theta_{1,N})^2}\cdot O\left(\frac{1}{N}\right)\cdot O_P(1)\cdot \{q_{N,1}(\al,\beta)\cdot\widetilde{\sigma}_{N,1}(\kl_3)\}\notag\\
    &\qquad{}\qquad{}+\frac{q_{N,1}(\al,\beta)\cdot\widetilde{\sigma}_{N,1}(\kl_3)}{\theta_{1,N}}\cdot T_{N,1}(\kl_3).
\end{align}
From Corollary \eqref{cor-clt-all}, $T_{N,1}(\kl_3)\darw N(0,1)$. Using Lemma \ref{lem-var} and for $(\al,\beta)\in R_{1,3}$, taking $q_{N,1}(\al,\beta) = N^{3(1-\al-\beta)/2}$, we can say that
\begin{align*}
    N^{3(1-\al-\beta)/2}\cdot \widetilde{\sigma}_{N,1}(\kl_3) = \frac{\sigma_{N,1}(\kl_3)}{N^{3(1-\al-\beta)/2}}  = O(1).
\end{align*}
Thus, using the above arguments for $(\al,\beta)\in R_{1,3}$ and using Slutsky's Theorem, we can say that
\begin{align*}
    N^{3(1-\al-\beta)/2}\cdot T_{N,1}(\kl_{1,2},\kl_3)\darw N(0,~\rho^2_{1,3}(\kl_3)),\quad\text{as $N\to\infty$,}
\end{align*}
where 
\begin{align}\label{v-1}
    \frac{N^{3(1-\al-\beta)/2}\cdot\widetilde{\sigma}_{N,1}(\kl_3)}{\theta_{1,N}} \to \rho_{1,3}(\kl_3),
\end{align}
and $\rho_{1,3}(\kl_3)$ defined in \eqref{var-rho-1-3-kap-1-4}. In case of $(\al,\beta)\in R_{1,4}\cup R_{1,5}$, using Theorem \ref{thm-2-main-temp}, we can write
\begin{align*}
    &T_{N,1}(\kl_3) \darw \frac{V_{1,\kl_3} - c^3\cdot \kappa(\Cb,\lamb:\kl_3)}{[c^3\cdot \kappa(\Cb,\lamb:\kl_3)]^{1/2}},\qquad{}\text{for $(\al,\beta)\in R_{1,4}$, and}\notag\\
    &T_{N,1}(\kl_3) \darw \frac{\left(1-c^3\right)\cdot U_{1,1,\kl_3} - c^3\cdot U_{2,1,\kl_3}}{[c^3\cdot\{1-c^3\}\cdot \kappa(\Cb,\lamb: \kl_3)]^{1/2}},\qquad{}\text{when $(\al,\beta)\in R_{1,5}$,}
\end{align*}
where, $V_{1,\kl_3}$, $U_{1,1,\kl_3}$ and $U_{2,1,\kl_3}$ are defined in Theorem \ref{thm-2-main-temp}, and $\kappa(\Cb,\lamb,\kl_3)$ is defined in \eqref{kap-def}. Note that, in this case, $q_{N,1}(\al,\beta) =1$ for $(\al,\beta)\in R_{1,4}\cup R_{1,5}$. Thus, from \eqref{piv-TN-imp} and Slutsky's Theorem, we can write, 
\begin{align*}
    T_{N,1}(\kl_{1,2},\kl_3) \darw V_{1,\al,\beta},~\text{for $(\al,\beta)\in R_{1,4}\cup R_{1,5}$.}
\end{align*}
where $V_{1,\al,\beta}$ is defined in \eqref{ll-cc-ind} and note that as $N\to\infty$, we can write,
\begin{align}\label{v-2}
    \widetilde{\sigma}^2_{N,1}(\kl_3) = \sigma^2_{N,1}(\kl_3) \to \begin{cases}
        c^3\cdot \kappa(\Cb,\lamb:\kl_3),&~\text{if $(\al,\beta)\in R_{1,4}$,}\\
        c^3\cdot\{1-c^3\}\cdot \kappa(\Cb,\lamb: \kl_3),&~\text{if $(\al,\beta)\in R_{1,5}$.}
    \end{cases}
\end{align}
Next we will show that for $(\al,\beta)\in \bigcup_{i=1}^5R_{1,i}$ the remainder terms $\widehat{\Upsilon}_{N,1}$ and $\Upsilon_{N,1}$ in \eqref{piv-cc-1} converge to zero in probability. Firstly, note that, from \eqref{rem-term} and Lemma \ref{SN-hat-tght}, we can write
\begin{align}
\label{rem-ind-hat}
\widehat{\Upsilon}_{N,1} &= \frac{2(\theta_{2,N} + \delta(\widehat{X}_{N,1,2} -\theta_{2,N}))}{(\theta_{1,N} + \delta(\widehat{X}_{N,1,1} - \theta_{1,N}))^3} \left\{\left(\widehat{X}_{N,1,1} - \theta_{1,N}\right)\right\}\cdot \left(\widehat{X}_{N,1,1} - \theta_{1,N}\right) \notag\\
&\quad {}+\frac{2}{(\theta_{1,N} + \delta(\widehat{X}_{N,1,1} - \theta_{1,N}))^2}\left\{\left(\widehat{X}_{N,1,1} - \theta_{1,N}\right)\right\}\left(\widehat{X}_{N,1,2} - \theta_{2,N}\right)\notag\\
&= \frac{2(\theta_{2,N} + \delta\cdot o_P(1))}{(\theta_{1,N} + \delta(\widehat{X}_{N,1,1} - \theta_{1,N}))^3} \left\{\left(\widehat{X}_{N,1,1} - \theta_{1,N}\right)\right\}\cdot o_P(1) \notag\\
&\quad{}+\frac{2}{(\theta_{1,N} + \delta(\widehat{X}_{N,1,1} - \theta_{1,N}))^2}\left\{\left(\widehat{X}_{N,1,1} - \theta_{1,N}\right)\right\}\cdot o_P(1).
\end{align}
Similarly, using Lemma \ref{SN-hat-tght}, we can say that
\begin{align}\label{rem-ind-pop}
    {\Upsilon}_{N,1} &= \frac{2(\theta_{2,N} + \delta\cdot o_P(1))}{(\theta_{1,N} + \delta({X}_{N,1} - \theta_{1,N}))^3} \left\{({X}_{N,1} - \theta_{1,N})\right\}\cdot o_P(1) \notag\\
    &\quad{}+\frac{2}{(\theta_{1,N} + \delta({X}_{N,1} - \theta_{1,N}))^2}\left\{({X}_{N,1} - \theta_{1,N})\right\}\cdot o_P(1).
\end{align}
A direct computation of the variances of $S_N(\kl_{1,2})$ and $S_N(\kl_3)$ shows that there exists an $M_6\geq 1$ such that for all $N>M_6$, we can write
\begin{align}\label{var-ext}
    &\omega^2_{N}(\kl_{1,2}) = \Var(X_{N,1}) = \frac{\Var(S_N(\kl_{1,2}))}{N^{6-4\beta}}\leq a_{10} \cdot\frac{N^{4-3\beta}}{N^{6-4\beta}} = \frac{a_{10}}{N^{2-\beta}}\quad\text{if $\beta\in [0,1]$, and}\notag\\
    &\omega^2_{N}(\kl_3)=\Var(X_{N,2}) =\frac{\Var(S_N(\kl_{3}))}{N^{6-6\beta}}\leq
    \begin{cases}
    &a_{11}\cdot\frac{N^{4-5\beta}}{N^{6-6\beta}}=\frac{a_{10}}{N^{2-\beta}}\quad\text{if  $\beta\in [0,1/2)$},\\
    &a_{12}\cdot\frac{N^{3-3\beta}}{N^{6-6\beta}}=\frac{a_{11}}{N^{3(1-\beta)}}\quad\text{if $\beta\in [1/2,1)$},\\
    &a_{13}\quad\text{if $\beta=1$,}
    \end{cases}
\end{align}
where, $a_{10} = a_{10}(\beta, \kl_{1,2}, \lamb, \Cb)$ and $a_i = a_i(\beta, \kl_{3}, \lamb, \Cb)$, $i=11,12,13$, are finite positive constants depending on the graph $H$ and the model parameters. When $(\al,\beta)\in R_{1,1}\cup R_{1,2}$, then $q_{N,1}(\al,\beta) = N^{\frac{1}{2}(1-\al)}$. Then from the expression \eqref{rem-ind-hat}, using \eqref{var-rate-ind},\eqref{var-Xhat}, \eqref{var-ext} and Lemma \ref{SN-hat-tght}, we can write
\begin{align*}
    N^{\frac{1}{2}(1-\al)}(\widehat{X}_{N,1,1} - \theta_{1,N})&= \frac{\widetilde{\sigma}_{N,1}(\kl_{1,2})}{N^{-\frac{1}{2}(1-\al)}}\cdot T_{N,1}(\kl_{1,2}) + \frac{\omega_{N}(\kl_{1,2})}{N^{-\frac{1}{2}(1-\al)}} \cdot\frac{(X_{1,N} - \theta_{1,N})}{\omega(\kl_{1,2})} \notag\\
    &=O\left(\frac{1}{N}\right) O_P(1) + \frac{O_P(1)}{N^{\frac{1}{2}(1-\beta+\al)}}= O_P(1).
\end{align*}
From \eqref{d-bdd}, on the event $A_{1,1,N}\cap B_{1,N}$, the denominator terms within $\widehat{\Upsilon}_{N,1}$ and $\Upsilon_N$ remain bounded away from zero. Hence, we can say that $N^{\frac{1}{2}(1-\al)}\widehat{\Upsilon}_{N,1} = o_P(1)$, when $(\al,\beta)\in R_{1,1}\cup R_{1,2}$. Similarly, using \eqref{var-ext} and Lemma \ref{SN-hat-tght}, from \eqref{rem-ind-pop}, we can say that 
\begin{align*}
    &N^{\frac{1}{2}(1-\al)}(X_{N,1} - \theta_{1,N})= \frac{\omega_{N}(\kl_{1,2})}{N^{-\frac{1}{2}(1-\al)}}\frac{X_{N,1} - \theta_{1,N}}{\omega_N(\kl_{1,2})} = \frac{1}{N^{\frac{1}{2}(1-\beta+\al)}}\cdot O_P(1).
\end{align*}
Similarly, from \eqref{d-bdd} we can say that $N^{\frac{1}{2}(1-\al)}{\Upsilon}_{N,1} = o_P(1)$, in the region $R_{1,1}\cup R_{1,2}$. It remains to show that $q_{N,1}(\al,\beta)(\widehat{\Upsilon}_{N,1} + \Upsilon_{N,1}) = o_P(1)$, when $(\al,\beta)\in R_{1,3}\cup R_{1,4}\cup R_{1,5}$. First note that, using \eqref{var-Xhat} and \eqref{var-ext}, for $N>\max\{M_5,M_6\}$, we can say the following 
\begin{align}\label{pop-samp-rat}
    \frac{\omega^2_{N}(\kl_{1,2})}{\widetilde{\sigma}^2_{N,1}(\kl_{3})}\leq\begin{cases}
        a_{14}\cdot N^{(2-3\beta-2\al) - (1-\beta) -\al}&\text{when $(\al,\beta)\in R_{1,3}$,}\\
        a_{15}\cdot N^{-(2-\beta)}&\text{when $(\al,\beta)\in R_{1,4}\cup R_{1,5}$.}
    \end{cases}
\end{align}
Then, from \eqref{rem-ind-hat} and using \eqref{var-ratio-ind}, we can write
\begin{align*}
    \frac{\widehat{\Upsilon}_{N,1}}{\widetilde{\sigma}_{N,1}(\kl_3)} &= \frac{2(\theta_{2,N} + \delta(\widehat{X}_{N,1,2} -\theta_{2,N}))}{(\theta_{1,N} + \delta(\widehat{X}_{N,2,1} - \theta_{1,N}))^3} \left\{\frac{\widetilde{\sigma}_{N,1}(\kl_{1,2})}{\widetilde{\sigma}_{N,1}(\kl_3)}\left(\frac{\widehat{X}_{N,1,1} - \theta_{1,N}}{\widetilde{\sigma}_{N,1}(\kl_{1,2})}\right)\right\}\cdot o_P(1) \notag\\
    &\quad {}+\frac{2}{(\theta_{1,N} + \delta(\widehat{X}_{N,1,1} - \theta_{1,N}))^2} \left\{\frac{\widetilde{\sigma}_{N,1}(\kl_{1,2})}{\widetilde{\sigma}_{N,1}(\kl_3)}\left(\frac{\widehat{X}_{N,1,1} - \theta_{1,N}}{\widetilde{\sigma}_{N,1}(\kl_{1,2})}\right)\right\}\cdot o_P(1)\notag\\
    & = \frac{2(\theta_{2,N} + \delta\cdot o_P(1))}{(\theta_{1,N} + \delta(\widehat{X}_{N,1,1} - \theta_{1,N}))^3} \left\{O\left(\frac{1}{N}\right)\cdot\left(\frac{\widehat{X}_{N,1,1} - \theta_{1,N}}{\widetilde{\sigma}_{N,1}(\kl_{1,2})}\right)\right\}\cdot o_P(1) \notag\\
    &\quad {}+\frac{2}{(\theta_{1,N} + \delta(\widehat{X}_{N,1,1} - \theta_{1,N}))^2} \left\{O\left(\frac{1}{N}\right)\cdot\left(\frac{\widehat{X}_{N,1,1} - \theta_{1,N}}{\widetilde{\sigma}_{N,1}(\kl_{1,2})}\right)\right\}o_P(1).
\end{align*}
In the above expression and on the event $A_{1,1,N}\cap B_{1,N}$ (cf. \eqref{d-bdd}), the denominator terms are bounded away from zero. Using Corollary \ref{cor-clt-all}, Lemma \ref{SN-hat-tght}, \eqref{var-ext} and \eqref{var-Xhat}, we can say that, for $(\al,\beta)\in R_{1,3}\cup R_{1,4}\cup R_{1,5}$,
\begin{align*}
    &\frac{\widehat{X}_{N,1,1} - \theta_{1,N}}{\widetilde{\sigma}_{N,1}(\kl_{1,2})} = T_{N,1}(\kl_{1,2}) + \frac{\omega_{N}(\kl_{1,2})}{\widetilde{\sigma}_{N,1}(\kl_{1,2})}\left(\frac{X_{N,1} - \theta_{1,N}}{\omega_N(\kl_{1,2})}\right) = O_P(1) + \frac{1}{N^{\frac{1}{2}(1-\beta+\al)}}\cdot O_P(1) = O_P(1).
\end{align*}
Hence, we can say that in the region $R_{1,3}\cup R_{1,4}\cup R_{1,5}$, $\frac{\widehat{\Upsilon}_{N,1}}{\widetilde{\sigma}_{N,1}(\kl_3)} = o_P(1)$. From \eqref{v-1} and  \eqref{v-2} we know that $q_{N,1}(\al,\beta)\cdot \widetilde{\sigma}_{N,1}(\kl_3) = O(1)$ when $(\al,\beta)\in R_{1,3}\cup R_{1,4}\cup R_{1,5}$. Thus, in this region $q_{N,1}(\al,\beta)\cdot \widehat{\Upsilon}_{N,1}=o_P(1)$. Similarly, from \eqref{rem-ind-pop}, using \eqref{pop-samp-rat} and Lemma \ref{SN-hat-tght}, we can write
\begin{align*}
    \frac{{\Upsilon}_{N,1}}{\widetilde{\sigma}_{N,1}(\kl_3)} &= \frac{2(\theta_{2,N} + \delta\cdot o_P(1))}{(\theta_{1,N} + \delta({X}_{N,1} - \theta_{1,N}))^3} \left\{\frac{\omega_N(\kl_{1,2})}{\widetilde{\sigma}_{N,1}(\kl_3)}\left(\frac{{X}_{N,1} - \theta_{1,N}}{\omega_N(\kl_{1,2})}\right)\right\}\cdot o_P(1) \notag\\
    &\quad {}+\frac{2}{(\theta_{1,N} + \delta({X}_{N,1} - \theta_{1,N}))^2}\left\{\frac{\omega_N(\kl_{1,2})}{\widetilde{\sigma}_{N,1}(\kl_3)}\left(\frac{{X}_{N,1} - \theta_{1,N}}{\omega_N(\kl_{1,2})}\right)\right\}\cdot o_P(1)\notag\\
    &= \frac{2(\theta_{2,N} + \delta\cdot o_P(1))}{(\theta_{1,N} + \delta({X}_{N,1} - \theta_{1,N}))^3} \left\{O\left(\frac{1}{N}\right)O_P(1)\right\}\cdot o_P(1) \notag\\
    &\quad {}+\frac{2}{(\theta_{1,N} + \delta({X}_{N,1} - \theta_{1,N}))^2}\left\{O\left(\frac{1}{N}\right)O_P(1)\right\}\cdot o_P(1) = o_P(1).
\end{align*}
Therefore, using \eqref{d-bdd}, we can say that $q_{N,1}(\al,\beta)\cdot \Upsilon_{N,1} = o_P(1)$, when $(\al,\beta)\in R_{1,3}\cup R_{1,4}\cup R_{1,5}$. Hence, we show that for $(\al,\beta)\in \bigcup_{i=1}^5 R_{1,i}$, the remainder terms in the expression \eqref{piv-cc-1} converge to zero in probability. Therefore, from the expression \eqref{piv-cc-1}, and using Slutsky's theorem, we can write, 
\begin{align*}
    &\pr\left(q_{N,1}(\al,\beta)\cdot T_{N,1}\left(\kl_{1,2},\kl_3\right)  + \frac{q_{N,1}(\al,\beta)}{2}\left\{\widehat{\Delta}_{N,1} +\Delta_N\right\}\leq t,~ A_{1,1,N}\cap B_{1,N}\right) + o(1)\notag\\
    &= \pr\left(q_{N,1}(\al,\beta)\cdot T_{N,1}\left(\kl_{1,2},\kl_3\right)  \leq t,~ A_{1,1,N}\cap B_{1,N}\right) + o(1)= \pr\left(q_{N,1}(\al,\beta)\cdot T_{N,1}\left(\kl_{1,2},\kl_3\right)  \leq t\right)+o(1).
\end{align*}
This completes the proof of Theorem \eqref{thm-cc-ind}.
\end{proof}

\begin{proof}[Proof of Theorem \ref{thm-cc-ego}]
We consider ego-centric network formation ($l=2$). Recall the definition of $R_{2,j},~j=1,\ldots,6$, from \eqref{rsets-def}. 
Using Lemma \ref{lem-var}, there exists an $M_7$ such that for all $N>M_7$, we can write
\begin{align}\label{var-rate-ego}
    &a_{16}\cdot N^{5-\al-4\beta}\leq \sigma^2_{N,2}(\kl_{1,2}) \leq a_{17}\cdot N^{5-\al-4\beta},~\text{when $(\al,\beta)\in \bigcup_{i=1}^6R_{2,i}$,}~\text{and}\notag\\
    &a_{18}\cdot N^{5-3\al-6\beta}\sigma^2_{N,2}(\kl_{3}) \leq a_{19}\cdot N^{5-3\al-6\beta},~\text{when $(\al,\beta)\in R_{2,1}\cup R_{2,2}$}\\
    &\sigma^2_{N,2}(\kl_{3}) \geq \begin{cases}
        a_{20}\cdot N^{4-2\al - 5\beta},~&\text{when $(\al,\beta)\in R_{2,3}$},\\
        a_{21}\cdot N^{3-2\al - 3\beta},~&\text{when $(\al,\beta)\in R_{2,4}$},\\
        a_{22},~&\text{when $(\al,\beta)\in R_{2,5}\cup R_{2,6}$},
    \end{cases}\notag
\end{align}
where $a_{16},a_{17},a_{18},a_{19},a_{20},a_{21},a_{22}$, are finite positive constants arising from Lemma \ref{lem-var}). Therefore, using the definition of $f_2(p_N:\kl_{1,2})$ and $f_2(p_N:\kl_3)$ (cf. \eqref{fl-wedge-tri-def}), we can write
\begin{align}\label{var-Xhat-1}
    &\widetilde{\sigma}^2_{N,2}(\kl_{1,2}) =  \Var\left(\widehat{X}_{N,2,1}-X_{N,1}\right) = \frac{\Var\left(\widehat{S}_{N,2}(\kl_{1,2}) - f_2(p_N:\kl_{1,2})S_N(\kl_{1,2})\right)}{N^{6-4\beta}f^2_2(p_N:\kl_{1,2})} \notag\\
    &=\frac{\Var\left(T_{N,2}(\kl_{1,2})\right)}{N^{6-2\al-4\beta}c^2(1+(1-c/N^{\al})(c/N^{\al}))}=\frac{\sigma^2_{N,2}(\kl_{1,2})}{N^{6-2\al-4\beta}c^2(1+(1-c/N^{\al})(c/N^{\al}))}\quad\text{and,}\notag\\
    &~\widetilde{\sigma}^2_{N,2}(\kl_{3})=\Var\left(\widehat{X}_{N,2,2}-X_{N,2}\right) = \frac{\Var\left(\widehat{S}_{N,2}(\kl_{3}) - f_2(p_N:\kl_{3})S_N(\kl_{3})\right)}{N^{6-6\beta}f^2_2(p_N:\kl_{3})} \notag\\
    &= \frac{\Var\left(T_{N,2}(\kl_{3})\right)}{N^{6-4\al-6\beta}c^4(3+2(c/N^{\al}))}=\frac{\sigma^2_{N,1}(\kl_{3})}{c^4\cdot N^{6-4\al-6\beta}(3+2(c/N^{\al}))}.
\end{align}
Using \eqref{var-rate-ego} and \eqref{var-Xhat-1}, we can claim that for all $N>M_7$,  
\begin{align}\label{var-ratio-ego}
   & \frac{a_{16}}{a_{19}}\leq \frac{\widetilde{\sigma}^2_{N,2}(\kl_{1,2})}{\widetilde{\sigma}^2_{N,2}(\kl_{3})} \leq \frac{a_{17}}{a_{18}},\quad \text{if $(\al,\beta)\in R_{2,1}\cup R_{2,2}$,~~  and,}\notag\\
    &\frac{\widetilde{\sigma}^2_{N,2}(\kl_{1,2})}{\widetilde{\sigma}^2_{N,2}(\kl_{3})} \leq \begin{cases}
        \frac{a_{17}}{a_{20}}N^{-\left(1-\al -\beta\right)},~&\text{when $(\al,\beta) \in {R}_{2,3}$}\\
        \frac{a_{17}}{a_{21}}N^{2(1-\frac{\al}{2} -\frac{3}{2}\beta)},~&\text{when $(\al,\beta) \in {R}_{2,4}$}\\
        \frac{a_{17}}{a_{22}}\frac{1}{N^{1-\al}},~&\text{when $(\al,\beta) \in {R}_{2,5}$.}\\
        \frac{a_{17}}{a_{22}}\frac{1}{N},&\text{when $(\al,\beta) \in {R}_{2,6}$}.
    \end{cases}
\end{align}
This implies that the limiting behavior of $T_{N,2}(\kl_{1,2},\kl_3)$ (see \eqref{piv-w-t}) is jointly influenced by the wedge and triangle count based terms $(\widehat{X}_{N,2,1} - X_{N,1})$ and $(\widehat{X}_{N,2,2} - X_{N,2})$, when $(\al,\beta)\in R_{2,1}\cup R_{2,2}$, and is influence only by triangle based term $(\widehat{X}_{N,2,1}-X_{N,2})$ when $(\al,\beta)\in R_{2,3}\cup R_{2,4}\cup R_{2,5}\cup R_{2,6}$. Thus, we separately analyze the limiting behavior of $T_{N,2}(\kl_{1,2},\kl_3)$ in these regions. Due to \eqref{piv-cc-1}, firstly we obtain the asymptotic distribution of 
\begin{align}\label{piv-TN-ego}
   q_{N,2}(\al,\beta)\cdot T_{N,2}(\kl_{1,2},\kl_3) = q_{N,2}\left[-\frac{\theta_{2,N}}{(\theta_{1,N})^2}\left(\widehat{X}_{N,2,1} - X_{N,1}\right)+\frac{1}{\theta_{1,N}}\left(\widehat{X}_{N,2,2} - X_{N,2}\right)\right]. 
\end{align}
In case of $(\al,\beta)\in R_{2,1}\cup R_{2,2}$, using Proposition \ref{prop-joint}, and using the fact \eqref{theta-def}, we can write,
 \begin{align*}
N^{\zeta(\al,\beta)}
\Var\!\bigl(T_{N,2}(\kl_{1,2},\kl_3)\bigr)
\;\to\;\rho^2_{2,j}(\kl_{1,2},\kl_3),
\qquad (\al,\beta)\in R_{2,j},\ j=1,2,
\end{align*}
where
\begin{align*}
    \zeta(\al,\beta) =\begin{cases}
        1,&\quad\text{if $(\al,\beta)\in R_{2,1}$,}\\
        (1-\al),&\quad\text{if $(\al,\beta)\in R_{2,2}$,}
    \end{cases}
\end{align*}
and $\rho^2_{2,j}(\kl_{1,2},\kl_3)$ defined in \eqref{var-rho-2-1} and \eqref{var-rho-2-2} respectively. In addition, using Proposition \eqref{prop-joint} we can say that 
$$
N^{\zeta(\al,\beta)/2}\, T_{N,2}(\kl_{1,2},\kl_3) \darw N\left(0,\rho^2_{2,j}(\kl_{1,2},\kl_3)\right),
\qquad \text{if}~~(\al,\beta)\in R_{2,j},\ j=1,2.
$$
Thus, when $(\al,\beta)\in R_{2,1}\cup R_{2,2}$, we will have $q_{N,2}(\al,\beta) = N^{\frac{1}{2}(1-\al)}$, for all $N\geq 1$, for $\al\in [0,1)$. Now consider the case where $(\al,\beta)\in R_{2,3}\cup R_{2,4}\cup R_{2,5}\cup R_{2,6}$. From \eqref{piv-TN-ego}, we can write
\begin{align*}
    \frac{T_{N,2}(\kl_{1,2},\kl_3)}{\widetilde{\sigma}_{N,1}(\kl_3)} &= -\frac{\theta_{2,N}}{(\theta_{1,N})^2}\frac{\widetilde{\sigma}_{N,1}(\kl_{1,2})}{\widetilde{\sigma}_{N,2}(\kl_3)}\frac{\left(\widehat{X}_{N,2,1} - X_{N,1}\right)}{\widetilde{\sigma}_{N,1}(\kl_{1,2})}+\frac{1}{\theta_{1,N}}\frac{1}{\widetilde{\sigma}_{N,2}(\kl_3)}\left(\widehat{X}_{N,2,2} - X_{N,2}\right)\notag\\
    &= -\frac{\theta_{2,N}}{(\theta_{1,N})^2}\frac{\widetilde{\sigma}_{N,2}(\kl_{1,2})}{\widetilde{\sigma}_{N,2}(\kl_3)}\cdot T_{N,2}(\kl_{1,2})+\frac{1}{\theta_{1,N}}\cdot T_{N,2}(\kl_3).
\end{align*}
Now using \eqref{var-ratio-ego} and Corollary \ref{cor-clt-all}, we can write 
\begin{align}\label{piv-TN-imp-1}
    \frac{T_{N,2}(\kl_{1,2},\kl_3)}{\widetilde{\sigma}_{N,2}(\kl_3)} &= -\frac{\theta_{2,N}}{(\theta_{1,N})^2}\cdot O\left(\frac{1}{N}\right)\cdot O_P(1)+\frac{1}{\theta_{1,N}}\cdot T_{N,2}(\kl_3)\notag\\
    \Rightarrow ~q_{N,2}(\al,\beta)\cdot T_{N,2}(\kl_{1,2},\kl_3) &= -\frac{\theta_{2,N}}{(\theta_{2,N})^2}\cdot O\left(\frac{1}{N}\right)\cdot O_P(1)\cdot \{q_{N,2}(\al,\beta)\cdot\widetilde{\sigma}_{N,2}(\kl_3)\}\notag\\
    &\qquad{}\qquad{}+\frac{q_{N,2}(\al,\beta)\cdot\widetilde{\sigma}_{N,2}(\kl_3)}{\theta_{1,N}}\cdot T_{N,2}(\kl_3).
\end{align}
From Corollary \eqref{cor-clt-all}, $T_{N,2}(\kl_3)\darw N(0,1)$. Using Lemma \ref{lem-var} and for
\begin{align*}
    q_{N,2}(\al,\beta) = \begin{cases}
        N^{(4-2\al-5\beta)/2},&~\text{if $(\al,\beta)\in R_{2,3}$}\\
        N^{(3-2\al -3\beta)/2},&~\text{if $(\al,\beta)\in R_{2,4}$,}
    \end{cases}
\end{align*}
we can say that
\begin{align*}
    N^{(4-2\al-5\beta)/2}\cdot \widetilde{\sigma}_{N,2}(\kl_3) = \frac{\sigma_{N,2}(\kl_3)}{N^{(2-2\al-\beta)/2}}  = O(1),~\text{and} ~N^{(3-2\al-3\beta)/2}\cdot \widetilde{\sigma}_{N,2}(\kl_3) = \frac{\sigma_{N,2}(\kl_3)}{N^{(3-2\al-3\beta)/2}}  = O(1).
\end{align*}
Thus, using above arguments for $(\al,\beta)\in R_{2,3}\cup R_{2,4}$ and using Slutsky's Theorem, we can say that
\begin{align*}
     &N^{(4-2\al-5\beta)/2}\cdot T_{N,1}(\kl_{1,2},\kl_3)\darw N(0,~\rho^2_{2,3}(\kl_3)),~\text{and}\\
    &N^{3-2\al-3\beta)/2}\cdot T_{N,1}(\kl_{1,2},\kl_3)\darw N(0,~\rho^2_{2,4}(\kl_3))~\text{as $N\to\infty$,}
\end{align*}
where, 
\begin{align}\label{v-11}
    \frac{N^{(4-2\al-5\beta)/2}\cdot\widetilde{\sigma}_{N,1}(\kl_3)}{\theta_{1,N}} \to \rho_{1,3}(\kl_3)\quad\text{and}\quad\frac{N^{3-2\al-3\beta)/2}\cdot\widetilde{\sigma}_{N,1}(\kl_3)}{\theta_{1,N}} \to \rho_{1,3}(\kl_3),
\end{align}
where $\rho_{2,3}(\kl_3)$ and $\rho_{2,4}(\kl_3)$ are defined in \eqref{var-rho-2-3-2-4}. In case of $(\al,\beta)\in R_{2,5}\cup R_{2,6}$, using the Theorem \ref{thm-2-main-temp}, we can write
\begin{align*}
    &T_{N,2}(\kl_3) \darw \frac{V_{2,\kl_3} - 3c^2\cdot \kappa(\Cb,\lamb:\kl_3)}{[3c^2\cdot \kappa(\Cb,\lamb:\kl_3)]^{1/2}}\qquad{}\text{for $(\al,\beta)\in R_{2,5}$, and,}\notag\\
    &T_{N,2}(\kl_3) \darw \frac{\left(1-(3c^2-2c^3)\right)\cdot U_{1,2,\kl_3} - (3c^2-2p^3)\cdot U_{2,2,\kl_3}}{[(3c^2-2c^3)\cdot\{1-c^3\}\cdot \kappa(\Cb,\lamb: \kl_3)]^{1/2}},\qquad{}\text{when $(\al,\beta)\in R_{2,6}$,}
\end{align*}
where $V_{2,\kl_3}$, $U_{1,2,\kl_3}$ and $U_{2,2,\kl_3}$ is defined in Theorem \ref{thm-2-main-temp}, and $\kappa(\Cb,\lamb,\kl_3)$ is defined in \eqref{kap-def}. Note that in this case, $q_{N,2}(\al,\beta) =1$ for $(\al,\beta)\in R_{2,5}\cup R_{2,6}$. Thus, from \eqref{piv-TN-imp} and Slutsky's Theorem, we can write, 
\begin{align*}
    T_{N,2}(\kl_{1,2},\kl_3) \darw V_{2,\al,\beta},~\text{for $(\al,\beta)\in R_{2,5}\cup R_{2,6}$.}
\end{align*}
where $V_{2,\al,\beta}$ is defined in \eqref{ll-cc-ego} and note that as $N\to\infty$, we can write,
\begin{align}\label{v-21}
    \widetilde{\sigma}^2_{N,2}(\kl_3) = \sigma^2_{N,2}(\kl_3) \to \begin{cases}
        3c^2\cdot \kappa(\Cb,\lamb:\kl_3),&~\text{if $(\al,\beta)\in R_{2,5}$,}\\
        (3c^2-2c^3)\cdot\{1-(3c^2-2c^3)\}\cdot \kappa(\Cb,\lamb: \kl_3),&~\text{if $(\al,\beta)\in R_{2,6}$.}
    \end{cases}
\end{align}
Next, we will show that for $(\al,\beta)\in \bigcup_{i=1}^6R_{2,i}$ the remainder term $\widehat{\Upsilon}_{N,2}$ in \eqref{piv-cc-1} converges to zero in probability. First, note that from \eqref{rem-term} and Lemma \ref{SN-hat-tght}, we can write
\begin{align}\label{rem-ego-hat}
    &\widehat{\Upsilon}_{N,2} = \frac{2(\theta_{2,N} + \delta(\widehat{X}_{N,2,2} -\theta_{2,N}))}{(\theta_{1,N} + \delta(\widehat{X}_{N,2,1} - \theta_{1,N}))^3} \left\{\left(\widehat{X}_{N,2,1} - \theta_{1,N}\right)\right\}\cdot \left(\widehat{X}_{N,2,1} - \theta_{1,N}\right) \notag\\
    &\quad {}+\frac{2}{(\theta_{1,N} + \delta(\widehat{X}_{N,2,1} - \theta_{1,N}))^2}\left\{\left(\widehat{X}_{N,2,1} - \theta_{1,N}\right)\right\}\left(\widehat{X}_{N,2,2} - \theta_{2,N}\right)\notag\\
    &= \frac{2(\theta_{2,N} + \delta\cdot o_P(1))}{(\theta_{1,N} + \delta(\widehat{X}_{N,2,1} - \theta_{1,N}))^3} \left\{\left(\widehat{X}_{N,2,1} - \theta_{1,N}\right)\right\}\cdot o_P(1) \notag\\
    &\quad {}+\frac{2}{(\theta_{1,N} + \delta(\widehat{X}_{N,2,1} - \theta_{1,N}))^2}\left\{\left(\widehat{X}_{N,2,1} - \theta_{1,N}\right)\right\}\cdot o_P(1).
\end{align}
When $(\al,\beta)\in R_{2,1}\cup R_{2,2}$, then $q_{N,2}(\al,\beta) = N^{\frac{1}{2}(1-\al)}$. Then from the expression \eqref{rem-ego-hat}, using \eqref{var-rate-ego},\eqref{var-Xhat-1}, \eqref{var-ext} and Lemma \ref{SN-hat-tght}, we can write
\begin{align*}
    N^{\frac{1}{2}(1-\al)}(\widehat{X}_{N,2,1} - \theta_{1,N})&= \frac{\widetilde{\sigma}_{N,2}(\kl_{1,2})}{N^{-\frac{1}{2}(1-\al)}}T_{N,2}(\kl_{1,2}) + \frac{\omega_{N}(\kl_{1,2})}{N^{-\frac{1}{2}(1-\al)}} \frac{X_{1,N} - \theta_{1,N}}{\omega(\kl_{1,2})} =O\left(\frac{1}{N}\right) O_P(1) + \frac{O_P(1)}{N^{\frac{1}{2}(1-\beta+\al)}}\notag\\
    &= O_P(1).
\end{align*}
From \eqref{d-bdd}, on the event $A_{1,2,N}\cap B_{1,N}$, the denominator term with in $\widehat{\Upsilon}_{N,2}$ remains bounded away from zero. Hence, we can say that $N^{\frac{1}{2}(1-\al)}\widehat{\Upsilon}_{N,2} = o_P(1)$, when $(\al,\beta)\in R_{2,1}\cup R_{2,2}$.
It remains to show that $q_{N,2}(\al,\beta)\widehat{\Upsilon}_{N,2} = o_P(1)$, when $(\al,\beta)\in R_{2,3}\cup R_{2,4}\cup R_{2,5}\cup R_{2,6}$. First note that, using \eqref{var-Xhat} and \eqref{var-ext}, for $N>\max\{M_6,M_7\}$, we can say the following 
\begin{align}\label{pop-samp-rat-1}
    \frac{\omega^2_{N}(\kl_{1,2})}{\widetilde{\sigma}^2_{N,2}(\kl_{3})}\leq\begin{cases}
        a_{23}\cdot N^{2(1-\al-2\beta)},~&\text{when $(\al,\beta)\in R_{2,3}$},\\
        a_{24}\cdot N^{1-2\al-2\beta},~&\text{when $(\al,\beta)\in R_{2,4}$},\\
        a_{25}\cdot N^{-(2-\beta)},~&\text{when $(\al,\beta)\in R_{2,5}\cup R_{2,6}$}.
    \end{cases}
\end{align}
Then, from \eqref{rem-ego-hat} and using \eqref{var-ratio-ego} and \eqref{pop-samp-rat-1}, we can write
\begin{align*}
    &\frac{\widehat{\Upsilon}_{N,2}}{\widetilde{\sigma}_{N,2}(\kl_3)} = \frac{2(\theta_{2,N} + \delta o_P(1))}{(\theta_{1,N} + \delta(\widehat{X}_{N,2,1} - \theta_{1,N}))^3} \left\{\frac{\widetilde{\sigma}_{N,1}(\kl_{1,2})}{\widetilde{\sigma}_{N,2}(\kl_3)}\left(\frac{\widehat{X}_{N,2,1} - \theta_{1,N}}{\widetilde{\sigma}_{N,2}(\kl_{1,2})}\right)\right\}\cdot o_P(1) \notag\\
    &\quad{}+\frac{2}{(\theta_{1,N} + \delta(\widehat{X}_{N,2,1} - \theta_{1,N}))^2} \left\{\frac{\widetilde{\sigma}_{N,2}(\kl_{1,2})}{\widetilde{\sigma}_{N,2}(\kl_3)}\left(\frac{\widehat{X}_{N,2,1} - \theta_{1,N}}{\widetilde{\sigma}_{N,2}(\kl_{1,2})}\right)\right\}\cdot o_P(1)\notag\\
    & = \frac{2(\theta_{2,N} + \delta\cdot o_P(1))}{(\theta_{1,N} + \delta(\widehat{X}_{N,2,1} - \theta_{1,N}))^3} \left\{O\left(\frac{1}{N}\right)\cdot\left(\frac{\widehat{X}_{N,2,1} - \theta_{1,N}}{\widetilde{\sigma}_{N,2}(\kl_{1,2})}\right)\right\}\cdot o_P(1) \notag\\
    &\quad{}+\frac{2}{(\theta_{1,N} + \delta(\widehat{X}_{N,2,1} - \theta_{1,N}))^2} \left\{O\left(\frac{1}{N}\right)\cdot\left(\frac{\widehat{X}_{N,2,1} - \theta_{1,N}}{\widetilde{\sigma}_{N,2}(\kl_{1,2})}\right)\right\}o_P(1).
\end{align*}
In the above expression and on the event $A_{1,2,N}\cap B_{1,N}$,  the denominator terms are bounded away from zero. Using Corollary \ref{cor-clt-all}, Lemma \ref{SN-hat-tght}, \eqref{var-ext} and \eqref{var-Xhat}, we can say that, for $(\al,\beta)\in R_{2,3}\cup R_{2,4}\cup R_{2,5}\cup R_{2,6}$,
\begin{align*}
    &\frac{\widehat{X}_{N,2,1} - \theta_{1,N}}{\widetilde{\sigma}_{N,2}(\kl_{1,2})} = T_{N,2}(\kl_{1,2}) + \frac{\omega_{N}(\kl_{1,2})}{\widetilde{\sigma}_{N,2}(\kl_{1,2})}\left(\frac{X_{N,1} - \theta_{1,N}}{\omega_N(\kl_{1,2})}\right) = O_P(1) + \frac{1}{N^{\frac{1}{2}(1-\beta+\al)}}\cdot O_P(1) = O_P(1).
\end{align*}
Hence, we can say that in the region $R_{2,3}\cup R_{2,4}\cup R_{2,5}\cup R_{2,6}$, $\frac{\widehat{\Upsilon}_{N,2}}{\widetilde{\sigma}_{N,2}(\kl_3)} = o_P(1)$. From \eqref{v-11} and  \eqref{v-21} we know that $q_{N,2}(\al,\beta)\cdot \widetilde{\sigma}_{N,2}(\kl_3) = O(1)$, when $(\al,\beta)\in R_{2,3}\cup R_{2,4}\cup R_{2,5}\cup R_{2,6}$. Thus, in this region $q_{N,2}(\al,\beta)\cdot \widehat{\Upsilon}_{N,2}=o_P(1)$.
Hence, for $(\al,\beta)\in \bigcup_{i=1}^6R_{2,i}$, the remainder term in the expression \eqref{piv-cc-1} converges to zero in probability. Therefore, from the expression \eqref{piv-cc-1}, using Slutsky's theorem, we can write, 
\begin{align*}
    &\pr\left(q_{N,2}(\al,\beta)\cdot T_{N,2}\left(\kl_{1,2},\kl_3\right)  + \frac{q_{N,2}(\al,\beta)}{2}\left\{\widehat{\Upsilon}_{N,2} +\Upsilon_N\right\}\leq t,~ A_{1,2,N}\cap B_{1,N}\right) + o(1)\notag\\
    &= \pr\left(q_{N,2}(\al,\beta)\cdot T_{N,2}\left(\kl_{1,2},\kl_3\right)  \leq t,~ A_{1,2,N}\cap B_{1,N}\right) + o(1)= \pr\left(q_{N,2}(\al,\beta)\cdot T_{N,2}\left(\kl_{1,2},\kl_3\right)  \leq t\right) + o(1).
\end{align*}
This completes the proof of Theorem \ref{thm-cc-ego}.
\end{proof}

\setappendix{C}
\refstepcounter{subsection}
\subsection*{Appendix C: Results from other sources}
\label{sec:app-C}

For the sake of completeness, we state certain results from other sources that have been widely used in our proofs. The first result (Theorem \ref{thm-ruc}) provides a non-asymptotic bound on the $d_1$ distance (see \eqref{d1-def}) between a scaled sum of dependent random variables and a standard normal random variable. The proof of Theorem \ref{thm-ruc} is available in \cite{bkr-89}.

\begin{theorem}[Theorem 6.32. in \cite{janson-rucinski-randomgraphs}]
\label{thm-ruc}
    Suppose that $S = \sum_{\al\in A}X_{\al}$, where $\{X_{\al}\}_{\al\in A}$ is a family of random variables with dependency graph $L$ and assume that, $\E(X_{\al}) =0$, for all $\al\in A$. Let $\sigma^2 = \Var(S)$, and assume that $0<\sigma^2<\infty$. Then, for some universal constant $C$ 
    \begin{align*}
        d_1\left(\sigma^{-1}S,Z\right) \leq \frac{C}{\sigma^3}\sum_{\al\in A}\sum_{\beta,\gamma\in N(\al)}\left[\E|X_{\al}X_{\beta}X_{\gamma}| +\E|X_{\al}X_{\beta}|\E|X_{\gamma}|\right],
    \end{align*}
where, $d_1$ is defined in \eqref{d1-def}, $N(\al)$ is the closed dependency neighborhood of $\al$ for any $\al\in A$, and $Z\sim N(0,1)$.    
\end{theorem}
\vspace{2mm}

The next result (Theorem \ref{thm-ruc-1}) provides a non-asymptotic bound on the total variation distance between a sum of dependent Bernoulli random variables and another Poisson random variable. 

\begin{theorem}[Theorem 6.23 form \cite{janson-rucinski-randomgraphs}]\label{thm-ruc-1}
    Suppose that $X = \sum_{\al\in A}\mathbf{1}_{\al}$, where the $\mathbf{1}_{\al}$ are random indicator variable with a dependency graph $L$. Then, with $\pi_{\al} = \E(\mathbf{1}_{\al})$ and $\lambda = \E(X) = \sum_{\al\in A}\pi_{\al}$ (and with the summation over ordered pairs $(\al,\beta)$),
    \begin{align*}
        d_{{TV}}(X,U) \leq \min(\lambda^{-1},1)\left(\Var(X) - \E(X) + 2\sum_{\al\in A}\sum_{\beta\in N(\al)}\pi_{\al}\pi_{\al} + 2 \sum_{\al\in A}\pi_{\al}^2\right),
    \end{align*}
where, $d_{TV}$ denotes the total variation distance (cf. \eqref{TV-dist}), $N(\al)$ denotes the closed dependency neighborhood of any $\al\in A$ and $U\sim$ Poisson $(\lambda)$.    
\end{theorem}
\vspace{2mm}

Theorem \ref{lem-mult-SB} provides a construction of the multivariate size biased coupling, which is used to obtain a multivariate Poisson approximation result for a random vector, where the components are sums of dependent Bernoulli random variables, and the components are themselves dependent. We state the result for a two-dimensional random vector. The proof is available in \cite{nualart2025}.

\begin{theorem}[Lemma 2.4 in \cite{nualart2025} for $d=2$]\label{lem-mult-SB}
Let $\mathbf{S} = {(S_1,S_2)}^\primet$, where, 
$$
S_1 = \sum_{j=1}^{n_1}X_j^1,\quad\text{and}\quad S_2 = \sum_{j=1}^{n_2} X_j^2,
$$
where, we assume that $\{X^{i}_j:1\leq j\leq n_i\}$, $i=1,2$, are Bernoulli random variables with $\pr(X_j^i = 1) =p_{i,j}\in(0,1)$. For all $i=1,2$ and $l=1,\ldots, n_i$, consider the following random vectors 
\begin{align*}
&\mathbf{S}^{1,(1,l)} = \left(\left(X_j^{1,(1,l)}\right)_{j=1,j\neq l}^{n_1}\right)^\primet\quad \text{and}\quad
 \mathbf{S}^{2,(2,l)} = \left(\left(X_j^{1,(2,l)}\right)_{j=1}^{n_1}, \left(X_j^{2,(2,l)}\right)_{j=1,j\neq l}^{n_2}\right)^\primet,
\end{align*}
which are defined on the same probability space as the random vectors 
    \begin{align*}
        &\mathbf{S}^{1,l} = \left((X_j^1)_{j=1,j\neq l}^{n_1}\right)^\primet\quad\text{and}\quad
     \mathbf{S}^{2,l} = \left(\left(X_j^1\right)_{j=1}^{n_1}, (X_j^2)_{j=1,j\neq l}^{n_2}\right)^\primet.
    \end{align*}
Assume that  
    \begin{align*}
        \mathbf{S}^{1,(1,l)} \eqd [\mathbf{S}^{1,l} \, |\, X^1_{l} = 1]\quad\text{and}\quad\mathbf{S}^{2,(2,l)} \eqd [\mathbf{S}^{2,l} \, |\, X^2_{l} = 1].
    \end{align*}
Let $I_1$ and $I_2$ be two independent random variables, which are also independent of all other random variables, such that for all $i=1,2$, and $j=1,\ldots,n_i$,
    \begin{align*}
        \pr(I_i = j) = \frac{p_{i,j}}{\lambda_i},~\text{where} ~\lambda_i=\E(S_i) = \sum_{j=1}^{n_i}p_{i,j}.
    \end{align*}
Define the two-dimensional triangular array, $\widetilde{\mathbf{S}}= \left(\widetilde{\mathbf{S}}^1,\widetilde{\mathbf{S}}^2\right)$, where 
    \begin{align*}
        &\widetilde{\mathbf{S}}^1 = \left(\sum_{j=1,j\neq I_1}^{n_1}X_j^{1,(1,I_1)}+1\right)\quad\text{and}\quad\widetilde{\mathbf{S}}^2 = \left(\sum_{j=1}^{n_1}X_j^{1,(2,I_2)}, \sum_{j=1,j\neq I_2}^{n_2}X_j^{2,(2,I_2)}+1\right)^\primet.
    \end{align*}
    Then, $\widetilde{\mathbf{S}}$ is a size-biased coupling of $\mathbf{S}$.
\end{theorem}
\vspace{2mm}

The next result (Theorem \ref{thm-mult-SB}) provides a bound on the Wasserstein distance $d_W$ (see \eqref{wass-dist}) between a bivariate random vector of the form $\mathbf{S}={(S_1,S_2)}^\primet$, considered in Theorem A.3 above, and a bivariate random vector $\mathbf{Z} = {(Z_1,Z_2)}^\primet$, where $Z_1$ and $Z_2$ are independent and each $Z_i$, $i=1,2$, is a Poisson random variable. A proof is available in \cite{nualart2025} (see Theorem 2.6) or  \cite{pianoforte2023} (see Theorem 1.1).

\begin{theorem}[Theorem 2.6 in \cite{nualart2025} for $d=2$]\label{thm-mult-SB}
    Let $\mathbf{S} = {(S_1,S_2)}^\primet$ be a two-dimensional random vector taking values in $\mathbb{N}^2_{0}$, with $\E(S_i) =\kappa_i >0$, for $i=1,2$. Let $\widetilde{\mathbf{S}}$ be a size-biased coupling of $\mathbf{S}$ (see Theorem \ref{assump-3} above). Then,
    \begin{equation*}
        \begin{aligned}
            d_{{W}}\left(\mathbf{S}, \mathbf{Z}\right) \leq \min\{1,\kappa_1\}\cdot\E\left|\widetilde{S}_1^1-1-S_1\right| + \min\{1,\kappa_2\}\cdot\E\left|\widetilde{S}_2^2-1-S_2\right| + 2\kappa_2\cdot\E\left|\widetilde{S}_1^2-S_1\right|, 
        \end{aligned}
    \end{equation*}
    where, $\mathbf{Z} = {(Z_1,Z_2)}^\primet$, $d_W$ is defined in \eqref{wass-dist}, and $Z_1$ and $Z_2$ are independent Poisson random variables with means $\kappa_1$ and $\kappa_2$ respectively.
\end{theorem}

\setappendix{D}
\refstepcounter{subsection}
\subsection*{Appendix D: Supplementary lemmas}
\label{sec:app-D}

\begin{lemma}\label{lem-1}
  Let $H$ be a fixed, undirected, simple and connected graph with $V(H)=\{1,\ldots,R\}$ and  edge set $E(H)\subseteq {[V(H)]}^2$, where $R\geq 2$, is a fixed positive integer. For each $t\in \{0,1,\ldots,R\}$, define the collection of ordered pairs of vectors
\begin{align}
\Mc_t = \left\{(\sbb_1,\sbb_2):\sbb_1,\sbb_2\in {[N]}_R,~ |A(\sbb_1)\cap A(\sbb_2)| = t\right\}.
\label{Etset}
\end{align}
Then, for $t\in \{2,\ldots, R\}$, there exists an $M\in \mathbb{N}$ such that  for all $N\geq M$,
\begin{align}
& L_t(\beta,H,\lamb,\Cb)\leq \frac{1}{N^{2R-t-\beta(2T - m(t:H))}}\sum_{(\sbb_1,\sbb_2)\in \Mc_t} \E\left(Y_{N}(\sbb_1:H)\cdot Y_{N}(\sbb_2:H)\right)\leq U_t(\beta,H,\lamb,\Cb).
\label{res-lem-1} 
\end{align}
where, $L_t(\beta,H,\lamb,\Cb)$ and $U_t(\beta,H,\lamb,\Cb)$ are two positive constants depending on $\beta$, $H$ and the model parameters $\lamb$ and $\Cb$.
\end{lemma}
\begin{proof}[Proof of Lemma \ref{lem-1}]
Since $|A(\sbb_1)\cap A(\sbb_2)| = t \in \{2,\ldots,R\}$, thus there exists $t$ common elements between $\sbb_1$ and $\sbb_2$. Fix any $\jbb=(j_1,\ldots,j_t)^\primet\in [R]_{<,t}$ and $\kbb=(k_1,\ldots,k_t)^\primet\in [R]_{<,t}$ (see \eqref{n-sub-r-2}), where $\jbb$ and $\kbb$ denote the positions at which the common elements appear within $\sbb_1$ and $\sbb_2$ respectively.
Consider the set $[t]=\{1,\ldots,t\}$, and define
\begin{align}
\mathcal{S}_t = \left\{\xi:[t]\mapsto (\xi(1),\ldots,\xi(t))\in {[t]}_t~\big|~\text{$\xi$ is a permutation of $\{1,\ldots,t\}$}\right\},
\label{scal-t}
\end{align}
as the collection of all possible permutations of $\{1,\ldots,t\}$. Note $|\mathcal{S}_t|=t!$. For each such $\jbb,\kbb\in {[R]}_{<,t}$ and any permutation $\xi\in \mathcal{S}_t$, define the following subsets of $\Mc_t$ (see \eqref{Etset}),
\begin{equation}
\label{Etset-1}
\begin{aligned}
\Mc_t(\jbb, \kbb) & = \left\{(\sbb_1,\sbb_2): \sbb_1,\sbb_2\in \Mc_t,\ A\left({(\sbb_1)}_{\jbb}\right) = A\left({(\sbb_2)}_{\kbb}\right)\right\},\quad\text{and}\\
\Mc_t(\jbb, \kbb, \xi) & = \left\{(\sbb_1,\sbb_2): \sbb_1,\sbb_2\in \Mc_t,\ s_{1,j_1}=s_{2,k_{\xi(1)}},\ldots,s_{1,j_t}=s_{2,k_{\xi(t)}}\right\}.
\end{aligned}
\end{equation}
For a fixed $t\in \{2,\ldots,R\}$, the sets $\Mc_t(\jbb,\kbb)$ and $\Mc_t(\jbb',\kbb')$ are disjoint if, $\jbb\neq \jbb'$ or $\kbb\neq \kbb'$.
Also note that, 
\begin{align*}
   \Mc_t= \bigcup_{\jbb,\kbb\in [R]_{<,t}}\Mc_t(\jbb,\kbb)\quad\text{and}\quad \Mc_t(\jbb,\kbb) = \bigcup_{\xi\in\mathcal{S}_t} \Mc_t(\jbb,\kbb,\xi).
\end{align*}
Fix any $(\sbb_1,\sbb_2)\in \Mc_t(\jbb,\kbb,\xi)$ and write, $\sbb_i = \left(s_{i,1},\ldots, s_{i,R}\right)^\primet\in [N]_R$, for $i=1,2$. Let $A(\sbb_1)\cap A(\sbb_2) = \{a_1,\ldots,a_t\}$, where $a_i$'s are distinct and $a_i\in [N]$ for each $i=1,\ldots,t$. For any $\jbb = (j_1,\ldots,j_t)^\primet\in {[R]}_{<,t}$ and $\kbb = (k_1,\ldots,k_t)^\primet\in {[R]}_{<,t}$, we write $A(\jbb) = \{j_1,\ldots,j_t\}$, and $A(\kbb) = \{k_1,\ldots,k_t\}$. For simplicity of notation, we will write the product, 
$$
\prod_{j\in {[R]}\setminus \{j_1,\ldots,j_t\}} \equiv \prod_{j\in {[R]}\setminus A(\jbb)},
$$
and similarly for the case of the product over $k\in {[R]}\setminus \{k_1,\ldots,k_t\}$.
Next, we use the sets $A(\jbb)$ and $A(\kbb)$ as vertex sets. 
From each of these sets, we construct induced subgraphs of $H$ that contain exactly $t$ vertices. 
We denote these induced subgraphs by $B(\jbb,t)$ and $B(\kbb,t)$
corresponding to vertex set $A(\jbb)$ and $A(\kbb)$, respectively.
With these definitions in place, l.h.s. of \eqref{res-lem-1} has the following expression.
\begin{align}\label{E-1}
    &\sum_{(\sbb_1,\sbb_2)\in \Mc_t}\E\left(Y_{N}(\sbb_1:H)Y_N(\sbb_2:H)\right)= \sum_{\jbb,\kbb\in [R]_{<,t}}\sum_{\xi\in \mathcal{S}_t}\sum_{(\sbb_1,\sbb_2)\in \Mc_t(\jbb,\kbb,\xi)}\E\left(\prod_{\{i,j\}\in E(H)}Y_{s_{1,i},s_{1,j}}\prod_{\{i,j\}\in E(H)}Y_{s_{2,i},s_{2,j}}\right)\notag\\
    &= \sum_{\jbb,\kbb\in [R]_{<,t}}\sum_{\xi\in \mathcal{S}_t}\sum_{(\sbb_1,\sbb_2)\in \Mc_t(\jbb,\kbb,\xi)}\E\left(\prod_{\{i,j\}\in E(B(\jbb,t))}Y_{s_{1,i},s_{1,j}}\prod_{\{i,j\}\in E(H)\setminus E(B(\jbb,t))}Y_{s_{1,i},s_{1,j}}\right.\notag\\
    &\left.\qquad{}\qquad{}\times\prod_{\{i,j\}\in E(B(\kbb,t))}Y_{s_{2,i},s_{2,j}}\prod_{\{i,j\}\in E(H)\setminus E(B(\kbb,t))}Y_{s_{2,i},s_{2,j}}\right)\notag\\
    &= \sum_{\jbb,\kbb\in [R]_{<,t}}\sum_{\xi\in \mathcal{S}_t}\sum_{(\sbb_1,\sbb_2)\in \Mc_t(\jbb,\kbb,\xi)}\E\left(\prod_{\{i,j\}\in E(B(\jbb,t))}Y_{s_{1,i},s_{1,j}}\prod_{\{i,j\}\in E(B(\kbb,t))}Y_{s_{2,i},s_{2,j}}\right)\notag\\
    &\qquad{}\qquad{}\times\E\left(\prod_{\{i,j\}\in E(H)\setminus E(B(\jbb,t))}Y_{s_{1,i},s_{1,j}}\right)\E\left(\prod_{\{i,j\}\in E(H)\setminus E(B(\kbb,t))}Y_{s_{2,i},s_{2,j}}\right).
\end{align}
In expression \eqref{E-1}, the product $\left(\prod_{{i,j}\in E(H)\setminus E(B(\jbb,t))}Y_{s_{1,i},s_{1,j}}\right)$ consists of random variables from the set $A_1 = {Y_{s_{1,i},s_{1,j}}: {i,j}\in E(H)\setminus E(B(\jbb,t))}$, and the product $\left(\prod_{{i,j}\in E(H)\setminus E(B(\kbb,t))}Y_{s_{2,i},s_{2,j}}\right)$ consists of random variables from the set $A_2={Y_{s_{2,i},s_{2,j}}: {i,j}\in E(H)\setminus E(B(\kbb,t))}$. For any edge ${i,j} \in E(H) \setminus E(B(\mathbf{j},t))$, two cases arise: either $i \in V(B(\mathbf{j},t))$ and $j \notin V(B(\mathbf{j},t))$, or both $i,j \notin V(B(\mathbf{j},t))$. Similarly, for the set $A_2$, for any edge ${a,b} \in E(H) \setminus E(B(\mathbf{k},t))$, two cases arise: either $a \in V(B(\mathbf{k},t))$ and $b \notin V(B(\mathbf{k},t))$, or both $a,b \notin V(B(\mathbf{k},t))$. Then, for all $(\sbb_1,\sbb_2)\in \Mc(\jbb,\kbb,\xi)$, the index sets of the random variables from $A_1$ and $A_2$ satisfy the following relation: $|{s_{1,i},s_{1,j}}\cap{s_{2,a},s_{2,b}}|$ is either $0$ or $1$, depending on the positions of the vertices $i,j$ and $a,b$. But we know that the random variables $Y_{s_{1,i},s_{1,j}}$ and $Y_{s_{2,a},s_{2,b}}$ are dependent if and only if $|{s_{1,i},s_{1,j}}\cap{s_{2,a},s_{2,b}}|=2$.
Next, we define two sets for fixed $t\in \{1,\ldots,R\}$,
\begin{equation}\label{def-Jt-Mt}
\begin{aligned}
    &J_t = \{(\jbb,\kbb):~\jbb,\kbb\in [R]_{<,t},~B(\jbb,t)\simeq B(\kbb,t)\}~\text{and},\\
    &M_t = \{(\jbb,\kbb):~\jbb,\kbb\in [R]_{<,t},~|E(B(\jbb,t))|=|E(B(\kbb,t))|=m(t:H)\}.
\end{aligned}
\end{equation}
where $m(t:H)$ defined in \eqref{mtH-def}.
Therefore, when $(\jbb,\kbb)\in J_t$, then for a fixed $\xi\in \mathcal{S}_t$, 
\begin{align*}
    \prod_{\{i,j\}\in E(B(\jbb,t))}Y_{s_{1,i},s_{1,j}}\prod_{\{i,j\}\in E(B(\kbb,t))}Y_{s_{2,i},s_{2,j}} = \prod_{\{i,j\}\in E(B(\jbb,t))}Y_{s_{1,i},s_{1,j}}.
\end{align*}
Therefore, the sum in \eqref{E-1} can be written as follows:
\begin{align}\label{E-2}
    &\sum_{(\sbb_1,\sbb_2)\in \Mc_t}\E\left(Y_N(\sbb_1:H)Y_N(\sbb_2:H)\right)\notag\\
    &= \sum_{\jbb,\kbb\in [R]_{<,t}}\sum_{\xi\in \mathcal{S}_t}\sum_{(\sbb_1,\sbb_2)\in \Mc_t(\jbb,\kbb,\xi)}\E\left(\prod_{\{i,j\}\in E(B(\jbb,t))}Y_{s_{1,i},s_{1,j}}\prod_{\{i,j\}\in E(B(\kbb,t))}Y_{s_{2,i},s_{2,j}}\right)\notag\\
    &\qquad{}\qquad{}\times\E\left(\prod_{\{i,j\}\in E(H)\setminus E(B(\jbb,t))}Y_{s_{1,i},s_{1,j}}\right)\E\left(\prod_{\{i,j\}\in E(H)\setminus E(B(\kbb,t))}Y_{s_{2,i},s_{2,j}}\right)\notag\\
    &= \sum_{(\jbb,\kbb)\in J_t}\sum_{\xi\in \mathcal{S}_t}\sum_{(\sbb_1,\sbb_2)\in \Mc_t(\jbb,\kbb,\xi)}\E\left(\prod_{\{i,j\}\in E(B(\jbb,t))}Y_{s_{1,i},s_{1,j}}\prod_{\{i,j\}\in E(B(\kbb,t))}Y_{s_{2,i},s_{2,j}}\right)\notag\\
    &\qquad{}\qquad{}\times\E\left(\prod_{\{i,j\}\in E(H)\setminus E(B(\jbb,t))}Y_{s_{1,i},s_{1,j}}\right)\E\left(\prod_{\{i,j\}\in E(H)\setminus E(B(\kbb,t))}Y_{s_{2,i},s_{2,j}}\right)\notag\\
    &\qquad{}+ \sum_{(\jbb,\kbb)\in J_t^c}\sum_{\xi\in \mathcal{S}_t}\sum_{(\sbb_1,\sbb_2)\in \Mc_t(\jbb,\kbb,\xi)}\E\left(\prod_{\{i,j\}\in E(B(\jbb,t))}Y_{s_{1,i},s_{1,j}}\prod_{\{i,j\}\in E(B(\kbb,t))}Y_{s_{2,i},s_{2,j}}\right)\notag\\
    &\qquad{}\qquad{}\times\E\left(\prod_{\{i,j\}\in E(H)\setminus E(B(\jbb,t))}Y_{s_{1,i},s_{1,j}}\right)\E\left(\prod_{\{i,j\}\in E(H)\setminus E(B(\kbb,t))}Y_{s_{2,i},s_{2,j}}\right)\notag\\
    &= \sum_{(\jbb,\kbb)\in J_t\cap M_t}\sum_{\xi\in \mathcal{S}_t}\sum_{(\sbb_1,\sbb_2)\in \Mc_t(\jbb,\kbb,\xi)}\E\left(\prod_{\{i,j\}\in E(B(\jbb,t))}Y_{s_{1,i},s_{1,j}}\prod_{\{i,j\}\in E(B(\kbb,t))}Y_{s_{2,i},s_{2,j}}\right)\notag\\
    &\qquad{}\qquad{}\times\E\left(\prod_{\{i,j\}\in E(H)\setminus E(B(\jbb,t))}Y_{s_{1,i},s_{1,j}}\right)\E\left(\prod_{\{i,j\}\in E(H)\setminus E(B(\kbb,t))}Y_{s_{2,i},s_{2,j}}\right)\notag\\
    &\quad{}+ \sum_{(\jbb,\kbb)\in J_t\cap M_t^c}\sum_{\xi\in \mathcal{S}_t}\sum_{(\sbb_1,\sbb_2)\in \Mc_t(\jbb,\kbb,\xi)}\E\left(\prod_{\{i,j\}\in E(B(\jbb,t))}Y_{s_{1,i},s_{1,j}}\prod_{\{i,j\}\in E(B(\kbb,t))}Y_{s_{2,i},s_{2,j}}\right)\notag\\
    &\qquad{}\qquad{}\times\E\left(\prod_{\{i,j\}\in E(H)\setminus E(B(\jbb,t))}Y_{s_{1,i},s_{1,j}}\right)\E\left(\prod_{\{i,j\}\in E(H)\setminus E(B(\kbb,t))}Y_{s_{2,i},s_{2,j}}\right)\notag\\
    &\qquad{}+ \sum_{(\jbb,\kbb)\in J_t^c}\sum_{\xi\in \mathcal{S}_t}\sum_{(\sbb_1,\sbb_2)\in \Mc_t(\jbb,\kbb,\xi)}\E\left(\prod_{\{i,j\}\in E(B(\jbb,t))}Y_{s_{1,i},s_{1,j}}\prod_{\{i,j\}\in E(B(\kbb,t))}Y_{s_{2,i},s_{2,j}}\right)\notag\\
    &\qquad{}\qquad{}\times\E\left(\prod_{\{i,j\}\in E(H)\setminus E(B(\jbb,t))}Y_{s_{1,i},s_{1,j}}\right)\E\left(\prod_{\{i,j\}\in E(H)\setminus E(B(\kbb,t))}Y_{s_{2,i},s_{2,j}}\right)\notag\\
    &= V_{1,N} + V_{2,N} + V_{3,N},~\text{(say).}
\end{align}
First, we analyze the $V_{1,N}$ form \ref{E-2}, that is, 
\begin{align}\label{V-1}
    &V_{1,N} = \sum_{(\jbb,\kbb)\in J_t\cap M_t}\sum_{\xi\in \mathcal{S}_t}\sum_{(\sbb_1,\sbb_2)\in \Mc_t(\jbb,\kbb,\xi)}\E\left(\prod_{\{i,j\}\in E(B(\jbb,t))}Y_{s_{1,i},s_{1,j}}\prod_{\{i,j\}\in E(B(\kbb,t))}Y_{s_{2,i},s_{2,j}}\right)\notag\\
    &\qquad{}\qquad{}\times\E\left(\prod_{\{i,j\}\in E(H)\setminus E(B(\jbb,t))}Y_{s_{1,i},s_{1,j}}\right)\E\left(\prod_{\{i,j\}\in E(H)\setminus E(B(\kbb,t))}Y_{s_{2,i},s_{2,j}}\right)\notag\\
    &= \sum_{(\jbb,\kbb)\in J_t\cap M_t}\sum_{\xi\in \mathcal{S}_t}\sum_{(\sbb_1,\sbb_2)\in \Mc_t(\jbb,\kbb,\xi)}\E\left(\prod_{\{i,j\}\in E(B(\jbb,t))}Y_{s_{1,i},s_{1,j}}\prod_{\{i,j\}\in E(H)\setminus E(B(\jbb,t))}Y_{s_{1,i},s_{1,j}}\right)\notag\\
    &\quad\qquad{}\times\E\left(\prod_{\{i,j\}\in E(H)\setminus E(B(\kbb,t))}Y_{s_{2,i},s_{2,j}}\right)\notag\\
    &= \sum_{(\jbb,\kbb)\in J_t\cap M_t}\sum_{\xi\in \mathcal{S}_t}\sum_{(\sbb_1,\sbb_2)\in \Mc_t(\jbb,\kbb,\xi)}\prod_{\{i,j\}\in E(B(\jbb,t))}\pi_{N,\al_{s_{1,i}},\al_{s_{1,j}}}\prod_{\{i,j\}\in E(H)\setminus E(B(\jbb,t))}\pi_{N,\al_{s_{1,i}},\al_{s_{1,j}}}\notag\\
    &\qquad{}\prod_{\{i,j\}\in E(H)\setminus E(B(\kbb,t))}\pi_{N,\al_{s_{2,i}},\al_{s_{2,j}}}\notag\\
    &= \sum_{(\jbb,\kbb)\in J_t\cap M_t}\sum_{\xi\in \mathcal{S}_t}\sum_{(\sbb_1,\sbb_2)\in \Mc_t(\jbb,\kbb,\xi)}\prod_{\{i,j\}\in E(H)}\pi_{N,\al_{s_{1,i}},\al_{s_{1,j}}}\prod_{\{i,j\}\in E(H)\setminus E(B(\kbb,t))}\pi_{N,\al_{s_{2,i}},\al_{s_{2,j}}}\notag\\
    &= \sum_{(\jbb,\kbb)\in J_t\cap M_t}\sum_{\xi\in \mathcal{S}_t}\sum_{(\sbb_1,\sbb_2)\in \Mc_t(\jbb,\kbb,\xi)}\sum_{\ub,\vb\in [K]^R}\prod_{\{i,j\}\in E(H)}\pi_{N,u_i,u_j}\prod_{\{i,j\}\in E(H)\setminus E(B(\kbb,t))}\pi_{N,v_i,v_j}\notag\\
    &\qquad{}\times  \left\{\prod_{r=1}^t\mathbf{1}(\al_{a_r}=u_{j_r})\prod_{j\in [R]\setminus A(\jbb)}\mathbf{1}(\al_{s_{1,j}}=u_j)\prod_{r=1}^t\mathbf{1}(\al_{a_r}=v_{k_{\xi(r)}})\prod_{k\in [R]\setminus A(\kbb)}\mathbf{1}(\al_{s_{2,k}}= v_k)\right\}\notag\\
    &= \sum_{(\jbb,\kbb)\in J_t\cap M_t}\sum_{\xi\in \mathcal{S}_t}\sum_{(\sbb_1,\sbb_2)\in \Mc_t(\jbb,\kbb,\xi)}\sum_{\ub,\vb\in [K]^R}\prod_{\{i,j\}\in E(H)}\pi_{N,u_i,u_j}\prod_{\{i,j\}\in E(H)\setminus E(B(\kbb,t))}\pi_{N,v_i,v_j}\prod_{r=1}^t\mathbf{1}(u_{j_r}=v_{k_{\xi(r)}})\notag\\
    &\qquad{}\times  \left\{\prod_{j\in [R]}\mathbf{1}(\al_{s_{1,j}}=u_j)\prod_{k\in [R]\setminus A(\kbb)}\mathbf{1}(\al_{s_{2,k}}= v_k)\right\}\notag\\
    &= \sum_{(\jbb,\kbb)\in J_t\cap M_t}\sum_{\xi\in \mathcal{S}_t}\sum_{\ub,\vb\in [K]^R}\prod_{\{i,j\}\in E(H)}\pi_{N,u_i,u_j}\prod_{\{i,j\}\in E(H)\setminus E(B(\kbb,t))}\pi_{N,v_i,v_j}\prod_{r=1}^t\mathbf{1}(u_{j_r}=v_{k_{\xi(r)}})\notag\\
    &\qquad{}\times  \sum_{(\sbb_1,\sbb_2)\in \Mc_t(\jbb,\kbb,\xi)}\left\{\prod_{j\in [R]}\mathbf{1}(\al_{s_{1,j}}=u_j)\prod_{k\in [R]\setminus A(\kbb)}\mathbf{1}(\al_{s_{2,k}}= v_k)\right\}\notag\\
    &= N^{2R-t}\sum_{(\jbb,\kbb)\in J_t\cap M_t}\sum_{\xi\in \mathcal{S}_t}\sum_{\ub,\vb\in [K]^R}N^{-\beta(2T -|E(B(\kbb,t))|)}\prod_{\{i,j\}\in E(H)}N^{\beta}\pi_{N,u_i,u_j}\prod_{\{i,j\}\in E(H)\setminus E(B(\kbb,t))}N^{\beta}\pi_{N,v_i,v_j}\notag\\
    &\qquad{}\times  \prod_{r=1}^t\mathbf{1}(u_{j_r}=v_{k_{\xi(r)}})\left\{\prod_{j\in [R]}\la_{N,u_j}\prod_{k\in [R]\setminus A(\kbb)}\la_{N,v_k} + O\left(\frac{1}{N}\right)\right\}\notag\\
    &= N^{2R-t-\beta(2T -m(t:H))}\sum_{(\jbb,\kbb)\in J_t\cap M_t}\sum_{\xi\in \mathcal{S}_t}\sum_{\ub,\vb\in [K]^R}\prod_{\{i,j\}\in E(H)}N^{\beta}\pi_{N,u_i,u_j}\prod_{\{i,j\}\in E(H)\setminus E(B(\kbb,t))}N^{\beta}\pi_{N,v_i,v_j}\notag\\
    &\qquad{}\times  \prod_{r=1}^t\mathbf{1}(u_{j_r}=v_{k_{\xi(r)}})\left\{\prod_{j\in [R]}\la_{N,u_j}\prod_{k\in [R]\setminus A(\kbb)}\la_{N,v_k} + O\left(\frac{1}{N}\right)\right\}.
\end{align}
Using Assumptions \ref{assump-2} and \ref{assump-3}, for all $\varepsilon>0$ there exists an $M_{1,t}$ such that for all $N>M_{1,t}$,
\begin{align*}
    &L_{t,1}(\beta,H,\lamb,\Cb) -\frac{\varepsilon}{3}
    < \sum_{(\jbb,\kbb)\in J_t\cap M_t}\sum_{\xi\in \mathcal{S}_t}\sum_{\ub,\vb\in [K]^R}\prod_{\{i,j\}\in E(H)}N^{\beta}\pi_{N,u_i,u_j}\prod_{\{i,j\}\in E(H)\setminus E(B(\kbb,t))}N^{\beta}\pi_{N,v_i,v_j}\notag\\
    &\qquad{}\times  \prod_{r=1}^t\mathbf{1}(u_{j_r}=v_{k_{\xi(r)}})\left\{\prod_{j\in [R]}\la_{N,u_j}\prod_{k\in [R]\setminus A(\kbb)}\la_{N,v_k} + O\left(\frac{1}{N}\right)\right\} < L_{t,1}(\beta,H,\lamb,\Cb) +\frac{\varepsilon}{3}.
\end{align*}
where, the constant $L_{t,1}(\beta,H,\lamb,\Cb)$ has the following form,
\begin{align*}
   L_{t,1}(\beta,H,\lamb,\Cb)= &\sum_{(\jbb,\kbb)\in J_t\cap M_t}\sum_{\xi\in \mathcal{S}_t}\sum_{\ub,\vb\in [K]^R}\prod_{\{i,j\}\in E(H)}c_{u_i,u_j}\prod_{\{i,j\}\in E(H)\setminus E(B(\kbb,t))}c_{v_i,v_j}\notag\\
    &\qquad{}\times  \prod_{r=1}^t\mathbf{1}(u_{j_r}=v_{k_{\xi(r)}})\cdot\prod_{j\in [R]}\la_{u_j}\prod_{k\in [R]\setminus A(\kbb)}\la_{v_k}.
\end{align*}
Thus, from \eqref{V-1}, we can write
\begin{align}\label{V-1-imp}
    L_{t,1}(\beta,H,\lamb,\Cb) -\frac{\varepsilon}{3} <\frac{V_{1,N}}{N^{2R-t-\beta(2T -m(t:H))}} < L_{t,1}(\beta,H,\lamb,\Cb) +\frac{\varepsilon}{3}
\end{align}
Next, from \eqref{E-2}, we analyze the explicit expression of $V_{2,N}$, that is,
\begin{align}\label{V-2-1}
    &V_{2,N} = \sum_{(\jbb,\kbb)\in J_t\cap M_t^c}\sum_{\xi\in \mathcal{S}_t}\sum_{(\sbb_1,\sbb_2)\in \Mc_t(\jbb,\kbb,\xi)}\E\left(\prod_{\{i,j\}\in E(B(\jbb,t))}Y_{s_{1,i},s_{1,j}}\prod_{\{i,j\}\in E(B(\kbb,t))}Y_{s_{2,i},s_{2,j}}\right)\notag\\
    &\qquad{}\qquad{}\times\E\left(\prod_{\{i,j\}\in E(H)\setminus E(B(\jbb,t))}Y_{s_{1,i},s_{1,j}}\right)\E\left(\prod_{\{i,j\}\in E(H)\setminus E(B(\kbb,t))}Y_{s_{2,i},s_{2,j}}\right)\notag\\
    &= \sum_{(\jbb,\kbb)\in J_t\cap M_t^c}\sum_{\xi\in \mathcal{S}_t}\sum_{(\sbb_1,\sbb_2)\in \Mc_t(\jbb,\kbb,\xi)}\E\left(\prod_{\{i,j\}\in E(H)}Y_{s_{1,i},s_{1,j}}\right)\E\left(\prod_{\{i,j\}\in E(H)\setminus E(B(\kbb,t))}Y_{s_{2,i},s_{2,j}}\right)\notag\\
    &= \sum_{(\jbb,\kbb)\in J_t\cap M_t^c}\sum_{\xi\in \mathcal{S}_t}\sum_{\ub,\vb\in [K]^R}\prod_{\{i,j\}\in E(H)}\pi_{N,u_i,u_j}\prod_{\{i,j\}\in E(H)\setminus E(B(\kbb,t))}\pi_{N,v_i,v_j}\prod_{r=1}^t\mathbf{1}(u_{j_r}=v_{k_{\xi(r)}})\notag\\
    &\qquad{}\times  \sum_{(\sbb_1,\sbb_2)\in \Mc_t(\jbb,\kbb,\xi)}\left\{\prod_{j\in [R]}\mathbf{1}(\al_{s_{1,j}}=u_j)\prod_{k\in [R]\setminus A(\kbb)}\mathbf{1}(\al_{s_{2,k}}= v_k)\right\}\notag\\
    &= \sum_{(\jbb,\kbb)\in J_t\cap M_t^c}\sum_{\xi\in \mathcal{S}_t}\sum_{\ub,\vb\in [K]^R}N^{2R-t-\beta(2T -|E(B(\kbb,t))|)}\prod_{\{i,j\}\in E(H)}N^{\beta}\pi_{N,u_i,u_j}\prod_{\{i,j\}\in E(H)\setminus E(B(\kbb,t))}N^{\beta}\pi_{N,v_i,v_j}\notag\\
    &\qquad{}\times \prod_{r=1}^t\mathbf{1}(u_{j_r}=v_{k_{\xi(r)}}) \left\{\prod_{j\in [R]}\la_{N,u_j}\prod_{k\in [R]\setminus A(\kbb)}\la_{N,v_k} + O\left(\frac{1}{N}\right)\right\}.
\end{align}
 For all $(\jbb,\kbb) \in J_t\cap M_t^c$, $|E(B(\jbb, t))|,|E(B(\kbb,t))| < m(t:H)$. Hence, using Assumption \ref{assump-2} and \ref{assump-3}, for all $\varepsilon>0$ there exists an $M_{2,t}$ such that for all $N>M_{2,t}$,
 \begin{align*}
     -\frac{\varepsilon}{3}<&\sum_{(\jbb,\kbb)\in J_t\cap M_t^c}\sum_{\xi\in \mathcal{S}_t}\sum_{\ub,\vb\in [K]^R}\frac{1}{N^{\beta(m(t:H) - |E(B(\kbb,t))|)}}\prod_{\{i,j\}\in E(H)}N^{\beta}\pi_{N,u_i,u_j}\prod_{\{i,j\}\in E(H)\setminus E(B(\kbb,t))}N^{\beta}\pi_{N,v_i,v_j}\notag\\
    &\qquad{}\times \prod_{r=1}^t\mathbf{1}(u_{j_r}=v_{k_{\xi(r)}}) \left\{\prod_{j\in [R]}\la_{N,u_j}\prod_{k\in [R]\setminus A(\kbb)}\la_{N,v_k} + O\left(\frac{1}{N}\right)\right\}<\frac{\varepsilon}{3}.
\end{align*}
Thus, from the expression \eqref{V-2-1}, we can write
\begin{align}\label{v-2-imp}
    -\frac{\varepsilon}{3} < \frac{V_{2,N}}{N^{2R-t-\beta(2T -m(t:H))}} < \frac{\varepsilon}{3}.
\end{align}
Now consider the term $V_{3,N}$ in \eqref{E-2}. In this case, $(\jbb,\kbb)\in J_t^c$. Here we can decompose $J_t^c$ into $j_t^\cap M_t$ and $J_t^c\cap M^c_t$. We focus only on the part $J_t^c\cap M_t$, because the effect of the part $J_t^c\cap M_t^c$ in the summation is negligible trivially. Since $B(\jbb,t)$ and $B(\kbb,t)$ are not isomorphic for all $\jbb,\kbb\in J_t^c\cap M_t$, then there exists at least one edge, say $\{r_1,s_1\}$ from $E(B(\jbb,t))$ and $\{r_2,s_2\}$ from $E(B(\kbb,t))$, for which if we define a bijection $\phi:V(B(\jbb,t))\mapsto V(B(\kbb,t))$, then $\{\phi(r_1),\phi(s_1)\}\neq \{r_2,s_2\}$; that is, some subgraph of $B(\kbb,t)$ is isomorphic to some subgraph of $B(\jbb,t)$. That implies, for $(\jbb,\kbb)\in J_t^c\cap M_t$
\begin{align*}
    &\E\left(\prod_{\{i,j\}\in E(B(\jbb,t))}Y_{s_{1,i},s_{1,j}}\prod_{\{i,j\}\in E(B(\kbb,t))}Y_{s_{2,i},s_{2,j}}\right)\notag\\
    &= \E\left(Y_{s_{1,r},s_{1,s}}\right)\E\left(Y_{s_{2,r_2},s_{2,s_2}}\right)\E\left(\prod_{\{i,j\}\in E(B(\jbb,t))\setminus \{r_1,s_1\}}Y_{s_{1,i},s_{1,j}}\prod_{\{i,j\}\in E(B(\kbb,t))\setminus\{r_2,s_2\}}Y_{s_{2,i},s_{2,j}}\right).
\end{align*}
Without loss of generality, if we assume that after removing the edges $\{r_1,s_1\}$ from the graph $B(\jbb,t)$ and $\{r_2,s_2\}$ from the graph $B(\kbb,t)$, the remaining graph becomes isomorphic, and $|E(B(\jbb,t))| =|E(B(\kbb,t))|= m(t:H)$ then we can write the following, 
\begin{align}\label{V-3}
&\sum_{(\jbb,\kbb)\in J_t^c\cap M_t}\sum_{\xi\in \mathcal{S}_t}\sum_{(\sbb_1,\sbb_2)\in \Mc_t(\jbb,\kbb,\xi)}\E\left(\prod_{\{i,j\}\in E(B(\jbb,t))}Y_{s_{1,i},s_{1,j}}\prod_{\{i,j\}\in E(B(\kbb,t))}Y_{s_{2,i},s_{2,j}}\right)\notag\\
    &\qquad{}\qquad{}\times\E\left(\prod_{\{i,j\}\in E(H)\setminus E(B(\jbb,t))}Y_{s_{1,i},s_{1,j}}\right)\E\left(\prod_{\{i,j\}\in E(H)\setminus E(B(\kbb,t))}Y_{s_{2,i},s_{2,j}}\right)\notag\\
    &=\sum_{(\jbb,\kbb)\in J_t^c\cap M_t}\sum_{\xi\in \mathcal{S}_t}\sum_{(\sbb_1,\sbb_2)\in \Mc_t(\jbb,\kbb,\xi)}\E\left(Y_{s_{1,r},s_{1,s}}\right)\E\left(Y_{s_{2,r_2},s_{2,s_2}}\right)\E\left(\prod_{\{i,j\}\in E(B(\jbb,t))\setminus \{r_1,s_1\}}Y_{s_{1,i},s_{1,j}}\right)\notag\\
    &\qquad{}\qquad{}\times\E\left(\prod_{\{i,j\}\in E(H)\setminus E(B(\jbb,t))}Y_{s_{1,i},s_{1,j}}\right)\E\left(\prod_{\{i,j\}\in E(H)\setminus E(B(\kbb,t))}Y_{s_{2,i},s_{2,j}}\right)\notag\\
    &= \sum_{(\jbb,\kbb)\in J_t^c\cap M_t}\sum_{\xi\in \mathcal{S}_t}\sum_{(\sbb_1,\sbb_2)\in \Mc_t(\jbb,\kbb,\xi)}\pi_{\al_{s_{1,r}},\al_{s_{1,s}}}\cdot \pi_{\al_{s_{2,r_2}},\al_{s_{2,s_2}}}\cdot\prod_{\{i,j\}\in E(B(\jbb,t))\setminus \{r_1,s_1\}}\pi_{\al_{s_{1,i}},\al_{s_{1,j}}}\notag\\
    &\qquad{}\qquad{}\times\prod_{\{i,j\}\in E(H)\setminus E(B(\jbb,t))}\pi_{\al_{s_{1,i}},\al_{s_{1,j}}}\prod_{\{i,j\}\in E(H)\setminus E(B(\kbb,t))}\pi_{\al_{s_{2,i}},\al_{s_{2,j}}}\notag\\
    &= N^{-2\beta}N^{-\beta(m(t:H) - 1)}N^{-2\beta(T-m(t:H))}\sum_{(\jbb,\kbb)\in J_t^c\cap M_t}\sum_{\xi\in \mathcal{S}_t}\sum_{(\sbb_1,\sbb_2)\in \Mc_t(\jbb,\kbb,\xi)}N^{2\beta}\pi_{\al_{s_{1,r}},\al_{s_{1,s}}}\cdot \pi_{\al_{s_{2,r_2}},\al_{s_{2,s_2}}}\notag\\
    &\qquad{}\qquad{}\times\prod_{\{i,j\}\in E(B(\jbb,t))\setminus \{r_1,s_1\}}N^{\beta}\pi_{\al_{s_{1,i}},\al_{s_{1,j}}}\prod_{\{i,j\}\in E(H)\setminus E(B(\jbb,t))}N^{\beta}\pi_{\al_{s_{1,i}},\al_{s_{1,j}}}\prod_{\{i,j\}\in E(H)\setminus E(B(\kbb,t))}N^{\beta}\pi_{\al_{s_{2,i}},\al_{s_{2,j}}}\notag\\
    &\leq \left(N^{\beta}\max_{i,j\in [K]}\pi_{N,i,j}\right)^{2T}\sum_{(\jbb,\kbb)\in J_t^c\cap M_t}\sum_{\xi\in \mathcal{S}_t}N^{2R-t-\beta(2T-m(t:H) +1)}.
\end{align}
Using \ref{assump-2} and \ref{assump-3}, for all $\varepsilon_t>0$ there exists an $M_3$ such that for all $N>M_3$,
\begin{align*}
   -\frac{\varepsilon_t}{3}< \left(N^{\beta}\max_{i,j\in [K]}\pi_{N,i,j}\right)^{2T}\sum_{(\jbb,\kbb)\in J_t^c\cap M_t}\sum_{\xi\in \mathcal{S}_t}\frac{N^{2R-t-\beta(2T-m(t:H) +1)}}{N^{2R-t-\beta(2T-m(t:H))}} <\frac{\varepsilon_t}{3}.
\end{align*}
The upper bound is applied in \eqref{V-3}, is applied over all the terms of $V_{3,N}$ and there are finite number of terms in $V_{3,N}$. Thus, from the expression in \eqref{V-3}, we can write,
\begin{align}\label{V-3-imp}
    -\frac{\varepsilon_t}{3}<0<\frac{V_{3,N}}{N^{2R-t-\beta(2T- m(t:H))}} <\frac{\varepsilon_t}{3}.
\end{align}
Therefore, combining  \eqref{V-1-imp}, \eqref{v-2-imp} and \eqref{V-3-imp}, and taking $N>\max\{M_1,M_2,M_3\}$, using Lemma \ref{lem-1} we can write
\begin{align*}
    L_{t,1}(\beta,H,\lamb,\Cb) -\varepsilon_t <\frac{\sum_{(\sbb_1,\sbb_2)\in \Mc_t} \E\left(Y_{N}(\sbb_1:H)\cdot Y_{N}(\sbb_2:H)\right)}{N^{2R-t-\beta(2T - m(t:H))}}< L_{t,1}(\beta,H,\lamb,\Cb) +\varepsilon_t.
\end{align*}
This completes the proof.
\end{proof}

\begin{lemma}\label{maxEdge-lem}
Let $H(\sbb_i)$ denote a graph with vertex set $V(H(\sbb_i)) = \{s_{i,1},\ldots,s_{i,R}\}$ and edge set $E(H(\sbb_i)) = \bigl\{\{s_{i,a},s_{i,b}\} : \{a,b\} \in E(H)\bigr\}$, for $i=1,2,3$. Assume that $|V(H(\sbb_1)) \cap V(H(\sbb_2))| = k_1$ and $|V(H(\sbb_3)) \cap (V(H(\sbb_1)) \cup V(H(\sbb_2)))| =k_2$, where $k_1,k_2 \in [R]$. Then
$$
|E(H(\sbb_1)) \cap E(H(\sbb_2))| \le m(k_1:H)\quad\text{and}\quad
|E(H(\sbb_3)) \cap (E(H(\sbb_1)) \cup E(H(\sbb_2)))| \le m(k_2:H),
$$
where $m(\cdot:H)$ is defined in~\eqref{mtH-def}.
\end{lemma}

\begin{proof}[Proof of Lemma \ref{maxEdge-lem}]
Let $B_1,\ldots, B_{\binom{R}{k_1}}$ denote all possible $k_1$ vertex subsets from $V(H)$. Let $H_r$ denote the corresponding induced subgraph formed by the vertices in $B_r$, for $r=1,\ldots,\binom{R}{k_1}$. 
There must exist some $r_0\in \{1,\ldots,\binom{R}{k_1}\}$, such that $H_{r_0}$ is isomorphic to $H(\sbb_1)$. Then $$|E(H(\sbb_1)) \cap E(H(\sbb_2))| \leq |E(H(\sbb_1))| =|E(H_{r_0})|\leq \max_{r}|E(H_r)|\leq m(k_1:H).$$
Similarly, There must exist some $r_1\in \{1,\ldots,\binom{R}{k_2}\}$, such that $H_{r_1}$ is isomorphic to $H(\sbb_3)$. Then
$$
|E(H(\sbb_3)) \cap (E(H(\sbb_1)) \cup E(H(\sbb_2)))| \leq |E(H(\sbb_3))| =|E(H_{r_1})|\leq \max_{r}|E(H_r)|\leq m(k_2:H),
$$
and it will lead to the second bound.
\end{proof}

Before stating the next result, we provide some needed definitions. Let $D$ be a vertex cover of a graph $H$, and assume $u\in D$. An edge $\{u,v\}\in E(H)$ is called a {\it private edge} of $u$ if $v\not\in D$ and $u$ is the only vertex in $D$ that covers the edge $\{u,v\}$. If $u\in D$ and $u$ does not have a private edge, then $u$ is redundant and $D\setminus\{u\}$ will be a vertex cover of $H$. A vertex cover $D$ of $H$ is called {\it minimal} if no proper subset of $D$ is a vertex cover of $H$. Any minimum vertex cover is minimal, but the converse is not true. Lemma \ref{VC-lem-0} deals with minimal vertex covers, but the statement is also true for any minimum vertex cover.
 
\begin{lemma}\label{VC-lem-0}
Let $H$ be a graph with a vertex set $V(H)$. A vertex cover $D\subseteq V(H)$ of graph $H$ is minimal if and only if each vertex $v\in D$ has a private edge in $H$.
\end{lemma}
\begin{proof}[Proof of Lemma \ref{VC-lem-0}]
    Suppose $D$ is an minimal vertex cover of $H$. If possible, assume that there exists a $v\in D$, such that $v$ has no private edge. This implies, all edges which are incident on $v$ must be covered by some other vertex is $D\setminus \{v\}$. Any other edge, which was not incident on $v$ was already covered by $D\setminus \{v\}$. Thus $D\setminus \{v\}$ is a vertex cover of $H$, and it is proper subset of $D$, which contradicts the minimality of $D$. Hence, there can not exist any such $v\in D$ and every vertex in $D$ must have a private edge.

    Conversely, assume that every $v\in D$ has a private edge. We will show that $D$ is minimal vertex cover of $H$. Suppose $D$ is not a minimal vertex cover of $H$, then there exists $v\in D$ such that $D_1= D\setminus \{v\}$ is still a vertex cover. But since $v\in D$ has a private edge $\{u,v\}$ with $u\notin D$. Therefore the edge $\{u,v\}$ will not be covered by $D_1$, contradicting the fact that $D_1$ is a vertex cover of $H$. Hence, no vertex can be removed, so $D$ must be minimal. This completes the proof. 
\end{proof}

\begin{lemma}\label{VC-lem-1}
Let $H$ be a connected graph with a minimum vertex cover size $\tau(H)$. Consider any $\sbb_1,\sbb_2,\sbb_3\in \nsr$. Let $H(\sbb_1)$, $H(\sbb_2)$ and $H(\sbb_3)$ be three isomorphic copies of $H$, with vertex sets $V(H(\sbb_1))=A(\sbb_1)$, $V(H(\sbb_2))=A(\sbb_2)$ and $V(H(\sbb_3))=A(\sbb_3)$, respectively. Suppose
\begin{equation*}
\left.\begin{aligned}
&\left|V(H(\sbb_1)) \cap V(H(\sbb_2))\right| = k_1,\quad\text{and}\\
&\left| V(H(\sbb_3)) \cap \left[V(H(\sbb_1)) \cup V(H(\sbb_2))\right] \right| = k_2, 
\end{aligned}
\right\}\quad \text{for some $k_1,k_2\in [R]$.}
\end{equation*}
Define the union graphs, 
$G_1 = \bigcup_{i=1}^2 H(\sbb_i)$ and $G_2 = \bigcup_{i=1}^3 H(\sbb_i)$.  

\begin{enumerate}
    \item[(a)] Let $D_1$ be any arbitrary minimum vertex cover of $G_1$. Then,
    \begin{align}
    |D_1| \geq 2\cdot\tau(H) -\min\{\tau(H),k_1\}.
\label{lowbdd-1}
\end{align}
\item[(b)] There exists a minimum vertex cover $D_2$ of $G_2$, such that the following holds
    \begin{align}
    |D_2| \geq 3\cdot\tau(H) -\min\{\tau(H),k_1\} -\min\{\tau(H),k_2\}.
\label{lowbdd-2}
\end{align}
\end{enumerate}
\end{lemma}

\begin{proof}[Proof of Lemma \eqref{VC-lem-1}] 
Consider part (a) of the Lemma. Let $D_1$ be any minimum vertex cover of $G_1 =H(\sbb_1)\cup H(\sbb_2)$. Then, firstly note that $|V(H(\sbb_i))\cap D_1|$
is a vertex cover of $H(\sbb_i)$ for both $i=1,2$. Thus we can write,
\begin{align*}
    |V(H(\sbb_i))\cap D_1| \geq \tau(H),~\text{for $i=1,2$,}\quad\text{and}\quad \left|\left(V(H(\sbb_1))\cup V(H(\sbb_1))\right)\cap D_1\right| =|D_1|\geq \tau(H).
\end{align*}
Also we know that, $|V(H(\sbb_1))\cap V(H(\sbb_2))| = k_1$ for some $k_1\in [R]$. Thus 
\begin{align*}
    |(V(H(\sbb_1))\cap V(H(\sbb_1)))\cap D_1| \leq k_1.
\end{align*}
Thus, we can write,
\begin{align*}
     &|D_1| = |D_1\cap (V(H(\sbb_1))\cup V(H(\sbb_2)))|\notag\\
     &= |D_1\cap V(H(\sbb_1))| + |D_1\cap V(H(\sbb_2))| - |D_1\cap (V(H(\sbb_1))\cap V(H(\sbb_2)))| \geq 2\cdot\tau(H) -k_1.
\end{align*}
Therefore, for any $\sbb_1$ and $\sbb_2$ satisfying $|V(H(\sbb_1))\cap V(H(\sbb_2))|=k_1$, we can write,
\begin{align*}
     |D_1| \geq \max\{\tau(H), 2\cdot\tau(H)-k_1\} = 2\cdot\tau(H) - \min\{\tau(H),k_1\}.
\end{align*}
This completes the proof of part (a). 

To prove part (b), we consider the following. Let $U_1$ be any minimum vertex cover of $H(\sbb_3)$. As $H$ and $H(\sbb_3)$ are isomorphic, we must have $|U_1|=\tau(H)$. Since $U_1$ is a minimum vertex cover for $H(\sbb_3)$, for each $v\in U_1$ there exists a {\it private edge} $e=\{v,a\}\in E(H(\sbb_3))$ (see Lemma \ref{VC-lem-0}). This means that $U_1\setminus \{v\}$ will not cover $e$. Consider the following collection of vertices in $V(H(\sbb_1))\cup V(H(\sbb_2))\cup V(H(\sbb_3))$, 
$$
A_1\equiv U_1\cup \left(\left\{V(H(\sbb_1))\cup V(H(\sbb_2))\right\}\setminus V(H(\sbb_3))\right).
$$ 
Note, $A_1$ is a vertex cover of $\bigcup_{i=1}^3H(\sbb_i)$. We can create a minimum vertex cover of $\bigcup_{i=1}^3H(\sbb_i)$ from $A_1$ by removing any vertex from $A_1$, if it does not have a private edge in $\cup_{i=1}^3E(H(\sbb_i))$. Obviously, any $v\in U_1\cap A_1$ can not be removed for the following reason. As $H(\sbb_3)$ is an induced subgraph of $\bigcup_{i=1}^3H(\sbb_i)$, any edge in $E(H(\sbb_3))$ is also an edge in $\bigcup_{i=1}^3E(H(\sbb_i))$. In addition, there exists a private edge $e=\{v,a\}\in E(H(\sbb_3))$, associated with the vertex $v$, where $a\not\in U_1$ and $a\not\in \{V(H(\sbb_1))\cup V(H(\sbb_2))\}\setminus V(H(\sbb_3))$. Thus, if $v$ is removed from $A_1$, the edge $e=\{v,a\}$ cannot be covered. Thus, the resulting minimum vertex cover can be constructed by removing redundant vertices from $A_1\setminus U_1$. This involves checking each $v\in A_1\setminus U_1 = \left\{V(H(\sbb_1))\cup V(H(\sbb_2))\right\}\setminus V(H(\sbb_3))$, and removing it if it does not have a private edge in $\bigcup_{i=1}^3 E(H(\sbb_i))$. The resulting minimum vertex cover of $\bigcup_{i=1}^3 H(\sbb_i)$, denoted by $D_2$, will also contain $U_1$. By construction, 
$$
D_2 = U_1 \cup \left\{\text{a collection of vertices from $\left(\left\{V(H(\sbb_1))\cup V(H(\sbb_2))\right\}\setminus V(H(\sbb_3))\right)$}\right\}.
$$
As a result, 
$$
\tau(H)= |U_1| = |D_2\cap V(H(\sbb_3))| \geq \left|D_2\cap V(H(\sbb_3))\cap \left(V(H(\sbb_1))\cup V(H(\sbb_2))\right)\right|.
$$
Obviously, $\left|D_2\cap V(H(\sbb_3))\cap\left(V(H(\sbb_1))\cup V(H(\sbb_2))\right)\right|\leq |V(H(\sbb_3))\cap (V(H(\sbb_2))\cup V(H(\sbb_1)))|\leq k_2$. Then we can write, 
\begin{align}\label{bdd-1}
\left|D_2 \cap V(H(\sbb_3))\cap \left(V(H(\sbb_1))\cup V(H(\sbb_2))\right)\right|\leq \min\{\tau(H),k_2\}.  \end{align}
Next, we can wan write,
\begin{align}\label{bdd-21}
    \left|D_2\cap \left(V(H(\sbb_1)) \cup V(H(\sbb_2))\right)\right| \geq \tau(H).
\end{align}
Also note that, $D_2\cap V(H(\sbb_i))$ is a vertex cover cover of $H(\sbb_i)$ for $i=1,2$. Therefore, we can claim $|D_2\cap V(H(\sbb_i))| \geq \tau(H)$. Thus, using the fact $|(V(\sbb_1)\cap V(\sbb_2))\cap D_2| \leq k_1$, we can write,
\begin{align}\label{bdd-3}
    &\left|D_2\cap \left(V(H(\sbb_1)) \cup V(H(\sbb_2))\right)\right| \geq |V(\sbb_1)\cap D_2| + |V(\sbb_2)\cap D_2| - |(V(\sbb_1)\cap V(\sbb_2))\cap D_2| \geq 2\tau(H) - k_1.
\end{align}
Combining \eqref{bdd-21} and \eqref{bdd-3}, we can write,
\begin{align}\label{bdd-4}
    \left|D_2\cap \left(V(H(\sbb_1)) \cup V(H(\sbb_2))\right)\right| \geq \max\{\tau(H),2\cdot\tau(H) -k_1\} = 2\cdot \tau(H) - \min\{\tau(H),k_1\}.
\end{align}
Therefore, combining \eqref{bdd-1} and \eqref{bdd-4}, we can write,
\begin{align*}
    &\left|D_2\right| = \left|D_2\cap \left(\bigcup_{i=1}^3V(H(\sbb_i))\right)\right|\notag\\
    &=\left|D_2\cap \left(V(H(\sbb_1))\cup V(H(\sbb_2))\right)\right| + \left|D_2\cap V(H(\sbb_3))\right| - \left|D_2\cap V(H(\sbb_3))\cap \left(V(H(\sbb_1))\cup V(H(\sbb_2))\right)\right|\notag\\
    &\geq 3\cdot\tau(H) - \min\{\tau(H),k_1\} - \min\{\tau(H),k_2\}.
\end{align*}
This completes the proof.
\end{proof}

\begin{lemma}[Attainment of lower bound \eqref{lowbdd-1}]
\label{VC-lem-2} 
Let $H$ be a fixed connected, simple, undirected graph with vertex set $
V(H)=\{1,\dots,R\}$, edge set $E(H)$ and minimum vertex cover size $\tau(H)$. Fix a $t\in[R]$. Consider any $\sbb_1,\sbb_2\in \nsr$ and assume that $|A(\sbb_1)\cap A(\sbb_2)| = t$. For $i=1,2$, define the graph $H(\sbb_i)$, with vertex set $A(\sbb_i)$ and edge set 
$$
E\big(H(\mathbf{s}_i)\big)=\big\{\{s_{i,a},s_{i,b}\} : \{a,b\}\in E(H)\big\}.
$$
Define the union graph $G(\sbb_1,\sbb_2) = H(\sbb_1)\cup H(\sbb_2)$. Let $\xi_0$ denote the identity permutation on $\{1,\ldots,t\}$. There exists an index vector $\mathbf{j}_0\in {[R]}_t$ such that, for any 
\begin{align}
(\mathbf{s}_1,\mathbf{s}_2)\in 
\mathcal{M}_t(\mathbf{j}_0,\mathbf{j}_0,\xi_0) = \left\{(\sbb_1,\sbb_2):~\sbb_1,\sbb_2\in\nsr, |A(\sbb_1)\cap A(\sbb_2)|=t,s_{1,k}=s_{2,k}~\text{for all $k\in A(\jbb_0)$}\right\},
\label{Mtj0-set}
\end{align}
the union graph $G(\sbb_1,\sbb_2)$ satisfies the following property: there exists a minimum vertex cover $D_0(\sbb_1,\sbb_2)$ of $G(\sbb_1,\sbb_2)$ such that
\begin{align}\label{lowbdd-at}
         |D_0(\sbb_1,\sbb_2)| = 2\cdot\tau(H) - \min\{\tau(H),t\}.
\end{align}
\end{lemma}


\begin{proof}[Proof of Lemma \ref{VC-lem-2}]
Firstly, consider the case where $t\in \{\tau(H),\ldots,R\}$. Without loss of generality, assume that $D_0 = \{1,\ldots,\tau(H)\}$ is a minimum vertex cover of $H$ (if needed, this can be achieved by relabeling of vertices). Define the vector
$$
\jbb_0 = {(1,\ldots,{\tau(H)}, {\tau(H)+1}, \ldots, t)}^\primet\in{[R]}_t,
$$
and pick any $(\sbb_1,\sbb_2)\in \mathcal{M}_t(\jbb_0,\jbb_0,\xi_0)$ (cf. \eqref{Mtj0-set}) and consider the union graph $G(\sbb_1,\sbb_2)$. Define the collection
$$
D_0(\sbb_1,\sbb_2) = \{s_{1,1},\ldots,s_{1,\tau(H)}\},
$$
consisting of the first $\tau(H)$ components of the vector $\sbb_1 = {(s_{1,1},\ldots,s_{1,R})}^\primet$. Due to the fact that the first $t$ components of $\sbb_1$ and $\sbb_2$ are equal, this implies $D_0(\sbb_1,\sbb_2)$ also consists of the first $\tau(H)$ components of $\sbb_2$. Since we have assumed $t\geq \tau(H)$, we have $|D_0(\sbb_1,\sbb_2)| = \tau(H) = 2\tau(H)-\min\{\tau(H),t\}$. We claim that $D_0(\sbb_1,\sbb_2)$ is a vertex cover of $G(\sbb_1,\sbb_2) = H(\sbb_1)\cup H(\sbb_2)$. This is because $D_0(\sbb_1,\sbb_2)$ is a vertex cover of both $H(\sbb_1)$ and $H(\sbb_2)$. This follows from the following reasoning.

If $\{s_{1,u},s_{1,v}\}\in E(H(\sbb_1))$, then due to the above described construction of $H(\sbb_1)$, $\{u,v\}\in E(H)$. As $D_0 = \{1,\ldots,\tau(H)\}$ is a minimum vertex cover of $H$, it implies that either $u$, or $v$, or both $u$ and $v$ are elements of $D_0=\{1,\ldots,\tau(H)\}$. From the definition of $D_0(\sbb_1,\sbb_2)$, it contains the first $\tau(H)$ components of $\sbb_1 = {(s_{1,1},\ldots,s_{1,R})}^\primet$, and this implies that, either $s_{1,u}$, or $s_{1,v}$, or both $s_{1,u}$ and $s_{1,v}$ are elements of $D_0(\sbb_1,\sbb_2)=\{s_{1,1},\ldots,s_{1,\tau(H)}\}$. This implies, $D_0(\sbb_1,\sbb_2)$ is a vertex cover of $H(\sbb_1)$. Similarly, if $\{s_{2,u},s_{2,v}\}\in E(H(\sbb_2))$, we must have $\{u,v\}\in E(H)$, and this implies that either $s_{2,u}$, or $s_{2,v}$, or both $s_{2,u}$ and $s_{2,v}$ are elements of $D_0(\sbb_1,\sbb_2)=\{s_{1,1},\ldots,s_{1,\tau(H)}\} = \{s_{2,1},\ldots,s_{2,\tau(H)}\}$, where the equality in the last step arises due to the use of the identity mapping $\xi_0$ in the definition of $\mathcal{M}_t(\jbb_0,\jbb_0,\xi_0)$ (cf. \eqref{Mtj0-set}). This implies, $D_0(\sbb_1,\sbb_2)$ is a vertex cover of $H(\sbb_2)$. As a result $D_0(\sbb_1,\sbb_2)$ is a vertex cover of the union graph $G(\sbb_1,\sbb_2)$. Now, using Lemma \ref{VC-lem-1}(a), we can claim that the lowest possible cardinality of any arbitrary minimum vertex cover of $G(\sbb_1,\sbb_2)$ will be $\tau(H)$. Since $D_0(\sbb_1,\sbb_2)$ is a vertex cover of $G(\sbb_1,\sbb_2)$ and $|D_0(\sbb_1,\sbb_2)| = \tau(H)$, it immediately implies that $D_0(\sbb_1,\sbb_2)$ will be a minimum vertex cover of $G(\sbb_1,\sbb_2)$.

Next, consider the case where $t\in \{1,\ldots,\tau(H)-1\}$. Without loss of generality we assume that $D_0 = \{1,\ldots,\tau(H)\}$ is a minimum vertex cover of $H$. Pick 
$$
\jbb_0 = {(1,\ldots,t)}^\primet.
$$
Consider any $(\sbb_1,\sbb_2)\in \mathcal{M}_t(\jbb_0,\jbb_0,\xi_0)$ (cf. \eqref{Mtj0-set}) and construct the union graph $G(\sbb_1,\sbb_2) = H(\sbb_1)\cup H(\sbb_2)$. Using previous arguments, we can claim 
$$
D_i(\sbb_i) = \{s_{i,1},\ldots,s_{i,\tau(H)}\},~i=1,2,
$$
is a minimum vertex cover of $H(\sbb_i)$, for $i=1,2$. According to the definition of $\mathcal{M}_t(\jbb_0,\jbb_0,\xi_0)$ (cf. \eqref{Mtj0-set}), $s_{1,k}=s_{2,k}$, for $k=1,\ldots,t$, and there is no other common component between $\sbb_1$ and $\sbb_2$. Define the following collection of vertices in $G(\sbb_1,\sbb_2)$,
$$
D_0(\sbb_1,\sbb_2) = D_1(\sbb_1)\cup D_2(\sbb_2).
$$
We claim that $D_0(\sbb_1,\sbb_2)$ is a minimal vertex cover of $G(\sbb_1,\sbb_2)$. Firstly, note that $D_0(\sbb_1,\sbb_2)$ is a vertex cover of both $G(\sbb_1,\sbb_2)$. Note that, as per definition of $\mathcal{M}_t(\jbb_0,\jbb_0,\xi_0)$
$$
V(\sbb_1,\sbb_2) = D_1(\sbb_1)\cap D_2(\sbb_2) = \{s_{1,1},\ldots,s_{1,\tau(H)}\}\cap \{s_{2,1},\ldots,s_{2,\tau(H)}\} = \{s_{1,1},\ldots,s_{1,t}\}.
$$
Write,
$$
D_0(\sbb_1,\sbb_2) = \left(D_1(\sbb_1)\setminus V(\sbb_1,\sbb_2)\right)\cup V(\sbb_1,\sbb_2)\cup \left(D_2(\sbb_2)\setminus V(\sbb_1,\sbb_2)\right),
$$
as the union of three disjoint sets. Note that $|D_0(\sbb_1,\sbb_2)| = (\tau(H)-t)+t+(\tau(H)-t) = 2\tau(H)-t = 2\tau(H)-\min\{t,\tau(H)\}$. Pick any $u\in V(\sbb_1,\sbb_2)$. Due to Lemma \ref{VC-lem-0} and due to the fact that $u\in D_1(\sbb_1)$ and $D_1(\sbb_1)$ is a minimum vertex cover of $H(\sbb_1)$, there exists a private edge $\{u,v_1\}\in E(H(\sbb_1))$ for vertex $u$. Similarly, there must exist a private edge $\{u,v_2\}\in E(H(\sbb_2))$. Obviously, $v_1\not\in D_1(\sbb_1)$, otherwise $\{u,v_1\}$ will not be a private edge of $u$. Similarly, $v_2\not\in D_2(\sbb_2)$. Also, we must have $v_1\neq v_2$. As a result, if $u$ is removed from $D_0(\sbb_1,\sbb_2)$, then there will be at least two edges in $E(H(\sbb_1)\cup H(\sbb_2))$ which will not be covered by $D_0(\sbb_1,\sbb_2)\setminus \{u\}$. Similarly, if $u\in D_1(\sbb_1)\setminus V(\sbb_1,\sbb_2)$, there will exist at least one edge in $E(H(\sbb_1)\cup H(\sbb_2))$ that will not be covered by $D_0(\sbb_1,\sbb_2)\setminus \{u\}$, and a similar argument applies for $u\in D_2(\sbb_2)\setminus V(\sbb_1,\sbb_2)$. As a result, we have established the minimality of the vertex cover $D_0(\sbb_1,\sbb_2)$. By Lemma \ref{VC-lem-1}(a), $D_0(\sbb_1,\sbb_2)$ also attains the lowest possible cardinality of any vertex cover of the union graph $G(\sbb_1,\sbb_2)$. Hence, the minimal vertex cover $D_0(\sbb_1,\sbb_2)$ is also a minimum vertex cover of $G(\sbb_1,\sbb_2)$. This completes the proof.

\end{proof}

\begin{lemma}
\label{VC-lem-3}
For any $R\geq 2$, the star graph $\kl_{1,R-1}$ and the complete graph $\kl_{R}$ on $R$ vertices satisfy Assumption \ref{assump-5}.
\end{lemma}

\begin{proof}[Proof of Lemma \ref{VC-lem-3}]
Consider a star graph $H = \kl_{1,R-1}$ with edges $\{\{1,2\},\{1,3\},\ldots,\{1,R\}\}$. Here, $\tau(\kl_{1,R-1}) = 1$, and the minimum vertex cover is $D_0 = \{1\}$. Consider the $t$ vertex induced subgraph $H_{r_0}$ of $H$, with vertex set $V(H_{r_0}) = \{1,i_2,\ldots,i_{t}\}$, for any $\{i_2,\ldots,i_t\}\subseteq \{2,\ldots,R\}$. Note, $H_{r_0}$ is also a star graph on $t$ vertices and has the maximum number of edge $m(t:H) = (t-1)$. Thus, $H_{r_0}\in \Gc_{t,\kl_{1,R-1}}$ and $D_0\subseteq V(H_{r_0}) = \{1,i_1,\ldots,i_{t-1}\}$ trivially. Since, $\tau(H)=1$, the case of $t<\tau(H)$ does not arise. 

When  $H=\kl_R$ (complete graph) then $\tau(\kl_{R}) =R-1$, and we can choose any arbitrary set of $(R-1)$ vertices from $[R]$ to create a minimum vertex cover $D_0 = \{i_1,\ldots,i_{R-1}\}$. First, note that any $H_{r}\in \Gc_{t,\kl_R}$ is a complete graph on $t$ vertices. When $t < R-1$, then vertices of the graph $H_r$ with $V(H_r) = \{i_1,\ldots,i_t\}$ also form a complete graph and $V(H_r)\subseteq D_0$. If $t=R-1$, then any $H_r$ with vertex set $V(H_r) = \{i_1,\ldots,i_{R-1}\}$ is a vertex cover of $\kl_R$. Therefore in this case $D_0 = V(H_r)$. When $t=R$, then $D_0 \subset V(\kl_R)$ trivially.  


\end{proof}

\begin{lemma}\label{lem-var-ego-1}
    Suppose that Assumption \ref{assump-5} holds. Let $H$ be a fixed, undirected, simple, and connected graph with $V(H)=\{1,\ldots,R\}$ and the edge set $E(H)\subseteq [V(H)]^2$, where $R\geq 2$, is a fixed positive integer. Then, for $t\in \{1,\ldots R\}$, for any $(\sbb_1,\sbb_2)\in \Mc_t$, there exists an $M$ such that for all $N\geq M$,
    \begin{align}
L_{t,2}(\al,\beta,H,\lamb,\Cb)
&\le 
\frac{
\displaystyle 
\sum_{(\sbb_1,\sbb_2)\in \Mc_t}
\E\left(W_{N,2}(\sbb_1:H)\cdot W_{N,2}(\sbb_2:H)\right)
\E\left(Y_N(\sbb_1:H)\cdot Y_N(\sbb_2:H)\right)
}{
N^{\displaystyle 
2R - t - \al\left(2\tau(H)-\min\{\tau(H),t\}\right)
-\beta\left(2T-m(t:H)\right)}
}
\notag\\
&\le 
U_{t,2}(\al,\beta,H,\lamb,\Cb).
\end{align}
where, $L_{t,2}(\al,\beta,H,\lamb,\Cb)$ and $U_{t,2}(\al,\beta,H,\lamb,\Cb)$ are two positive constants depending on $\al,\beta,H$ and the model parameters $\lamb$ and $\Cb$. 
 \end{lemma}
\begin{proof}
    Using the same approach and notation as in the proof of Lemma \ref{lem-1} and using the definition \eqref{def-Jt-Mt}, we can write
    \begin{align}\label{ego-A}
        &\sum_{(\sbb_1,\sbb_2)\in \Mc_t}\E\left(W_{N,2}(\sbb_1:H)W_{N,2}(\sbb_2:H)\right)\cdot \E\left(Y_N(\sbb_1:H)Y_{N}(\sbb_2:H)\right)\notag\\
        &= \sum_{\jbb,\kbb\in [R]_{<,t}}\sum_{\xi\in \mathcal{S}_t}\sum_{(\sbb_1,\sbb_2)\in \Mc_t(\jbb,\kbb,\xi)}\E\left(W_{N,2}(\sbb_1:H)W_{N,2}(\sbb_2:H)\right)\cdot \E\left(Y_N(\sbb_1:H)Y_{N}(\sbb_2:H)\right)\notag\\
        &= \sum_{\jbb,\kbb\in [R]_{<,t}}\sum_{\xi\in \mathcal{S}_t}\sum_{(\sbb_1,\sbb_2)\in \Mc_t(\jbb,\kbb,\xi)}\E\left(\prod_{\{i,j\}\in E(H(\sbb_1))\cup E(H(\sbb_2))}\max\{W_{N,i},W_{N,j}\}\right)\cdot \E\left(Y_N(\sbb_1:H)Y_{N}(\sbb_2:H)\right)\notag\\
        &= \sum_{\jbb,\kbb\in [R]_{<,t}}\sum_{\xi\in \mathcal{S}_t}\sum_{(\sbb_1,\sbb_2)\in \Mc_t(\jbb,\kbb,\xi)}\left(\sum_{D\in \Dc(H(\sbb_1)\cup H(\sbb_2))}p_N^{|D|}(1-p_N)^{2R-|D|}\right)\cdot \E\left(Y_N(\sbb_1:H)Y_{N}(\sbb_2:H)\right)\notag\\
        &= \sum_{(\jbb,\kbb)\in J_t\cap M_t}\sum_{\xi\in \mathcal{S}_t}\sum_{(\sbb_1,\sbb_2)\in \Mc_t(\jbb,\kbb,\xi)}\left(\sum_{D\in \Dc(H(\sbb_1)\cup H(\sbb_2))}p_N^{|D|}(1-p_N)^{2R-|D|}\right)\cdot \E\left(Y_N(\sbb_1:H)Y_{N}(\sbb_2:H)\right)\notag\\
        &+\sum_{(\jbb,\kbb)\in (J_t\cap M_t)^c}\sum_{\xi\in \mathcal{S}_t}\sum_{(\sbb_1,\sbb_2)\in \Mc_t(\jbb,\kbb,\xi)}\left(\sum_{D\in \Dc(H(\sbb_1)\cup H(\sbb_2))}p_N^{|D|}(1-p_N)^{2R-|D|}\right)\cdot \E\left(Y_N(\sbb_1:H)Y_{N}(\sbb_2:H)\right)\notag\\
        & = A_{1,N} + A_{2,N},~\text{(say)}.
    \end{align}
First, we analyze the design component of the term $A_{1,N}$ from \eqref{ego-A}. To analyze the term, using Lemma \ref{VC-lem-2} and Assumption \ref{assump-5}, we now show the existence of $(\jbb_0,\kbb_0)\in J_t\cap M_t$ and a permutation map $\xi_0\in \mathcal{S}_t$ for which there is at least one minimum vertex cover of $H(\sbb_1)\cup H(\sbb_2)$ achieving the lower bound provided in \eqref{lowbdd-at}, that is, of size $2\cdot\tau(H) -\min\{\tau(H),t\}$ for $(\sbb_1,\sbb_2)\in \Mc_t(\jbb_0,\kbb_0,\xi_0)$. Fix $t\in \{\tau(H),\ldots,R\}$. Now, since $H$ satisfies Assumption \ref{assump-5}, we can say that there exist a $t$ vertex induced subgraph $H_{r_0}\in \Gc_{t,H}$ (cf. \eqref{Gcal-t}) and a minimum vertex cover $D_0=\{i_1,\ldots,i_{\tau(H)}\}$ of $H$, such that 
$D_0 \subseteq V(H_{r_0})\subseteq [R]$. Then, we can write 
$$
V(H_{r_0})=A(\jbb_0), \quad \text{where}\quad\jbb_0 = {(i_1,\ldots,i_{\tau(H)}, j_{\tau(H)+1}, \ldots, j_t)}^\primet, 
$$
and $j_{\tau(H)+1}, \ldots, j_t \in [R]_t$, denote the remaining vertices of $H_{r_0}$. Take another position vector $\kbb_0 = \jbb_0$ and the identity permutation map $\xi_0\in \mathcal{S}_t$. By construction, $(\jbb_0,\kbb_0)= (\jbb_0,\jbb_0)\in J_t\cap M_t$ (cf. \eqref{def-Jt-Mt}), and by the definition of $J_t$ and $M_t$, a $t$ vertex induced subgraph of $H$, that is, $B(\jbb_0,t) = H_{r_0}$ and $|E(H_{r_0})| = |E(B(\jbb_0,t))| = m(t:H)$. Then for any $(\sbb_1,\sbb_2)\in \Mc_t(\jbb_0,\jbb_0,\xi_0)$, we can construct a union graph $H(\sbb_1)\cup H(\sbb_2)$. Using Lemma \ref{VC-lem-2}, we can claim that there is at least one minimum vertex cover of $H(\sbb_1)\cup H(\sbb_2)$, say $D^\star$, which satisfies \eqref{lowbdd-at}. 

 Next, fix $t\in \{1,\ldots,\tau(H)-1\}$. As $H$ satisfies Assumption \ref{assump-5}, there exist a $H_{r_0}\in \Gc_{t,H}$ and a minimum vertex cover $D_0 = \{i_1,\ldots,i_{\tau(H)}\}$ of $H$ such that
\begin{align*}
        V(H_{r_0})= \{i_1,\ldots,i_{t}\}\subset D_0.
\end{align*}
In this case, take $\jbb_0 = (i_1,\ldots,i_t)^\primet$ and note that $A(\jbb_0) = V(H_{r_0})$. Take $\kbb_1 = (i_1,\ldots,i_t)^\primet = \jbb_0$ and consider the identity permutation  $\xi_0\in \mathcal{S}_t$. By construction, $(\jbb_0,\jbb_0)\in J_t\cap M_t$.  For this $(\sbb_{1,0},\sbb_{2,0})\in \Mc_t(\jbb_0,\jbb_0,\xi_0)$, construct the union graph $H(\sbb_{1,0})\cup H(\sbb_{2,0})$. Using Lemma \ref{VC-lem-2}, we can claim that, there is at least one minimum vertex cover of $H(\sbb_1)\cup H(\sbb_2)$, say $D^\star$, which satisfies \eqref{lowbdd-at}.
Then, for fixed $t\in [R]$, there is at least one $(\jbb_0,\kbb_0)\in J_t\cap M_t$ and $\xi_0\in \mathcal{S}_t$ such that for all $(\sbb_1,\sbb_2)\in \Mc_t(\jbb_0,\kbb_0,\xi_0)$, we can write the following. 
\begin{align}\label{VC-A}
    &\sum_{D\in \Dc(H(\sbb_1)\cup H(\sbb_2))}p_N^{|D|}(1-p_N)^{2R-|D|}\notag\\
    &= \sum_{D\in \Dc_{\min}(H(\sbb_1)\cup H(\sbb_2))}p_N^{|D|}(1-p_N)^{2R-|D|} + \sum_{D\in \Dc(H(\sbb_1)\cup H(\sbb_2))\setminus\Dc_{\min}(H(\sbb_1)\cup H(\sbb_2))}p_N^{|D|}(1-p_N)^{2R-|D|}\notag\\
    &= |\Dc_{\min}(H(\sbb_1)\cup H(\sbb_2))|\cdot p_N^{|D^{\star}|}(1-p_N)^{2R-|D^{\star}|} + \sum_{D\in \Dc(H(\sbb_1)\cup H(\sbb_2))\setminus\Dc_{\min}(H(\sbb_1)\cup H(\sbb_2))}p_N^{|D|}(1-p_N)^{2R-|D|}\notag\\
    &= p_N^{|D^{\star}|}\bigg\{|\Dc_{\min}(H(\sbb_1)\cup H(\sbb_2))|\cdot (1-p_N)^{2R-|D^{\star}|}\notag\\
    &\qquad{}\qquad{}+ \sum_{D\in \Dc(H(\sbb_1)\cup H(\sbb_2))\setminus\Dc_{\min}(H(\sbb_1)\cup H(\sbb_2))}p_N^{|D|-|D^{\star}|}(1-p_N)^{2R-|D|}\bigg\}\notag\\
    &= p_N^{2\tau(H) - \min\{\tau(H),t\}}\cdot |\Dc_{\min}(H(\sbb_1)\cup H(\sbb_2))| \left(1+ O\left(\frac{1}{N^{\al}}\right)\right).
\end{align}
In the above expression $ |\Dc_{\min}(H(\sbb_1)\cup H(\sbb_2))|$ is bounded by a finite quantity $|\Dc(H)|^2$, which is a collection of vertex covers of $H$. For simplicity, we can write $a_{1} = |\Dc_{\min}(H(\sbb_1)\cup H(\sbb_2))|$, and note that the structure of the graph $H(\sbb_1)\cup H(\sbb_2)$ is solely dependent on the choice of the triplet $(\jbb,\kbb,\xi)$ and therefore the value of $a_1$. Without loss of generality, since the sets $J_t\cap M_t$ and $\mathcal{S}_t$ are finite sets, assume that for one choice $(\jbb_0,\kbb_0)\in J_t\cap M_t$ and $\xi_0\in \mathcal{S}_t$ there exists a minimum vertex cover $D^\star$ in the collection $\mathcal{D}_{\min}(H(\sbb_1)\cup H(\sbb_2))$ for $(\sbb_1,\sbb_2)\in \Mc(\jbb_0,\kbb_0,\xi_0)$ such that $|D^\star| = 2\tau(H) - \min\{\tau(H),t\}$. Then, the term $A_{1,N}$ (see \eqref{ego-A}) can be written as follows.
\begin{align}\label{VC-A-1}
    &A_{1,N} = \sum_{(\jbb,\kbb)\in J_t\cap M_t}\sum_{\xi\in \mathcal{S}_t}\sum_{(\sbb_1,\sbb_2)\in \Mc_t(\jbb,\kbb,\xi)}\left(\sum_{D\in \Dc(H(\sbb_1)\cup H(\sbb_2))}p_N^{|D|}(1-p_N)^{|D|}\right)\cdot \E\left(Y_N(\sbb_1:H)Y_{N}(\sbb_2:H)\right)\notag\\
    &= \sum_{(\sbb_1,\sbb_2)\in \Mc_t(\jbb_0,\kbb_0,\xi_0)}\left(\sum_{D\in \Dc(H(\sbb_1)\cup H(\sbb_2))}p_N^{|D|}(1-p_N)^{2R-|D|}\right)\cdot \E\left(Y_N(\sbb_1:H)Y_{N}(\sbb_2:H)\right) + \notag\\
    &+\sum_{\substack{(\jbb,\kbb)\in J_t\cap M_t\\ (\jbb,\kbb)\neq(\jbb_0,\kbb_0)}}\sum_{\substack{\xi\in \mathcal{S}_t\\
    \xi\neq \xi_0}}\sum_{(\sbb_1,\sbb_2)\in \Mc_t(\jbb,\kbb,\xi)}\left(\sum_{D\in \Dc(H(\sbb_1)\cup H(\sbb_2))}p_N^{|D|}(1-p_N)^{2R-|D|}\right)\cdot \E\left(Y_N(\sbb_1:H)Y_{N}(\sbb_2:H)\right)\notag\\
    &= \sum_{(\sbb_1,\sbb_2)\in \Mc_t(\jbb_0,\kbb_0,\xi_0)}\left(\sum_{D\in \Dc(H(\sbb_1)\cup H(\sbb_2))}p_N^{|D|}(1-p_N)^{2R-|D|}\right)\cdot \E\left(Y_N(\sbb_1:H)Y_{N}(\sbb_2:H)\right) + A_{2,1,N},~\text{(say)}\notag\\
    &= a_1\cdot p_N^{2\tau(H) - \min\{\tau(H),t\}}\cdot \left(1+ O\left(\frac{1}{N^{\al}}\right)\right)\sum_{(\sbb_1,\sbb_2)\in \Mc_t(\jbb_0,\kbb_0,\xi_0)} \E\left(Y_N(\sbb_1:H)Y_{N}(\sbb_2:H)\right) + A_{2,1,N}\notag\\
    &= A_{1,1,N} + A_{2,1,N},~\text{(say)}
\end{align}
Firstly, we analyze the term $A_{1,1,N}$ from \eqref{VC-A-1}. From the expression \eqref{V-1-imp} (cf. Lemma \ref{lem-1}), we know that for all $\varepsilon>0$ there exists an $M$ such that for all $N\geq M$, and for all $(\jbb,\kbb)\in J_t\cap M_t$ and $\xi\in \mathcal{S}_t$, we can write,
\begin{align*}
    L_{t,1}(\beta,H,\lamb,\Cb) -\varepsilon < \frac{\sum_{(\jbb,\kbb)\in J_t\cap M_t}\sum_{\xi\in \mathcal{S}_t}\sum_{(\sbb_1,\sbb_2)\in \Mc_t(\jbb,\kbb,\xi)} \E\left(Y_N(\sbb_1:H)Y_{N}(\sbb_2:H)\right)}{N^{2R-t-\beta(2T-m(t:H))}} < U_{t,1}(\beta,H,\lamb,\Cb) -\varepsilon.
\end{align*}
Using this fact, for fixed $t\in [R]$, and for $(\jbb_0,\jbb_0)\in J_t\cap M_t$,  $\xi_0\in \mathcal{S}_t$ and recall the fact $p_N= c\cdot N^{-\al}$, from \eqref{VC-A-1}, for all $\varepsilon_1>0$ there exists an $M_1$ such that for all $N>M_1$, we can write,
\begin{align}\label{VC-bdd-ego-1}
     &L_{t,1}(\beta,H,\lamb,\Cb) -\varepsilon < \frac{\sum_{(\sbb_1,\sbb_2)\in \Mc_t(\jbb_0,\jbb_0,\xi_0)} \E\left(Y_N(\sbb_1:H)Y_{N}(\sbb_2:H)\right)}{N^{2R-t-\beta(2T-m(t:H))}} < U_{t,1}(\beta,H,\lamb,\Cb) -\varepsilon\notag\\
     &\Rightarrow L_{t,1}(\beta,H,\lamb,\Cb) -\varepsilon < \frac{a_1\cdot p_N^{2\tau(H) - \min\{\tau(H),t\}}\sum_{(\sbb_1,\sbb_2)\in \Mc_t(\jbb_0,\jbb_0,\xi_0)} \E\left(Y_N(\sbb_1:H)Y_{N}(\sbb_2:H)\right)}{a_1\cdot p_N^{2\tau(H) - \min\{\tau(H),t\}}\cdot N^{2R-t-\beta(2T-m(t:H))}}\notag\\
     &\qquad{}\qquad{}\qquad{}\qquad{}\qquad{}< U_{t,1}(\beta,H,\lamb,\Cb) -\varepsilon\notag\\
    &\Rightarrow~\frac{L_{t,1}(\beta,H,\lamb,\Cb) -\varepsilon}{[a_1\cdot c^{2\tau(H) -\min\{\tau(H),t\}}]^{-1}} < \frac{A_{1,1,N}}{N^{2R-t - \al(2\tau(H) - \min\{\tau(H),t\}) -\beta(2T-m(t:H))}} < \frac{U_{t,1}(\beta,H,\lamb,\Cb) -\varepsilon}{[a_1\cdot c^{2\tau(H) -\min\{\tau(H),t\}}]^{-1}}\notag\\
    &\Rightarrow~L_{t,2}(\al,\beta,H,\lamb,\Cb) -\varepsilon_1 < \frac{A_{1,1,N}}{N^{2R-t - \al(2\tau(H) - \min\{\tau(H),t\}) - \beta(2T-m(t:H))}} < U_{t,2}(\al,\beta,H,\lamb,\Cb) -\varepsilon_1.
\end{align}
Next we analyze the term $A_{2,1,N}$ from the expression \eqref{VC-A-1}. Since the sets $J_t\cap M_t$ and $\mathcal{S}_t$ are finite and we can assume without loss of generality that the only choice of $(\jbb,\kbb)$ and $\xi$ in $J_t\cap M_t$ and $\mathcal{S}_t$ respectively are $(\jbb_0,\jbb_0)$ and $\xi_0$, for which we can construct a union graph $H(\sbb_1)\cup H(\sbb_2)$ using $(\sbb_1,\sbb_2)\in \Mc_{t}(\jbb_0,\kbb_0,\xi_0)$ and there is at least one minimum vertex cover of $H(\sbb_1)\cup H(\sbb_2)$ of size $2\tau(H) -\min\{\tau(H),t\}$. In the expression of $A_{2,1,N}$, we are summing over $(\jbb,\kbb,\xi)$ triplet which are not equal to the triplet $(\jbb_0,\jbb_0,\xi_0)$, therefore union graphs formed by this triplet has minimum vertex covers which is strictly greater than $2\tau(H) - \min\{\tau(H),t\}$.  
But the exponent of $N$ from the model component in the expression $A_{2,1,N}$ remains the same, that is, $2R-t -\beta(2T-m(t,H))$. Hence, for all $\varepsilon_2>0$, there exists an $M_2$ such that for all $N>M_2$, we can write
\begin{align}\label{VC-bdd-ego-2}
    -\varepsilon_2 < \frac{A_{2,1,N}}{N^{2R-t - \al(2\tau(H) - \min\{\tau(H),t\}) - \beta(2T-m(t:H))}} < \varepsilon_2.
\end{align}
Next, we analyze the term $A_{2,N}$ from \eqref{ego-A}. In this case $(\jbb,\kbb)\in (J_t\cap M_t)^c$. We have already proved in \eqref{v-2-imp} and \eqref{V-3-imp}, that the exponent of $N$ is larger than $2R-t - \beta(2T-m(t:H))$ for all $(\jbb,\kbb)\in (J_t\cap M_t)^c$. Irrespective of the value of the exponent of $N$ in the design based term, which is always greater or equal to $2\tau(H)-\min\{\tau(H),t\}$.  Thus, for all $\varepsilon_3>0$, there exists a $M_3$ such that for all $N>M_3$, we can write,
\begin{align}\label{VC-bdd-ego-3}
    -\varepsilon_3 < \frac{A_{2,N}}{N^{2R-t - \al(2\tau(H) - \min\{\tau(H),t\}) - \beta(2T-m(t:H))}} < \varepsilon_3.
\end{align}
Therefore, combining \eqref{VC-bdd-ego-1}, \eqref{VC-bdd-ego-2} and \eqref{VC-bdd-ego-3}, for all $N>\max\{M_1,M_2,M_3\}$, from the expression \eqref{ego-A}, we can write
\begin{align*}
    L_{t,2}(\al,\beta,H,\lamb,\Cb) -\varepsilon &< \frac{
\displaystyle 
\sum_{(\sbb_1,\sbb_2)\in \Mc_t}
\E\left(W_{N,2}(\sbb_1:H)\cdot W_{N,2}(\sbb_2:H)\right)
\E\left(Y_N(\sbb_1:H)\cdot Y_N(\sbb_2:H)\right)
}{N^{2R - t - \al\left(2\tau(H)-\min\{\tau(H),t\}\right)-\beta\left(2T-m(t:H)\right)}}\notag\\ &< U_{t,2}(\al,\beta,H,\lamb,\Cb) -\varepsilon.
\end{align*}
This completes the proof.
\end{proof}

\begin{lemma}[Growth rate of the $\sigma^2_{N,l}(H)$]\label{lem-var}
    Let $H$ be a fixed, undirected, simple, and connected graph with $V(H)=\{1,\ldots,R\}$ and the edge set $E(H)\subseteq {[V(H)]}^2$, where $R\geq 2$, is a fixed positive integer, and $|E(H)| =T$. Suppose that Assumptions \ref{assump-1}, \ref{assump-2}, \ref{assump-3}, and \ref{assump-4} hold and   recall the definition of $\sigma^2_{N,l}(H)$ (cf. \eqref{sig-temp-0}). Then
\begin{enumerate}
    \item[(a)] In the induced case ($l=1$), there exists an $M$, such that for all $N>M$, 
\begin{align}\label{var-bdd-ind}
    L(1,\al,\beta,H,\lamb,\Cb) 
    < \frac{\sigma^2_{N,1}(H)}
    {N^{\displaystyle -2g_{1,\al,\beta,H}(R) + g_{1,\al,\beta,H}(t_1(\al,\beta;H))}} 
    < U(1,\al,\beta,H,\lamb,\Cb).
\end{align}
\item[(b)] In the ego-centric case ($l=2$), along with the Assumption \ref{assump-5} over graph $H$, there exists an $M$, such that for all $N>M$, 
\begin{align}\label{var-bdd}
    L(2,\al,\beta,H,\lamb,\Cb) 
    < \frac{\sigma^2_{N,2}(H)}
    {N^{\displaystyle -2g_{2,\al,\beta,H}(R) + g_{2,\al,\beta,H}(t_2(\al,\beta;H))}} 
    < U(2,\al,\beta,H,\lamb,\Cb).
\end{align}
\end{enumerate}
where, $L(l,\al,\beta,H,\lamb,\Cb)$ and $U(l,\al,\beta,H,\lamb,\Cb)$ are two positive constants depending on network formation design, sparsity level $\al$ and $\beta$, graph $H$ and the SBM parameters. Note that $g_{l,\al,\beta, H}(\cdot)$ and $t_l(\al,\beta;H)$ for $l=1,2$, are defined in \eqref{delta-bdd-temp}.
\end{lemma}
\begin{proof}[Proof of Lemma \ref{lem-var}]
For $l=1,2$,
\begin{align}\label{var-exp-gen}
    &\sigma^2_{N,l}(H) =\Var\left(\widehat{S}_{N,l}(H) - f_{l}(p_N:H)\cdot S_N(H)\right) = \Var\left(\sum_{\sbb\in \nsr}\widetilde{W}_{N,l}(\sbb:H)\cdot Y_N(\sbb:H)\right)\notag\\
    &= \sum_{t=1}^R\sum_{(\sbb_1,\sbb_2)\in \Mc_t}\cov\left(\widetilde{W}_{N,l}(\sbb_1:H)Y_N(\sbb_1:H),\widetilde{W}_{N,l}(\sbb_2:H)Y_N(\sbb_2:H)\right)\notag\\
    &=\sum_{t=1}^R\sum_{(\sbb_1,\sbb_2)\in \Mc_t}\cov\left(W_{N,l}(\sbb_1:H),W_{N,l}(\sbb_2;H)\right)\E\left(Y_N(\sbb_1:H)Y_{N}(\sbb_2:H)\right),
\end{align}
where, $\Mc_t$ defined in \eqref{Etset}. firstly, we analyze the design component from the expression \eqref{var-exp-gen}. We begin with the induced case ($l=1$), that is for fixed $t\in [R]$ and $(\sbb_1,\sbb_2)\in \Mc_t$, we can write,
\begin{align*}
    \cov\left(W_{N,1}(\sbb_1:H),W_{N,1}(\sbb_2;H)\right) = p_N^{2R-t}(1-p_N^t) = \left(\frac{c}{N^{\al}}\right)^{2R-t}\left(1-\frac{c^t}{N^{t\al}}\right).
\end{align*}
Now, for all $\varepsilon_t>0$, there exists an $M_{1,\varepsilon_t}$ such for all $N>M_{1,\varepsilon_t}$, we can write,
\begin{align}\label{cov-bdd-ind}
  c^{R} -  \varepsilon_t\le c^{2R-t} - \varepsilon_t\leq \frac{\cov\left(W_{N,1}(\sbb_1:H),W_{N,1}(\sbb_2;H)\right)}{N^{-\al(2R-t)}}  \le c^{2R-t} +\varepsilon_t\le c^{2R} +\varepsilon_t,~\text{for all $t\in [R]$}.
\end{align}
Then using \eqref{cov-bdd-ind} and multiplying both side by $N^{-\al(2R-t)}\E(Y_N(\sbb_1:H)Y_N(\sbb_2:H))$ and then taking summation over $t$ and $(\sbb_1,\sbb_2)\in \Mc_t$ and applying Lemma \ref{lem-1}, we can write,
\begin{align*}
    &c^{R} -  \max\{\varepsilon_t:t\in [R]\}\le \frac{\cov\left(W_{N,1}(\sbb_1:H),W_{N,1}(\sbb_2;H)\right)}{N^{-\al(2R-t)}}  \le c^{2R} +\min\{\varepsilon_t:t\in [R]\}\notag\\
    &\Rightarrow\left\{c^{R} -  \max\{\varepsilon_t:t\in [R]\}\right\}\sum_{t=1}^RN^{-\al(2R-t)}\sum_{(\sbb_1,\sbb_2)\in \Mc_t}\E(Y_{N}(\sbb_1:H)Y_{N}(\sbb_2:H))\notag\\
    &\le\sigma^2_{N,1}(H) \le\left\{c^{R} -  \min\{\varepsilon_t:t\in [R]\}\right\}\sum_{t=1}^RN^{-\al(2R-t)}\sum_{(\sbb_1,\sbb_2)\in \Mc_t}\E(Y_{N}(\sbb_1:H)Y_{N}(\sbb_2:H))\notag\\
    &\Rightarrow L(1,\al,\beta,H,\lamb,\Cb) \sum_{t=1}^RN^{\displaystyle -2g_{1,\al,\beta,H}(R) + g_{1,\al,\beta,H}(t_1(\al,\beta;H))}
    \le \sigma^2_{N,1}(H)\notag\\
    &\qquad{}\le U(1,\al,\beta,H,\lamb,\Cb) \sum_{t=1}^RN^{\displaystyle -2g_{1,\al,\beta,H}(R) + g_{1,\al,\beta,H}(t_1(\al,\beta;H))}\notag\\
    &\Rightarrow L(1,\al,\beta,H,\lamb,\Cb) 
    \le \frac{\sigma^2_{N,1}(H)}
    {N^{\displaystyle -2g_{1,\al,\beta,H}(R) + g_{1,\al,\beta,H}(t_1(\al,\beta;H))}} 
    \le U(1,\al,\beta,H,\lamb,\Cb),
\end{align*}
where, $L(1,\al,\beta,H,\lamb,\Cb) = (c^{R} -  \max\{\varepsilon_t:t\in [R]\})\cdot L(\beta,H,\lamb,\Cb)$ and $L(1,\al,\beta,H,\lamb,\Cb) = (c^{2R} -  \min\{\varepsilon_t:t\in [R]\})\cdot U(\beta,H,\lamb,\Cb)$. This completes the proof for the induced case.

In the ego-centric case or for $l=2$ and for $H$ satisfying Assumption \ref{assump-5}, using Lemma \ref{lem-var-ego-1}, there exist an $M_{t,2}$ such that for all $N\geq M_{t,2}$, we can write,
\begin{align}\label{cov-bdd-ego}
L_{t,2}(\al,\beta,H,\lamb,\Cb)
&\le 
\frac{\displaystyle \sum_{(\sbb_1,\sbb_2)\in \Mc_t}
\E\left(W_{N,2}(\sbb_1:H)\cdot W_{N,2}(\sbb_2:H)\right)
\E\left(Y_N(\sbb_1:H)\cdot Y_N(\sbb_2:H)\right)
}{N^{\displaystyle 
2R - t - \al\left(2\tau(H)-\min\{\tau(H),t\}\right)-\beta\left(2T-m(t:H)\right)}
}
\notag\\
&\le 
U_{t,2}(\al,\beta,H,\lamb,\Cb)\notag\\
L(2,\al,\beta,H,\lamb,\Cb)
&\le 
\frac{\displaystyle \sum_{(\sbb_1,\sbb_2)\in \Mc_t}
\E\left(W_{N,2}(\sbb_1:H)\cdot W_{N,2}(\sbb_2:H)\right)
\E\left(Y_N(\sbb_1:H)\cdot Y_N(\sbb_2:H)\right)
}{N^{\displaystyle 
2R - t - \al\left(2\tau(H)-\min\{\tau(H),t\}\right)-\beta\left(2T-m(t:H)\right)}
}
\notag\\
&\le 
U(2,\al,\beta,H,\lamb,\Cb),
\end{align}
where, $L(2,\al,\beta,H,\lamb,\Cb) = \max\{L_{t,2}(\al,\beta,H,\lamb,\Cb):t\in [R]\}$ and $U(2,\al,\beta,H,\lamb,\Cb) = \max\{U_{t,2}(\al,\beta,H,\lamb,\Cb):t\in [R]\}$. Now, multiplying both side by $N^{\displaystyle 
2R - t - \al\left(2\tau(H)-\min\{\tau(H),t\}\right)-\beta\left(2T-m(t:H)\right)}$ in \eqref{cov-bdd-ego}, then subtracting the term $N^{2R-t} f_2^2(p_N:H)\cdot (\E(Y_N(\sbb:H)))^2$ from both side and taking summation over $t\in [R]$, we get the following, for $N\ge \max\{M_{t,2}:t\in [R]\}$
\begin{align*}
 &L(2,\al,\beta,H,\lamb,\Cb) \sum_{t=1}^RN^{\displaystyle -2g_{2,\al,\beta,H}(R) + g_{2,\al,\beta,H}(t_2(\al,\beta;H))}
    \le \sigma^2_{N,2}(H)\notag\\
    &\qquad{}\le U(2,\al,\beta,H,\lamb,\Cb) \sum_{t=1}^RN^{\displaystyle -2g_{2,\al,\beta,H}(R) + g_{2,\al,\beta,H}(t_2(\al,\beta;H))}\notag\\
&\Rightarrow L(2,\al,\beta,H,\lamb,\Cb) 
    \le \frac{\sigma^2_{N,2}(H)}
    {N^{\displaystyle -2g_{2,\al,\beta,H}(R) + g_{2,\al,\beta,H}(t_2(\al,\beta;H))}} 
    \le U(2,\al,\beta,H,\lamb,\Cb),    
\end{align*}
where, $L(2,\al,\beta,H,\lamb,\Cb)$ and $U(2,\al,\beta,H,\lamb,\Cb)$ are two positive constants depending on network formation design and the SBM parameters. This completes the proof of ego-centric case.
\end{proof}

\begin{lemma}\label{SN-hat-tght}
    Suppose that Assumptions \ref{assump-1},\ref{assump-2}, \ref{assump-3} and \ref{assump-4} hold. Let $H$ be a fixed, undirected, simple graph with $R$ vertices and $T$ edges.  $\widehat{S}_{N,l}(H)$ and $S_N(H)$ denote the estimated and population subgraph $H$ count (cf. \eqref{s-all-main}), respectively. Then 
    \begin{enumerate}
        \item[(a)] for induced ($l=1$) and ego-centric ($l=2$) network formation,
        \begin{align*}
            \frac{\widehat{S}_{N,l}(H) - \E\left(\widehat{S}_{N,l}(H)\right)}{N^{-g_{l,\al,\beta,H}(R)}}   = o_P(1).  
        \end{align*}
        For $l = 2$ case, in addition to all previously stated assumptions, the graph $H$ must satisfy Assumption~\ref{assump-5}. 
        \item[(b)] Moreover, the population subgraph count $S_N(H)$ (cf. \eqref{s-all-main}) satisfies the following,
        \begin{align*}
           \frac{S_N(H) - \E(S_N(H))}{\sqrt{\Var(S_{N}(H))}} = O_P(1)\quad\text{and}\quad  \frac{{S}_{N}(H) - \E\left({S}_{N}(H)\right)}{N^{R-T\beta}} = o_P(1).
        \end{align*}
    \end{enumerate}
\end{lemma}
\begin{proof}[Proof of Lemma \ref{SN-hat-tght}]
    First note that
    \begin{align*}
        g_{l,\al,\beta,H}(R) = \begin{cases}
            -R(1-\al) +T\beta,&~\text{if $l=1$,}\\
            -R +\tau(H)\cdot \al + T\beta ,&~\text{if $l=2$}.
        \end{cases}
    \end{align*}
First we will prove part (a). Using Assumption \ref{assump-2}, \ref{assump-3}, and Lemma \ref{lem-1}, we can write the following
\begin{align*}
    &\Var\left(\widehat{S}_{N,1}(H)\right) \leq \sum_{t=1}^R\sum_{(\sbb_1,\sbb_2)\in \mathcal{M}_t}\E\left(W_{N,1}(\sbb_1:H)W_{N,1}(\sbb_2:H)\right)\E\left(Y_N(\sbb_1:H)Y_N({\sbb_2:H})\right)\notag\\
    &=\sum_{t=1}^Rc^{2R-t}N^{-\al(2R-t)}\sum_{(\sbb_1,\sbb_2)\in \mathcal{M}_t}\E\left(Y_N(\sbb_1:H)Y_N({\sbb_2:H})\right)\notag\\
    &\qquad{}\qquad{}\leq c^{2R}\cdot \left(N^{\beta}\max_{i,j\in [K]}\pi_{N,i,j}\right)^{2T}\sum_{t=1}^R N^{-2g_{1,\al,\beta,H}(R) + g_{1,\al,\beta,H}(t)}\notag\\
    &\Rightarrow~\frac{\Var\left(\widehat{S}_{N,1}(H)\right)}{N^{-2g_{1,\al,\beta,H}(R)}} \leq c^{2R}\left(N^{\beta}\max_{i,j\in [K]}\pi_{N,i,j}\right)^{2T}\sum_{t=1}^R N^{g_{l,\al,\beta,H}(t)}\to 0,~\text{as $N\rai$.}
\end{align*}
Therefore, from the above expression we can write,
\begin{align*}
    \frac{\widehat{S}_{N,1}(H) - \E\left(\widehat{S}_{N,1}(H)\right)}{N^{-g_{1,\al,\beta,H}(R)}}   = o_P(1).
\end{align*}
Next, in the ego-centric case ($l=2$), using Assumption \ref{assump-2}, \ref{assump-3} and Lemma \ref{lem-var-ego-1}, we can write, 
\begin{align*}
    \Var\left(\widehat{S}_{N,2}(H)\right) &\leq \sum_{t=1}^R\sum_{(\sbb_1,\sbb_2)\in \mathcal{M}_t}\E\left(W_{N,2}(\sbb_1:H)W_{N,2}(\sbb_2:H)\right)\E\left(Y_N(\sbb_1:H)Y_N({\sbb_2:H})\right)\notag\\
    &\leq c^{2\tau(H)}|\mathcal{D}(H)|^2\left(N^{\beta}\max_{i,j\in [K]}\pi_{N,i,j}\right)^{2T}\sum_{t=1}^R N^{-2g_{2,\al,\beta,H}(R) + g_{2,\al,\beta,H}(t)}\notag\\
    \Rightarrow~\frac{\Var\left(\widehat{S}_{N,2}(H)\right)}{N^{-2g_{2,\al,\beta,H}(R)}} &\leq c^{2\tau(H)}|\mathcal{D}(H)|^2\left(N^{\beta}\max_{i,j\in [K]}\pi_{N,i,j}\right)^{2T}\sum_{t=1}^R N^{g_{2,\al,\beta,H}(t)}\to 0,~\text{as $N\rai$.}
\end{align*}
Therefore, from the above expression we can write,
\begin{align*}
    \frac{\widehat{S}_{N,2}(H) - \E\left(\widehat{S}_{N,2}(H)\right)}{N^{-g_{2,\al,\beta,H}(R)}}   = o_P(1).
\end{align*}
This completes the proof of part (a). Next, we prove part (b).
Note that,
\begin{align*}
    \Var(S_N(H)) &= \sum_{t=2}^R\sum_{(\sbb_1,\sbb_2)\in \mathcal{M}_t}\cov\left(Y_{N}(\sbb_1:H),Y_N(\sbb_2:H)\right)\notag\\
    &= \sum_{t=2}^R\sum_{(\sbb_1,\sbb_2)\in \mathcal{M}_t}\left\{\E\left(Y_{N}(\sbb_1:H)Y_N(\sbb_2:H)\right) - \left(\E\left(Y_{N}(\sbb_1:H)\right)\right)^2\right\}. 
\end{align*}
Using Assumption \ref{assump-2} and Lemma \ref{lem-1}, there exist and $M$ such that for all $N>M$, we can write,
\begin{align}\label{var-pop}
    L(\beta,H,\lamb,\Cb)\sum_{t=2}^RN^{2R-t-\beta(2T-m(t:H))} \leq \Var(S_N(H)) \leq U(\beta,H,\lamb,\Cb)\sum_{t=2}^RN^{2R-t-\beta(2T-m(t:H))} 
\end{align}
where, $L(\beta,H,\lamb,\Cb) = \min_t\{L_t(\beta,H,\lamb,\Cb)\}$ and $U(\beta,H,\lamb,\Cb) = \max_t\{U_t(\beta,H,\lamb,\Cb)\}$ are positive constants (see Lemma \ref{lem-1})
Next, we apply Theorem \ref{thm-ruc} over the pivotal quantity,
$$
T_N(H) = \frac{S_N(H) - \E(S_N(H))}{\sqrt{\Var(S_{N}(H))}},
$$
and we get the following,
\begin{align}\label{d1-pop}
    &d_1\left(T_N(H,~Z)\right) \leq \frac{1}{[\Var(S_{N}(H)]^{3/2}}\sum_{\sbb_1\in \nsr}\sum_{\sbb_2,\sbb_3\in N(\sbb_1)}\left(\E\left|\prod_{i=1}^3\widetilde{Y}_N(\sbb_i:H)\right|+\E\left|\prod_{i=1}^2\widetilde{Y}_N(\sbb_i:H)\right|\E|\widetilde{Y}_N(\sbb_3:H)|\right)\notag\\
    &\leq \frac{1}{[\Var(S_{N}(H)]^{3/2}}\sum_{k_1,k_2=2}^R\sum_{(\sbb_1,\sbb_2,\sbb_3)\in \Mc_{k_1,k_2}}\left(\E\left|\prod_{i=1}^3\widetilde{Y}_N(\sbb_i:H)\right|+\E\left|\prod_{i=1}^2\widetilde{Y}_N(\sbb_i:H)\right|\E|\widetilde{Y}_N(\sbb_3:H)|\right)\notag\\
    &=\frac{1}{[\Var(S_{N}(H)]^{3/2}}\left(P_{N,1} + P_{N,2}\right),~\text{(say)}
\end{align}
where, $\widetilde{Y}_N(\sbb:H) = Y_N(\sbb:H) - \E(Y_N(\sbb:H))$. Note that, by similar argument as in \eqref{exp-W} and using \eqref{exp-Y},
\begin{align}\label{P1P2}
    &\E\left|\prod_{i=1}^3\widetilde{Y}_N(\sbb_i:H)\right| \leq 8\cdot\left(N^{\beta}\max_{i,j\in [K]}\pi_{N,i,j}\right)^{3T}N^{-\beta(3T - m(k_1:H) - m(k_2:H))}\notag\\
&\E\left|\prod_{i=1}^2\widetilde{Y}_N(\sbb_i:H)\right|\E|\widetilde{Y}_N(\sbb_3:H)|\leq 8\cdot\left(N^{\beta}\max_{i,j\in [K]}\pi_{N,i,j}\right)^{3T}N^{-\beta(3T - m(k_1:H) - m(k_2:H))}
\end{align}
Then from the expression in the the r.h.s. of \eqref{d1-pop} and  combining \eqref{var-pop} and \eqref{P1P2}, there exist an $M$ such that for all $N>M$, we can write the following
\begin{align*}
    &d_1\left(T_N(H),Z\right) \leq O\left(\frac{1}{N^{\frac{1}{2}(t(\beta;H) -\beta\cdot m(t(\beta;H):H))}}\right) = O\left(\frac{1}{\sqrt{N^{t(\beta:H)(1-m(H)\cdot \beta)}}}\right).
\end{align*}
where, $t(\beta;H) = \argmaxA_{t\in \{2,\ldots,R\}} (t - \beta\cdot m(t:H))$  and $m(H) = m(t(\beta;H):H)/t(\beta;H)$.
Thus, for $\beta <1/m(H)$, $T_N(H) = O_P(1)$. This completes the proof of first part of (c). To prove the second assertion in part (c) of Lemma \ref{SN-hat-tght}. Note that for any graph fixed graph $H$ and for any $\beta<1/m(H)$, from the expression \eqref{var-pop}, we can write,
\begin{align*}
    \Var\left(\frac{S_N(H)}{N^{R-T\beta}}\right) \leq U(\beta,H,\lamb,\Cb)\frac{1}{N^{2R-2T\beta}}\sum_{t=2}^RN^{2R-t-\beta(2T-m(t:H))}\to 0,\quad\text{as $N\to \infty$}.
\end{align*}
Hence, for $0\leq \beta<1/m(H)$, we can claim that, $(S_N(H) - \E(S_N(H)))/N^{R-T\beta} = o_p(1)$.
For strictly balanced graph $H$, $m(H) = T/R$ and for $\beta = R/T$, from the expression \eqref{var-pop}, we can write,
\begin{align*}
    \Var\left(S_N(H)\right) \leq U(\beta,H,\lamb,\Cb)\left(1+ O\left(\frac{1}{N}\right)\right).
\end{align*}
Thus, using \ref{assump-2}, for $N\rai$, we can claim $\E(S_N) = O(1)$ and $\Var(S_N(H))=O(1)$. In addition note that, $\E(S_N(H) - \E(S_N(H))) = 0$. Then by Chebyshev inequality, we can write,
$$
\pr(|S_N(H) - \E(S_N(H))| > M)
\le \frac{\Var(S_N(H))}{M^2} = \frac{O(1)}{M^2}.
$$
Hence, for all $M>0$, we can write, $ \sup_N \pr(|S_N(H)| > M)\le \frac{O(1)}{M^2}$.
Then we can claim that
$$
\sup_N \pr(|S_N(H)| > M) < \varepsilon. 
$$
This completes the proof.

\end{proof}

\begin{lemma}\label{var-exp-bdd}
Suppose that Assumptions \ref{assump-1}, \ref{assump-2}, \ref{assump-3} and \ref{assump-4} hold. Let $H$ be a fixed, undirected, simple, and connected graph with $V(H) = \{1,\ldots,R\}$ and the edge set $E(H) \subseteq [V(H)]^2$, where $R\geq 2$ and $|E(H)| = T$. It is given that $H$ is a strictly balanced graph and 
    $$
    \sigma^2_{N,l}(H) = \Var\left(\widehat{S}_{N,l}(H) - f_l(p_N:H)\cdot S_N(H)\right)\quad\text{for $l=1,2$}.
    $$
    Then 
    \begin{enumerate}
        \item[(i)] If $(\al,\beta) = (0,R/T)$, then 
        \begin{align*}
            \sigma^2_{N,l}(H) \to f_l(c:H)(1-f_l(c:H))\cdot \kappa_N(\Cb, \lamb:H)\quad\text{as $N\to\infty$.}
        \end{align*}
        \item[(ii)] If $(\al,\beta)\in F_l(H)$, defined in \eqref{F1F2-set} and in addition for $l=2$, the graph $H$ satisfies \ref{assump-5}, then
        \begin{align*}
            \sigma^2_{N,l}(H) \to \phi_l(c:H)\cdot \kappa(\Cb,\lamb:H)\quad\text{as $N\to\infty$.}
        \end{align*}
        where $\phi_l(c:H)$ defined in \eqref{exp-B-1-B-2} and $\kappa(\Cb,\lamb:H)$ defined in \eqref{kap-def}.
        \item[(iii)] For $\tau(H) = 1$, and $\al = 1$, in the ego-centric sub-network formation ($l=2$), 
        \begin{align*}
            &\sigma^2_{N,2}(H)\notag\\
            &= N^{2(R-1)(1-\beta)}\left\{c\sum_{\substack{\ub\in [K]^R\\(v_2,\ldots,v_R)\in [K]^{R-1}}}\prod_{j=2}^RN^{2\beta}\pi_{N,u_1,u_j}\pi_{N,u_1,v_j}\prod_{j=1}^R\la_{N,u_j}\prod_{j=2}^R\la_{N,v_j}\left(1 +O\left(\frac{1}{N}\right)\right) + O\left(\frac{1}{N}\right)\right\}.
        \end{align*}
    \end{enumerate}
\end{lemma}

\begin{proof}[Proof of Lemma \ref{var-exp-bdd}]
    For part (a), when $\al=0$ and $\beta = R/T$ and graph $H$ is strictly balanced with $R$ vertices and $T$ edges. In this case $p_N = c \in (0,1)$. Then for $l=1,2$, we can write,
    \begin{align*}
        &\sigma^2_{N,l}(H) = \Var\left(\widehat{S}_{N,l}(H) - f_l(c:H)\cdot S_N(H)\right) \notag\\
        &= \Var\left(\sum_{\sbb\in \nsr}\left(W_{N,l}(\sbb:H) - f_l(c:H)\right)Y_N(\sbb:H)\right)\notag\\
        &= \sum_{t=1}^R\sum_{(\sbb_1,\sbb_2)\in \mathcal{M}_t}\cov\left(W_{N,l}(\sbb_1:H),W_{N,l}(\sbb_2:H)\right)\cdot \E(Y_N(\sbb_1:H)\cdot Y_{N}(\sbb_2:H))\notag\\
        &= \sum_{t=1}^{R-1}\sum_{(\sbb_1,\sbb_2)\in \mathcal{M}_t}\cov\left(W_{N,l}(\sbb_1:H),W_{N,l}(\sbb_2:H)\right)\cdot \E(Y_N(\sbb_1:H)\cdot Y_{N}(\sbb_2:H)) \notag\\
        &+ \sum_{(\sbb_1,\sbb_2)\in \mathcal{M}_R}\cov\left(W_{N,l}(\sbb_1:H),W_{N,l}(\sbb_2:H)\right)\cdot \E(Y_N(\sbb_1:H)\cdot Y_{N}(\sbb_2:H))\notag\\
        &= \sum_{t=1}^{R-1}\sum_{(\sbb_1,\sbb_2)\in \mathcal{M}_t}\cov\left(W_{N,l}(\sbb_1:H),W_{N,l}(\sbb_2:H)\right)\cdot \E(Y_N(\sbb_1:H)\cdot Y_{N}(\sbb_2:H)) \notag\\
        &+ \sum_{\sbb\in \nsr}\Var\left(W_{N,l}(\sbb:H)\right)\cdot \E(Y_N(\sbb:H))\notag\\
        &= \sum_{t=1}^{R-1}\sum_{(\sbb_1,\sbb_2)\in \mathcal{M}_t}\cov\left(W_{N,l}(\sbb_1:H),W_{N,l}(\sbb_2:H)\right)\cdot \E(Y_N(\sbb_1:H)\cdot Y_{N}(\sbb_2:H)) \notag\\
        &+ N^{R-T\beta}f_l(c:H)(1-f_l(c:H))\sum_{\ub\in [K]^R}\prod_{\{i,j\}\in E(H)}N^{\beta}\pi_{N,u_i,u_j} \left(\prod_{i=1}^R\la_{N,u_i}+O\left(\frac{1}{N}\right)\right)\notag\\
        &= \sum_{t=1}^{R-1}\sum_{(\sbb_1,\sbb_2)\in \mathcal{M}_t}\cov\left(W_{N,l}(\sbb_1:H),W_{N,l}(\sbb_2:H)\right)\cdot \E(Y_N(\sbb_1:H)\cdot Y_{N}(\sbb_2:H)) \notag\\
        &+ f_l(c:H)(1-f_l(c:H))\sum_{\ub\in [K]^R}\prod_{\{i,j\}\in E(H)}N^{\beta}\pi_{N,u_i,u_j} \left(\prod_{i=1}^R\la_{N,u_i}+O\left(\frac{1}{N}\right)\right)\notag\\
        &= J_{N,l,1} + J_{N,l,2},~\text{(say)}.
    \end{align*}
Next, we will show that $J_{N,l,1}$ converges to zero as $N\rai$. Using the facts $\al=0$ and $\beta=R/T$,  we can write 
\begin{align}\label{J1}
    &\left|J_{N,l,1}\right|\leq \sum_{t=1}^{R-1}\sum_{(\sbb_1,\sbb_2)\in \mathcal{M}_t}\left|\cov\left(W_{N,l}(\sbb_1:H),W_{N,l}(\sbb_2:H)\right)\cdot \E(Y_N(\sbb_1:H)\cdot Y_{N}(\sbb_2:H))\right|\notag\\
    &\leq \sum_{t=1}^{R-1}\sum_{(\sbb_1,\sbb_2)\in \mathcal{M}_t}\left|\E\left(W_{N,l}(\sbb_1:H)\cdot W_{N,l}(\sbb_2:H)\right)\cdot \E(Y_N(\sbb_1:H)\cdot Y_{N}(\sbb_2:H))\right|\notag\\
    &\leq \sum_{t=1}^{R-1}\sum_{(\sbb_1,\sbb_2)\in \mathcal{M}_t}\E(Y_N(\sbb_1:H)\cdot Y_{N}(\sbb_2:H))\leq \sum_{t=1}^{R}\left(N^{\beta}\max_{i,j\in [K]^R} \pi_{N,i,j}\right)^{2T}\cdot N^{2R-t - \beta(2T-m(t:H))}\notag\\
    & = \left(N^{\beta}\max_{i,j\in [K]^R} \pi_{N,i,j}\right)^{2T}\sum_{t=1}^{R-1}N^{-t +(R/T)m(t:H)}.
\end{align}
In the above expression $t/m(t:H) > R/T$  for $t\in \{1,\ldots,R-1\}$ since $H$ is a strictly balanced graph. Thus, using Assumption \ref{assump-2}, we can claim that $|J_{N,l,1}|\to 0$, as $N\rai$. Also note that using Assumptions \ref{assump-2} and \ref{assump-3} we can say that
\begin{align*}
    J_{N,l,2} &\to f_l(c:H)(1-f_l(c:H))\sum_{\ub\in [K]^R}\prod_{\{i,j\}\in E(H)}c_{u_i,u_j} \cdot\prod_{i=1}^R\la_{u_i}\notag\\ 
    &=f_l(c:H)(1-f_l(c:H))\cdot \kappa(\Cb,\lamb:H),~\text{as $N\rai$ for $l=1,2$,}
\end{align*}
where $\kappa(\Cb,\lamb:H)$ defined in \eqref{kap-def}.
Hence, combining both the term $J_{N,l,1}$ and $J_{N,l,2}$ as $N\rai$, we can say that
\begin{align*}
    \sigma^2_{N,l}(H) \to f_l(c:H)(1-f_l(c:H))\cdot \kappa(\Cb,\lamb:H),~\text{for $l=1,2$.}
\end{align*}
This completes the proof of part (a). For part (b), $\al\in (0,1)$. In the induced case, it is assumed that the sparsity levels $(\al,\beta)$ are within the boundary region
$$
F_{1}(H) = \left\{(\al,\beta): 0<\al<1,~\beta = \frac{R\cdot(1-\al)}{T}\right\},\quad\text{(see \eqref{F1F2-set}).}
$$
In the ego-centric case, $(\al,\beta)$ are within the boundary region 
\begin{align*}
F_2(H) &= \left\{(\al,\beta): 0 <\al < \widetilde{\al}_H,~\beta = \left(1-\frac{\tau(H)\cdot\al}{R}\right)\cdot\frac{R}{T}\right\},
\end{align*} 
where $\widetilde{\al}_H$ is defined \eqref{al-tild-def} and the graph $H$ satisfies \ref{assump-5}. Then $\sigma^2_{N,l}(H)$ has the following expression,
\begin{align*}
    &\sigma^2_{N,l}(H) = \sum_{t=1}^R\sum_{(\sbb_1,\sbb_2)\in \mathcal{M}_t}\cov\left(W_{N,l}(\sbb_1:H),W_{N,l}(\sbb_2:H)\right)\cdot \E(Y_N(\sbb_1:H)\cdot Y_{N}(\sbb_2:H))\notag\\
    &= \sum_{t=1}^{R-1}\sum_{(\sbb_1,\sbb_2)\in \mathcal{M}_t}\cov\left(W_{N,l}(\sbb_1:H),W_{N,l}(\sbb_2:H)\right)\cdot \E(Y_N(\sbb_1:H)\cdot Y_{N}(\sbb_2:H))\notag\\
    &\qquad{} +\sum_{(\sbb_1,\sbb_2)\in \mathcal{M}_R}\cov\left(W_{N,l}(\sbb_1:H),W_{N,l}(\sbb_2:H)\right)\cdot \E(Y_N(\sbb_1:H)\cdot Y_{N}(\sbb_2:H))\notag\\
    &= \sum_{t=1}^{R-1}\sum_{(\sbb_1,\sbb_2)\in \mathcal{M}_t}\cov\left(W_{N,l}(\sbb_1:H),W_{N,l}(\sbb_2:H)\right)\cdot \E(Y_N(\sbb_1:H)\cdot Y_{N}(\sbb_2:H))\notag\\
    &\qquad{} +\sum_{\sbb\in \nsr}\Var\left(W_{N,l}(\sbb:H)\right)\cdot \E(Y_N(\sbb:H))\notag\\
    &= \sum_{t=1}^{R-1}\sum_{(\sbb_1,\sbb_2)\in \mathcal{M}_t}\cov\left(W_{N,l}(\sbb_1:H),W_{N,l}(\sbb_2:H)\right)\cdot \E(Y_N(\sbb_1:H)\cdot Y_{N}(\sbb_2:H))\notag\\
    &\qquad{}+ N^{R-T\beta}f_l(p_N:H)(1-f_l(p_N:H))\sum_{\ub\in [K]^R}\prod_{\{i,j\}\in E(H)}N^{\beta}\pi_{N,u_i,u_j} \left(\prod_{i=1}^R\la_{N,u_i}+O\left(\frac{1}{N}\right)\right)\notag\\
    &= B_{N,l,1} + B_{N,l,2},~\text{(say)}.
\end{align*}
First, we will show that $B_{N,l,1}$ converges to zero as $N\rai$. Using the fact $(\al,\beta)\in F_1(H)$, we can write
\begin{align*}
    &|B_{N,1,1}| \leq \sum_{t=1}^{R-1}\sum_{(\sbb_1,\sbb_2)\in \mathcal{M}_t}|\cov\left(W_{N,1}(\sbb_1:H),W_{N,1}(\sbb_2:H)\right)\cdot \E(Y_N(\sbb_1:H)\cdot Y_{N}(\sbb_2:H))|\notag\\
    &\leq \sum_{t=1}^{R-1}\sum_{(\sbb_1,\sbb_2)\in \mathcal{M}_t}\E\left(W_{N,1}(\sbb_1:H)\cdot W_{N,1}(\sbb_2:H)\right)\cdot \E(Y_N(\sbb_1:H)\cdot Y_{N}(\sbb_2:H))\notag\\
    &\leq a_{\max}(H)\cdot c^{2R} N^{-g_{1,\al,\beta,H}(R)}\sum_{t=1}^{R-1}N^{{}g_{1,\al,\beta,H}(t)},
\end{align*}
and in case of $l=2$, when $(\al,\beta)\in F_2(H)$ and $H$ satisfy Assumption \ref{assump-5}, we can write
\begin{align*}
    &|B_{N,2,1}| \leq \sum_{t=1}^{R-1}\sum_{(\sbb_1,\sbb_2)\in \mathcal{M}_t}|\cov\left(W_{N,2}(\sbb_1:H),W_{N,2}(\sbb_2:H)\right)\cdot \E(Y_N(\sbb_1:H)\cdot Y_{N}(\sbb_2:H))|\notag\\
    &\leq \sum_{t=1}^{R-1}\sum_{(\sbb_1,\sbb_2)\in \mathcal{M}_t}\E\left(W_{N,2}(\sbb_1:H)\cdot W_{N,2}(\sbb_2:H)\right)\cdot \E(Y_N(\sbb_1:H)\cdot Y_{N}(\sbb_2:H))\notag\\
    &\leq a_{\max}(H)\cdot c^{2\tau(H)}\left|\mathcal{D}(H)\right|^2 N^{-g_{2,\al,\beta,H}(R)}\sum_{t=1}^{R-1}N^{{}g_{2,\al,\beta,H}(t)} 
\end{align*}
where, $a_{\max}(H)$ is a finite positive constant (defined in \eqref{V-2}) and $|\mathcal{D}(H)|$ is the size of the collection of all possible vertex covers of $H$, which is a finite quantity. Next, when $(\alpha,\beta) \in F_l(H)$, we have $g_{l,\alpha,\beta,H}(R)=0$ for $l=1,2$. Moreover, for $(\alpha,\beta) \in F_l(H)$, $l=1,2$, and $t \in \{1,\ldots,R-1\}$, writing $\beta$ in terms of $\al$ we have the following.
\begin{align*}
    g_{1,\alpha,\beta,H}(t)
    &= (1-\alpha)\left(-\frac{t}{m(t:H)} + \frac{R}{T}\right)m(t:H),\quad\text{and}\\
    g_{2,\alpha,\beta,H}(t)
    &= m(t:H)\left\{\left(-\frac{t}{m(t:H)} + \frac{R}{T}\right)
    + \left(\frac{\min(\tau(H),t)}{m(t:H)} - \frac{\tau(H)}{T}\right)\alpha\right\}.
\end{align*}
Since $\alpha \in (0,1)$ in the induced case and
$\alpha \in (0,\widetilde{\alpha}_H)$ in the egocentric case, it follows that $g_{l,\alpha,\beta,H}(t) < 0$ for all $t \in \{1,\ldots,R-1\}$. Thus, $B_{N,l,1} \to 0$ as $N\rai$. Using Assumptions \ref{assump-2} and \ref{assump-3}, and the fact $(\al,\beta)\in F_l(H)$ we can write
\begin{align*}
    B_{N,l,2} &= N^{R-T\beta}f_l(p_N:H)(1-f_l(p_N:H))\sum_{\ub\in [K]^R}\prod_{\{i,j\}\in E(H)}N^{\beta}\pi_{N,u_i,u_j} \left(\prod_{i=1}^R\la_{N,u_i}+B\left(\frac{1}{N}\right)\right) \notag\\
    &\to  \phi_{l}(c:H)\cdot \kappa\left(\Cb,\lamb:H\right),~\text{for $l=1,2$,}
\end{align*}
where $\kappa(\Cb,\lamb:H)$ defined in \eqref{kap-def}.
Thus, combining the term $B_{N,l,1}$ and $B_{N,l,2}$ as $N\to\infty$, we can write 
\begin{align*}
    \sigma^2_{N,l}(H)\to \phi_l(c:H)\cdot \kappa(\Cb,\lamb:H).
\end{align*}
This completes the proof of part (b). For part (c), $\al = 1$ and $\beta\in [0,1)$. In this case $\tau(H) = 1$, so from Lemma \ref{lem-4} we can claim that $H=\kl_{1,R-1}$. Then $\sigma_{N,2}(\kl_{1,R-1})$ has the following expression 
\begin{align*}
    &\sigma^2_{N,2}(\kl_{1,R-1}) =\Var\left(\widehat{S}_{N,2}(\kl_{1,R-1}) - f_2(p_N:\kl_{1,R-1})\cdot S_N(\kl_{1,R-1})\right)\notag\\
    &= \sum_{t=2}^{R}\sum_{(\sbb_1,\sbb_2)\in \mathcal{M}_t}\cov\left(W_{N,2}(\sbb_1:\kl_{1,R-1}),W_{N,2}(\sbb_2:\kl_{1,R-1})\right)\cdot \E(Y_N(\sbb_1:\kl_{1,R-1})\cdot Y_{N}(\sbb_2:\kl_{1,R-1}))\notag\\
    &\qquad{} +\sum_{(\sbb_1,\sbb_2)\in \mathcal{M}_1}\cov\left(W_{N,2}(\sbb_1:\kl_{1,R-1}),W_{N,2}(\sbb_2:\kl_{1,R-1})\right)\cdot \E(Y_N(\sbb_1:\kl_{1,R-1})\cdot Y_{N}(\sbb_2:\kl_{1,R-1}))\notag\\
    &= \sum_{t=2}^{R}\sum_{(\sbb_1,\sbb_2)\in \mathcal{M}_t}\cov\left(W_{N,2}(\sbb_1:\kl_{1,R-1}),W_{N,2}(\sbb_2:\kl_{1,R-1})\right)\cdot \E(Y_N(\sbb_1:\kl_{1,R-1})\cdot Y_{N}(\sbb_2:\kl_{1,R-1}))\notag\\
    &\qquad{} +\sum_{(\sbb_1,\sbb_2)\in \mathcal{M}_1}\cov\left(W_{N,2}(\sbb_1:\kl_{1,R-1}),W_{N,2}(\sbb_2:\kl_{1,R-1})\right)\cdot \E(Y_N(\sbb_1:\kl_{1,R-1}))\cdot\E(Y_{N}(\sbb_2:\kl_{1,R-1}))\notag\\
    &= Q_{N,1} +Q_{N,2},~\text{(say).}
\end{align*}
Using Assumptions \ref{assump-2} and \ref{assump-3}, first analyze the term $Q_{N,1}$, that is,
\begin{align*}
    &|Q_{N,1}| = \sum_{t=2}^{R}\sum_{(\sbb_1,\sbb_2)\in \mathcal{M}_t}|\cov\left(W_{N,2}(\sbb_1:\kl_{1,R-1}),W_{N,2}(\sbb_2:\kl_{1,R-1})\right)\cdot \E(Y_N(\sbb_1:\kl_{1,R-1})\cdot Y_{N}(\sbb_2:\kl_{1,R-1}))|\notag\\
    &\leq \sum_{t=2}^{R}\sum_{(\sbb_1,\sbb_2)\in \mathcal{M}_t}\E\left(W_{N,2}(\sbb_1:\kl_{1,R-1})W_{N,2}(\sbb_2:\kl_{1,R-1})\right)\cdot \E(Y_N(\sbb_1:\kl_{1,R-1})\cdot Y_{N}(\sbb_2:\kl_{1,R-1}))\notag\\
    &\leq c\left(N^{\beta}\max_{i,j\in [K]^R}\pi_{N,i,j}\right)^{2(R-1)}\sum_{t=2}^{R}N^{(2R-t-1)(1-\beta)} \leq c(R-1)\left(N^{\beta}\max_{i,j\in [K]^R}\pi_{N,i,j}\right)^{2(R-1)}N^{(2R-3)(1-\beta)}\notag\\
    &\Rightarrow Q_{N,1} = O\left(N^{(2R-3)(1-\beta)}\right).
\end{align*}
Next, using Assumptions \ref{assump-2} and \ref{assump-3} in the expression $Q_{N,2}$ we get the following expression,
\begin{align*}
   & Q_{N,2}\notag\\
   &= \sum_{(\sbb_1,\sbb_2)\in \mathcal{M}_1}\cov\left(W_{N,2}(\sbb_1:\kl_{1,R-1}),W_{N,2}(\sbb_2:\kl_{1,R-1})\right)\cdot \E(Y_N(\sbb_1:\kl_{1,R-1}))\cdot\E(Y_{N}(\sbb_2:\kl_{1,R-1}))\notag\\
   &= \sum_{(\sbb_1,\sbb_2)\in \mathcal{M}_1}\left\{ \cov\left(W_{s_{1,1}} +(1-W_{s_{1,1}})\prod_{j=2}^RW_{s_{1,j}},W_{s_{1,1}} +(1-W_{s_{1,1}})\prod_{j=2}^RW_{s_{2,j}}\right)\right.\notag\\
   &\qquad{}\qquad{}\left.\times\prod_{j=2}^R\pi_{N,\al_{s_{1,1}}, \al_{s_{1,j}}}\pi_{N,\al_{s_{1,1}}, \al_{s_{2,j}}} + \right.\notag\\
   &\left.\qquad{} + \cov\left(W_{s_{1,1}} +(1-W_{s_{1,1}})\prod_{j=2}^RW_{s_{1,j}},W_{s_{2,1}} +(1-W_{s_{2,1}})W_{s_{1,1}}\prod_{j=3}^RW_{s_{2,j}}\right)\prod_{j=2}^R\pi_{N,\al_{s_{1,1}}, \al_{s_{1,j}}}\right.\notag\\
   &\left.\qquad{}\qquad{}\times \pi_{\al_{s_{2,1}},\al_{s_{1,1}}}\prod_{j=3}^R\pi_{\al_{s_{2,1}},\al_{s_{2,j}}}\right\}\notag\\
   &= \sum_{(\sbb_1,\sbb_2)\in \mathcal{M}_1} \cov\left(W_{s_{1,1}} +(1-W_{s_{1,1}})\prod_{j=2}^RW_{s_{1,j}},W_{s_{1,1}} +(1-W_{s_{1,1}})\prod_{j=2}^RW_{s_{2,j}}\right)\notag\\
   &\qquad{}\qquad{}\times\prod_{j=2}^R\pi_{N,\al_{s_{1,1}}, \al_{s_{1,j}}}\pi_{N,\al_{s_{1,1}}, \al_{s_{2,j}}} + O\left(\frac{1}{N}\right)\notag\\
   &= c\cdot N^{2(R-1)(1-\beta)}\sum_{\substack{\ub\in [K]^R\\
   (v_2,\ldots,v_R)\in [K]^{R-1}}}\prod_{j=2}^RN^{2\beta}\pi_{u_1,u_j}\pi_{u_1,v_j}\prod_{j=1}^R\la_{N,u_j}\prod_{j=2}^R\la_{N,v_j}\left(1 +O\left(\frac{1}{N}\right)\right) + O\left(\frac{1}{N}\right).
\end{align*}
Therefore, combining the term $Q_{N,1}$ and $Q_{N,2}$ we get,
\begin{align*}
    &\sigma^2_{N,2}(\kl_{1,R-1})\notag\\
    & = N^{2(R-1)(1-\beta)}\left\{c\sum_{\substack{\ub\in [K]^R\\(v_2,\ldots,v_R)\in [K]^{R-1}}}\prod_{j=2}^RN^{2\beta}\pi_{u_1,u_j}\pi_{u_1,v_j}\prod_{j=1}^R\la_{N,u_j}\prod_{j=2}^R\la_{N,v_j}\left(1 +O\left(\frac{1}{N}\right)\right) + O\left(\frac{1}{N}\right)\right\}.
\end{align*}
This completes the proof of part (c). 
\end{proof}

\begin{lemma}\label{WL-TV}
    Suppose $\{Z_N:N\leq 1\}$ is a sequence of random variables, where $Z_N\sim \text{Poisson}(\la_N)$ and $Z$ is another random variable such that $Z \sim \text{Poisson}(\la)$. Assume that $Z_N\darw Z$, as $N\to\infty$. Then 
    $$
    d_{TV} (Z_N,Z) \to 0, ~\text{as $N\rai$}.
    $$
\end{lemma}
\begin{proof}[Proof of Lemma \ref{WL-TV}]
Fix $N\geq 1$. Let $\hat{Z}\eqd {Z}$ and $\hat{Z}_N \eqd {Z}_N$. Define two random variables $X_N^{+} = \hat{Z}_N - \hat{Z}$ if $\lambda_N>\lambda$, and $X_N^{-} = \hat{Z} - \hat{Z}_N$ if $\lambda_N<\lambda$. Here $X^{+}_N$ and $X_N^{-}$ are independent of $\hat{Z}$ and $X^{+}_N\sim \text{Poisson}(\la_N-\la)$ and $X^{-}_N\sim \text{Poisson}(\la-\la_N)$. Then $(\hat{Z},\hat{Z}_N)$ is a coupling of two random variable $Z$ and $Z_N$.  Let $A$ be an event $[\hat{Z}_N = \hat{Z}]$. Then for all $B\in \mathcal{B}(\mathbb{R})$, we can write,
\begin{align}\label{ZneqZN}
    &\pr\left(Z_N\in B\right) - \pr\left(Z\in B\right) = \pr\left(\hat{Z}_N \in B\right) - \pr(\hat{Z}\in B)= \pr\left(\hat{Z}_N \in B, A^c\right) - \pr\left(\hat{Z}\in B, A^c\right)\notag\\
    &\leq \pr(A^c)= \pr(\hat{Z}_N \neq \hat{Z}) = \pr\left(\left|\hat{Z}_N - \hat{Z}\right|>0\right). 
\end{align}
Next, when $\la_N>\la$, we can write, 
\begin{align}\label{ZneqZN-1}
    \pr\left(\left|\hat{Z}_N - \hat{Z}\right|>0\right) = \pr\left(X^{+}_{N}>0\right) = 1- \pr\left(X^{+}_{N} = 0\right) = 1-e^{-(\la_N-\la)},
\end{align}
and when $\la_N<\la$, we can write, 
\begin{align}\label{ZneqZN-2}
    \pr\left(\left|\hat{Z}_N - \hat{Z}\right|>0\right) = \pr\left(X^{-}_{N}>0\right) = 1- \pr\left(X^{-}_{N} = 0\right) = 1-e^{-(\la-\la_N)}.
\end{align}
Therefore, combining \eqref{ZneqZN-1} and \eqref{ZneqZN-2}, from \eqref{ZneqZN}, for all $B\in \mathcal{B}(\mathbb{R})$, we can write,
\begin{align*}
    \left|\pr\left(Z_N\in B\right) - \pr\left(Z\in B\right)\right|\leq \pr(|\hat{Z}_N - \hat{Z}|>0) = 1-e^{-|\la_N-\la|}.
\end{align*}
Hence, taking supremum over $B\in \mathcal{B}(\mathbb{R})$, from \eqref{TV-dist}, we can say that for any $N\geq 1$,
\begin{align*}
    d_{TV}(Z_N,Z)= \sup_{B\in \mathcal{B}(\mathbb{R})}\bigg|\pr\left(Z_N\in B\right) - \pr(Z\in B)\bigg| \leq 1 - e^{-|\la_N - \la|}.
\end{align*}
Since, $Z_N \darw Z$, as $N\rai$, thus $\la_N = \E(Z_N) \to E(Z)=\la$. Hence, using the previous inequality, and letting $N\rai$, we can claim that $d_{TV}(Z_N,Z) \to 0$. This completes the proof.  
\end{proof}

\begin{lemma}\label{lem-4}
Let $H$ be a simple, connected and undirected graph with vertex set 
$V(H) = \{1,\dots,R\}$. Then the size of minimum vertex cover of $H$, $\tau(H) = 1$ (cf. \eqref{def-tauH}) if and only if $H=\kl_{1,R-1}$ or $R$-star graph. 
\end{lemma}

\begin{proof}
Assume that $\tau(H)=1$ and $\{v\}$ is the minimum vertex cover of $H$. Then, every edge of $H$ must have $v$ as an endpoint. Thus, no edge connects two vertices in $V(H)\setminus \{v\}$. Since $H$ is connected and has at least one edge, each vertex in $V(H)\setminus\{v\}$ must be adjacent to $v$. Therefore, $H$ consists of the center vertex $v$ with all other vertices attached only to $v$, {\it i.e.}, $H$ is a star graph.

Conversely, assume that $H$ is a star graph with the center vertex $v$. Every edge of $H$ is incident to $v$, so $\{v\}$ is a vertex cover. Since the graph has at least one edge, the empty set cannot be a vertex cover. Thus, the minimum vertex cover set has size $1$. This completes the proof.
\end{proof}

\begin{prop}[CLT for linear combination of estimated wedge and triangle density]\label{prop-joint}
    Suppose that Assumptions \ref{assump-1}, \ref{assump-2}, \ref{assump-3} and \ref{assump-4} hold. Let $Z\sim N(0,1)$. Fix a $(t_1,t_2)^\primet\in\mathbb{R}^2$, define 
    \begin{align}\label{lin-comb-1}
    T_{N,l}(\kl_{1,2},\kl_3) &= \frac{t_1}{N^{3-2\beta}f_{l}(p_N:\kl_{1,2})}\left(\widehat{S}_{N,l}(\kl_{1,2}) + f_{l}(p_N,\kl_{1,2})S_N(\kl_{1,2})\right)\notag\\
    &+ \frac{t_2}{N^{3-3\beta}f_{l}(p_N:\kl_{1,2})}\left(\widehat{S}_{N,l}(\kl_{3}) + f_{l}(p_N,\kl_{3})S_N(\kl_{3})\right),\quad\text{and}\notag\\
    \sigma^2_{N,l}(\kl_{1,2}, \kl_3) &= \Var(T_{N,l}(\kl_{1,2},\kl_3)), \quad\text{for $l=1,2$.}
    \end{align}
 Then
    \begin{enumerate}
        \item[(a)] In the induced case ($l=1$), for $(\al,\beta)\in R_{1,1}\cup R_{1,2}$, (see \eqref{rsets-def})
        \begin{align*}
            d_1\left(\frac{T_{N,1}(\kl_{1,2}, \kl_3)}{\sigma_{N,1}(\kl_{1,2},\kl_3)},~Z\right)\to 0,~\text{as $N\to\infty$.}
        \end{align*}
        \item[(b)] In the ego-centric case ($l=2$), for $(\al,\beta)\in R_{2,1}\cup R_{2,2}$ (see \eqref{rsets-def}),
        \begin{align*}
            d_1\left(\frac{T_{N,2}(\kl_{1,2}, \kl_3)}{\sigma_{N,2}(\kl_{1,2},\kl_3)},~Z\right)\to 0,~\text{as $N\to\infty$.}
        \end{align*}
    \end{enumerate}
\end{prop}

\begin{proof}[Proof of Proposition \ref{prop-joint}]
Fix a $(t_1,t_2)^\primet\in \mathbb{R}^2$. 
For $l=1$, from \eqref{lin-comb-1}, we can write,
\begin{align*}
    &T_{N,1}(\kl_{1,2},\kl_3) = \frac{t_1}{N^{3-2\beta}p_N^3}\left(\widehat{S}_{N,l}(\kl_{1,2}) + p_N^3S_N(\kl_{1,2})\right)\notag\\
    &+ \frac{t_2}{N^{3-3\beta}p_N^3}\left(\widehat{S}_{N,l}(\kl_{3}) + p_N^3S_N(\kl_{3})\right)\notag\\
    &= \sum_{\sbb\in [N]_3}\widetilde{W}_{N,1}(\sbb:\kl_3)\left(\frac{t_1}{N^{3(1-\al)-2\beta}}\Yb_{N}(\sbb:\kl_{1,2}) + \frac{t_2}{N^{3(1-\al)-3\beta}}\Yb_N(\sbb:\kl_3)\right),
\end{align*}
where $\widetilde{W}_{N,l}(\sbb:H)$ for $l=1,2$ and $Y_N(\sbb:H)$ are defined in \eqref{piv-1}. 
In the above expression $\E(T_{N,1}(\kl_{1,2},\kl_3)) = 0$.
Now, when $(\al,\beta) \in \{(\al,\beta): \al\in [0,1),~\al+\frac{\beta}{2/3}\leq 1\} = R_{1,1}\cup R_{1,2}$ (see. \eqref{rsets-def}), then using Assumptions \ref{assump-2} and \ref{assump-3} and Lemma \ref{lem-var}, there exists an $M_{1}$, such that for all $N>M_{1}$, we can write
\begin{align}\label{var-joint-ind}
    &\Var(T_{N,1}(\kl_{1,2},\kl_3))\notag\\
    &\geq a_{1,t_1,t_2}\left( \frac{t_1^2}{N^{6(1-\al) - 4\beta}}N^{5(1-\al) -4\beta}+ \frac{t_2^2}{N^{6(1-\al) -6\beta}} N^{5(1-\al) - 6\beta}-\frac{2t_1t_2}{N^{6(1-\al) -5\beta}} N^{5(1-\al) - 5\beta}\right)\notag\\
    &= \frac{a_{1,t_1,t_2}}{N^{1-\al}}\left(t_1^2+t_2^2 -2t_1t_2\right),
\end{align}
where $a_{1,t_1,t_2}$ is a finite positive constant depending on the design and model parameters. Then applying Theorem \ref{thm-ruc} over $T_{N,1}(\kl_{1,2}, \kl_{3})/\sigma_{N,1}(\kl_{1,2},\kl_3)$ we get the following
\begin{align}\label{joint-bdd-ind}
    &d_{1}\left(\frac{T_{N,1}(\kl_{1,2}, \kl_{3})}{\sigma_{N,1}(\kl_{1,2},\kl_3)},~Z\right)\notag\\
    &\leq \frac{1}{\left(\sigma_{N,1}(\kl_{1,2},\kl_3)\right)^3}\sum_{\sbb_1\in [N]_3}\sum_{\sbb_2\in N(\sbb_1)}\sum_{\sbb_3\in N(\sbb_1)}\left(\E\left|\prod_{i=1}^3\widetilde{W}_{N,1}(\sbb_i:\kl_3)\right|\right.\notag\\
    &\left.\qquad{}\times \E\left|\prod_{i=1}^3\left\{\frac{t_1}{N^{3(1-\al)-2\beta}}\Yb_{N}(\sbb_i:\kl_{1,2}) + \frac{t_2}{N^{3(1-\al)-3\beta}}\Yb_N(\sbb_i:\kl_{3})\right\}\right| + \E\left|\prod_{i=1}^2\widetilde{W}_{N,1}(\sbb_i:\kl_3)\right|\E\left|\widetilde{W}_{N,1}(\sbb_3:\kl_3)\right|\right.\notag\\
    &\left.\times\E\left|\prod_{i=1}^2\left\{\frac{t_1}{N^{3(1-\al)-2\beta}}\Yb_{N}(\sbb_i:\kl_{1,2}) + \frac{t_2}{N^{3(1-\al)-3\beta}}\Yb_N(\sbb_i:\kl_3)\right\}\right|\right.\notag\\
    &\left.\qquad{}\qquad{}\times \E\left|\left\{\frac{t_1}{N^{3(1-\al)-2\beta}}\Yb_{N}(\sbb_3:\kl_{1,2}) + \frac{t_2}{N^{3(1-\al)-3\beta}}\Yb_N(\sbb_3:\kl_3)\right\}\right|\right)\notag\\
    &\leq \frac{1}{\left(\sigma_{N,1}(\kl_{1,2},\kl_3)\right)^3}\sum_{k_1,k_2=1}^3\sum_{(\sbb_1,\sbb_2,\sbb_3)\in \mathcal{M}_{k_1,k_2}}\left(\E\left|\prod_{i=1}^3\widetilde{W}_{N,1}(\sbb_i:\kl_3)\right|\right.\notag\\
    &\left.\qquad{}\times \E\left|\prod_{i=1}^3\left\{\frac{t_1}{N^{3(1-\al)-2\beta}}\Yb_{N}(\sbb_i:\kl_{1,2}) + \frac{t_2}{N^{3(1-\al)-3\beta}}\Yb_N(\sbb_i:\kl_{3})\right\}\right| + \E\left|\prod_{i=1}^2\widetilde{W}_{N,1}(\sbb_i:\kl_3)\right|\E\left|\widetilde{W}_{N,1}(\sbb_3:\kl_3)\right|\right.\notag\\
    &\left.\times\E\left|\prod_{i=1}^2\left\{\frac{t_1}{N^{3(1-\al)-2\beta}}\Yb_{N}(\sbb_i:\kl_{1,2}) + \frac{t_2}{N^{3(1-\al)-3\beta}}\Yb_N(\sbb_i:\kl_3)\right\}\right|\right.\notag\\
    &\left.\qquad{}\qquad{}\times \E\left|\left\{\frac{t_1}{N^{3(1-\al)-2\beta}}\Yb_{N}(\sbb_3:\kl_{1,2}) + \frac{t_2}{N^{3(1-\al)-3\beta}}\Yb_N(\sbb_3:\kl_3)\right\}\right|\right)\notag\\
    &=\frac{1}{\left(\sigma_{N,1}(\kl_{1,2},\kl_3)\right)^3}\left(J_{N,1,1}+J_{N,1,2}\right),~\text{(say)}.
\end{align}
Next, consider the first term of \eqref{joint-bdd-ind}, using \eqref{prod-W-ind}, we can write
\begin{align}\label{E-1-1}
     J_{N,1,1}&\leq c^{9}\sum_{k_1,k_2=1}^3\sum_{(\sbb_1,\sbb_3,\sbb_3)\in \mathcal{M}_{k_1,k_2}} N^{-\al(9 - k_1-k_2)}\left(\frac{r_1|t_1|^3}{N^{9(1-\al) -6\beta}}\E\left(\Yb_N(\sbb_1:\kl_{1,2})\Yb_N(\sbb_2:\kl_{1,2})\Yb_N(\sbb_3:\kl_{1,2})\right)\right.\notag\\
    &\qquad{}\left.+ \frac{r_2t_1^2|t_2|}{N^{9(1-\al) -7\beta}}\E\left(\Yb_N(\sbb_1:\kl_{1,2})\Yb_N(\sbb_2:\kl_{1,2})\Yb_N(\sbb_3:\kl_3)\right)\right.\notag\\
    &\qquad{}\qquad{}\left.+\frac{r_3t_1^2|t_2|}{N^{9(1-\al) -7\beta}}\E\left(\Yb_N(\sbb_1:\kl_3)\Yb_N(\sbb_2:\kl_{1,2})\Yb_N(\sbb_3:\kl_{1,2})\right)\right.\notag\\
    &\left.\qquad{}+ \frac{r_4t_2^2|t_1|}{N^{9(1-\al) -8\beta}}\E\left(\Yb_N(\sbb_1:\kl_3)\Yb_N(\sbb_2:\kl_3)\Yb_N(\sbb_3:\kl_{1,2})\right)\right.\notag\\
    &\left.\qquad{}\qquad{}+ \frac{r_5t_2^2|t_1|}{N^{9(1-\al) -8\beta}}\E\left(\Yb_N(\sbb_1:\kl_{1,2})\Yb_N(\sbb_2:\kl_3)\Yb_N(\sbb_3:\kl_3)\right)\right.\notag\\
    &\left.\qquad{}\qquad{}\qquad{}+ \frac{r_6|t_2|^3}{N^{9(1-\al) -9\beta}}\E\left(\Yb_N(\sbb_1:\kl_3)\Yb_N(\sbb_2:\kl_3)\Yb_N(\sbb_3:\kl_3)\right)\right)\notag\\
    &= \left(\frac{r_1|t_1|^3}{N^{9(1-\al) -6\beta}}L_{N,1,1}+ \frac{r_2t_1^2|t_2|}{N^{9(1-\al) -7\beta}}L_{N,1,2}+\frac{r_3t_1^2|t_2|}{N^{9(1-\al) -7\beta}}L_{N,1,3}+ \frac{r_4t_2^2|t_1|}{N^{9(1-\al) -8\beta}}L_{N,1,4}\right.\notag\\
    &\left.\qquad{}\qquad{}\qquad{}+ \frac{r_5t_2^2|t_1|}{N^{9(1-\al) -8\beta}}L_{N,1,5}+ \frac{r_6|t_2|^3}{N^{9(1-\al) -9\beta}}L_{N,1,6}\right),~\text{(say)}
\end{align}
where $r_1,\ldots,r_6$ are the finite positive constants. In \eqref{E-1-1}, for the terms $L_{N,1,1}$ and $L_{N,1,6}$, using \eqref{exp-Y} for $H=\kl_{1,2}$ and $\kl_3$ respectively, there exists an $M_2$ such that for all $N\geq M_2$ and for $(\al,\beta)\in R_{1,1}\cup R_{1,2}$ (cf. \eqref{rsets-def}), we can write,  
\begin{align}\label{p-1-6}
    &L_{N,1,1}\leq c^9\left(N^{\beta}\max_{i,j\in [K]}\pi_{N,i,j}\right)^{6}N^{\max_{t\in [3]}\{(9-2t)\al -\beta(6-2m(t:H))\}}=  c^9\left(N^{\beta}\max_{i,j\in [K]}\pi_{N,i,j}\right)^{6}N^{7(1-\al) - 6\beta}\quad\text{and,}\notag\\
    &L_{N,1,6}\leq c^9\left(N^{\beta}\max_{i,j\in [K]}\pi_{N,i,j}\right)^{9}N^{\max_{t\in [3]}\{(9-2t)\al -\beta(9-2m(t:H))\}}= c^9\left(N^{\beta}\max_{i,j\in [K]}\pi_{N,i,j}\right)^{9} N^{7(1-\al) - 9\beta}.
\end{align}
In order to analyze the remaining terms of \eqref{E-1-1},
recall for any $k_1,k_2\in \{1,2,3\}$, definition of the set $\Mc_{k_1,k_2}$ defined in \eqref{Mck-setdef}. Consider any $(\sbb_1,\sbb_2,\sbb_3)\in \Mc_{k_1,k_2}$. Define the following class of maps,
\begin{align*}
\mathcal{F}(\sbb_i) = \left\{\phi: \text{$\phi$ is a bijection from $\{1,2,3\}$ to $A(\sbb_i)$}\right\},\quad \text{for $i=1,2,3$.}
\end{align*} 
For any $\phi_i\in\mathcal{F}(\sbb_i)$, we construct two graphs $H_i(\phi_i)$ and $G_i(\phi_i)$  with $V(H_i(\phi_i))= V(G_i(\phi_i)) = A(\sbb_i)$ for $i=1,2,3$, and \begin{align}\label{Hi-def}
     H_i(\phi_i)\simeq \kl_{1,2}\quad\text{and}\quad G(\phi_i)\simeq \kl_3,\quad \text{for i=1,2,3.}
\end{align}
Then define the following maps,
\begin{equation}
\begin{aligned}\label{mtH-prime-def}
&m^\prime(k_1:\kl_{1,2},\kl_{3})=m^\prime(k_1:\kl_{3},\kl_{1,2})=\max\left\{|E(G_1(\phi_1))\cap E(H_2(\phi_2))|:\phi_i\in \mathcal{F}(\sbb_i),~i=1,2\right\},\\
    &m^\prime(k_2:\kl_{1,2},\kl_{1,2},\kl_{3})=\max\left\{|E(G_3(\phi_3))\cap (E(H_1(\phi_1))\cup E(H_2(\phi_2)) )|:\phi_i\in \mathcal{F}(\sbb_i),~i=1,2,3\right\},\\
    &m^\prime(k_2:\kl_{3},\kl_{1,2},\kl_{1,2})=\max\left\{|E(H_3(\phi_3))\cap (E(G_1(\phi_1))\cup E(H_2(\phi_2)) )|:\phi_i\in \mathcal{F}(\sbb_i),~i=1,2,3\right\},\\
    &m^\prime(k_2:\kl_{3},\kl_{3},\kl_{1,2})=\max\left\{|E(H_3(\phi_3))\cap (E(G_1(\phi_1))\cup E(G_2(\phi_2)) )|:\phi_i\in \mathcal{F}(\sbb_i),~i=1,2,3\right\},~\text{and}\\
    &m^\prime(k_2:\kl_{1,2},\kl_{3},\kl_{3})=\max\left\{|E(G_3(\phi_3))\cap (E(H_1(\phi_1))\cup E(G_2(\phi_2)) )|:\phi_i\in \mathcal{F}(\sbb_i),~i=1,2,3\right\}
\end{aligned}    
\end{equation}
Using \eqref{mtH-prime-def}, define a map $e_{1,i}(k_1,k_2:\al,\beta)$ for $i=1,2,3,4$, in the following way,
\begin{equation*}
\begin{aligned}
    e_{1,1}(k_1,k_2:\al,\beta) &= (9-k_1-k_2)(1-\al) - \beta(7 - m(k_1:\kl_{1,2}) - m^\prime(k_2:\kl_{1,2},\kl_{1,2},\kl_{3})),\\
    e_{1,3}(k_1,k_2:\al,\beta) &= (9-k_1-k_2)(1-\al) - \beta(7 - m^\prime(k_1:\kl_{1,2},\kl_3) - m^\prime(k_2:\kl_3,\kl_{1,2},\kl_{1,2})),\\
    e_{1,4}(k_1,k_2:\al,\beta) &= (9-k_1-k_2)(1-\al) - \beta(7 - m(k_1:\kl_{1,2}) - m^\prime(k_2:\kl_{3},\kl_{3},\kl_{1,2})),~\text{and}\\
    e_{1,5}(k_1,k_2:\al,\beta) &= (9-k_1-k_2)(1-\al) - \beta(7 - m^\prime(k_1:\kl_{3},\kl_{1,2}) - m^\prime(k_2:\kl_{1,2},\kl_{3},\kl_{3})),
\end{aligned}    
\end{equation*}
where $m(k_1:H)$ defined in \eqref{mtH-def}, if we follow the same definition as in \eqref{mtH-prime-def}, then we can write $m(k_1:H) = m^\prime(k_1:H,H)$. Now, using a similar argument as in \eqref{exp-Y}, there exists a $M_3$ such that for all $N>M_3$, we can write
\begin{align}\label{L-ind}
    &L_{N,1,2}\leq c^9\left(N^{\beta}\max_{i,j\in [K]}\pi_{N,i,j}\right)^{7}\sum_{k_1,k_2=1}^3N^{e_{1,1}(k_1,k_2:\al,\beta)},~~ L_{N,1,3}\leq c^9\left(N^{\beta}\max_{i,j\in [K]}\pi_{N,i,j}\right)^{7}\sum_{k_1,k_2=1}^3N^{e_{1,2}(k_1,k_2:\al,\beta)},\\
    &L_{N,1,4}\leq c^9\left(N^{\beta}\max_{i,j\in [K]}\pi_{N,i,j}\right)^{8}\sum_{k_1,k_2=1}^3N^{e_{1,3}(k_1,k_2:\al,\beta)},\text{and}~ L_{N,1,5}\leq c^9\left(N^{\beta}\max_{i,j\in [K]}\pi_{N,i,j}\right)^{8}\sum_{k_1,k_2=1}^3N^{e_{1,4}(k_1,k_2:\al,\beta)}.
\end{align} 
The values of $e_{1,1}(k_1,k_2:\al,\beta)$ are listed in Table \ref{tab:e-1-1}, the values of $e_{1,2}(k_1,k_2:\al,\beta)$ are listed in Table \ref{tab:e-2-2}, the values of $e_{1,3}(k_1,k_2:\al,\beta)$ are listed in Table \ref{tab:e-3-3} and the values of $e_{1,4}(k_1,k_2:\al,\beta)$ are listed in Table \ref{tab:e-4-4}, for all possible choices of $k_1,k_2\in\{1,2,3\}$. From the above mentioned four table and for $(\al,\beta)\in R_{1,1}\cup R_{1,2}$, we obtained the pair  
\begin{align*}
    (k_1,k_2) \in \argmaxA_{k_1,k_2\in [3]} \{e_{1,i}(k_1,k_2:\al,\beta)\},\quad\text{for $i=1,2,3,4$.} 
\end{align*}
After evaluating the above quantity, we get $(k_1,k_2) = (1,1)$ from all four tables. Therefore, from \eqref{L-ind}, using the value $(k_1,k_2)=(1,1)$, we get the following
\begin{equation}
\label{p-2-4}
\begin{aligned}
    &L_{N,1,2}\leq c^9\left(N^{\beta}\max_{i,j\in [K]}\pi_{N,i,j}\right)^{7}N^{7(1-\al) - 7\beta},\quad L_{N,1,3}\leq c^9\left(N^{\beta}\max_{i,j\in [K]}\pi_{N,i,j}\right)^{7}N^{7(1-\al) - 7\beta}~\text{and,}~\\
    &L_{N,1,4}\leq c^9\left(N^{\beta}\max_{i,j\in [K]}\pi_{N,i,j}\right)^{8}N^{7(1-\al) - 8\beta}~L_{N,1,5}\leq c^9\left(N^{\beta}\max_{i,j\in [K]}\pi_{N,i,j}\right)^{8}N^{7(1-\al) - 8\beta}.
\end{aligned}
\end{equation}
Therefore, using \eqref{var-joint-ind}, \eqref{p-1-6} and \eqref{p-2-4}, from \eqref{E-1-1}, for all $N>\max\{M_1,M_2,M_3\}$, we can write
\begin{align*}
    \frac{J_{N,1,1}}{\left(\sigma_{N,1}(\kl_{1,2},\kl_3)\right)^3} &\leq \frac{1}{a_{1,t_1,t_2}}N^{\frac{3}{2}(1-\al)}O\left(\frac{1}{N^{2(1-\al)}}\right)  =O\left(\frac{1}{N^{\frac{1}{2}(1-\al)}}\right).
\end{align*}
Using a similar argument, it can be shown that there exists an $M_4$ such that for all $N>M_4$, $J_{N,1,2}/\left(\sigma_{N,1}(\kl_{1,2},\kl_3)\right)^3 \leq O(N^{-\frac{1}{2}(1-\al)})$ (cf. \eqref{joint-bdd-ind}). Hence, for $(\al,\beta) \in R_{1,1}\cup R_{1,2}$, for all $N>\max\{M_1,M_2,M_3,M_4\}$, we can write
\begin{align*}
    d_{1}\left(\frac{T_{N,1}(\kl_{1,2}, \kl_{3})}{\sigma_{N,1}(\kl_{1,2},\kl_3)},~Z\right) \leq O\left(\frac{1}{N^{\frac{1}{2}(1-\al)}}\right).
\end{align*}
This completes the proof of part (a). For $l=2$, from \eqref{lin-comb-1}, we can write, 
\begin{align}\label{piv-joint-ego}
    &T_{N,2}(\kl_{1,2},\kl_3) = \frac{t_1}{N^{3-2\beta}f_{2}(p_N:\kl_{1,2})}\left(\widehat{S}_{N,2}(\kl_{1,2}) + f_{2}(p_N:\kl_{1,2})S_N(\kl_{1,2})\right)\notag\\
    &+ \frac{t_2}{N^{3-3\beta}f_2(p_N:\kl_3)}\left(\widehat{S}_{N,2}(\kl_{3}) + f_2(p_N:\kl_3)S_N(\kl_{3})\right)\notag\\
    &= \sum_{\sbb\in \nsr}\left(\frac{t_1}{N^{3-\al-2\beta}c(1+(1-p_N)p_N)}\widetilde{W}_{N,2}(\sbb:\kl_{1,2})\Yb_{N}(\sbb:\kl_{1,2})\right.\notag\\
    &\qquad{}\qquad{}\left.+ \frac{t_2}{N^{3-2\al-3\beta}c^2(3-2p_N)}\widetilde{W}_{N,2}(\sbb:\kl_3)\Yb_N(\sbb:\kl_3)\right).
\end{align}
Form the above expression $\E(T_{N,2}(\kl_{1,2},\kl_3)) = 0$. 
Next, for $(\al,\beta)\in R_{2,1}\cup R_{2,2}$ (cf. \eqref{rsets-def}), using Lemma \ref{lem-var}, there exists a $M_{5}$ such that for all $N>M_{5}$, we can write
\begin{align}\label{var-joint-ego}
    &\sigma^2_{N,2}(\kl_{1,2},\kl_3) = \Var\left(T_{N,2}(\kl_{1,2},\kl_3)\right)\notag\\
    &>a_{2,t_1,t_2}\left( \frac{t_1^2}{N^{6-2\al-4\beta}}N^{5-\al-4\beta}+\frac{t_2^2}{N^{6-4\al-6\beta}}N^{5-3\al-5\beta}+\frac{2t_1t_2}{N^{6-3\al-6\beta}}N^{5-2\al-5\beta}\right)\notag\\
    &= \frac{a_{2,t_1,t_2}}{N^{1-\al}}(t_1^2 + t_2^2 + 2t_1t_2).
\end{align}
Then applying Theorem \ref{thm-ruc} over $T_{N,2}(\kl_{1,2},\kl_3)/\sigma_{N,2}(\kl_{1,2},\kl_3)$, we get the following
\begin{align}\label{joint-bdd-ego}
     &d_{1}\left(\frac{T_{N,2}(\kl_{1,2}, \kl_{3})}{\sigma_{N,2}(\kl_{1,2},\kl_3)},~Z\right)\notag\\
     &\leq \frac{1}{\sigma^3_{N,2}(\kl_{1,2},\kl_3)}\sum_{k_1,k_2=1}^3\sum_{(\sbb_1,\sbb_2,\sbb_3)\in \Mc_{k_1,k_2}}\left\{\E\left|\prod_{i=1}^3\left(\frac{t_1\cdot \widetilde{W}_{N,2}(\sbb_i:\kl_{1,2})\Yb_{N}(\sbb_i:\kl_{1,2})}{N^{3-\al-2\beta}c(1+(1-p_N)p_N)}\right.\right.\right.\notag\\
    &\qquad{}\qquad{}\left.\left.\left.+ \frac{t_2\cdot \widetilde{W}_{N,2}(\sbb_i:\kl_3)\Yb_N(\sbb_i:\kl_3)}{N^{3-2\al-3\beta}c^2(3-2p_N)}\right)\right|+ \E\left|\prod_{i=1}^2\left(\frac{t_1\cdot \widetilde{W}_{N,2}(\sbb_i:\kl_{1,2})\Yb_{N}(\sbb_i:\kl_{1,2})}{N^{3-\al-2\beta}c(1+(1-p_N)p_N)}\right.\right.\right.\notag\\
    &\qquad{}\qquad{}\left.\left.\left.+ \frac{t_2\cdot \widetilde{W}_{N,2}(\sbb_i:\kl_3)\Yb_N(\sbb_i:\kl_3)}{N^{3-2\al-3\beta}c^2(3-2p_N)}\right)\right|\E\left|\left(\frac{t_1\cdot \widetilde{W}_{N,2}(\sbb_3:\kl_{1,2})\Yb_{N}(\sbb_3:\kl_{1,2})}{N^{3-\al-2\beta}c(1+(1-p_N)p_N)}\right.\right.\right.\notag\\
    &\qquad{}\qquad{}\left.\left.\left.+ \frac{t_2\cdot \widetilde{W}_{N,2}(\sbb_3:\kl_3)\Yb_N(\sbb_3:\kl_3)}{N^{3-2\al-3\beta}c^2(3-2p_N)}\right)\right|\right\}\notag\\
    &=\frac{1}{\left(\sigma_{N,2}(\kl_{1,2},\kl_3)\right)^3}\left(J_{N,2,1}+J_{N,2,2}\right),~\text{(say)}.
\end{align}
Next, considering the term $J_{N,2,1}$ from \eqref{joint-bdd-ego}, we can write
\begin{align}\label{E-2-2}
    J_{N,2,1} &= \left(\frac{r_1|t_1|^3}{N^{9-3\al -6\beta}}L_{N,2,1}+ \frac{r_2t_1^2|t_2|}{N^{9-4\al -7\beta}}L_{N,2,2}+\frac{r_3t_1^2|t_2|}{N^{9-4\al -7\beta}}L_{N,2,3}+ \frac{r_4t_2^2|t_1|}{N^{9-5\al-8\beta}}L_{N,2,4}\right.\notag\\
    &\left.\qquad{}\qquad{}\qquad{}+ \frac{r_5t_2^2|t_1|}{N^{9-5\al -8\beta}}L_{N,2,5}+ \frac{r_6|t_2|^3}{N^{9-6\al-9\beta}}L_{N,2,6}\right),~\text{(say)}
\end{align}
where, $L_{N,2,i},i=1,\ldots,6$ has the following expression,
\begin{align}\label{L-def-ego}
    &L_{N,2,1} =\sum_{\star} \E\left|\prod_{i=1}^3\widetilde{W}_{N,2}(\sbb_i:\kl_{1,2})\right|\cdot \E\left(\prod_{i=1}^6Y_N(\sbb_i:\kl_{1,2})\right),\notag\\
    &L_{N,2,2} = \sum_{\star}\E\left|\widetilde{W}_{N,2}(\sbb_1:\kl_{1,2})\widetilde{W}_{N,2}(\sbb_2:\kl_{1,2})\widetilde{W}_{N,2}(\sbb_1:\kl_{3})\right|\E(Y_N(\sbb_1:\kl_{1,2})Y_N(\sbb_2:\kl_{1,2})Y_N(\sbb_3:\kl_{3})),\notag\\
    &L_{N,2,3} = \sum_{\star}\E\left|\widetilde{W}_{N,2}(\sbb_1:\kl_{3})\widetilde{W}_{N,2}(\sbb_2:\kl_{1,2})\widetilde{W}_{N,2}(\sbb_1:\kl_{1,2})\right|\E(Y_N(\sbb_1:\kl_{3})Y_N(\sbb_2:\kl_{1,2})Y_N(\sbb_3:\kl_{1,2}))\notag\\
    &L_{N,2,4} = \sum_{\star}\E\left|\widetilde{W}_{N,2}(\sbb_1:\kl_{3})\widetilde{W}_{N,2}(\sbb_2:\kl_{3})\widetilde{W}_{N,2}(\sbb_1:\kl_{1,2})\right|\E(Y_N(\sbb_1:\kl_{3})Y_N(\sbb_2:\kl_{3})Y_N(\sbb_3:\kl_{1,2})),\notag\\
    &L_{N,2,5} =\sum_{\star} \E\left|\widetilde{W}_{N,2}(\sbb_1:\kl_{1,2})\widetilde{W}_{N,2}(\sbb_2:\kl_{3})\widetilde{W}_{N,2}(\sbb_1:\kl_{3})\right|\E(Y_N(\sbb_1:\kl_{1,2})Y_N(\sbb_2:\kl_{3})Y_N(\sbb_3:\kl_{3})),\notag\\
    &L_{N,2,6} =\sum_{\star} \E\left|\prod_{i=1}^3\widetilde{W}_{N,2}(\sbb_i:\kl_{3})\right|\cdot \E\left(\prod_{i=1}^6Y_N(\sbb_i:\kl_{3})\right).
\end{align}
In the above expression, we denote $$\sum_{k_1,k_2=1}^3\sum_{(\sbb_1,\sbb_2,\sbb_3)\in \Mc_{k_1,k_2}}\equiv\sum_{\star}.$$
In \eqref{L-def-ego}, the model components are the same as the induced case described in \eqref{E-1-1}. From the proof of Theorem \ref{thm-1-main} and similar \eqref{p-1-6}, there exists an $M_6$ such that for all $N>M_6$, for $(\al,\beta)\in R_{2,1}\cup R_{2,2}$, we can write,
\begin{align}\label{p-2-6}
    L_{N,2,1}\leq c\left(N^{\beta}\max_{i,j\in [K]}\pi_{N,i,j}\right)^6 N^{7-\al-6\beta},~\text{and}~L_{N,2,6}\leq c^4\left(N^{\beta}\max_{i,j\in [K]}\pi_{N,i,j}\right)^9 N^{7-4\al-6\beta}.
\end{align}
In order to analyze the remaining terms in \eqref{E-2-2}, recall, for any $k_1,k_2\in \{1,2,3\}$, definition of the set $\Mc_{k_1,k_2}$ defined in \eqref{Mck-setdef}. Consider any $(\sbb_1,\sbb_2,\sbb_3)\in \Mc_{k_1,k_2}$. First, we define the following union graphs using graphs defined in \eqref{Hi-def}, for a tuple of maps $\phib = (\phi_1,\phi_2,\phi_2)^\primet$, that is,
\begin{align}\label{uni-graphs}
    &Q(\kl_{1,2},\kl_{1,2},\kl_3:\phib) = H_1(\phi_1) \cup H_2(\phi_2) \cup G_3(\phi_3),\quad
    Q(\kl_{3},\kl_{1,2},\kl_{1,2}:\phib) = G_1(\phi_1) \cup H_2(\phi_2) \cup H_3(\phi_3),\notag\\
    &Q(\kl_{3},\kl_{3},\kl_{1,2}:\phib) = G_1(\phi_1) \cup G_2(\phi_2) \cup H_3(\phi_3)~\text{and}~
    Q(\kl_{1,2},\kl_{3},\kl_{3}:\phib) = H_1(\phi_1) \cup G_2(\phi_2) \cup G_3(\phi_3).
\end{align}
From \eqref{mtH-prime-def}, we get a tuple of maps $\phib^\star \equiv \phib^\star_{k_1,k_2} = (\phi^\star_1,\phi^\star_2,\phi^\star_2)^\primet$ for which the maximum value, $m^\prime(k_1,\cdot,\cdot)+m^\prime(k_2:\cdot,\cdot,\cdot)$ is obtained for a choice of $k_1,k_2\in \{1,2,3\}$. Then define the size of the minimum vertex cover of the union graph as follows,
\begin{align}\label{tau-uni}
    &\tau\left(\kl_{1,2},\kl_{1,2},\kl_{3}\right) = \tau\left(Q(\kl_{1,2},\kl_{1,2},\kl_3:\phib^\star)\right),\quad\tau\left(\kl_{3},\kl_{1,2},\kl_{1,2}\right) = \tau\left(Q(\kl_{3},\kl_{1,2},\kl_{1,2}:\phib^\star)\right),\notag\\
    &\tau\left(\kl_{3},\kl_{3},\kl_{1,2}\right) = \tau\left(Q(\kl_{3},\kl_{3},\kl_{1,2}:\phib^\star)\right),~\text{and}~\tau\left(\kl_{1,2},\kl_{3},\kl_{3}\right) = \tau\left(Q(\kl_{1,2},\kl_{3},\kl_{3}:\phib^\star)\right).
\end{align}
Then we define the following maps,
\begin{equation*}
\begin{aligned}
    e_{2,1}(k_1,k_2:\al,\beta) &= (9-k_1-k_2)-\al\cdot\tau\left(\kl_{1,2},\kl_{1,2},\kl_{3}\right) - \beta\cdot(7 - m(k_1:\kl_{1,2}) - m^\prime(k_2:\kl_{1,2},\kl_{1,2},\kl_{3})),\\
    e_{2,3}(k_1,k_2:\al,\beta) &= (9-k_1-k_2)-\al\cdot\tau\left(\kl_{3},\kl_{1,2},\kl_{1,2}\right) - \beta\cdot(7 - m^\prime(k_1:\kl_{1,2},\kl_3) - m^\prime(k_2:\kl_3,\kl_{1,2},\kl_{1,2})),\\
    e_{2,4}(k_1,k_2:\al,\beta) &= (9-k_1-k_2)-\al\cdot\tau\left(\kl_{3},\kl_{3},\kl_{1,2}\right) - \beta\cdot(7 - m(k_1:\kl_{1,2}) - m^\prime(k_2:\kl_{3},\kl_{3},\kl_{1,2})),~\text{and}\\
    e_{2,5}(k_1,k_2:\al,\beta) &= (9-k_1-k_2)-\al\cdot\tau\left(\kl_{1,2},\kl_{3},\kl_{3}\right) - \beta\cdot(7 - m^\prime(k_1:\kl_{3},\kl_{1,2}) - m^\prime(k_2:\kl_{1,2},\kl_{3},\kl_{3})),
\end{aligned}    
\end{equation*}
where $m(k_1:H)$ is defined in \eqref{mtH-def}. Now, there exists a $M_7$ such that for all $N>M_7$, we can write
\begin{align}\label{L-ego}
    &L_{N,2,2}\leq c^2\left(N^{\beta}\max_{i,j\in [K]}\pi_{N,i,j}\right)^{7}\sum_{k_1,k_2=1}^3N^{e_{2,1}(k_1,k_2:\al,\beta)},~~ L_{N,1,3}\leq c^2\left(N^{\beta}\max_{i,j\in [K]}\pi_{N,i,j}\right)^{7}\sum_{k_1,k_2=1}^3N^{e_{2,2}(k_1,k_2:\al,\beta)},\notag\\
    &L_{N,2,4}\leq c^3\left(N^{\beta}\max_{i,j\in [K]}\pi_{N,i,j}\right)^{8}\sum_{k_1,k_2=1}^3N^{e_{2,3}(k_1,k_2:\al,\beta)},\text{and}~ L_{N,1,5}\leq c^3\left(N^{\beta}\max_{i,j\in [K]}\pi_{N,i,j}\right)^{8}\sum_{k_1,k_2=1}^3N^{e_{2,4}(k_1,k_2:\al,\beta)}.
\end{align}
The values of $e_{2,1}(k_1,k_2:\al,\beta)$ are listed in Table \ref{tab:e-1-1}, the values of $e_{2,2}(k_1,k_2:\al,\beta)$ are listed in Table \ref{tab:e-2-2}, the values of $e_{2,3}(k_1,k_2:\al,\beta)$ are listed in Table \ref{tab:e-3-3} and the values of $e_{2,4}(k_1,k_2:\al,\beta)$ listed in Table \ref{tab:e-4-4} for all possible choices of $k_1,k_2\in\{1,2,3\}$. Using an argument similar to the one in the induced case, using the values of $e_{2,i}(k_1,k_2:\al,\beta)$ from the table mentioned above, we obtained a pair $(k_1,k_2)$ from evaluating
\begin{align*}
    (k_1,k_2) \in \argmaxA_{k_1,k_2\in [3]} \{e_{2,i}(k_1,k_2:\al,\beta)\},\quad\text{for $i=1,2,3,4$.} 
\end{align*}
From all tables, for $(\al,\beta)\in R_{2,1}\cup R_{2,2}$, we get $(k_1,k_2) = (1,1)$. Then, using the value $(k_1,k_2) = (1,1)$, from \eqref{L-ego}, we can write
\begin{equation}
\label{p-2-4-1}
\begin{aligned}
    &L_{N,2,2}\leq c^2\left(N^{\beta}\max_{i,j\in [K]}\pi_{N,i,j}\right)^{7}N^{7-2\al - 7\beta},\quad L_{N,2,3}\leq c^2\left(N^{\beta}\max_{i,j\in [K]}\pi_{N,i,j}\right)^{7}N^{7-2\al - 7\beta}~\text{and,}~\\
    &L_{N,2,4}\leq c^3\left(N^{\beta}\max_{i,j\in [K]}\pi_{N,i,j}\right)^{8}N^{7-3\al - 8\beta}~L_{N,2,5}\leq c^3\left(N^{\beta}\max_{i,j\in [K]}\pi_{N,i,j}\right)^{8}N^{7-3\al - 8\beta}.
\end{aligned}
\end{equation}
Therefore, using \eqref{var-joint-ego}, \eqref{p-2-6} and \eqref{p-2-4-1}, from \eqref{E-2-2}, for all $N>\max\{M_5,M_6,M_7\}$, from \eqref{E-2-2}, we can write
\begin{align*}
    \frac{J_{N,2,1}}{\left(\sigma_{N,1}(\kl_{1,2},\kl_3)\right)^3} &\leq \frac{N^{\frac{3}{2}(1-\al)}}{a_{2,t_1,t_2}}\cdot O\left(\frac{1}{N^{2(1-\al)}}\right) =O\left(\frac{1}{N^{\frac{1}{2}(1-\al)}}\right).
\end{align*}
Using a similar argument, it can be shown that there exists an $M_8$ such that for all $N>M_8$, $J_{N,2,2}/\left(\sigma_{N,2}(\kl_{1,2},\kl_3)\right)^3 \leq O(N^{-\frac{1}{2}(1-\al)})$ (cf. \eqref{joint-bdd-ego}). Hence, for $(\al,\beta) \in R_{2,1}\cup R_{2,2}$, for all $N>\max\{M_5,M_6,M_7,M_8\}$, from \eqref{joint-bdd-ego}, we can write
\begin{align*}
    d_{1}\left(\frac{T_{N,2}(\kl_{1,2}, \kl_{3})}{\sigma_{N,2}(\kl_{1,2},\kl_3)},~Z\right) \leq O\left(\frac{1}{N^{\frac{1}{2}(1-\al)}}\right).
\end{align*}
This completes the proof.
\end{proof}

\begin{lemma}
\label{lem-sup-eq}
Assume $F$ and $G$ are both cumulative distribution functions (cdf's) on $\rl$, with, $G(x) = \int_{-\infty}^x g(u)~du$, for each $x\in\rl$, where $g(\cdot)$ is a probability density function (pdf). Then,
\begin{align*}
\sup_{x\in\rl} |F(x)-G(x)| = \sup_{x\in\rl} |F(x^-)-G(x^{-})|,
\end{align*}
where, $F(x^-) = \lim_{h\rightarrow 0} F(x-h)$ and $G(x^{-})$ is also defined similarly.  
\end{lemma}

\begin{proof}[Proof of Lemma \ref{lem-sup-eq}]
Firstly we establish uniform continuity of $G(\cdot)$. Fix $\eps>0$, we need to show that there exists a $\delta>0$, such that $|x-y|<\delta$, will imply $|G(x)-G(y)|<\eps$. Let $x<y$, with $(y-x)<\delta$. Note that, 
\begin{align*}
|G(y)-G(x)| &= \int_x^y g(u)~du
= \int_x^y g(u)\cdot \mathbf{1}\left(g(u)>M\right)~du + \int_x^y g(u)\cdot \mathbf{1}\left(g(u)\leq M\right)~du,
\end{align*}
for any $M>0$, which is to be chosen later. As $g$ is integrable, hence by DCT, there exists a $M_\eps>0$, such that 
$$
\int_\rl g(u)\cdot\mathbf{1}(g(u)>M_\eps)~du <\frac{\eps}{2},
$$
where the choice of $M_\eps$ does not depend on $x$ and $y$, and only depends on $\eps$. Thus, with this choice of $M = M_\eps$, we obtain
$$
|G(y)-G(x)| < \frac{\eps}{2} + M_\eps\cdot |y-x| \leq \frac{\eps}{2} + M_\eps \cdot \delta.
$$
Choose, $\delta = \eps/(2M_\eps)$. This leads to 
$$
G(y)-G(x) < \frac{\eps}{2} + \frac{\eps}{2} = \eps.
$$
As $x$ and $y$ are arbitrary and $\delta$ depends only on $\eps$, we have established uniform continuity of $G$. 
Let $X$ and $Y$ denote the underlying random variables corresponding to the cdf's $F$ and $G$ respectively. Then for each $\eps>0$,
$$
F(x^-) = \pr(X<x) \leq F(x) \leq \pr(X<x+\eps) = F({(x+\eps)}^-).
$$
Uniform continuity implies continuity, {\it i.e.}, $G(x^-) = \lim_{h\rightarrow 0} G(x-h) = G(x)$. Write, $
\Delta = \sup_{x}|F(x^-)-G(x^-)|$. Then,
\begin{align*}
& F(x^-)-G(x)\leq F(x)-G(x)\leq F({(x+\eps)}^-)-G(x)\\
\Rightarrow & \ -|F(x^-)-G(x)|\leq F(x^-)-G(x)\leq F(x)-G(x)\leq F({(x+\eps)}^-)-G(x+\eps)+G(x+\eps)-G(x)\\
\Rightarrow & \ -\Delta \leq F(x)-G(x)\leq \sup_{t}|F(t^-)-G(t^-)| + \sup_t|G(t+\eps)-G(t)| = \Delta + \Delta_{1,\eps}\quad\text{(say),}\\
\Rightarrow & \ -\Delta-\Delta_{1,\eps} \leq -\Delta \leq F(x)-G(x)\leq \Delta + \Delta_{1,\eps},\quad\text{for all $x$ and all $\eps>0$,}\\
\Rightarrow & \ \sup_{x}|F(x)-G(x)|\leq \Delta + \Delta_{1,\eps}.
\end{align*}
As $G$ is uniformly continuous, letting $\eps\downarrow 0$, $\Delta_{1,\eps}\rightarrow 0$. This implies, $\sup_{x}|F(x)-G(x)|\leq \Delta$. To show the reverse inequality, we write for each $x$ and $\eps>0$,
\begin{align*}
& \pr(X\leq x-\eps)\leq \pr(X<x)\leq \pr(X\leq x)\leq \pr(X\leq x+\eps)\\
\Rightarrow &\ F(x-\eps)\leq F(x^-)\leq F(x+\eps)\\
\Rightarrow &\ F(x-\eps)-G(x)\leq F(x^-)-G(x)\leq F(x+\eps)-G(x)\\
\Rightarrow &\ F(x-\eps)-G(x-\eps)+G(x-\eps)-G(x)\leq F(x^-)-G(x)\leq F(x+\eps)-G(x+\eps)+G(x+\eps)-G(x)\\
\Rightarrow &\ -\sup_t|F(t)-G(t)|-\sup_t|G(t-\eps)-G(t)|\leq F(x^-)-G(x)\leq \sup_t|F(t)-G(t)|+\sup_t|G(t+\eps)-G(t)|\\
\Rightarrow &\ -\sup_t|F(t)-G(t)|-\Delta_{2,\eps}\leq F(x^-)-G(x^-)\leq \sup_t|F(t)-G(t)| + \Delta_{1,\eps}\quad\text{(say),}\\
\Rightarrow &\ -\sup_t|F(t)-G(t)|-\Delta_{\eps}\leq F(x^-)-G(x^-) \leq \sup_t|F(t)-G(t)| + \Delta_{\eps},\quad\text{where $\Delta_\eps = \max\{\Delta_{1,\eps},\Delta_{2,\eps}\}$}\\
\Rightarrow &\ \Delta=\sup_{x}|F(x^-)-G(x^-)|\leq \sup_t|F(t)-G(t)|+\Delta_{\eps},\\
\Rightarrow &\ \Delta\leq \sup_x|F(x)-G(x)|,\quad\text{as $\Delta_\eps\rightarrow 0$ when $\eps\downarrow 0$.}
\end{align*}
This establishes the reverse inequality, indicating the equality of both suprema. 
\end{proof}

\begin{prop}\label{samp-CLT}
    Suppose that Assumptions \ref{assump-1}, \ref{assump-2}, \ref{assump-3} and \ref{assump-4} hold. Let $Z\sim N(0,1)$. For $H = \kl_{1,2}$ (wedge), consider the estimated wedge count $\widehat{S}_{N,l}(\kl_{1,2})$ for $l=1,2$ and define $\psi^2_{N,l}(H)=\Var\left(\widehat{S}_{N,l}(\kl_{1,2})\right)$ for $l=1,2$. Then 
    \begin{align*}
        d_1\left(\frac{\widehat{S}_{N,l}(\kl_{1,2}) -\E\left(\widehat{S}_{N,l}(\kl_{1,2})\right)}{\psi_{N,l}(\kl_{1,2})},~Z\right)\to 0,~\text{as}~N\to\infty,~\text{for $l=1,2$}.
    \end{align*}
\end{prop}

\begin{proof}
    For $l=1,2$, define,
    \begin{align*}
        Z_{N,l}(\kl_{1,2}) &= \frac{\widehat{S}_{N,l}(\kl_{1,2}) - \E\left(\widehat{S}_{N,l}(\kl_{1,2})\right)}{\psi_{N,l}(\kl_{1,2})}\notag\\
        &= \frac{1}{\psi_{N,l}(\kl_{1,2})}\sum_{\sbb\in \nsr} \left(\Wb_{N,l}(\sbb:\kl_{1,2})\Yb_N(\sbb:\kl_{1,2}) - \E\left(\Wb_{N,l}(\sbb:\kl_{1,2})\Yb_N(\sbb:\kl_{1,2})\right)\right).
    \end{align*}
Then, the variance of $\widehat{S}_{N,l}(\kl_{1,2})$ can be written as follows. 
\begin{align}\label{est-SN-var}
    &\psi^2_{N,l}(\kl_{1,2})=\Var\left(\widehat{S}_{N,l}(\kl_{1,2})\right) = \E\left(\left(\sum_{\sbb\in \nsr }\Wb_{N,l}(\sbb:H)\Yb_N(\sbb:H)\right)^2\right)\notag\\
    &\qquad{}- \left(\E\left(\sum_{\sbb\in \nsr }\Wb_{N,l}(\sbb:H)\Yb_N(\sbb:H)\right)\right)^2\notag\\
    &= \sum_{\sbb_1,\sbb_2\nsr}\E\left(\Wb_{N,l}(\sbb_1:H)\Wb_{N,l}(\sbb_2:H)\right)\E\left(\Yb_N(\sbb_1:H)\Yb_N(\sbb_2:H)\right)\notag\\
    &\qquad{}- \left(\sum_{\sbb\in \nsr }\E\left(\Wb_{N,l}(\sbb:H)\Yb_N(\sbb:H)\right)\right)^2\notag\\
&=\sum_{\sbb_1,\sbb_2\nsr}\E\left(\widetilde{\Wb}_{N,l}(\sbb_1:H)\widetilde{\Wb}_{N,l}(\sbb_2:H)\right)\E\left(\Yb_N(\sbb_1:H)\Yb_N(\sbb_2:H)\right)\notag\\
&+2\sum_{\sbb_1,\sbb_2\nsr}\E\left({\Wb}_{N,l}(\sbb_1:\kl_{1,2})\right)\E\left({\Wb}_{N,l}(\sbb_2:\kl_{1,2})\right)\E\left(\Yb_N(\sbb_1:\kl_{1,2})\Yb_N(\sbb_2:\kl_{1,2})\right)\notag\\
    &\qquad{}- \left(\sum_{\sbb\in \nsr }\E\left(\Wb_{N,l}(\sbb:\kl_{1,2})\right)\E\left(\Yb_N(\sbb:\kl_{1,2})\right)\right)^2\notag\\
    &= \sigma^2_{N,l}(\kl_{1,2})+2\sum_{\sbb_1,\sbb_2\nsr}\E\left({\Wb}_{N,l}(\sbb_1:\kl_{1,2})\right)\E\left({\Wb}_{N,l}(\sbb_2:\kl_{1,2})\right)\E\left(\Yb_N(\sbb_1:\kl_{1,2})\Yb_N(\sbb_2:\kl_{1,2})\right)\notag\\
    &+\left(\sum_{\sbb\in \nsr }\E\left(\Wb_{N,l}(\sbb:\kl_{1,2})\right)\E\left(\Yb_N(\sbb:\kl_{1,2})\right)\right)^2\notag\\
    &= \sigma^2_{N,l}(\kl_{1,2})+2f^2_l(p_N:\kl_{1,2})\sum_{\sbb_1,\sbb_2\nsr}\E\left(\Yb_N(\sbb_1:\kl_{1,2})\Yb_N(\sbb_2:\kl_{1,2})\right)\notag\\
    &\qquad{}\qquad{}+f^2_l(p_N:\kl_{1,2})\left(\sum_{\sbb\in \nsr }\E\left(\Yb_N(\sbb:\kl_{1,2})\right)\right)^2.
\end{align}
Therefore, using Lemma \ref{lem-1} and Lemma \eqref{lem-var} (one can use exact expression of $\sigma^2_{N,l}(\kl_{1,2})$ provided in \eqref{var-wedge-ind} and \eqref{var-wedge-ego}), for large enough $N$, we can write
\begin{align*}
    \frac{\psi^2_{N,l}(\kl_{1,2})}{\sigma^2_{N,l}(\kl_{1,2})} = 1 + O\left(\frac{1}{N}\right).
\end{align*}
Using the Theorem \ref{thm-ruc} over $Z_{N,l}(\kl_{1,2})$, we can write
\begin{align*}
    &d_1\left(Z_{N,l}(\kl_{1,2}),~N(0,1)\right) \leq \frac{1}{\psi^3_{N,l}(\kl_{1,2})}\sum_{t_1,t_2=1}^3\sum_{(\sbb_1,\sbb_2,\sbb_3)\in \Mc_{t_1,t_2}}\bigg\{\E\prod_{i=1}^3\bigg|\Wb_{N,l}(\sbb_i:\kl_{1,2})\Yb_N(\sbb_i:\kl_{1,2})\notag\\
    &\qquad{}- \E\left(\Wb_{N,l}(\sbb_i:\kl_{1,2})\Yb_N(\sbb_i:\kl_{1,2})\right)\bigg| 
    + \E\Bigg|\bigg\{\Wb_{N,l}(\sbb_1:\kl_{1,2})\Yb_N(\sbb_1:\kl_{1,2})\\
    &- \E\left(\Wb_{N,l}(\sbb_1:\kl_{1,2})\Yb_N(\sbb_1:\kl_{1,2})\right)\bigg\}
    \bigg\{\Wb_{N,l}(\sbb_2:\kl_{1,2})\Yb_N(\sbb_2:\kl_{1,2}) 
    - \E\left(\Wb_{N,l}(\sbb_2:\kl_{1,2})\Yb_N(\sbb_2:\kl_{1,2})\right)\bigg\}\Bigg|\notag\\
    &\times\E\left|\Wb_{N,l}(\sbb_3:\kl_{1,2})\Yb_N(\sbb_3) 
    - \E\left(\Wb_{N,l}(\sbb_3)\Yb_N(\sbb_3)\right)\right|\bigg\}\notag\\
    &\leq \frac{\sigma^3_{N,l}(\kl_{1,2})}{\psi^3_{N,l}(\kl_{1,2})}\frac{1}{\sigma^3_{N,l}(\kl_{1,2})}
    \sum_{t_1,t_2=1}^3\sum_{(\sbb_1,\sbb_2,\sbb_3)\in \Mc_{t_1,t_2}}\Bigg(
    \E\left(\Wb_{N,l}(\sbb_1)\Wb_{N,l}(\sbb_2)\Wb_{N,l}(\sbb_3)\right)\\
    &\qquad\qquad\qquad\qquad\cdot \E\left(\Yb_N(\sbb_1)\Yb_N(\sbb_2)\Yb_N(\sbb_3)\right)\notag\\
    &\qquad +\E\left(\Wb_{N,l}(\sbb_1)\Wb_{N,l}(\sbb_2)\right)
    \E\left(\Yb_N(\sbb_1)\Yb_N(\sbb_2)\right)
    \E\left(\Wb_{N,l}(\sbb_3)\right)
    \E\left(\Yb_{N}(\sbb_3)\right)\Bigg).
\end{align*}
The above bound is similar to the bound obtained in the proof of Theorem \ref{thm-1-main} and the upper bound goes to zero under similar assumption required in Theorem \ref{thm-1-main} and Corollary \ref{cor-clt-all}. In addition, note that, since $\psi^2_{N,l}(\kl_{1,2})/\sigma^2_{N,l}(\kl_{1,2})\to 1$ as $N\rai$, then by Lemma \ref{lem-var}, as $N\to\infty$,
\begin{align*}
    \frac{\psi^2_{N,l}(\kl_{1,2})}{N^{-g_{l,\al,\beta,H}(R) + g_{l,\al,\beta,H}(t_l(\al,\beta;H))}} \to \psi_l^2(\kl_{1,2})\in (0,\infty),~\text{for $l=1,2$.}
\end{align*}
This completes the proof.
    
\end{proof}

\setappendix{E}
\refstepcounter{subsection}
\subsection*{Appendix E: Variance expressions for some common target subgraphs}
\label{sec:app-E}

We provide expression of $\sigma^2_{N,l}(H)$ for some common target subgraphs, such as the edge, wedge, and triangle at any choice of $(\al,\beta)\in C_l(H)$ (cf. \eqref{c-set-ind} and \eqref{c-set-ego}). Similar expressions for any complete and star graph are also provided. The expressions are provided in terms of the SBM edge probabilities $((\pi_{N,u,v}:1\leq u,v\leq K))$ and recall $N^{\beta}\cdot \PiB_N\asymp \Cb$, the class size proportions $\la_{N,1},\ldots,\la_{N,K}$ and the node sampling probability $p_N = c/N^{\al}$. These variance expressions are exact and valid for all choices of $N$. 
\begin{enumerate}
\item {\bf Edge:} $H=\kl_2$. In this case,
    \begin{equation}
    \begin{aligned}\label{var-edge}
    \sigma^2_{N,1}(\kl_2) &= 4\cdot N^3\cdot p_N^3(1-p_N)\cdot \sum_{u,v,w=1}^K\pi_{N,u,v}\pi_{N,u,w}\cdot \la_{N,u}\la_{N,v}\la_{N,w}\\
    &\qquad{}+ 2\cdot N^2\cdot p_N^2(1-p_N^2)\sum_{u,v=1}^K \pi_{N,u,v}\cdot \la_{N,u}\la_{N,v},\quad\text{for $(\al,\beta)\in C_1(\kl_2)$ and}\\
    \sigma_{N,2}^2(\kl_2) &= 4\cdot N^3\cdot p_N(1-p_N^3)\sum_{u,v,w=1}^K\pi_{N,u,v}\pi_{N,u,w}\cdot\la_{N,u}\la_{N,v}\la_{N,w}\\
    &\qquad{}+ 2\cdot N^2\cdot p_N(2-p_N)(1-p_N)^2\sum_{u,v=1}^K \pi_{N,u,v}\cdot \la_{N,u}\la_{N,v},\quad\text{for $(\al,\beta)\in C_2(\kl_2)$.} 
\end{aligned}
\end{equation}
\item {\bf Wedge:} $H=\kl_{1,2}$. In this case,
\begin{equation}
\begin{aligned}\label{var-wedge-ind}
    \sigma_{N,1}^2(\kl_{1,2}) &= N^5\cdot p_N^5(1-p_N) \sum_{u,v,w,x,y=1}^K\pi_{N,u,v}\pi_{N,u,w}\cdot (\pi_{N,u,x}\pi_{N,u,y} + 4\cdot \pi_{N,u,x}\pi_{N,x,y}\\
    &\qquad{}\qquad{}+4\cdot \pi_{N,v,x} \pi_{N,x,y})\cdot \la_{N,u} \la_{N,v} \la_{N,w}\la_{N,x}\la_{N,y}+\\
    &+ N^4\cdot p_N^4(1 - p_N^2)\sum_{u,v,w,x=1}^K\pi_{N,u,v}\pi_{N,u,w}\cdot (4\cdot\pi_{N,u,x} + 4\cdot\pi_{N,v,x}\\
    &\qquad{}\qquad{}+2\cdot\pi_{N,v,x}\pi_{N,u,x} +8\cdot\pi_{N,v,x} \pi_{N,w,x})\cdot \la_{N,u} \la_{N,v}\la_{N,w}\la_{N,x}+\\
    &+N^3\cdot p_N^3(1 - p_N^3)\sum_{u,v,w=1}^K \pi_{N,u,v}\pi_{N,u,w}\cdot (2  + 4\cdot\pi_{N,v,w}) \cdot\la_{N,u}\la_{N,v}\la_{N,w},\quad\text{for $(\al,\beta)\in C_1(\kl_{1,2})$,}
\end{aligned}
\end{equation}
and
\begin{align}
    \sigma_{N,2}^2(\kl_{1,2}) &= N^5\cdot \sum_{u,v,w,x,y=1}^K\pi_{N,u,v}\pi_{N,u,w}\cdot\{p_N(1 - p_N)^3\pi_{N,u,x}\pi_{N,u,y}\notag\\
        &\qquad{}\qquad{}+4\cdot p_N^2(1 - p_N)^3(1 + p_N)\pi_{N,u,x}\pi_{N,x,y}\notag\\
        &\qquad{}\qquad{}\qquad{}+4\cdot p_N^3(1 - p_N)^3\pi_{N,v,x}\pi_{N,x,y}\}\cdot \la_{N,u}\la_{N,v}\la_{N,w}\la_{N,x}\la_{N,y}\notag\\
        &+N^4\sum_{u,v,w,x=1}^N\pi_{N,u,v}\pi_{N,u,w}\cdot\{4\cdot p_N(1-p_N)^2(1 + p_N - p_N^3)\pi_{N,u,x}\notag\\
        &\qquad{}\qquad{}\qquad{}+ 4\cdot p_N^2(1 - p_N)(2 - p_N^2)\cdot \pi_{N,v,x}\notag\\
        &+2\cdot p_N^2(1 - p_N)^3\pi_{v,x}\pi_{N,u,x} +8\cdot p_N^2(1 - p_N)^3(1 + p_N)\cdot \pi_{N,v,x}\pi_{N,w,x}\}\cdot\la_{N,u}\la_{N,v}\la_{N,w}\la_{N,x}\notag\\
        &\qquad{}\qquad{}+ N^3\sum_{u,v,w=1}^K\pi_{N,u,v}\pi_{N,u,w}\cdot \{2\cdot(p_N + (1-p_N)p_N^2)\cdot(1 - (p_N + (1-p_N)p_N^2))\notag\\
        &\qquad{}\qquad{}+ 4\cdot p_N^2(1-p_N)^2(2-p_N^2)\cdot\pi_{N,v,w}\}\cdot\la_{N,u}\la_{N,v}\la_{N,w}, \quad\text{for $(\al,\beta)\in C_2(\kl_{1,2})$.}
       \label{var-wedge-ego} 
\end{align}

\item {\bf Triangle:} $H=\kl_3$. In this case,
\begin{align}
\label{var-tri-ind}
    \sigma_{N,1}^2(\kl_3) &= 9\cdot N^5\cdot p_N^5(1-p_N)\notag\\
&\qquad{}\qquad{}\times\sum_{u,v,w,x,y=1}^K\pi_{N,u,v}\pi_{N,v,w}\pi_{N,w,u}\pi_{N,u,x}\pi_{N,x,y}\pi_{N,y,u}\cdot \la_{N,u}\la_{N,v}\la_{N,w}\la_{N,x}\la_{N,y} +\notag\\
    &+18\cdot N^4\cdot p_N^4(1-p_N^2)\sum_{u,v,w,x=1}^K\pi_{N,u,v}\pi_{N,v,w}\pi_{N,w,u}\pi_{N,v,x}\pi_{N,x,u}\cdot\la_{N,u}\la_{N,v}\la_{N,w}\la_{N,x}+\notag\\
    &+6\cdot N^3\cdot p_N^3(1-p_N^3)\sum_{u,v,w=1}^K \pi_{N,u,v}\pi_{N,v,w}\pi_{N,w,u}\cdot \la_{N,u}\la_{N,v}\la_{N,w},\quad\text{for $(\al,\beta)\in C_1(\kl_3)$}
 \end{align}
 and
 \begin{align}
 \label{var-tri-ego}
    \sigma^2_{N,2}(\kl_3) &= 36\cdot N^5\cdot p_N^3(1-p_N)^3\notag\\
    &\qquad{}\qquad{}\times\sum_{u,v,w,x=1}^K \pi_{N,u,v}\pi_{N,v,w}\pi_{N,w,u}\pi_{N,u,x}\pi_{N,x,y}\pi_{N,y,u}\cdot \la_{N,u}\la_{N,v}\la_{N,w}\la_{N,x}\la_{N,y}+\notag\\
        &+18\cdot N^4\cdot (p_N^2+2p_N^3-2p_N^4-9p_N^4 +12p_N^5-4p^6)\notag\\
        &\qquad{}\qquad{}\sum_{u,v,w,x=1}^K \pi_{N,u,v}\pi_{N,v,w}\pi_{N,w,u}\pi_{N,v,x}\pi_{N,x,u}\cdot\la_{N,u}\la_{N,v}\la_{N,w}\la_{N,x}\notag\\
        &+6\cdot N^3\cdot (3\cdot p_N^2 - 2\cdot p_N^3)(1-(3\cdot p_N^2 - 2\cdot p_N^3))\notag\\
        &\qquad{}\qquad{}\qquad{}\times\sum_{u,v,w=1}^K \pi_{N,u,v}\pi_{N,v,w}\pi_{N,w,u}\cdot \la_{N,u}\cdot \la_{N,v}\la_{N,w},\quad\text{for $(\al,\beta)\in C_2(\kl_3)$.}
\end{align}

\item {\bf Star graph:} $H=\kl_{1,R-1}$, for any $R\geq 2$. Recall the definition of the sets $C_{l,j}(H)$, $l=1,2$, $j=1,\ldots,R$ (cf. \eqref{Clj-def}). In this case, we obtain asymptotic expressions,
\begin{align}
\label{var-star-1}
\sigma^2_{N,1}(\kl_{1,R-1})
&\asymp
\begin{cases}
N^{2R-1} p_N^{2R-1}(1-p_N)
\displaystyle\sum_{a,b=1}^R
\sum_{\ub,\vb\in [K]^R}
\mathbf{1}(u_a = v_b) \\[0.6em]
\qquad\cdot
\left(
\prod_{j=2}^R \pi_{u_1,u_j}
\prod_{j=2}^R \pi_{v_1,v_j}
\right)
\prod_{j=1}^R \la_{u_j}
\prod_{\substack{k=1\\ k\neq b}}^R \la_{v_k},
\\[1.2em]
\hfill\text{when}\quad (\al,\beta)\in C_{1,1}(\kl_{1,R-1}),
\\[2ex]
N^{R} p_N^{R}(1-p_N)^{R} \, R
\displaystyle\sum_{\ub\in [K]^R}
\left(
\prod_{j=2}^R \pi_{u_1,u_j}
\right)
\prod_{j=1}^R \la_{u_j},
\\[1.2em]
\hfill \text{when}\quad (\al,\beta)\in C_{1,R}(\kl_{1,R-1}).
\end{cases}\\
\sigma^2_{N,2}(\kl_{1,R-1})
&\asymp
\begin{cases}
N^{2R-1} p_N
\displaystyle\sum_{a,b=1}^R
\sum_{\ub,\vb\in [K]^R}
\mathbf{1}(u_a = v_b) \\[0.6em]
\qquad\cdot
\left(
\prod_{j=2}^R \pi_{u_1,u_j}
\prod_{j=2}^R \pi_{v_1,v_j}
\right)
\prod_{j=1}^R \la_{u_j}
\prod_{\substack{k=1\\ k\neq b}}^R \la_{v_k},\quad\text{when}\quad (\al,\beta)\in C_{2,1}(\kl_{1,R-1}),
\\[2ex]
N^{R} (p_N+ (1-p_N)p_N^{R-1})(1-(p_N+ (1-p_N)p_N^{R-1})) \, (R-1)!\\
\qquad{}\qquad{}\displaystyle\sum_{\ub\in [K]^R}
\left(
\prod_{j=2}^R \pi_{u_1,u_j}
\right)
\prod_{j=1}^R \la_{u_j},\qquad\text{when}\quad (\al,\beta)\in C_{2,R}(\kl_{1,R-1}).
\end{cases}
\label{var-star-2}
\end{align}

\item {\bf Complete graph:} $H=\kl_{R}$, for any $R\geq 3$. Here,
\begin{align}
\sigma^2_{N,1}(\kl_{R})
&\asymp
\begin{cases}
N^{2R-1} p_N^{2R-1}(1-p_N)
\displaystyle\sum_{a,b=1}^R
\sum_{\ub,\vb\in [K]^R}
\mathbf{1}(u_a=v_b) \\[0.6em]
\qquad\cdot
\left(
\prod_{1\leq i<j\leq R} \pi_{u_i,u_j}
\prod_{1\leq i<j\leq R} \pi_{v_i,v_j}
\right)
\prod_{j=1}^R \la_{u_j}
\prod_{\substack{k=1\\ k\neq b}}^R \la_{v_k},
\\[1.2em]
\hfill \text{when}\quad(\al,\beta)\in C_{1,1}(\kl_{R}),
\\[2ex]
N^{R} p_N^{R}(1-p_N)^{R} \, R
\displaystyle\sum_{\ub\in [K]^R}
\left(
\prod_{1\leq i<j\leq R} \pi_{u_1,u_j}
\right)
\prod_{j=1}^R \la_{u_j},
\\[1.2em]
\hfill \text{when}\quad (\al,\beta)\in C_{1,R}(\kl_{R}).
\end{cases}
\label{var-comp-1}
\end{align}

\begin{align}
\sigma^2_{N,2}(\kl_{R})
&\asymp
\begin{cases}
N^{2R-1} p_N^{2R-3}
\displaystyle\sum_{a,b=1}^R
\sum_{\ub,\vb\in [K]^R}
\mathbf{1}(u_a=v_b) \\[0.6em]
\qquad\cdot
\left(
\prod_{1\leq i<j\leq R} \pi_{u_i,u_j}
\prod_{1\leq i<j\leq R} \pi_{v_i,v_j}
\right)
\prod_{j=1}^R \la_{u_j}
\prod_{\substack{k=1\\ k\neq b}}^R \la_{v_k},
\\[1.2em]
\hfill \text{when}\quad (\al,\beta)\in C_{2,1}(\kl_{R}),
\\[2ex]
N^{2R-1} p_N^{R-1}
\displaystyle\sum_{(\jbb,\kbb)\in J_t\cap M_t}\sum_{\xi\in\mathcal{S}_t}
\sum_{\ub,\vb\in [K]^R}
\prod_{i=1}^{R-1}\mathbf{1}(u_{j_i}=v_{k_{\xi(i)}}) \\[0.6em]
\qquad\cdot
\left(
\prod_{\{i,j\}\in E(\kl_R)} \pi_{u_i,u_j}
\prod_{\{i,j\}\in E(\kl_R)\setminus B(\kbb,R-1)} \pi_{v_i,v_j}
\right)
\prod_{j=1}^R \la_{u_j}
\prod_{k\in [R]\setminus A(\kbb)}^R \la_{v_k},
\\[1.2em]
\hfill \text{when}\quad (\al,\beta)\in C_{2,R-1}(\kl_{R}),
\\[2ex]
N^{R}  \left(R\cdot p_N^{R-1}(1-p_N) +p_N^R\right)\left(1-\left(R\cdot p_N^{R-1}(1-p_N) +p_N^R\right)\right) R!\\
\qquad{}\qquad{}\times\displaystyle\sum_{\ub\in [K]^R}
\left(
\prod_{1\leq i<j\leq R} \pi_{u_1,u_j}
\right)
\prod_{j=1}^R \la_{u_j},
\\[1.2em]
\hfill \text{when}\quad (\al,\beta)\in C_{2,R}(\kl_{R}).
\end{cases}
\label{var-comp-2}
\end{align}

\item {\bf Line graph:} $H=\mathbb{L}_{R}$. The variance expression is provided only for the induced case, since the line graph fails to satisfy Assumption \ref{assump-5}.
\begin{align}
\sigma^2_{N,1}(\mathbb{L}_R)
&\asymp
\begin{cases}
N^{2R-1} p_N^{2R-1}(1-p_N)
\displaystyle\sum_{a,b=1}^R
\sum_{\ub,\vb\in [K]^R}
\mathbf{1}(u_a = v_b) \\[0.6em]
\qquad\cdot
\left(
\prod_{j=1}^{R-1} \pi_{u_{j},u_{j+1}}
\prod_{j=1}^{R-1} \pi_{v_j,v_{j+1}}
\right)
\prod_{j=1}^R \la_{u_j}
\prod_{\substack{k=1\\ k\neq b}}^R \la_{v_k},
\\[1.2em]
\hfill \text{when}\quad (\al,\beta)\in C_{1,1}(\mathbb{L}_R),
\\[2ex]
N^{R} p_N^{R}(1-p_N)^{R} \, R
\displaystyle\sum_{\ub\in [K]^R}
\left(
\prod_{j=1}^{R-1} \pi_{u_{j},u_{j+1}}
\right)
\prod_{j=1}^R \la_{u_j},
\\[1.2em]
\hfill \text{when}\quad (\al,\beta)\in C_{1,R}(\mathbb{L}_R).
\end{cases}
\label{var-path-1}
\end{align}
\end{enumerate}

\newpage

\begin{center}
\rotatebox{90}{%
  \begin{minipage}{\textheight}
  \centering
  \begin{tabular}{cccccccc}
    \toprule
    $k_1$ &
    $k_2$ &
    $\left(\twostar{s_{1,1}}{s_{1,2}}{s_{1,3}},
            \twostar{s_{1,1}}{s_{1,2}}{s_{1,3}},
            \trianglegraph{s_{3,1}}{s_{3,2}}{s_{3,3}}\right)$ &
    $\tau\left(\kl_{1,2},\kl_{1,2},\kl_{3}\right)$ &
    $m(k_1:\kl_{1,2})$ &
    $m'(k_2:\kl_{1,2},\kl_{1,2},\kl_{3})$ &
    $e_{1,1}(k_1,k_2:\al,\beta)$ &
    $e_{2,1}(k_1,k_2:\al,\beta)$ \\
    \midrule

    1 & 1 &
    \twostar{1}{2}{3},\twostar{1}{4}{5},\trianglegraph{1}{6}{7} &
    2 & 0 & 0 &
    $7(1-\al)-7\beta$ &
    $7-2\al-7\beta$ \\

    1 & 2 &
    \twostar{1}{2}{3},\twostar{1}{4}{5},\trianglegraph{1}{2}{6} &
    2 & 0 & 1 &
    $6(1-\al)-6\beta$ &
    $6-2\al-6\beta$ \\

    1 & 3 &
    \twostar{1}{2}{3},\twostar{1}{4}{5},\trianglegraph{1}{2}{3} &
    2 & 0 & 2 &
    $5(1-\al)-5\beta$ &
    $5-2\al-5\beta$ \\

    2 & 1 &
    \twostar{1}{2}{3},\twostar{1}{2}{4},\trianglegraph{1}{5}{6} &
    2 & 1 & 0 &
    $6(1-\al)-6\beta$ &
    $6-2\al-6\beta$ \\

    2 & 2 &
    \twostar{1}{2}{3},\twostar{1}{2}{4},\trianglegraph{1}{2}{5} &
    2 & 1 & 1 &
    $5(1-\al)-5\beta$ &
    $5-2\al-5\beta$ \\

    2 & 3 &
    \twostar{1}{2}{3},\twostar{1}{2}{4},\trianglegraph{1}{2}{3} &
    2 & 1 & 2 &
    $4(1-\al)-4\beta$ &
    $4-2\al-4\beta$ \\

    3 & 1 &
    \twostar{1}{2}{3},\twostar{1}{2}{3},\trianglegraph{1}{4}{5} &
    2 & 2 & 0 &
    $5(1-\al) -5\beta$ &
    $5-2\al-5\beta$ \\

    3 & 2 &
    \twostar{1}{2}{3},\twostar{1}{2}{3},\trianglegraph{1}{2}{4} &
    2 & 2 & 1 &
    $4(1-\al)-4\beta$ &
    $4-2\al-4\beta$ \\

    3 & 3 &
    \twostar{1}{2}{3},\twostar{1}{2}{3},\trianglegraph{1}{2}{3} &
    2 & 2 & 2 &
    $3(1-\al)-3\beta$ &
    $3-2\al-3\beta$ \\

    \bottomrule
  \end{tabular}

  \captionof{table}{%
  Values of $e_{1,1}(k_1,k_2:\al,\beta)$ and
  $e_{2,1}(k_1,k_2:\al,\beta)$ for different choices of $k_1,k_2$. See Proposition \ref{prop-joint}. Quantities in column (4) and (6) are defined in \eqref{tau-uni} and \eqref{mtH-prime-def} respectively.}
  \label{tab:e-1-1}
  \end{minipage}
}
\end{center}

\begin{center}
\rotatebox{90}{%
  \begin{minipage}{\textheight}
  \centering
  \begin{tabular}{cccccccc}
    \toprule
    $k_1$ &
    $k_2$ &
    $\left(\trianglegraph{s_{1,1}}{s_{1,2}}{s_{1,3}},
            \twostar{s_{1,1}}{s_{1,2}}{s_{1,3}},
            \twostar{s_{3,1}}{s_{3,2}}{s_{3,3}}\right)$ &
    $\tau\left(\kl_{3},\kl_{1,2},\kl_{1,2}\right)$ &
    $m^\prime(k_1:\kl_{3},\kl_{1,2})$ &
    $m'(k_2:\kl_3,\kl_{1,2},\kl_{1,2})$ &
    $e_{1,2}(k_1,k_2:\al,\beta)$ &
    $e_{2,2}(k_1,k_2:\al,\beta)$ \\
    \midrule

    1 & 1 &
    \trianglegraph{1}{2}{3},\twostar{1}{4}{5},\twostar{1}{6}{7} &
    2 & 0 & 0 &
    $7(1-\al)-7\beta$ &
    $7-2\al-7\beta$ \\

    1 & 2 &
    \trianglegraph{1}{2}{3},\twostar{1}{4}{5},\twostar{1}{2}{6} &
    2 & 0 & 1 &
    $6(1-\al)-6\beta$ &
    $6-2\al-6\beta$ \\

    1 & 3 &
    \trianglegraph{1}{2}{3},\twostar{1}{4}{5},\twostar{1}{2}{3} &
    2 & 0 & 2 &
    $5(1-\al)-5\beta$ &
    $5-2\al-5\beta$ \\

    2 & 1 &
    \trianglegraph{1}{2}{3},\twostar{1}{2}{4},\twostar{1}{5}{6} &
    2 & 1 & 0 &
    $6(1-\al)-6\beta$ &
    $6-2\al-6\beta$ \\

    2 & 2 &
    \trianglegraph{1}{2}{3},\twostar{1}{2}{4},\twostar{1}{2}{5} &
    2 & 1 & 1 &
    $5(1-\al)-5\beta$ &
    $5-2\al-5\beta$ \\

    2 & 3 &
    \trianglegraph{1}{2}{3},\twostar{1}{2}{4},\twostar{1}{2}{3} &
    2 & 1 & 2 &
    $4(1-\al)-4\beta$ &
    $4-2\al-4\beta$ \\

    3 & 1 &
    \trianglegraph{1}{2}{3},\twostar{1}{2}{3},\twostar{1}{4}{5} &
    2 & 2 & 0 &
    $5(1-\al) -5\beta$ &
    $5-2\al-5\beta$ \\

    3 & 2 &
    \trianglegraph{1}{2}{3},\twostar{1}{2}{3},\twostar{1}{2}{4} &
    2 & 2 & 1 &
    $4(1-\al)-4\beta$ &
    $4-2\al-4\beta$ \\

    3 & 3 &
    \trianglegraph{1}{2}{3},\twostar{1}{2}{3},\twostar{1}{2}{3} &
    2 & 2 & 2 &
    $3(1-\al)-3\beta$ &
    $3-2\al-3\beta$ \\

    \bottomrule
  \end{tabular}

  \captionof{table}{%
  Values of $e_{1,1}(k_1,k_2:\al,\beta)$ and
  $e_{2,1}(k_1,k_2:\al,\beta)$ for different choices of $k_1,k_2$. See Proposition \ref{prop-joint}. Quantities in column (4) and (6) are defined in \eqref{tau-uni} and \eqref{mtH-prime-def} respectively.}
  \label{tab:e-2-2}
  \end{minipage}
}
\end{center}

\begin{center}
\rotatebox{90}{%
  \begin{minipage}{\textheight}
  \centering
  \begin{tabular}{cccccccc}
    \toprule
    $k_1$ &
    $k_2$ &
    $\left(\trianglegraph{s_{1,1}}{s_{1,2}}{s_{1,3}},
            \trianglegraph{s_{1,1}}{s_{1,2}}{s_{1,3}},
            \twostar{s_{3,1}}{s_{3,2}}{s_{3,3}}\right)$ &
    $\tau\left(\kl_{3},\kl_{3},\kl_{1,2}\right)$ &
    $m(k_1:\kl_3)$ &
    $m'(k_2:\kl_3,\kl_{3},\kl_{1,2})$ &
    $e_{1,2}(k_1,k_2:\al,\beta)$ &
    $e_{2,2}(k_1,k_2:\al,\beta)$ \\
    \midrule

    1 & 1 &
    \trianglegraph{1}{2}{3},\trianglegraph{1}{4}{5},\twostar{1}{6}{7} &
    3 & 0 & 0 &
    $7(1-\al)-8\beta$ &
    $7-3\al-8\beta$ \\

    1 & 2 &
    \trianglegraph{1}{2}{3},\trianglegraph{1}{4}{5},\twostar{1}{2}{6} &
    3 & 0 & 1 &
    $6(1-\al)-7\beta$ &
    $6-3\al-7\beta$ \\

    1 & 3 &
    \trianglegraph{1}{2}{3},\trianglegraph{1}{4}{5},\twostar{1}{2}{3} &
    3 & 0 & 2 &
    $5(1-\al)-6\beta$ &
    $5-3\al-6\beta$ \\

    2 & 1 &
    \trianglegraph{1}{2}{3},\trianglegraph{1}{2}{4},\twostar{1}{5}{6} &
    2 & 1 & 0 &
    $6(1-\al)-7\beta$ &
    $6-2\al-7\beta$ \\

    2 & 2 &
    \trianglegraph{1}{2}{3},\trianglegraph{1}{2}{4},\twostar{1}{2}{5} &
    2 & 1 & 1 &
    $5(1-\al)-6\beta$ &
    $5-2\al-6\beta$ \\

    2 & 3 &
    \trianglegraph{1}{2}{3},\trianglegraph{1}{2}{4},\twostar{1}{2}{3} &
    2 & 1 & 2 &
    $4(1-\al)-5\beta$ &
    $4-2\al-5\beta$ \\

    3 & 1 &
    \trianglegraph{1}{2}{3},\trianglegraph{1}{2}{3},\twostar{1}{4}{5} &
    2 & 3 & 0 &
    $5(1-\al) -5\beta$ &
    $5-2\al-5\beta$ \\

    3 & 2 &
    \trianglegraph{1}{2}{3},\trianglegraph{1}{2}{3},\twostar{1}{2}{4} &
    2 & 3 & 1 &
    $4(1-\al)-4\beta$ &
    $4-2\al-4\beta$ \\

    3 & 3 &
    \trianglegraph{1}{2}{3},\trianglegraph{1}{2}{3},\twostar{1}{2}{3} &
    2 & 3 & 2 &
    $3(1-\al)-3\beta$ &
    $3-2\al-3\beta$ \\

    \bottomrule
  \end{tabular}

  \captionof{table}{%
  Values of $e_{1,1}(k_1,k_2:\al,\beta)$ and
  $e_{2,1}(k_1,k_2:\al,\beta)$ for different choices of $k_1,k_2$. See Proposition \ref{prop-joint}. Quantities in column (4) and (6) are defined in \eqref{tau-uni} and \eqref{mtH-prime-def} respectively.}
  \label{tab:e-3-3}
  \end{minipage}
}
\end{center}

\begin{center}
\rotatebox{90}{%
  \begin{minipage}{\textheight}
  \centering
  \begin{tabular}{cccccccc}
    \toprule
    $k_1$ &
    $k_2$ &
    $\left(\twostar{s_{1,1}}{s_{1,2}}{s_{1,3}},
            \trianglegraph{s_{1,1}}{s_{1,2}}{s_{1,3}},
            \trianglegraph{s_{3,1}}{s_{3,2}}{s_{3,3}}\right)$ &
    $\tau\left(\kl_{1,2},\kl_{3},\kl_{3}\right)$ &
    $m^\prime(k_1:\kl_{1,2},\kl_3)$ &
    $m'(k_2:\kl_{1,2},\kl_{3},\kl_{3})$ &
    $e_{1,2}(k_1,k_2:\al,\beta)$ &
    $e_{2,2}(k_1,k_2:\al,\beta)$ \\
    \midrule

    1 & 1 &
    \twostar{1}{2}{3},\trianglegraph{1}{4}{5},\trianglegraph{1}{6}{7} &
    3 & 0 & 0 &
    $7(1-\al)-8\beta$ &
    $7-3\al-8\beta$ \\

    1 & 2 &
    \twostar{1}{2}{3},\trianglegraph{1}{4}{5},\trianglegraph{1}{2}{6} &
    3 & 0 & 1 &
    $6(1-\al)-7\beta$ &
    $6-3\al-7\beta$ \\

    1 & 3 &
    \twostar{1}{2}{3},\trianglegraph{1}{4}{5},\trianglegraph{1}{2}{3} &
    3 & 0 & 2 &
    $5(1-\al)-6\beta$ &
    $5-3\al-6\beta$ \\

    2 & 1 &
    \twostar{1}{2}{3},\trianglegraph{1}{2}{4},\trianglegraph{1}{5}{6} &
    3 & 1 & 0 &
    $6(1-\al)-7\beta$ &
    $6-3\al-7\beta$ \\

    2 & 2 &
    \twostar{1}{2}{3},\trianglegraph{1}{2}{4},\trianglegraph{1}{2}{5} &
    2 & 1 & 1 &
    $5(1-\al)-6\beta$ &
    $5-2\al-6\beta$ \\

    2 & 3 &
    \twostar{1}{2}{3},\trianglegraph{1}{2}{4},\trianglegraph{1}{2}{3} &
    2 & 1 & 2 &
    $4(1-\al)-5\beta$ &
    $4-2\al-5\beta$ \\

    3 & 1 &
    \twostar{1}{2}{3},\trianglegraph{1}{2}{3},\trianglegraph{1}{4}{5} &
    3 & 2 & 0 &
    $5(1-\al) -6\beta$ &
    $5-3\al-6\beta$ \\

    3 & 2 &
    \twostar{1}{2}{3},\trianglegraph{1}{2}{3},\trianglegraph{1}{2}{4} &
    2 & 2 & 1 &
    $4(1-\al)-5\beta$ &
    $4-2\al-5\beta$ \\

    3 & 3 &
    \twostar{1}{2}{3},\trianglegraph{1}{2}{3},\trianglegraph{1}{2}{3} &
    2 & 2 & 3 &
    $3(1-\al)-3\beta$ &
    $3-2\al-3\beta$ \\

    \bottomrule
  \end{tabular}

  \captionof{table}{%
  Values of $e_{1,1}(k_1,k_2:\al,\beta)$ and
  $e_{2,1}(k_1,k_2:\al,\beta)$ for different choices of $k_1,k_2$. See Proposition \ref{prop-joint}. Quantities in column (4) and (6) are defined in \eqref{tau-uni} and \eqref{mtH-prime-def} respectively.}
  \label{tab:e-4-4}
  \end{minipage}
}
\end{center}

\end{document}